%% file: book.tex
\documentclass[graybox,envcountchap,sectrefs]{svmono}

\newtheorem{thm}{Theorem}[section]

\newtheorem{lem}[thm]{Lemma}

\newtheorem{rem}[thm]{Remark}

\newtheorem{defn}[thm]{Definition}
\newtheorem{ex}[thm]{Example}

\newcommand{\R}{\mathbb{R}}

\newcommand{\degb}{\operatorname{deg_{{}_{B}}}}
\newcommand{\degls}{\operatorname{deg_{{}_{LS}}}}
\newcommand{\dist}{\operatorname{dist}}
\newcommand{\ind}{\operatorname{ind_{{}_{FP}}}}
\newcommand{\sgn}{\operatorname{sgn}}

\usepackage{amsmath}
\usepackage{amssymb}
\usepackage{graphics}

\usepackage{mathptmx}
\usepackage{helvet}
\usepackage{courier}
\usepackage{type1cm}

\usepackage{makeidx}         
\usepackage{graphicx}        
\graphicspath{{Figures/}}
\usepackage{multicol}        
\usepackage[bottom]{footmisc}

\makeindex             

\begin{document}

\author{Naseer Ahmad Asif, PhD}
\title{Existence and Multiplicity Results for Systems of Singular Boundary Value Problems}
\subtitle{-- Monograph --}
\maketitle

\frontmatter

\include{Preface}

\tableofcontents

\mainmatter

\include{Ch1}
\include{Ch2}
\include{Ch3}
\include{Ch4}
\include{Ch5}

\backmatter

\bibliographystyle{plain}

\include{references}

\end{document}

%% file: Preface.tex
\chapter*{Preface}

Singular boundary value problems (SBVPs) arise in various fields of Mathematics, Engineering and Physics such as boundary layer theory, gas dynamics, nuclear physics, nonlinear optics, etc. The present monograph is devoted to systems of SBVPs for ordinary differential equations (ODEs). It presents existence theory for a variety of problems having unbounded nonlinearities in regions where their solutions are searched for. The main focus is to establish the existence of positive solutions. The results are based on regularization and sequential procedure.

\vskip 0.5em

First chapter of this monograph describe the motivation for the study of SBVPs. It also include some available results from functional analysis and fixed point theory. The following chapters contain results from author's PhD thesis, National University of Sciences and Technology, Islamabad, Pakistan. These results provide the existence of positive solutions for a variety of systems of SBVPs having singularity with respect to independent and/or dependent variables as well as with respect to the first derivatives of dependent variables.

%% file: Ch1.tex
\pagenumbering{arabic}
\chapter[Introduction and Preliminaries]{Introduction and Preliminaries}\label{ch1}

Many problems in applied sciences are modeled by singular boundary value problems (SBVPs). For example, in the study of rotating flow \cite{hallam}, in the theory of viscous fluids \cite{callegarifiedman}, in the study of pseudoplastic fluids \cite{callegari,nachman}, in boundary
layer theory \cite{callegarinachman1,schlichting,vajravelu,wang}, the
theory of shallow membrane caps \cite{baxrob,dickey1,johnson}, in
pre-breakdown of gas discharge \cite{diekmann}, the turbulent flow
of a gas in a porous medium \cite{esteban}, can be represented by
SBVPs. Further, many mathematical models of various applications from nuclear physics, plasma physics, nonlinear optics, fluid mechanics, chemical reactor theory, predator-prey interactions \cite{blossey,farkas,marletts,turchin,zws} are systems of time dependent partial differential equations (PDEs) subject to initial and/or boundary conditions. In the investigation of stationary solutions, these models of systems of PDEs can be reduced to systems of SBVPs.

\vskip 0.5em

In the scenario of the above mentioned models of various phenomenon,
the theory of SBVPs has become much more important. In this
monograph, we present existence results for positive solutions to
various systems of BVPs for nonlinear ODEs. We provide sufficient
conditions for the existence of at least one and two solutions for the singular systems of nonlinear ODEs subject to various type of boundary conditions (BCs) both on finite and infinite domains. We use the classical tools of functional analysis including the fixed point theory and the theory of the fixed points index. The rest of this chapter is devoted to the basic study of these notions. In the following section, we present
some definitions and notions from functional analysis. Moreover, some famous fixed point results such as Schauder's fixed point theorem and the
Guo-Krasnosel'skii fixed point theorem are also included, \cite{amr,amann,Dk,Dj,fonseca,GL,HPC}.

\section{Some basic definitions and known results}\label{secfa}

\begin{defn}
A subset $\Omega$ of a Banach space $\mathcal{B}$ is said to be compact if and only if every sequence $\{x_{n}\}\subset \Omega$ has a convergent subsequence with limit in $\Omega$. Moreover, $\Omega$ is relatively compact if $\overline{\Omega}$ is compact.
\end{defn}

\begin{defn}
Let $\Omega$ be a subset of a Banach space
$\mathcal{B}$. A map $T:\Omega\rightarrow\mathcal{B}$ is compact if
$T$ maps every bounded subset of $\Omega$ into a relatively compact
subset of $\mathcal{B}$. Moreover, $T$ is completely continuous if
$T$ is continuous and compact.
\end{defn}

\begin{defn}
A nonempty subset $K$ of a Banach space
$\mathcal{B}$ is a retract of $\mathcal{B}$ if there exist a
continuous map $r:\mathcal{B}\rightarrow K$, a retraction, such that
$\left.r\right|_{K}=I_{K}$, where $I_{K}$ is identity map on $K$.
\end{defn}

\begin{defn}
Let $\mathcal{B}$ be a real Banach space. A
nonempty, closed and convex set $P\subset\mathcal{B}$ is said to be
a cone if the following axioms are satisfied:

\begin{description}
\item[$(\mathbf{P_{1}})$] $\alpha\,x\in P$ for all $x\in P$ and $\alpha\geq0$,
\item[$(\mathbf{P_{2}})$] $x,-x\in P$ implies $x=0$.
\end{description}
\end{defn}

\begin{thm}\label{arzela}
\textbf{(Arzel\`{a}-Ascoli theorem)} Let $\Omega$ be a compact
subset of $\R^{n}$. A set $\mathfrak{M}$ of continuous functions on $\Omega$ is
relatively compact in $C(\Omega)$ if and only if $\mathfrak{M}$ is a family of uniformly bounded and equicontinuous functions.
\end{thm}

\vskip 0.5cm

Now we recall the notion of degree for continuous maps. The degree of a map in finite dimensional spaces is known as the Brouwer degree. Let $\Omega$ be a bounded and open subset of a finite dimensional Banach space $(\mathbb{B},\|\cdot\|)$. Let $T_{0}:\overline{\Omega}\rightarrow\mathbb{B}$ be a $C(\overline{\Omega})\cup C^{1}(\Omega)$ map, and $S_{0}=\{x\in\Omega:J_{0}(x)=0\}$ be the set of all critical points of the map $T_{0}$, where $J_{0}(x):=\det T_{0}'(x)$ is the Jacobean of $T_{0}$ at $x$. If $y\notin T_{0}(\partial\Omega\cup S_{0})$, then the Brouwer degree is defined as
\begin{align*}
\degb(T_{0},\Omega,y)=\sum_{x\in T_{0}^{-1}(y)}\sgn J_{0}(x),
\end{align*}
which corresponds to the number of solutions of $T_{0}(x)=y$ in $\Omega$. However, if $T_{0}\in C(\overline{\Omega})\cap C^{2}(\Omega)$, $y\notin T_{0}(\partial\Omega)$ and $y\in T_{0}(S_{0})$, then  $\degb(T_{0},\Omega,y)=\degb(T_{0},\Omega,z)$, where $z\notin T_{0}(S_{0})$ such that $\|z-y\|<\dist(y,T_{0}(\partial\Omega))$. Further, if $T_{0}\in C(\overline{\Omega})$ and $y\notin T_{0}(\partial\Omega\cup S_{0})$, then $\degb(T_{0},\Omega,y)=\degb(T_{1},\Omega,y),$ where $T_{1}\in C(\overline{\Omega})\cap C^{2}(\Omega)$ satisfies $\sup_{x\in\Omega}\|T_{0}(x)-T_{1}(x)\|<\dist(y,T_{0}(\partial\Omega))$.

\vskip 0.5em

In 1934, J. Leray and J. Schauder \cite{JLJS} extended the notion of degree to infinite dimensional spaces. They proved that along with continuity some compactness condition for the map is required. However, this is more suitable for a map of the form $I-T$, where $I$ is the identity map and $T$ is a completely continuous map. For defining the Leray-Schauder degree, the following theorem \cite{fonseca} is helpful to approximate a compact map with a finite-dimensional map.

\begin{thm}\label{thapp}
Assume that $\Omega$ is an open and bounded subset of a real Banach space $(\mathcal{B},\|\cdot\|)$ and $T:\overline{\Omega}\rightarrow\mathcal{B}$ is a completely continuous map. Then, for every $\varepsilon>0$, there exist a
finite-dimensional space $\mathbb{B}$ and a continuous map
$T_{\varepsilon}:\overline{\Omega}\rightarrow\mathbb{B}$ such that
$\|T(x)-T_{\varepsilon}(x)\|<\varepsilon$ for every
$x\in\overline{\Omega}$.
\end{thm}

Let $\Omega$ be a bounded and open subset of a Banach space
$\mathcal{B}$. Let $T:\overline{\Omega}\rightarrow\mathcal{B}$ be a
completely continuous mapping and $y\notin(I-T)(\partial\Omega)$.
The Leray-Schauder degree of $I-T$ over $\Omega$ at point $y$ is
defined as
\begin{align*}
\degls(I-T,\Omega,y)=\degb(I-T_{\varepsilon},\Omega,y),
\end{align*}
where $T_{\varepsilon}$ is an approximation of $T$ in a finite dimensional space such that $\|T(x)-T_{\varepsilon}(x)\|<\varepsilon:=\dist(y,(I-T)(\partial\Omega))$.
When $T:\overline{\Omega}\rightarrow K$ be a completely continuous map
such that $0\notin(I-T)(\partial\Omega)$, where $K$ is a retract of a Banach space $\mathcal{B}$ and $\Omega$ is an open subset of $K$, then for any retraction $r:\mathcal{B}\rightarrow K$ the Leray-Schauder degree $\degls(I-T\circ r,r^{-1}(\Omega),0)$ is known as the fixed point index of the map $T$ over $\Omega$ with respect to the retract $K$ and is denoted by $\ind(T,\Omega,K)$. The following are the most significant properties of the fixed point index for completely continuous maps.

\vskip 0.5em

\begin{description}
\item[$(\mathbf{F_{1}})$]\textbf{Normalization:} For every constant map $T$ mapping $\overline{\Omega}$ into $\Omega$,
\begin{align*}\ind(T,\Omega,K)=1.\end{align*}
\item[$(\mathbf{F_{2}})$]\textbf{Additivity:} For any disjoint open subsets $\Omega_{1}$ and $\Omega_{2}$ of $\Omega$ such that $T$ has no fixed
point on $\overline{\Omega}\setminus(\Omega_{1}\cup\Omega_{2})$,
\begin{align*}\ind(T,\Omega,K)=\ind(T,\Omega_{1},K)+\ind(T,\Omega_{2},K).\end{align*}
\item[$(\mathbf{F_{3}})$]\textbf{Homotopy:} For every compact interval $[a,b]\subset\R$ and every compact map
$h:[a,b]\times\overline{\Omega}\rightarrow K$ such that $h(\tau,x)\neq x$ for $(\tau,x)\in[a,b]\times\partial\Omega$,
\begin{align*}\ind(h(\tau,\cdot),\Omega,K)\end{align*}
is well defined and is independent of $\tau\in[a,b]$.
\item[$(\mathbf{F_{4}})$]\textbf{Solution:} If $\ind(T,\Omega,K)\neq 0$, then $T$ has at least one fixed point in $\Omega$.
\item[$(\mathbf{F_{5}})$]\textbf{Permanence:} If $K_{1}$ is a retract of $K$ and $T(\overline{\Omega})\subset K_{1}$, then
\begin{align*}\ind(T,\Omega,K)=\ind(T,\Omega\cap K_{1},K_{1}).\end{align*}
\item[$(\mathbf{F_{6}})$]\textbf{Excision:} For every open set $\Omega_{1}\subset\Omega$ such that $T$ has no fixed point in
$\overline{\Omega}\setminus\Omega_{1}$, then
\begin{align*}\ind(T,\Omega,K)=\ind(T,\Omega_{1},K).\end{align*}
\end{description}

\begin{thm}\label{schauder}
\textbf{(Schauder's fixed point theorem)} Let $X$ be a nonempty,
closed, bounded and convex subset of a Banach space $\mathcal{B}$
and $T:X\rightarrow X$ be a completely continuous map. Then, $T$ has
a fixed point in $X$.
\end{thm}

\begin{proof}
For some $x_{*}\in X$, consider a map $h:[0,1]\times X\rightarrow X$ defined by
\begin{align*}h(\tau,x)=(1-\tau)x_{*}+\tau Tx.\end{align*}
Then, the conclusion follows from the homotopy property $(\mathbf{F_{3}})$ together with the normalization property $(\mathbf{F_{1}})$.
\end{proof}

\begin{lem}\label{lemindexone}
Let $\Omega$ be an open and bounded set in a real Banach space
$\mathcal{B}$, $P$ be a cone of $\mathcal{B}$, $0\in\Omega$ and
$T:\overline{\Omega}\cap P\rightarrow P$ be a completely continuous
map. Suppose $x\neq\lambda Tx$, for any $x\in\partial\Omega\cap P$,
$\lambda\in(0,1]$. Then, the fixed point index $\ind(T,\Omega\cap
P,P)=1$.
\end{lem}


\begin{lem}\label{lemindex0}
Let $\Omega$ be an open and bounded set in a real Banach space
$(\mathcal{B},\|\cdot\|)$, $P$ be a cone of $\mathcal{B}$,
$0\in\Omega$ and $T:\overline{\Omega}\cap P\rightarrow P$ be a
completely continuous map. Suppose there exist a $v\in
P\setminus\{0\}$ with $x\neq Tx+\delta v$ for every $\delta>0$ and
$x\in\partial\Omega\cap P$. Then, the fixed point index
$\ind(T,\Omega\cap P,P)=0$.
\end{lem}

\begin{proof}
Let $\mu=\sup\{\|Tx\|:x\in\Omega\cap P\}$ and
$\rho=\sup\{\|x\|:x\in\Omega\}$. Choose
$\delta_{1}>(\mu+\rho)/\|v\|$ and define a map
$h:[0,1]\times(\overline{\Omega}\cap P)\rightarrow P$ as
\begin{align*}h(\tau,x)=T(x)+\tau\delta_{1}v.\end{align*}
Then, by the homotopy property $(\mathbf{F_{3}})$, we
obtain
\begin{align*}\ind(T,\Omega\cap P,P)=\ind(T+\delta_{1}v,\Omega\cap P,P).\end{align*}
Now, if $\ind(T,\Omega\cap P,P)\neq 0$, then there exist an element
$x\in\Omega\cap P$ such that $x=Tx+\delta_{1}v$. Consequently,
\begin{align*}\|x\|=\|Tx+\delta_{1}v\|\geq\delta_{1}\|v\|-\|Tx\|\geq\delta_{1}\|v\|-\mu>\rho,\end{align*}
a contradiction. Hence, $\ind(T,\Omega\cap P,P)=0$.
\end{proof}

\begin{lem}\label{lemindexzero}
Let $\Omega$ be a bounded and open set in a real Banach space
$(\mathcal{B},\|\cdot\|)$, $P$ be a cone of $\mathcal{B}$,
$0\in\Omega$ and $T:\overline{\Omega}\cap P\rightarrow P$ be a
completely continuous map. Suppose $Tx\npreceq x$, for any
$x\in\partial\Omega\cap P$. Then, the fixed point index
$\ind(T,\Omega\cap P,P)=0$.
\end{lem}

\begin{proof}
We claim that, for some $v\in P\setminus\{0\}$, $x\neq Tx+\delta v$ for every $\delta>0$ and $x\in\partial\Omega\cap P$. Suppose there exist some $\delta_{0}>0$ and $x_{0}\in\partial\Omega\cap P$ such that
$x_{0}=Tx_{0}+\delta_{0}v$. Then, $x_{0}=Tx_{0}+\delta_{0}v\succ Tx_{0}$, a contradiction as $x_{0}$ can not be mapped toward the origin under $T$. Hence, by Lemma \ref{lemindex0}, $\ind(T,\Omega\cap P,P)=0$.
\end{proof}


\begin{thm}\label{thmguo}
\textbf{(Guo-Krasnosel'skii fixed point theorem)} Let $P$ be a
cone of a real Banach space $(\mathcal{B},\|\cdot\|)$. Let $\Omega_{1}$,
$\Omega_{2}$ be bounded and open neighborhoods of $0\in\mathcal{B}$
such that ${\Omega}_{1}\subset\Omega_{2}$. Suppose that
$T:(\overline{\Omega}_{2}\setminus\Omega_{1})\cap P\rightarrow P$ is
completely continuous such that one of the following conditions
holds:
\begin{description}
\item[$(i)$] $\|Tx\|\leq\|x\|$ for $x\in \partial\Omega_{1}\cap P$, $\|Tx\|\geq\|x\|$ for $x\in \partial\Omega_{2}\cap P$.
\item[$(ii)$] $\|Tx\|\leq\|x\|$ for $x\in \partial\Omega_{2}\cap P$, $\|Tx\|\geq\|x\|$ for $x\in \partial\Omega_{1}\cap P$.
\end{description}
Then, $T$ has a fixed point in
$(\overline{\Omega}_{2}\setminus\Omega_{1})\cap P$.
\end{thm}
\begin{proof}
Assume that $(i)$ holds. If $T$ has a fixed point on $\partial\Omega_{1}\cup\partial\Omega_{2}$ then proof is complete. Suppose $Tx\neq x$ for all $x\in\partial\Omega_{1}\cup\partial\Omega_{2}$. First of all we show that $x\neq\lambda Tx$
for $\lambda\in(0,1]$ and $x\in \partial\Omega_{1}\cap P$. Suppose,
$x_{1}=\lambda_{1}Tx_{1}$ for some $\lambda_{1}\in(0,1)$ and
$x_{1}\in \partial\Omega_{1}\cap P$. Then,
$\|x_{1}\|=\lambda_{1}\|Tx_{1}\|<\|Tx_{1}\|\leq\|x_{1}\|$, a
contradiction. Hence, by Lemma \ref{lemindexone}, the fixed point
index $\ind(T,\Omega_{1}\cap P,P)=1$.

\vskip 0.5em

We claim that there exist a $v\in P\setminus\{0\}$ with $x\neq
Tx+\delta v$ for every $\delta>0$ and $x\in \partial\Omega_{2}\cap P$.
Suppose, $x_{2}=Tx_{2}+\delta_{2}v$ for some $\delta_{2}>0$ and
$x_{2}\in \partial\Omega_{2}\cap P$. Then,
$\|x_{2}\|=\|Tx_{2}+\delta_{2}v\|>\|Tx_{2}\|\geq\|x_{2}\|$, a
contradiction. Therefore, by Lemma \ref{lemindex0}, the fixed point
index $\ind(T,\Omega_{2}\cap P,P)=0$.

\vskip 0.5em

Thus, by the additivity property of fixed point index
$(\mathbf{F_{2}})$, we obtain
\begin{align*}
\ind(T,(\Omega_{2}\setminus\Omega_{1})\cap
P,P)=\ind(T,\Omega_{2}\cap P,P)-\ind(T,\Omega_{1}\cap P,P)=0-1=-1.
\end{align*}
Thus, $T$ has a fixed point in $(\overline{\Omega}_{2}\setminus\Omega_{1})\cap P$. The proof for $(ii)$ is similar.
\end{proof}


\vskip 0.5em

For each $x\in C[0,1]\cap C^{1}(0,1]$, we write
$\|x\|=\max_{t\in[0,1]}|x(t)|$, $\|x\|_{1}=\sup_{t\in(0,1]}t|x'(t)|$ and $\|x\|_{2}=\max\{\|x\|,\|x\|_{1}\}$. Further, for each
$x\in C^{1}[0,1]$, we write $\|x\|_{3}=\max\{\|x\|,\|x'\|\}$. The following results are known \cite{agarwaloregan,liuyan1,wj,xian1,YRA,YRA1}.

\begin{lem}\label{lembanach}
$(\mathcal{E},\|\cdot\|_{2})$ is a Banach space.
\end{lem}

\begin{lem}\label{lemspace1}
If $x\in\mathcal{E}$, then $|x'(t)|\leq\frac{\|x\|_{2}}{t}$ for all
$t\in(0,1]$.
\end{lem}


\begin{lem}\label{lemcone1}
If $x\in P:=\{x\in\mathcal{E}:x(t)\geq
t\|x\|\text{ for all }t\in[0,1],\,x(1)\geq\|x\|_{1}\}$, then
$\|x\|_{2}=\|x\|$.
\end{lem}


\begin{lem}\label{lemmax1}
Let $\sigma\in C(0,1)$ and $\sigma>0$ on $(0,1)$ with
$\int_{0}^{1}\sigma(t)dt<+\infty$. Then,
\begin{align*}
t\max_{\tau\in[0,1]}\int_{0}^{1}G(\tau,s)\sigma(s)ds\leq&\int_{0}^{1}G(t,s)\sigma(s)ds\text{ for }t\in[0,1],\\
\sup_{\tau\in(0,1]}
\tau\int_{\tau}^{1}\sigma(s)ds\leq&\max_{t\in[0,1]}\int_{0}^{1}G(t,s)\sigma(s)ds,
\end{align*}
where
\begin{align*}G(t,s)=\begin{cases}s,\,& 0\leq s\leq t\leq1,\\
 t,\,& 0\leq t\leq s\leq1.\end{cases}\end{align*}
\end{lem}



\begin{lem}\label{lemcone2}
If $x\in P:=\{x\in
C^{1}[0,1]:x(t)\geq\gamma\|x\|\text{ for all }t\in[0,1],\,x(0)\geq\frac{b}{a}\|x'\|\}$,
then $x(t)\geq\gamma\rho \|x\|_{3}$ for all $t\in[0,1]$, where
$\gamma=\frac{b}{a+b}$, $\varrho=\frac{1}{\max\{1,\frac{a}{b}\}}$,
$a,b>0$.
\end{lem}



\begin{lem}\label{lemmax2}
Let $\sigma\in C(0,1)$ and $\sigma>0$ on $(0,1)$ with
$\int_{0}^{1}\sigma(t)dt<+\infty$. Then,
\begin{align*}
\gamma\max_{\tau\in[0,1]}\int_{0}^{1}G(\tau,s)\sigma(s)ds\leq&\int_{0}^{1}G(t,s)\sigma(s)ds\text{ for }t\in[0,1],\,\gamma=\frac{b}{a+b},\,a,b>0,\\
\frac{b}{a}\max_{\tau\in[0,1]}\int_{\tau}^{1}\sigma(s)ds=&\int_{0}^{1}G(0,s)\sigma(s)ds,
\end{align*}
where
\begin{align*}
G(t,s)=\frac{1}{a}\begin{cases}b+as,\,& 0\leq s\leq t\leq1,\\
b+at,\,& 0\leq t\leq s\leq1.\end{cases}
\end{align*}
\end{lem}


\begin{lem}\label{lemaconcave}
Let $x\in C^{1}[0,1]\cap C^{2}(0,1)$ satisfies $x''<0$ on $(0,1)$,
$x(0)=0$, $x'(1)=a\geq0$. Then, $x(t)\geq tx(1)$ for $t\in[0,1]$.
\end{lem}


\begin{lem}\label{tp}
The Green's function
\begin{equation}\label{GH1}
H(t,s)=
\begin{cases} \frac{t(1-s)}{1-\alpha\eta}-
\frac{\alpha t(\eta-s)}{1-\alpha\eta}-(t-s),\, &0\leq s\leq t\leq
1,\,s\leq \eta,\\
\frac{t(1-s)}{1-\alpha\eta}- \frac{\alpha t(\eta-s)}{1-\alpha\eta},
\,&0 \leq t\leq s\leq 1,\,s\leq \eta,\\
\frac{t(1-s)}{1-\alpha\eta}, \,&0\leq t\leq s\leq 1,\,s\geq \eta,\\
\frac{t(1-s)}{1-\alpha\eta}-(t-s), \,&0\leq s\leq t\leq 1,\,s\geq
\eta.
\end{cases}\end{equation}
satisfies
\begin{description}
\item[(i)]$H(t,s)\leq\mu s(1-s),\hspace{0.4cm}(t,s)\in[0,1]\times[0,1]$,
\item[(ii)]$H(t,s)\geq\nu s(1-s),\hspace{0.45cm}(t,s)\in[\eta,1]\times[0,1]$,
\item[(iii)]$H(t,s)\geq\nu t(1-t)s(1-s),\hspace{0.45cm}(t,s)\in[0,1]\times[0,1]$,
\end{description}
where
\begin{align*}
\mu:=\frac{\max\{1,\alpha\}}{1-\alpha\eta}>0,\,\nu:=\frac{\min\{1,\alpha\}\min\{\eta,1-\eta\}}{1-\alpha\eta}>0.
\end{align*}
\end{lem}


%% file: Ch2.tex
\chapter[Singular Systems of ODEs with Nonlocal BCs]{Singular Systems of Ordinary Differential Equations with Nonlocal Boundary Conditions}\label{ch2}

Existence theory for nonlinear three-point boundary value problems (BVPs) was initiated by Gupta \cite{gupta}. Since then the study of nonlinear regular multi-point BVPs has attracted the attention of many researchers; see for example, \cite{bl,kp,lb,llw1,mr,mm,wj,liuqiu,zhang} for scalar equations, and for systems of BVPs, see \cite{cheungwong,dr,kw}. Recently, the study of SBVPs has also attracted much attention. An excellent resource with an extensive bibliography was produced by Agarwal and O'Regan \cite{ao}. In Sections \ref{existenceone} and \ref{existencetwo}, we study the following systems of SBVPs
\begin{equation}\label{k203}\begin{split}
-x''(t)&=f(t,y(t)),\hspace{0.4cm}t\in(0,1),\\
-y''(t)&=g(t,x(t)),\hspace{0.4cm}t\in(0,1),\\
x(0)&=0,\,x(1)=\alpha x(\eta),\\
y(0)&=0,\,y(1)=\alpha y(\eta),
\end{split}\end{equation}
and
\begin{equation}\label{k204}\begin{split}
-x''(t)&=f(t,x(t),y(t)),\hspace{0.4cm}t\in(0,1),\\
-y''(t)&=g(t,x(t),y(t)),\hspace{0.4cm}t\in(0,1),\\
x(0)&=0,\,x(1)=\alpha x(\eta),\\
y(0)&=0,\,y(1)=\alpha y(\eta)
\end{split}\end{equation}
where $\eta\in(0,1)$, $0<\alpha\eta<1$. For the system of SBVPs \eqref{k203}, we assume that $f,g:(0,1)\times(0,\infty)\rightarrow(0,\infty)$ are continuous, $f(t,0)$ and $g(t,0)$ are not identically $0$. For the system of SBVPs \eqref{k204}, we assume that $f,g:(0,1)\times(0,\infty)\times(0,\infty)\rightarrow(0,\infty)$ are continuous. Further, both $f$ and $g$ are allowed to be singular at $t=0$, $t=1$, $x=0$ and/or $y=0$.  By singularity we mean that the nonlinearities $f$ and $g$ are allowed to be unbounded at $t=0$, $t=1$, $x=0$ and $y=0$.

\vskip 0.5em

Further, in Section \ref{sec-couple-four-point}, we present existence result
for the following coupled system of SBVPs subject to four-point coupled BCs
\begin{equation}\label{4.0.1}\begin{split}
-x''(t)&=f(t,x(t),y(t)),\hspace{0.4cm}t\in(0,1),\\
-y''(t)&=g(t,x(t),y(t)),\hspace{0.45cm}t\in(0,1),\\
x(0)&=0,\,x(1)=\alpha y(\xi),\\
y(0)&=0,\,y(1)=\beta x(\eta),
\end{split}\end{equation}
where the parameters $\alpha$, $\beta$, $\xi$, $\eta$ satisfy $\xi,\eta\in(0,1)$, $0<\alpha\beta\xi\eta<1$. We assume that the nonlinearities $f,g:(0,1)\times[0,\infty)\times[0,\infty)\rightarrow[0,\infty)$ are continuous and allowed to be singular at $t=0$ or $t=1$.


\section{Sufficient conditions for the existence of at least one positive solution}\label{existenceone}

In this section, we establish existence of positive solution to the system of SBVPs \eqref{k203}. We say $(x,y)$ is a positive solution to system of SBVPs \eqref{k203} if $(x,y)\in(C[0,1]\cap C^{2}(0,1))\times(C[0,1]\cap C^{2}(0,1))$, $x>0$ and $y>0$ on $(0,1]$, $(x,y)$ satisfies \eqref{k203}. Let $n_{0}>\max\{\frac{1}{\eta},\frac{1}{1-\eta},\frac{2-\alpha}{1-\alpha\eta}\}$
be a fixed positive integer. For each $x\in\mathcal{E}_{n}:=C[\frac{1}{n},1-\frac{1}{n}]$, we write
$\|x\|_{{}_{\mathcal{E}_{n}}}=\max\{|x(t)|:t\in[\frac{1}{n},1-\frac{1}{n}]\}$,
where $n\in\{n_{0},n_{0}+1,n_{0}+2,\cdots\}$. Clearly,
$\mathcal{E}_{n}$ with the norm $\|\cdot\|_{{}_{\mathcal{E}_{n}}}$
is a Banach space. Define a cone $K_{n}$ of $\mathcal{E}_{n}$ as
\begin{align*}
K_{n}=\{x\in\mathcal{E}_{n}:x(t)\geq0\text{ and }x''(t)\leq0\text{ for }t\in[\frac{1}{n},1-\frac{1}{n}]\}.
\end{align*}
For any real constant $r>0$, define an open neighborhood of $0\in\mathcal{E}_{n}$ of radius $r$ by
\begin{align*}
\Omega_{r}=\{x\in\mathcal{E}_{n}:\|x\|_{{}_{\mathcal{E}_{n}}}<r\}.
\end{align*}

\begin{lem}\label{lemHn}
For $z\in\mathcal{E}_{n}$, the linear BVP
\begin{equation}\label{11k}\begin{split}
-u''(t)&=z(t),\hspace{0.4cm}t\in[\frac{1}{n},1-\frac{1}{n}],\\
u(\frac{1}{n})&=0,\,u(1-\frac{1}{n})=\alpha u(\eta),
\end{split}\end{equation}
has integral representation
\begin{equation}\label{12k}
u(t)=\int_{1/n}^{1-1/n}H_{n}(t,s)z(s)ds,
\end{equation}
where the Green's function $H_{n}:[\frac{1}{n},1-\frac{1}{n}]\times
[\frac{1}{n},1-\frac{1}{n}]\rightarrow [0,\infty)$ is defined by
\begin{equation}\label{13k}H_{n}(t,s)=\begin{cases}
\frac{(t-\frac{1}{n})((1-\frac{1}{n}-s)-\alpha(\eta-s))}
{1-\frac{2}{n}+\frac{\alpha}{n}-\alpha\eta}-(t-s), &\frac{1}{n}\leq
s\leq t\leq 1-\frac{1}{n},\,s\leq \eta,\\
\frac{(t-\frac{1}{n})((1-\frac{1}{n}-s)-\alpha(\eta-s))}
{1-\frac{2}{n}+\frac{\alpha}{n}-\alpha\eta}, &\frac{1}{n}\leq t\leq
s\leq 1-\frac{1}{n},\,s\leq \eta,\\
\frac{(t-\frac{1}{n})(1-\frac{1}{n}-s)}{1-\frac{2}{n}+\frac{\alpha}{n}-\alpha\eta},
&\frac{1}{n}\leq t\leq s\leq 1-\frac{1}{n},\,s\geq \eta,\\
\frac{(t-\frac{1}{n})(1-\frac{1}{n}-s)}{1-\frac{2}{n}+\frac{\alpha}{n}-\alpha\eta}-(t-s),
&\frac{1}{n}\leq s\leq t\leq 1-\frac{1}{n},\,s\geq \eta.
\end{cases}\end{equation}
\end{lem}


\begin{lem}\label{lem2.4}
The Green's function $H_{n}$ satisfies
\begin{description}
\item[(i)]$H_{n}(t,s)\leq\mu_{n}(s-\frac{1}{n})(1-\frac{1}{n}-s),
\hspace{0.4cm}(t,s)\in[\frac{1}{n},1-\frac{1}{n}]\times[\frac{1}{n},1-\frac{1}{n}]$,
\item[(ii)]$H_{n}(t,s)\geq\nu_{n}(s-\frac{1}{n})(1-\frac{1}{n}-s),\hspace{0.45cm}
(t,s)\in[\eta,1-\frac{1}{n}]\times[\frac{1}{n},1-\frac{1}{n}]$,
\end{description}
where
\begin{align*}\mu_{n}:=\frac{\max\{1,\alpha\}}{1-\frac{2}{n}+\frac{\alpha}{n}-\alpha\eta}>0,\,\nu_{n}
:=\frac{\min\{1,\alpha\}\min\{\eta-\frac{1}{n},1-\frac{1}{n}-\eta\}}{1-\frac{2}{n}+\frac{\alpha}{n}
-\alpha\eta}>0.\end{align*}
\end{lem}

Now we consider the modified system of non-singular BVPs
\begin{equation}\label{1eq2.6}\begin{split}
-x''(t)&=f(t,\max\{y(t)+\frac{1}{n},\frac{1}{n}\}),\hspace{0.4cm}t\in[\frac{1}{n},1-\frac{1}{n}],\\
-y''(t)&=g(t,\max\{x(t)+\frac{1}{n},\frac{1}{n}\}),\hspace{0.4cm}t\in[\frac{1}{n},1-\frac{1}{n}],\\
x(\frac{1}{n})&=0,\,x(1-\frac{1}{n})=\alpha x(\eta),\\
y(\frac{1}{n})&=0,\,y(1-\frac{1}{n})=\alpha y(\eta),
\end{split}\end{equation}
which in view of Lemma \ref{lemHn}, can be expressed as a system of integral equations
\begin{equation}\begin{split}\label{eq2.7}
x(t)&=\int_{1/n}^{1-1/n}H_{n}(t,s)f(s,\max\{y(s)+\frac{1}{n},\frac{1}{n}\})ds,\hspace{0.4cm}t\in[\frac{1}{n},1-\frac{1}{n}],\\
y(t)&=\int_{1/n}^{1-1/n}H_{n}(t,s)g(s,\max\{x(s)+\frac{1}{n},\frac{1}{n}\})ds,\hspace{0.4cm}t\in[\frac{1}{n},1-\frac{1}{n}].
\end{split}\end{equation}

Thus, $(x_{n},y_{n})\in\mathcal{E}_{n}\times \mathcal{E}_{n}$ is a solution of \eqref{1eq2.6} if and only if $(x_{n},y_{n})$ is a solution of \eqref{eq2.7}. Define maps $A_{n},B_{n},T_{n}:\mathcal{E}_{n}\rightarrow
K_{n}$ by
\begin{equation}\begin{split}\label{eq2.8}
(A_{n}y)(t)&=\int_{1/n}^{1-1/n}H_{n}(t,s)f(s,\max\{y(s)+\frac{1}{n},\frac{1}{n}\})ds,\hspace{0.4cm}t\in[\frac{1}{n},1-\frac{1}{n}],\\
(B_{n}x)(t)&=\int_{1/n}^{1-1/n}H_{n}(t,s)g(s,\max\{x(s)+\frac{1}{n},\frac{1}{n}\})ds,\hspace{0.4cm}t\in[\frac{1}{n},1-\frac{1}{n}],\\
(T_{n}x)(t)&=(A_{n}(B_{n}x))(t),\hspace{4.7cm}t\in[\frac{1}{n},1-\frac{1}{n}].
\end{split}\end{equation}
If $u_{n}\in K_{n}$ is a fixed point of $T_{n}$; then the system of
BVPs \eqref{1eq2.6} has a solution $(x_{n},y_{n})$ given by
\begin{align*}
x_{n}(t)&=u_{n}(t),\hspace{1.15cm}t\in[\frac{1}{n},1-\frac{1}{n}],\\
y_{n}(t)&=(B_{n}u_{n})(t),\hspace{0.5cm}t\in[\frac{1}{n},1-\frac{1}{n}].
\end{align*}

Assume that the following holds:

\begin{description}

\item[$(\mathbf{A}_{1})$] there exist $K,L\in C((0,1),(0,\infty))$ and $F,G\in C((0,\infty),(0,\infty))$ such that
\begin{align*}f(t,u)\leq K(t)F(u),\hspace{0.4cm}g(t,u)\leq L(t)G(u),\hspace{0.4cm}t\in(0,1),\;u\in(0,\infty),\end{align*}
where
\begin{align*}a:=\int_{0}^{1}t(1-t)K(t)dt<+\infty,\hspace{0.4cm}b:=\int_{0}^{1}t(1-t)L(t)dt<+\infty.\end{align*}

\end{description}

\begin{lem}\label{lem2.5}
Assume that $(\mathbf{A}_{1})$ holds. Then the map $T_{n}:\overline{\Omega}_{r}\cap K_{n}\rightarrow K_{n}$ is completely continuous.
\end{lem}

\begin{proof}Clearly, for any $u\in K_{n}$, we have $(T_{n}u)(t)\geq0$
and $(T_{n}u)''(t)\leq0$ for $t\in[\frac{1}{n},1-\frac{1}{n}]$.
Consequently, $T_{n}u\in K_{n}$ for all $u\in K_{n}$. Now, we show
that $T_{n}:\overline{\Omega}_{r}\cap K_{n}\rightarrow K_{n}$ is
uniformly bounded and equicontinuous. We introduce
\begin{equation}\label{18k}\begin{split}
d_{n}&=b\,\mu_{n}\max_{u\in[0,r]} G(u+\frac{1}{n}),\\
\omega_{n}&=\int_{1/n}^{1-1/n}f(t,\frac{1}{n}+\int_{1/n}^{1-1/n}H_{n}(t,s)g(s,u(s)+\frac{1}{n})ds)dt.
\end{split}\end{equation}
For any $u\in\overline{\Omega}_{r}\cap K_{n}$, using \eqref{eq2.8},
$(\mathbf{A}_{1})$ and $(i)$ of Lemma \ref{lem2.4}, we have
\begin{align*}\begin{split}
&(T_{n}u)(t)=\int_{1/n}^{1-1/n}H_{n}(t,s)f(s,\frac{1}{n}+\int_{1/n}^{1-1/n}H_{n}(s,\tau)g(\tau,u(\tau)+\frac{1}{n})d\tau)ds\\
&\leq\int_{1/n}^{1-1/n}H_{n}(t,s)K(s)F(\frac{1}{n}+\int_{1/n}^{1-1/n}H_{n}(s,\tau)g(\tau,u(\tau)+\frac{1}{n})d\tau)ds\\
&\leq\mu_{n}\int_{1/n}^{1-1/n}(s-\frac{1}{n})(1-\frac{1}{n}-s)K(s)F(\frac{1}{n}+\int_{1/n}^{1-1/n}H_{n}(s,\tau)g(\tau,u(\tau)+\frac{1}{n})d\tau)ds.
\end{split}\end{align*}
But, using $(\mathbf{A}_{1})$, $(i)$ of Lemma \ref{lem2.4} and
\eqref{18k},
\begin{align*}
0&\leq\int_{1/n}^{1-1/n}H_{n}(t,s)g(s,u(s)+\frac{1}{n})ds\\
&\leq\int_{1/n}^{1-1/n}H_{n}(t,s)L(s)G(u(s)+\frac{1}{n})ds\\
&\leq\mu_{n}\int_{1/n}^{1-1/n}(s-\frac{1}{n})(1-\frac{1}{n}-s)L(s)G(u(s)+\frac{1}{n})ds\\
&\leq\mu_{n}\max_{u\in[0,r]}G(u+\frac{1}{n})\int_{1/n}^{1-1/n}(s-\frac{1}{n})(1-\frac{1}{n}-s)L(s)ds\\
&\leq b\mu_{n}\max_{u\in[0,r]}G(u+\frac{1}{n})=d_{n}.
\end{align*}
Therefore,
\begin{align*}
(T_{n}u)(t)&\leq\mu_{n}\max_{u\in[0,d_{n}]}F(u+\frac{1}{n})\int_{1/n}^{1-1/n}(s-\frac{1}{n})(1-\frac{1}{n}-s)L(s)ds\\
&\leq a\mu_{n}\max_{u\in[0,d_{n}]}F(u+\frac{1}{n}),
\end{align*}
which implies that
\begin{align*}\|T_{n}u\|_{{}_{\mathcal{E}_{n}}}\leq a\mu_{n}\max_{u\in[0,d_{n}]}F(u+\frac{1}{n}),\end{align*}
that is, $T_{n}(\overline{\Omega}_{r}\cap K_{n})$ is uniformly
bounded. To show $T_{n}(\overline{\Omega}_{r}\cap K_{n})$ is
equicontinuous, let $t_{1},t_{2}\in[\frac{1}{n},1-\frac{1}{n}]$.
Since the Green's function $H_{n}$ is uniformly continuous on
$[\frac{1}{n},1-\frac{1}{n}]\times[\frac{1}{n},1-\frac{1}{n}]$, therefore, $T_{n}(\overline{\Omega}_{r}\cap K_{n})$ is equicontinuous. By
Theorem \ref{arzela}, $T_{n}(\overline{\Omega}_{r}\cap K_{n})$ is
relatively compact. Hence, $T_{n}$ is a compact map. Now, we show that $T_{n}$ is continuous. Let $u_{m},u\in\overline{\Omega}_{r}\cap K_{n}$ such that
\begin{align*}
\|u_{m}-u\|_{{}_{\mathcal{E}_{n}}}\rightarrow 0 \text{ as } m\rightarrow
+\infty.
\end{align*}
Using \eqref{eq2.8} and $(i)$ of Lemma
\ref{lem2.4}, we have
\begin{equation*}\begin{split}
&|(T_{n}u_{m})(t)-(T_{n}u)(t)|=\\
&\left|\int_{1/n}^{1-1/n}H_{n}(t,s)\big(f(s,\frac{1}{n}+\int_{1/n}^{1-1/n}H_{n}(s,\tau)g(\tau,u_{m}(\tau)+\frac{1}{n})d\tau)\right.\\
&\left.-f(s,\frac{1}{n}+\int_{1/n}^{1-1/n}H_{n}(s,\tau)g(\tau,u(\tau)+\frac{1}{n})d\tau)\big)ds\right|\\
&\leq\mu_{n}\int_{1/n}^{1-1/n}(s-\frac{1}{n})(1-\frac{1}{n}-s)\\
&\left|f(s,\frac{1}{n}+\int_{1/n}^{1-1/n}H_{n}(s,\tau)g(\tau,u_{m}(\tau)+\frac{1}{n})d\tau)\right.\\
&\left.-f(s,\frac{1}{n}+\int_{1/n}^{1-1/n}H_{n}(s,\tau)g(\tau,u(\tau)+\frac{1}{n})d\tau)\right|ds
\end{split}\end{equation*}
Consequently,
\begin{align*}\begin{split}
&\|T_{n}u_{m}-T_{n}u\|_{{}_{\mathcal{E}_{n}}}\leq\mu_{n}\int_{1/n}^{1-1/n}(s-\frac{1}{n})(1-\frac{1}{n}-s)\left|f(s,\frac{1}{n}+\int_{1/n}^{1-1/n}H_{n}(s,\tau)\right.\\
&\left.g(\tau,u_{m}(\tau)+\frac{1}{n})d\tau)-f(s,\frac{1}{n}+\int_{1/n}^{1-1/n}H_{n}(s,\tau)g(\tau,u(\tau)+\frac{1}{n})d\tau)\right|ds.
\end{split}\end{align*}
By the Lebesgue dominated convergent theorem, it follows that
\begin{align*}
\|T_{n}u_{m}-T_{n}u\|_{{}_{\mathcal{E}_{n}}}\rightarrow0\text{ as
}m\rightarrow+\infty,
\end{align*}
that is, $T_{n}:\overline{\Omega}_{r}\cap K_{n}\rightarrow K_{n}$ is
a continuous. Hence, $T_{n}:\overline{\Omega}_{r}\cap
K_{n}\rightarrow K_{n}$ is completely continuous.
\end{proof}
Assume that
\begin{description}
\item[$(\mathbf{A}_{2})$] there exist $\alpha_{1},\alpha_{2}\in(0,\infty)$ with $\alpha_{1}\alpha_{2}\leq1$ such that
\begin{align*}\lim_{u\rightarrow\infty}\frac{F(u)}{u^{\alpha_{1}}}\rightarrow0,\hspace{0.4cm}
\lim_{u\rightarrow\infty}\frac{G(u)}{u^{\alpha_{2}}}\rightarrow
0,\end{align*}
\item[$(\mathbf{A}_{3})$] there exist $\beta_{1},\beta_{2}\in(0,\infty)$ with
$\beta_{1}\beta_{2}\geq1$ such that
\begin{align*}\liminf_{u\rightarrow 0^{+}}\min_{t\in[\eta,1]}\frac{f(t,u)}{u^{\beta_{1}}}>0,\hspace{0.4cm}
\liminf_{u\rightarrow0^{+}}\min_{t\in[\eta,1]}\frac{g(t,u)}{u^{\beta_{2}}}>0.\end{align*}
\end{description}

\begin{thm}\label{thmfirst}
Assume that $(\mathbf{A}_{1})-(\mathbf{A}_{3})$ hold. Then the system of SBVPs \eqref{k203} has a positive solution.
\end{thm}

\begin{proof}By $(\mathbf{A}_{2})$, there exist real constants $c_{1},c_{2},c_{3},c_{4}>0$  such that
\begin{equation}
\label{k226}2^{2\alpha_{1}+\alpha_{1}\alpha_{2}}ab^{\alpha_{1}}\mu_{n}^{\alpha_{1}+1}c_{1}c_{3}^{\alpha_{1}}<1,
\end{equation}
and
\begin{equation}\label{k227}F(u+\frac{1}{n})\leq c_{1}(u+\frac{1}{n})^{\alpha_{1}}+c_{2},\,G(u+\frac{1}{n})\leq
c_{3}(u+\frac{1}{n})^{\alpha_{2}}+c_{4}\text{ for }u\geq0.
\end{equation}
In view of \eqref{k226}, we choose a real constant $R>0$ such that
\begin{equation}\label{k228}\small{
R\geq\frac{a\mu_{n}c_{2}+2^{2\alpha_{1}}ab^{\alpha_{1}}\mu_{n}^{\alpha_{1}+1}c_{1}c_{4}^{\alpha_{1}}
+2^{\alpha_{1}}a\mu_{n}n^{-\alpha_{1}}c_{1}
+2^{2\alpha_{1}+\alpha_{1}\alpha_{2}}ab^{\alpha_{1}}\mu_{n}^{\alpha_{1}+1}n^{-\alpha_{1}\alpha_{2}}c_{1}c_{3}^{\alpha_{1}}}
{1-2^{2\alpha_{1}+\alpha_{1}\alpha_{2}}ab^{\alpha_{1}}\mu_{n}^{\alpha_{1}+1}c_{1}c_{3}^{\alpha_{1}}}.}
\end{equation}
For any $u\in\partial\Omega_{R}\cap K_{n}$, using \eqref{eq2.8},
$(\mathbf{A}_{1})$ and \eqref{k227}, it follows that
\begin{align*}
(T_{n}u)(t)&=\int_{1/n}^{1-1/n}H_{n}(t,s)f(s,\frac{1}{n}+\int_{1/n}^{1-1/n}H_{n}(s,\tau)g(\tau,u(\tau)+\frac{1}{n})d\tau)ds\\
&\leq\int_{1/n}^{1-1/n}H_{n}(t,s)K(s)F(\frac{1}{n}+\int_{1/n}^{1-1/n}H_{n}(s,\tau)g(\tau,u(\tau)+\frac{1}{n})d\tau)ds\\
&\leq\int_{1/n}^{1-1/n}H_{n}(t,s)K(s)(c_{1}(\frac{1}{n}+\int_{1/n}^{1-1/n}H_{n}(s,\tau)g(\tau,u(\tau)+\frac{1}{n})d\tau)^{\alpha_{1}}+c_{2})ds\\
&=c_{1}\int_{1/n}^{1-1/n}H_{n}(t,s)K(s)(\frac{1}{n}+\int_{1/n}^{1-1/n}H_{n}(s,\tau)g(\tau,u(\tau)+\frac{1}{n})d\tau)^{\alpha_{1}}ds\\
&+c_{2}\int_{1/n}^{1-1/n}H_{n}(t,s)K(s)ds.
\end{align*}
Again using $(\mathbf{A}_{1})$ and \eqref{k227}, we obtain
\begin{align*}
(T_{n}u)(t)&\leq c_{1}\int_{1/n}^{1-1/n}H_{n}(t,s)K(s)(\frac{1}{n}+\int_{1/n}^{1-1/n}H_{n}(s,\tau)L(\tau)G(u(\tau)+\frac{1}{n})d\tau)^{\alpha_{1}}ds\\
&+c_{2}\int_{1/n}^{1-1/n}H_{n}(t,s)K(s)ds\\
&\leq c_{1}\int_{1/n}^{1-1/n}H_{n}(t,s)K(s)(\frac{1}{n}+\int_{1/n}^{1-1/n}H_{n}(s,\tau)L(\tau)(c_{3}(u(\tau)+\frac{1}{n})^{\alpha_{2}}+c_{4})d\tau)^{\alpha_{1}}ds\\
&+c_{2}\int_{1/n}^{1-1/n}H_{n}(t,s)K(s)ds\\
&\leq c_{1}\int_{1/n}^{1-1/n}H_{n}(t,s)K(s)ds(\frac{1}{n}+\int_{1/n}^{1-1/n}H_{n}(s,\tau)L(\tau)d\tau(c_{3}(R+\frac{1}{n})^{\alpha_{2}}+c_{4}))^{\alpha_{1}}\\
&+c_{2}\int_{1/n}^{1-1/n}H_{n}(t,s)K(s)ds.
\end{align*}
Employing $(i)$ of Lemma \ref{lem2.4} and $(\mathbf{A}_{1})$, leads
to
\begin{align*}
(T_{n}u)(t)&\leq c_{1}\mu_{n}\int_{1/n}^{1-1/n}(s-\frac{1}{n})(1-\frac{1}{n}-s)K(s)ds(\frac{1}{n}+\mu_{n}\int_{1/n}^{1-1/n}(\tau-\frac{1}{n})(1-\frac{1}{n}-\tau)\\
&L(\tau)d\tau(c_{3}(R+\frac{1}{n})^{\alpha_{2}}+c_{4}))^{\alpha_{1}}+c_{2}\mu_{n}\int_{1/n}^{1-1/n}(s-\frac{1}{n})(1-\frac{1}{n}-s)K(s)ds\\
&\leq
a\mu_{n}c_{1}(\frac{1}{n}+b\,\mu_{n}(c_{3}(R+\frac{1}{n})^{\alpha_{2}}+c_{4}))^{\alpha_{1}}+a\mu_{n}c_{2}.
\end{align*}
But,
\begin{align*}(\frac{1}{n}+b\,\mu_{n}(c_{3}(R+\frac{1}{n})^{\alpha_{2}}+c_{4}))^{\alpha_{1}}\leq2^{\alpha_{1}}(\frac{1}{n^{\alpha_{1}}}+b^{\alpha_{1}}\,\mu_{n}^{\alpha_{1}}(c_{3}(R+\frac{1}{n})^{\alpha_{2}}+c_{4})^{\alpha_{1}}).\end{align*}
Therefore,
\begin{align*}
&(T_{n}u)(t)\leq
2^{\alpha_{1}}a\mu_{n}c_{1}(\frac{1}{n^{\alpha_{1}}}+b^{\alpha_{1}}\,\mu_{n}^{\alpha_{1}}(c_{3}(R+\frac{1}{n})^{\alpha_{2}}+c_{4})^{\alpha_{1}})+a\mu_{n}c_{2}.
\end{align*}
Also,
\begin{align*}
(c_{3}(R+\frac{1}{n})^{\alpha_{2}}+c_{4})^{\alpha_{1}}&\leq2^{\alpha_{1}}(c_{3}^{\alpha_{1}}(R+\frac{1}{n})^{\alpha_{1}\alpha_{2}}+c_{4}^{\alpha_{1}})\\
&\leq2^{\alpha_{1}}(2^{\alpha_{1}\alpha_{2}}c_{3}^{\alpha_{1}}(R^{\alpha_{1}\alpha_{2}}+\frac{1}{n^{\alpha_{1}\alpha_{2}}})+c_{4}^{\alpha_{1}}).
\end{align*}
Consequently,
\begin{align*}
(T_{n}u)(t)&\leq 2^{\alpha_{1}}a\mu_{n}c_{1}(\frac{1}{n^{\alpha_{1}}}+2^{\alpha_{1}}b^{\alpha_{1}}\,\mu_{n}^{\alpha_{1}}(2^{\alpha_{1}\alpha_{2}}c_{3}^{\alpha_{1}}(R^{\alpha_{1}\alpha_{2}}+\frac{1}{n^{\alpha_{1}\alpha_{2}}})+c_{4}^{\alpha_{1}}))+a\mu_{n}c_{2}\\
&=2^{\alpha_{1}}a\mu_{n}n^{-\alpha_{1}}c_{1}+2^{2\alpha_{1}}ab^{\alpha_{1}}\mu_{n}^{\alpha_{1}+1}c_{1}(2^{\alpha_{1}\alpha_{2}}c_{3}^{\alpha_{1}}(R^{\alpha_{1}\alpha_{2}}+\frac{1}{n^{\alpha_{1}\alpha_{2}}})+c_{4}^{\alpha_{1}})+a\mu_{n}c_{2}\\
&=2^{\alpha_{1}}a\mu_{n}n^{-\alpha_{1}}c_{1}+2^{2\alpha_{1}+\alpha_{1}\alpha_{2}}ab^{\alpha_{1}}\mu_{n}^{\alpha_{1}+1}c_{1}c_{3}^{\alpha_{1}}(R^{\alpha_{1}\alpha_{2}}+\frac{1}{n^{\alpha_{1}\alpha_{2}}})+2^{2\alpha_{1}}ab^{\alpha_{1}}\mu_{n}^{\alpha_{1}+1}c_{1}c_{4}^{\alpha_{1}}+a\mu_{n}c_{2}\\
&=2^{\alpha_{1}}a\mu_{n}n^{-\alpha_{1}}c_{1}+2^{2\alpha_{1}+\alpha_{1}\alpha_{2}}ab^{\alpha_{1}}\mu_{n}^{\alpha_{1}+1}c_{1}c_{3}^{\alpha_{1}}R^{\alpha_{1}\alpha_{2}}\\
&+2^{2\alpha_{1}+\alpha_{1}\alpha_{2}}ab^{\alpha_{1}}\mu_{n}^{\alpha_{1}+1}n^{-\alpha_{1}\alpha_{2}}c_{1}c_{3}^{\alpha_{1}}+2^{2\alpha_{1}}ab^{\alpha_{1}}\mu_{n}^{\alpha_{1}+1}c_{1}c_{4}^{\alpha_{1}}+a\mu_{n}c_{2}.
\end{align*}
Using \eqref{k228}, we obtain
\begin{equation}\label{k229}\|T_{n}u\|_{{}_{\mathcal{E}_{n}}}\leq\|u\|_{{}_{\mathcal{E}_{n}}}\text{ for all
}u\in\partial\Omega_{R}\cap K_{n}.
\end{equation}
Now, by $(\mathbf{A}_{3})$, there exist constants $c_{5},c_{6}>0$
and $\rho\in(0,R)$ such that
\begin{equation}\label{eq3.5}f(t,x)\geq
c_{5}x^{\beta_{1}},\,g(t,x)\geq c_{6}x^{\beta_{2}}\text{ for
}x\in[0,\rho],\,t\in[\eta,1].
\end{equation}
Choose
\begin{equation}\label{k2211}
r_{n}=\min\big\{\rho,\nu_{n}^{\beta_{1}+1}n^{-\beta_{1}\beta_{2}}c_{5}c_{6}^{\beta_{1}}
(\int_{\eta}^{1-1/n}(s-\frac{1}{n})(1-\frac{1}{n}-s)ds)^{\beta_{1}+1}\big\}.
\end{equation}
For any $u\in\partial\Omega_{r_{n}}\cap K_{n}$ and
$t\in[\eta,1-\frac{1}{n}]$, using \eqref{eq2.8} and \eqref{eq3.5},
we have
\begin{align*}\begin{split}
(T_{n}u)(t)&=\int_{1/n}^{1-1/n}H_{n}(t,s)f(s,\frac{1}{n}+\int_{1/n}^{1-1/n}H_{n}(s,\tau)g(\tau,u(\tau)+\frac{1}{n})d\tau)ds\\
&\geq c_{5}\int_{1/n}^{1-1/n}H_{n}(t,s)(\frac{1}{n}+\int_{1/n}^{1-1/n}H_{n}(s,\tau)g(\tau,u(\tau)+\frac{1}{n})d\tau)^{\beta_{1}}ds\\
&\geq c_{5}\int_{1/n}^{1-1/n}H_{n}(t,s)(\int_{1/n}^{1-1/n}H_{n}(s,\tau)g(\tau,u(\tau)+\frac{1}{n})d\tau)^{\beta_{1}}ds\\
&\geq
c_{5}c_{6}^{\beta_{1}}\int_{1/n}^{1-1/n}H_{n}(t,s)(\int_{1/n}^{1-1/n}H_{n}(s,\tau)(u(\tau)+\frac{1}{n})^{\beta_{2}}d\tau)^{\beta_{1}}ds\\
&\geq
n^{-\beta_{1}\beta_{2}}c_{5}c_{6}^{\beta_{1}}\int_{1/n}^{1-1/n}H_{n}(t,s)(\int_{1/n}^{1-1/n}H_{n}(s,\tau)d\tau)^{\beta_{1}}ds\\
&\geq
n^{-\beta_{1}\beta_{2}}c_{5}c_{6}^{\beta_{1}}\int_{\eta}^{1-1/n}H_{n}(t,s)(\int_{\eta}^{1-1/n}H_{n}(s,\tau)d\tau)^{\beta_{1}}ds.
\end{split}\end{align*}
Employing $(ii)$ of Lemma \ref{lem2.4}, we get
\begin{align*}\begin{split}
(T_{n}u)(t)&\geq \nu_{n}^{\beta_{1}+1}n^{-\beta_{1}\beta_{2}}c_{5}c_{6}^{\beta_{1}}\int_{\eta}^{1-1/n}(s-\frac{1}{n})(1-\frac{1}{n}-s)ds(\int_{\eta}^{1-1/n}(\tau-\frac{1}{n})(1-\frac{1}{n}-\tau)d\tau)^{\beta_{1}}\\
&=\nu_{n}^{\beta_{1}+1}n^{-\beta_{1}\beta_{2}}c_{5}c_{6}^{\beta_{1}}(\int_{\eta}^{1-1/n}(s-\frac{1}{n})(1-\frac{1}{n}-s)ds)^{\beta_{1}+1}.
\end{split}\end{align*}
Using \eqref{k2211}, we obtain
\begin{equation}\label{k2212}\|T_{n}u\|_{{}_{\mathcal{E}_{n}}}\geq \|u\|_{{}_{\mathcal{E}_{n}}} \text{ for all }
u\in\partial\Omega_{r_{n}}\cap K_{n}.\end{equation} In view of
\eqref{k229}, \eqref{k2212} and by Theorem \ref{thmguo}, $T_{n}$ has
a fixed point
$u_{n}\in(\overline{\Omega}_{R}\setminus\Omega_{r_{n}})\cap K_{n}$.
Note that
\begin{equation}\label{eq3.7a}r_{n}\leq\|u_{n}\|_{{}_{\mathcal{E}_{n}}}\leq R\end{equation}
and $r_{n}\rightarrow 0$ as $n\rightarrow \infty$. Thus, we have
exhibited a uniform bound for each $u_{n}\in\mathcal{E}_{n}$, and
$\{u_{m}\}_{m\geq n}$ is uniformly bounded on $[\frac{1}{n},
1-\frac{1}{n}]$.

\vskip 0.5em

Now, we show that $\{u_{m}\}_{m\geq n}$, is equicontinuous on
$[\frac{1}{n},1-\frac{1}{n}]$. For
$t\in[\frac{1}{n},1-\frac{1}{n}]$, consider the integral equation
\begin{equation*}\label{29k}\begin{split}
u_{m}(t)=u_{m}(\frac{1}{n})+\frac{u_{m}(1-\frac{1}{n})-\alpha
u_{m}(\eta)-(1-\alpha)u_{m}(\frac{1}{n})}{1-\frac{2}{n}+
\frac{\alpha}{n}-\alpha\eta}(t-\frac{1}{n})+\int_{1/n}^{1-1/n}H_{n}(t,s)\tilde{f}(s)ds,
\end{split}\end{equation*} where
$\tilde{f}(t)=f(t,\frac{1}{n}+\int_{1/n}^{1-1/n}H_{n}(t,s)g(s,\frac{1}{n}+u_{m}(s))ds)$.

\vskip 0.5em

Which can also be written as
\begin{align*}
&u_{m}(t)=u_{m}(\frac{1}{n})+\frac{u_{m}(1-\frac{1}{n})-\alpha u_{m}(\eta)-(1-\alpha)u_{m}(\frac{1}{n})}{1-\frac{2}{n}+\frac{\alpha}{n}-\alpha\eta}(t-\frac{1}{n})+\frac{t-\frac{1}{n}}{1-\frac{2}{n}+\frac{\alpha}{n}-\alpha\eta}\\
&\int_{1/n}^{1-1/n}(1-\frac{1}{n}-s)\tilde{f}(s)ds-\frac{\alpha(t-\frac{1}{n})}{1-\frac{2}{n}+\frac{\alpha}{n}-\alpha\eta}\int_{1/n}^{\eta}(\eta-s)\tilde{f}(s)ds-\int_{1/n}^{t}(t-s)\tilde{f}(s)ds.
\end{align*}
Differentiating with respect to $t$, we get
\begin{align*}
u_{m}'(t)&=\frac{u_{m}(1-\frac{1}{n})-\alpha u_{m}(\eta)-(1-\alpha)u_{m}(\frac{1}{n})}{1-\frac{2}{n}+\frac{\alpha}{n}-\alpha\eta}+\frac{1}{1-\frac{2}{n}+\frac{\alpha}{n}-\alpha\eta}\int_{1/n}^{1-1/n}(1-\frac{1}{n}-s)\\
&\tilde{f}(s)ds-\frac{\alpha}{1-\frac{2}{n}+\frac{\alpha}{n}-\alpha\eta}\int_{1/n}^{\eta}(\eta-s)\tilde{f}(s)ds-\int_{1/n}^{t}\tilde{f}(s)ds,
\end{align*}
which implies that
\begin{align*}
\|u_{m}'\|_{{}_{\mathcal{E}_{n}}}&\leq\frac{2(1+\alpha)R}{1-\frac{2}{n}+\frac{\alpha}{n}-\alpha\eta}+\frac{1}{1-\frac{2}{n}+\frac{\alpha}{n}-\alpha\eta}\int_{1/n}^{1-1/n}(1-\frac{1}{n}-s)\tilde{f}(s)ds\\
&+\frac{\alpha}{1-\frac{2}{n}+\frac{\alpha}{n}-\alpha\eta}\int_{1/n}^{\eta}(\eta-s)\tilde{f}(s)ds+\int_{1/n}^{1-1/n}\tilde{f}(s)ds.
\end{align*}
Hence, $\{u_{m}\}_{m\geq n}$ is equicontinuous on
$[\frac{1}{n},1-\frac{1}{n}]$.

\vskip 0.5em

For $m\geq n$, we define
\begin{align*}v_{m}(t)=\begin{cases}
u_{m}(\frac{1}{n}),&0\leq t\leq \frac{1}{n},\\
u_{m}(t),&\frac{1}{n}\leq t\leq 1-\frac{1}{n},\\
\alpha u_{m}(\eta),&1- \frac{1}{n}\leq t\leq 1.
\end{cases}\end{align*}
Since $v_{m}$ is a constant extension of $u_{m}$ to $[0,1]$, the
sequence $\{v_{m}\}$ is uniformly bounded and equicontinuous on
$[0,1]$. Thus, there exists a subsequence $\{v_{m_{k}}\}$ of
$\{v_{m}\}$ converging uniformly to $v\in C[0,1]$. We introduce the notation
\begin{align*}\begin{split}
x_{m_{k}}(t)&=v_{m_{k}}(t),\,y_{m_{k}}(t)=\int_{1/m_{k}}^{1-1/m_{k}}H_{m_{k}}(t,s)g(s,v_{m_{k}}(s)+\frac{1}{m_{k}})ds,\\
x(t)&=\lim_{m_{k}\rightarrow\infty}x_{m_{k}}(t),\,y(t)=\lim_{m_{k}\rightarrow\infty}y_{m_{k}}(t).
\end{split}\end{align*}
For $t\in [0,1]$ consider the integral equations
\begin{align*}
x_{m_{k}}(t)&=\int_{1/m_{k}}^{1-1/m_{k}}H_{m_{k}}(t,s)f(t,y_{m_{k}}(s)+\frac{1}{m_{k}})ds,\\
y_{m_{k}}(t)&=\int_{1/m_{k}}^{1-1/m_{k}}H_{m_{k}}(t,s)g(t,x_{m_{k}}(s)+\frac{1}{m_{k}})ds.
\end{align*} Letting $m_{k}\rightarrow\infty$, we have
\begin{align*}
x(t)&=\int_{0}^{1}H(t,s)f(t,y(s))ds,\hspace{0.4cm}t\in[0,1],\\
y(t)&=\int_{0}^{1}H(t,s)g(s,x(s))ds,\hspace{0.4cm}t\in[0,1].
\end{align*}
Moreover,
\begin{align*}x(0)=0,\,x(1)=\alpha x(\eta),\,y(0)=0,\,y(1)=\alpha y(\eta).\end{align*}
Hence, $(x,y)$ is a solution of the system of BVPs \eqref{k203}. Moreover, since the Green's function $H$ is positive on $(0,1)\times(0,1)$, $f(t,0)$ and $g(t,0)$ are not identically $0$, it follows that $x>0$ and $y>0$ on $(0,1]$.
\end{proof}

\begin{ex} Let
\begin{align*}
f(t,y)=\frac{1}{t(1-t)}\left(\frac{1}{y}+3y^{1/3}\right),\hspace{0.4cm}g(t,x)=\frac{1}{t(1-t)}
\left(\frac{1}{x}+4x\right)
\end{align*} and
$\alpha=2,\,\eta=\frac{1}{3}$. Choose
\begin{align*}
K(t)=L(t)=\frac{1}{t(1-t)},\hspace{0.4cm}F(y)=\frac{1}{y}+3y^{1/3},\hspace{0.4cm}G(x)=\frac{1}{x}+4x,
\end{align*}
and $\alpha_{1}=\frac{1}{2}$, $\alpha_{2}=2$, $\beta_{1}=\beta_{2}=1$. Clearly, $(\mathbf{A}_{1})-(\mathbf{A}_{3})$ are satisfied. Hence, by Theorem \ref{thmfirst}, the system of SBVPs \eqref{k203} has a positive solution.
\end{ex}

\vskip 0.5em

Assume that
\begin{description}
\item[$(\mathbf{A}_{4})$] $f(t,u)$, $G(u)$ are non-increasing with respect
to $u$ and for each fixed $n\in\{n_{0},n_{0}+1,n_{0}+2,\cdots\}$,
there exists a constant $\rho_{n}>0$ such that
\begin{align*}
f(t,\frac{1}{n}+b\,\mu_{n}G(\frac{1}{n}))\geq
\rho_{n}(\nu_{n}\int_{\eta}^{1-1/n}(s-\frac{1}{n})(1-\frac{1}{n}-s)ds)^{-1},\hspace{0.4cm}t\in[\frac{1}{n},1-\frac{1}{n}].
\end{align*}
\end{description}

\begin{thm}\label{thmsecond}
Assume that $(\mathbf{A}_{1})$, $(\mathbf{A}_{2})$ and
$(\mathbf{A}_{4})$ hold. Then the system of SBVPs \eqref{k203} has a positive solution.
\end{thm}

\begin{proof}For any $u\in \partial\Omega_{\rho_{n}}\cap K_{n}$,
using \eqref{eq2.8}, $(i)$ of Lemma \ref{lem2.4} and
$(\mathbf{A}_{1})$, we have
\begin{align*}\begin{split}
(T_{n}u)(t)&=\int_{1/n}^{1-1/n}H_{n}(t,s)f(s,\frac{1}{n}+\int_{1/n}^{1-1/n}H_{n}(s,\tau)g(\tau,u(\tau)+\frac{1}{n})d\tau)ds\\
&\geq\int_{1/n}^{1-1/n}H_{n}(t,s)f(s,\frac{1}{n}+\mu_{n}\int_{1/n}^{1-1/n}(\tau-\frac{1}{n})(1-\frac{1}{n}-\tau)g(\tau,u(\tau)+\frac{1}{n})d\tau)ds\\
&\geq\int_{1/n}^{1-1/n}H_{n}(t,s)f(s,\frac{1}{n}+\mu_{n}\int_{1/n}^{1-1/n}(\tau-\frac{1}{n})(1-\frac{1}{n}-\tau)L(\tau)G(u(\tau)+\frac{1}{n})d\tau)ds\\
&\geq\int_{1/n}^{1-1/n}H_{n}(t,s)f(s,\frac{1}{n}+\mu_{n}G(\frac{1}{n})\int_{1/n}^{1-1/n}(\tau-\frac{1}{n})(1-\frac{1}{n}-\tau)L(\tau)d\tau)ds\\
&\geq\int_{1/n}^{1-1/n}H_{n}(t,s)f(s,\frac{1}{n}+b\,\mu_{n}\,G(\frac{1}{n}))ds.
\end{split}\end{align*}
Now in view of $(\mathbf{A}_{4})$, we have
\begin{align*}\begin{split}
(T_{n}u)(t)&\geq \rho_{n}\int_{1/n}^{1-1/n}H_{n}(t,s)ds(\nu_{n}\int_{\eta}^{1-1/n}(\tau-\frac{1}{n})(1-\frac{1}{n}-\tau)d\tau)^{-1}\\
&\geq
\rho_{n}\nu_{n}\int_{\eta}^{1-1/n}(s-\frac{1}{n})(1-\frac{1}{n}-s)ds(\nu_{n}\int_{\eta}^{1-1/n}(\tau-\frac{1}{n})(1-\frac{1}{n}-\tau)d\tau)^{-1}\\
&=\rho_{n},
\end{split}\end{align*}
which implies that
\begin{equation}\label{eq3.9}\|T_{n}u\|_{{}_{\mathcal{E}_{n}}}\geq \|u\|_{{}_{\mathcal{E}_{n}}}\text{ for all } u\in \partial\Omega_{\rho_{n}}\cap K_{n}.\end{equation}
In view of $(\mathbf{A}_{2})$, we can choose $R>\rho_{n}$ such that
\eqref{k229} holds. Hence, in view of \eqref{k229}, \eqref{eq3.9}
and by Theorem \ref{thmguo}, $T_{n}$ has a fixed point $u_{n}\in
(\overline{\Omega}_{R}\setminus\Omega_{\rho_{n}})\cap K_{n}$. Now,
following the same procedure as done in Theorem \ref{thmfirst}, the system of SBVPs \eqref{k203} has a positive solution.
\end{proof}

\begin{ex} Let
\begin{align*}
f(t,y)=\frac{e^{\frac{1}{y}}}{t(1-t)},\hspace{0.4cm}g(t,x)=\frac{e^{\frac{1}{x}}}{t(1-t)}
\end{align*} and
$\alpha=2,\,\eta=\frac{1}{3}$. Choose
\begin{align*}
K(t)=L(t)=\frac{1}{t(1-t)},\hspace{0.4cm}F(y)=e^{\frac{1}{y}},\hspace{0.4cm}G(x)=e^{\frac{1}{x}},
\end{align*}
$\rho_{n}\leq\frac{4(n-3)}{n}e^{\frac{n}{1+6ne^{n}}}\int_{1/3}^{1-1/n}(s-1/n)(1-1/n-s)ds$.
Then $(\mathbf{A}_{1}),$ $(\mathbf{A}_{2})$ and
$(\mathbf{A}_{4})$ are satisfied. Hence, by Theorem
\ref{thmsecond}, the system of SBVPs \eqref{k203} has a
positive solution.
\end{ex}

Assume that
\begin{description}
\item[$(\mathbf{A}_{5})$] $F(u)$, $g(t,u)$ are non-increasing with respect
to $u$ and for each fixed $n\in\{n_{0},n_{0}+1,n_{0}+2,\cdots\}$,
there exists a constant $M>0$ such that
\begin{align*}
a\,\mu_{n}F(\nu_{n}\int_{\eta}^{1-1/n}(s-\frac{1}{n})(1-\frac{1}{n}-s)g(s,M+\frac{1}{n})ds)\leq
M.
\end{align*}
\end{description}

\begin{thm}\label{thmthird}
Assume that $(\mathbf{A}_{1})$, $(\mathbf{A}_{3})$ and
$(\mathbf{A}_{5})$ holds. Then the system of SBVPs \eqref{k203}
has a positive solution.
\end{thm}

\begin{proof}
For any $u\in\partial\Omega_{M}\cap K_{n}$, using \eqref{eq2.8},
$(\mathbf{A}_{1})$ and $(\mathbf{A}_{5})$, we obtain
\begin{align*}\begin{split}
(T_{n}u)(t)&=\int_{1/n}^{1-1/n}H_{n}(t,s)f(s,\frac{1}{n}+\int_{1/n}^{1-1/n}H_{n}(s,\tau)g(\tau,u(\tau)+\frac{1}{n})d\tau)ds\\
&\leq\int_{1/n}^{1-1/n}H_{n}(t,s)K(s)F(\frac{1}{n}+\int_{1/n}^{1-1/n}H_{n}(s,\tau)g(\tau,u(\tau)+\frac{1}{n})d\tau)ds\\
&\leq\int_{1/n}^{1-1/n}H_{n}(t,s)K(s)F(\int_{1/n}^{1-1/n}H_{n}(s,\tau)g(\tau,u(\tau)+\frac{1}{n})d\tau)ds\\
&\leq\int_{1/n}^{1-1/n}H_{n}(t,s)K(s)F(\int_{1/n}^{1-1/n}H_{n}(s,\tau)g(\tau,M+\frac{1}{n})d\tau)ds\\
&\leq\int_{1/n}^{1-1/n}H_{n}(t,s)K(s)F(\int_{\eta}^{1-1/n}H_{n}(s,\tau)g(\tau,M+\frac{1}{n})d\tau)ds.
\end{split}\end{align*}
Employing $(ii)$ of Lemma \ref{lem2.4} and $(\mathbf{A}_{5})$,
leads to
\begin{align*}
(T_{n}u)(t)\leq
F(\nu_{n}\int_{\eta}^{1-1/n}(\tau-\frac{1}{n})(1-\frac{1}{n}-\tau)g(\tau,M+\frac{1}{n})d\tau)\int_{1/n}^{1-1/n}H_{n}(t,s)K(s)ds.
\end{align*} Now, using $(i)$ of Lemma \ref{lem2.4},
$(\mathbf{A}_{1})$ and $(\mathbf{A}_{5})$, we obtain
\begin{align*}\begin{split}
(T_{n}u)(t)&\leq\mu_{n}F(\nu_{n}\int_{\eta}^{1-1/n}(\tau-\frac{1}{n})(1-\frac{1}{n}-\tau)g(\tau,M+\frac{1}{n})d\tau)\\
&\int_{1/n}^{1-1/n}(s-\frac{1}{n})(1-\frac{1}{n}-s)K(s)ds\\
&\leq a\mu_{n}F(\nu_{n}\int_{\eta}^{1-1/n}(\tau-\frac{1}{n})(1-\frac{1}{n}-\tau)g(\tau,M+\frac{1}{n})d\tau)\\
&\leq M,
\end{split}\end{align*}
which implies that
\begin{equation}\label{eq3.10}\|T_{n}u\|\leq\|u\|_{{}_{\mathcal{E}_{n}}}
\text{ for all } u\in \partial\Omega_{M}\cap K_{n}.\end{equation} By
$(\mathbf{A}_{3})$, we can choose $r_{n}\in(0,M)$ such that
\eqref{k2212} holds. Hence, in view of \eqref{eq3.10}, \eqref{k2212}
and by Theorem \ref{thmguo}, $T_{n}$ has a fixed point
$u_{n}\in(\overline{\Omega}_{M}\setminus\Omega_{r_{n}})\cap K_{n}$.
Consequently, the system of SBVPs \eqref{k203} has a positive solution.
\end{proof}

\begin{ex} Let
\begin{align*}
f(t,y)=\begin{cases}
\frac{ye^{\frac{1}{y}}}{t(1-t)},&y\leq1,\\
\frac{e}{t(1-t)},&y>1,
\end{cases}\hspace{0.4cm}g(t,x)=\begin{cases}
\frac{xe^{\frac{1}{x}}}{t(1-t)},&x\leq1,\\
\frac{e}{t(1-t)},&x>1,
\end{cases}
\end{align*}
$\alpha=2$ $\eta=\frac{1}{3}$. Choose $\beta_{1}=\beta_{2}=1$,
\begin{align*}
K(t)=L(t)=\frac{1}{t(1-t)},\hspace{0.4cm}F(y)=\begin{cases}
ye^{\frac{1}{y}},&y\leq1,\\
e,&y>1,
\end{cases}\hspace{0.4cm}G(x)=\begin{cases}
xe^{\frac{1}{x}},&x\leq1,\\
e,&x>1,
\end{cases}
\end{align*}
\begin{align*}M\geq\max\{1,\,6\;F(e(1-\frac{3}{n})\int_{1/3}^{1-1/n}\frac{(t-\frac{1}{n})(1-\frac{1}{n}-t)}{t(1-t)}ds)\}.\end{align*}
Then $(\mathbf{A}_{1}),$ $(\mathbf{A}_{3})$ and
$(\mathbf{A}_{5})$ are satisfied. Hence, by Theorem \ref{thmthird}, the system of SBVPs \eqref{k203} has a positive solution.
\end{ex}

\begin{thm}\label{thmfourth}
Assume that $(\mathbf{A}_{1})$, $(\mathbf{A}_{4})$ and
$(\mathbf{A}_{5})$ hold. Then the system of SBVPs \eqref{k203} has a positive solution.
\end{thm}

\begin{proof}By $(\mathbf{A}_{1})$ and $(\mathbf{A}_{4})$, we obtain \eqref{eq3.9}. By $(\mathbf{A}_{5})$ we
can choose a constant $M>\rho_{n}$ such that \eqref{eq3.10} holds.
Then by Theorem \ref{thmguo}, $T_{n}$ has a fixed point
$u_{n}\in(\overline{\Omega}_{M}\setminus\Omega_{\rho_{n}})\cap
K_{n}$. Consequently, the system of SBVPs \eqref{k203} has a positive solution.
\end{proof}

\begin{ex} Let
\begin{align*}
f(t,y)=\frac{1}{t(1-t)}\frac{1}{\sqrt{y}},\hspace{0.4cm}g(t,x)=\frac{1}{t(1-t)}\frac{1}{x^{2}}
\end{align*} and
$\alpha=2,\,\eta=\frac{1}{3}$. Choose
\begin{align*}
K(t)=L(t)=\frac{1}{t(1-t)},\hspace{0.4cm}F(y)=\frac{1}{\sqrt{y}},\;\;\;\;G(x)=\frac{1}{x^{2}}.
\end{align*}
Choose constants $\rho_{n}$ and $M$ such that
\begin{align*}
\rho_{n}&\leq\frac{4(n-3)}{\sqrt{n(6n^{3}+1)}}\int_{1/3}^{1-1/n}(t-\frac{1}{n})(1-\frac{1}{n}-t)dt,\\
M&\geq\frac{1}{n}\big(\frac{1}{6}\big((1-\frac{3}{n})\int_{1/3}^{1-1/n}\frac{(t-\frac{1}{n})(1-\frac{1}{n}-t)}{t(1-t)}dt\big)^{1/2}-1\big)^{-1}.
\end{align*}
Then $(\mathbf{A}_{1}),$ $(\mathbf{A}_{4})$ and
$(\mathbf{A}_{5})$ are satisfied. Hence, by Theorem
\ref{thmfourth}, the system of BVPs \eqref{k203} has a
positive solution.
\end{ex}

\section{Sufficient conditions for the existence of at least one positive solution to a more general singular systems}\label{existencetwo}

In this section, we establish the existence of positive solution for the system of SBVPs \eqref{k204}. We say $(x,y)$ is a positive solution of the system of SBVPs \eqref{k204} if $(x,y)\in(C[0,1]\cap C^{2}(0,1))\times(C[0,1]\cap C^{2}(0,1))$, $x>0$ and $y>0$ on $(0,1]$, $(x,y)$ satisfies \eqref{k204}. For $x\in C[0,1]$, we write $\|x\|=\max_{t\in[0,1]}|x(t)|$. For any real constant $r>0$, we define an open neighborhood of $0\in C[0,1]$ as \begin{align*}\Omega_{r}=\{x\in C[0,1]:\|x\|<r\}.\end{align*} Define a cone $K$ of $C[0,1]$ as
\begin{align*}
K=\{x\in C[0,1]:x(t)\geq\,t(1-t)\,\gamma\,\|x\|\text{ for }t\in[0,1]\},
\end{align*}
where
\begin{align*}
0<\gamma:=\frac{\min\{1,\alpha\}\min\{\eta,1-\eta\}}{\max\{1,\alpha\}}<1.
\end{align*}
For each $(x,y)\in C[0,1]\times C[0,1]$, we write $\|(x,y)\|=\|x\|+\|y\|$. Clearly, $(C[0,1]\times C[0,1],\|\cdot\|)$ is a Banach space and $K\times K$ is a cone of $C[0,1]\times C[0,1]$.

\vskip 0.5em

For $n\in\{1,2,\cdots\}$, consider the following system of SBVPs
\begin{equation}\label{k231}\begin{split}
-x&''(t)=f(t,\max\{x(t)+\frac{1}{n},\frac{1}{n}\},\max\{y(t)+\frac{1}{n},\frac{1}{n}\}),\hspace{0.4cm}t\in[0,1],\\
-y&''(t)=g(t,\max\{x(t)+\frac{1}{n},\frac{1}{n}\},\max\{y(t)+\frac{1}{n},\frac{1}{n}\}),\hspace{0.4cm}t\in[0,1],\\
x&(0)=y(0)=0,\,x(1)=\alpha x(\eta),\,y(1)=\alpha y(\eta).
\end{split}\end{equation}
The system of BVPs \eqref{k231} can be expressed as an equivalent system of integral equations
\begin{equation}\begin{split}\label{k232}
&x(t)=\int_{0}^{1}H(t,s)f(s,\max\{x(s)+\frac{1}{n},\frac{1}{n}\},\max\{y(s)+\frac{1}{n},\frac{1}{n}\})ds,\\&
y(t)=\int_{0}^{10}H(t,s)g(s,\max\{x(s)+\frac{1}{n},\frac{1}{n}\},\max\{y(s)+\frac{1}{n},\frac{1}{n}\})ds,
\end{split}\end{equation}
where the Green's function $H$ is represented by \eqref{GH1}. By a solution of the system of BVPs \eqref{k231}, we mean a solution of the corresponding system of integral equations \eqref{k232}. Define a map
$T_{n}:C[0,1]\times C[0,1]\rightarrow K\times K$ by
\begin{equation}\label{k233}
T_{n}(x,y)=(A_{n}(x,y),B_{n}(x,y)),
\end{equation}
where the maps  $A_{n},B_{n}:C[0,1]\times C[0,1]\rightarrow K$ are defined by
\begin{equation}\label{eq2.10}\begin{split}
&A_{n}(x,y)(t)=\int_{0}^{1}H(t,s)f(s,\max\{x(s)+\frac{1}{n},\frac{1}{n}\},\max\{y(s)+\frac{1}{n},\frac{1}{n}\})ds,\\
&B_{n}(x,y)(t)=\int_{0}^{1}H(t,s)g(s,\max\{x(s)+\frac{1}{n},\frac{1}{n}\},\max\{y(s)+\frac{1}{n},\frac{1}{n}\})ds.
\end{split}\end{equation}
Clearly, if $(x_{n},y_{n})\in C[0,1]\times C[0,1]$ is a fixed point of $T_{n}$, then $(x_{n},y_{n})$ also a solution of the system of BVPs \eqref{k231}.

\vskip 0.5em

Assume that

\begin{description}

\item[$(\mathbf{A}_{6})$] for each $t\in(0,1)$, $f(t,x,y)$ and $g(t,x,y)$ are non-increasing with respect to $x$ and $y$, $f(\cdot,1,1),g(\cdot,1,1)\in C((0,1),(0,\infty))$ and
\begin{align*}
a:=&\int_{0}^{1}t(1-t)f(t,t(1-t),t(1-t))dt<+\infty,\\
b:=&\int_{0}^{1}t(1-t)g(t,t(1-t),t(1-t))dt<+\infty,
\end{align*}

\item[$(\mathbf{A}_{7})$] there exist real constants $\alpha_{i}$, $\beta_{i}$ with $\alpha_{i}\leq0\leq\beta_{i}$, $i=1,2$, such that for all $t\in(0,1)$, $x,y\in(0,\infty)$,
\begin{align*}\begin{split}
c^{\,\beta_{1}}&f(t,x,y)\leq f(t,c\,x,y)\leq
c^{\alpha_{1}}f(t,x,y),\hspace{0.4cm}\text{ if }0<c\leq1,\\
c^{\alpha_{1}}&f(t,x,y)\leq f(t,c\,x,y)\leq
c^{\,\beta_{1}}f(t,x,y),\hspace{0.37cm}\text{ if }c\geq1,\\
c^{\,\beta_{2}}&f(t,x,y)\leq f(t,x,c\,y)\leq
c^{\alpha_{2}}f(t,x,y),\hspace{0.4cm}\text{ if }0<c\leq1,\\
c^{\alpha_{2}}&f(t,x,y)\leq f(t,x,c\,y)\leq
c^{\,\beta_{2}}f(t,x,y),\hspace{0.37cm}\text{ if }c\geq1;
\end{split}\end{align*}

\item[$(\mathbf{A}_{8})$] there exist real constants $\gamma_{i}$, $\rho_{i}$ with
$\gamma_{i}\leq0\leq\rho_{i}$, $i=1,2$, such that for all $t\in(0,1)$, $x,y\in(0,\infty)$,
\begin{align*}\begin{split}
c^{\,\rho_{1}}&g(t,x,y)\leq g(t,c\,x,y)\leq
c^{\gamma_{1}}g(t,x,y),\hspace{0.4cm}\text{ if }0<c\leq1,\\
c^{\gamma_{1}}&g(t,x,y)\leq g(t,c\,x,y)\leq
c^{\,\rho_{1}}g(t,x,y),\hspace{0.37cm}\text{ if }c\geq1,\\
c^{\,\rho_{2}}&g(t,x,y)\leq g(t,x,c\,y)\leq
c^{\gamma_{2}}g(t,x,y),\hspace{0.4cm}\text{ if }0<c\leq1,\\
c^{\gamma_{2}}&g(t,x,y)\leq g(t,x,c\,y)\leq
c^{\,\rho_{2}}g(t,x,y),\hspace{0.37cm}\text{ if }c\geq1.
\end{split}\end{align*}

\end{description}

\begin{lem}
Assume that $(\mathbf{A}_{6})-(\mathbf{A}_{8})$ holds. Then the map $T_{n}:(\overline{\Omega}_{r_{1}}\times\overline{\Omega}_{r_{2}})\cap(K\times K)\rightarrow K\times K$ is completely continuous.
\end{lem}

\begin{proof}Clearly, $T_{n}(x,y)\in K\times K$ for all $(x,y)\in K\times K$. Now, we show that the map $A_{n}:(\overline{\Omega}_{r_{1}}\times\overline{\Omega}_{r_{2}})\cap(K\times K)\rightarrow K$ is uniformly bounded and equicontinuous. For $(x,y)\in(\overline{\Omega}_{r_{1}}\times\overline{\Omega}_{r_{2}})\cap(K\times K)$,
using \eqref{eq2.10}, $(\mathbf{A}_{6})$, $(\mathbf{A}_{7})$ and Lemma \ref{tp}, $A_{n}((\overline{\Omega}_{r_{1}}\times\overline{\Omega}_{r_{2}})\cap(K\times K))$ is uniformly bounded. Similarly, using \eqref{eq2.10}, $(\mathbf{A}_{6})$, $(\mathbf{A}_{8})$ and Lemma \ref{tp}, we can show that $B_{n}((\overline{\Omega}_{r_{1}}\times\overline{\Omega}_{r_{2}})\cap(K\times K))$ is also uniformly bounded. Thus,
$T_{n}((\overline{\Omega}_{r_{1}}\times\overline{\Omega}_{r_{2}})\cap(K\times K))$ is
uniformly bounded. Since the Green's function $H$ is uniformly continuous on $[0,1]\times[0,1]$, therefore, $T_{n}((\overline{\Omega}_{r_{1}}\times\overline{\Omega}_{r_{2}})\cap(K\times K))$ is
equicontinuous. Thus by Theorem \ref{arzela}, it follows that $T_{n}((\overline{\Omega}_{r_{1}}\times\overline{\Omega}_{r_{2}})\cap(K\times K))$ is
relatively compact. Hence, $T_{n}$ is a compact map.

\vskip 0.5em

Now, we show that $T_{n}$ is continuous. Let $(x_{m},y_{m}),(x,y)\in
K\times K$ such that
\begin{align*}
\|(x_{m},y_{m})-(x,y)\|\rightarrow 0 \text{ as } m\rightarrow +\infty.
\end{align*} Using \eqref{eq2.10} and $(i)$ of Lemma \ref{lem2.4}, we have
\begin{align*}\begin{split}
&|A_{n}(x_{m},y_{m})(t)-A_{n}(x,y)(t)|=\left|\int_{0}^{1}H(t,s)(f(s,x_{m}(s)+\frac{1}{n},y_{m}(s)+\frac{1}{n})-f(s,x(s)+\frac{1}{n},y(s)+\frac{1}{n}))ds\right|\\
&\leq\int_{0}^{1}H(t,s)\left|f(s,x_{m}(s)+\frac{1}{n},y_{m}(s)+\frac{1}{n})-f(s,x(s)+\frac{1}{n},y(s)+\frac{1}{n})\right|ds\\
&\leq\mu\int_{0}^{1}s(1-s)\left|f(s,x_{m}(s)+\frac{1}{n},y_{m}(s)+\frac{1}{n})-f(s,x(s)+\frac{1}{n},y(s)+\frac{1}{n})\right|ds.
\end{split}\end{align*}
Consequently,
\begin{align*}\begin{split}
\|A_{n}(x_{m},y_{m})-A_{n}(x,y)\|\leq&\mu\int_{0}^{1}s(1-s)\left|f(s,x_{m}(s)+\frac{1}{n},y_{m}(s)+\frac{1}{n})-f(s,x(s)+\frac{1}{n},y(s)+\frac{1}{n})\right|ds.
\end{split}\end{align*}
By Lebesgue dominated convergence theorem, it follows that
\begin{equation}\label{eq2.14}
\|A_{n}(x_{m},y_{m})-A_{n}(x,y)\|\rightarrow 0 \text{ as } m\rightarrow +\infty.
\end{equation}
Similarly, by using \eqref{eq2.10} and $(i)$ of Lemma \ref{lem2.4},
we have
\begin{equation}\label{eq2.15}
\|B_{n}(x_{m},y_{m})-B_{n}(x,y)\|\rightarrow 0 \text{ as } m\rightarrow +\infty.
\end{equation}
From \eqref{eq2.14}, \eqref{eq2.15} and \eqref{k233}, it follows that
\begin{align*}
\|T_{n}(x_{m},y_{m})-T_{n}(x,y)\|\rightarrow 0 \text{ as } m\rightarrow +\infty,
\end{align*}
that is, $T_{n}:(\overline{\Omega}_{r_{1}}\times\overline{\Omega}_{r_{2}})\cap(K\times
K)\rightarrow K\times K$ is continuous. Hence, $T_{n}:(\overline{\Omega}_{r_{1}}\times\overline{\Omega}_{r_{2}})\cap(K\times K)\rightarrow
K\times K$ is completely continuous.
\end{proof}

\begin{thm}\label{tm232}
Assume that $(\mathbf{A}_{6})-(\mathbf{A}_{8})$ hold. Then the system of SBVPs \eqref{k204} has a positive solution.
\end{thm}

\begin{proof}Choose a real constants $R_{1}>0$ and $R_{2}>0$ such that
\begin{equation}\label{eq3.1}\begin{split}
R_{1}&\geq\max\left\{1,\left(\mu a\gamma^{\alpha_{1}+\alpha_{2}}c_{7}^{\alpha_{1}+\alpha_{2}-\beta_{1}-\beta_{2}}R_{2}^{\alpha_{2}}\right)^{\frac{1}{1-\alpha_{1}}}\right\},\\
R_{2}&\geq\max\left\{1,\left(\mu b\gamma^{\gamma_{1}+\gamma_{2}}c_{7}^{\gamma_{1}+\gamma_{2}-\rho_{1}-\rho_{2}}R_{1}^{\gamma_{1}}\right)^{\frac{1}{1-\gamma_{2}}}\right\}.
\end{split}\end{equation}
where $c_{7}\in(0,1]$ such that $c_{7}R_{1}\leq1$ and $c_{7}R_{2}\leq1$. For any $(x,y)\in\partial(\Omega_{R_{1}}\times\Omega_{R_{2}})\cap(K\times K)$, using \eqref{eq2.10}, $(\mathbf{A}_{6})$, $(\mathbf{A}_{7})$, \eqref{eq3.1} and $(i)$ of Lemma \ref{tp}, we have
\begin{align*}\begin{split}
A_{n}(x,y)(t)&=\int_{0}^{1}H(t,s)f(s,x(s)+\frac{1}{n},y(s)+\frac{1}{n})ds\\
&\leq\mu\int_{0}^{1}s(1-s)f(s,s(1-s)\gamma\|x\|,s(1-s)\gamma\|y\|)ds\\
&\leq R_{1}
\end{split}\end{align*}
which implies that
\begin{equation}\label{eq3.2}
\|A_{n}(x,y)\|\leq\|x\|\text{ for all }(x,y)\in\partial(\Omega_{R_{1}}\times\Omega_{R_{2}})\cap(K\times K).\end{equation}
Similarly, using \eqref{eq2.10}, $(\mathbf{A}_{6})$, $(\mathbf{A}_{8})$, $(i)$ of Lemma \ref{lem2.4}, \eqref{eq3.1}, we obtain
\begin{equation}\label{eq3.3}
\|B_{n}(x,y)\|\leq\|y\|\text{ for all }(x,y)\in\partial(\Omega_{R_{1}}\times\Omega_{R_{2}})\cap(K\times K).\end{equation}
From \eqref{eq3.2}, \eqref{eq3.3} and \eqref{k233}, it follows that
\begin{equation}\label{eq3.4}
\|T_{n}(x,y)\|\leq\|(x,y)\|\text{ for all
}(x,y)\in\partial(\Omega_{R_{1}}\times\Omega_{R_{2}})\cap(K\times K).\end{equation}
Choose a real constants $r_{n}\in(0,R_{1})$ and $s_{n}\in(0,R_{2})$ such that
\begin{equation}\label{k2316}\begin{split}
r_{n}&\leq c_{8}^{\beta_{1}+\beta_{2}-\alpha_{1}-\alpha_{2}}n^{-\beta_{1}-\beta_{2}}\nu\int_{\eta}^{1}s(1-s)f(s,1,1)ds\\
s_{n}&\leq c_{8}^{\rho_{1}+\rho_{2}-\gamma_{1}-\gamma_{2}}n^{-\rho_{1}-\rho_{2}}\nu\int_{\eta}^{1}s(1-s)g(s,1,1)ds
\end{split}\end{equation}
where $c_{8}\in(0,1]$ is such that $c_{8}(r_{n}+\frac{1}{n})\leq1$ and $c_{8}(s_{n}+\frac{1}{n})\leq1$. For any
$(x,y)\in\partial(\Omega_{r_{n}}\times\Omega_{s_{n}})\cap(K\times K)$, using
\eqref{eq2.10}, $(\mathbf{A}_{7})$ and Lemma \ref{tp}, we have
\begin{align*}\begin{split}
A_{n}(x,y)(t)&=\int_{0}^{1}H(t,s)f\left(s,x(s)+\frac{1}{n},y(s)+\frac{1}{n}\right)ds\\
&=\int_{0}^{1}H(t,s)f\left(s,c_{8}\frac{x(s)+\frac{1}{n}}{c_{8}},c_{8}\frac{y(s)+\frac{1}{n}}{c_{8}}\right)ds\\
&\geq c_{8}^{\beta_{1}+\beta_{2}-\alpha_{1}-\alpha_{2}}n^{-\beta_{1}-\beta_{2}}\nu\int_{\eta}^{1}s(1-s)f(s,1,1)ds\geq r_{n}
\end{split}\end{align*}
Thus, in view of \eqref{k2316}, it follows that
\begin{equation}\label{eq3.6}\|A_{n}(x,y)\|\geq\|x\|\text{ for all
}(x,y)\in\partial(\Omega_{r_{n}}\times\Omega_{s_{n}})\cap(K\times K).
\end{equation}
Similarly, using \eqref{eq2.10}, $(\mathbf{A}_{8})$ and Lemma \ref{tp}, we get
\begin{equation}\label{eq3.7}\|B_{n}(x,y)\|\geq\|y\|\text{ for all
}(x,y)\in\partial(\Omega_{r_{n}}\times\Omega_{s_{n}})\cap(K\times K).
\end{equation}
From \eqref{eq3.6} and \eqref{eq3.7}, we obtain
\begin{equation}\label{eq3.8}\|T_{n}(x,y)\|\geq\|(x,y)\|\text{ for all
}(x,y)\in\partial(\Omega_{r_{n}}\times\Omega_{s_{n}})\cap(K\times K).
\end{equation}
In view of \eqref{eq3.4}, \eqref{eq3.8} and by Theorem \ref{thmguo},
$T_{n}$ has a fixed point $(x_{n},y_{n})\in(\overline{\Omega_{R_{1}}\times\Omega_{R_{2}}}\setminus(\Omega_{r_{n}}\times\Omega_{s_{n}}))\cap(K\times
K)$. Further
\begin{equation}\label{eq3.7a}r_{n}\leq\|x_{n}\|\leq R_{1},\,\,s_{n}\leq\|y_{n}\|\leq R_{2},\end{equation}
where $r_{n},\,s_{n}\rightarrow 0$ as $n\rightarrow \infty$. Thus, $\{(x_{n},y_{n})\}_{n=1}^{\infty}$ bounded uniformly on $[0,1]$. Moreover, since the Green's function $H$ is uniformly continuous on $[0,1]\times[0,1]$, therefore $\{(x_{n},y_{n})\}_{n=1}^{\infty}$ is equicontinuous on $[0,1]$. Thus, there exists a subsequence
$\{(x_{n_{k}},y_{n_{k}})\}$ of $\{(x_{n},y_{n})\}$ converging
uniformly to $(x,y)\in C[0,1]\times C[0,1]$. Now, for $t\in [0,1]$ consider the integral equations
\begin{align*}
x_{n_{k}}(t)&=\int_{0}^{1}H(t,s)f(t,x_{n_{k}}(s)+\frac{1}{n_{k}},y_{n_{k}}(s)+\frac{1}{n_{k}})ds,\\
y_{n_{k}}(t)&=\int_{0}^{1}H(t,s)g(t,x_{n_{k}}(s)+\frac{1}{n_{k}},y_{n_{k}}(s)+\frac{1}{n_{k}})ds,
\end{align*} as $n_{k}\rightarrow\infty$, we have
\begin{align*}
x(t)&=\int_{0}^{1}H(t,s)f(t,x(s),y(s))ds,\hspace{0.4cm}t\in[0,1],\\
y(t)&=\int_{0}^{1}H(t,s)g(s,x(s),y(s))ds,\hspace{0.4cm}t\in[0,1].
\end{align*}
Moreover,
\begin{align*}x(0)=0,\,x(1)=\alpha x(\eta),\,y(0)=0,\,y(1)=\alpha y(\eta).\end{align*}
Hence, $(x,y)$ is a solution of the system of BVPs \eqref{k204}. Moreover, since $f,g:(0,1)\times(0,\infty)\times(0,\infty)\rightarrow(0,\infty)$ and the Green's function $H$ is positive on $(0,1)\times(0,1)$, it follows that $x>0$ and $y>0$ on $(0,1]$.
\end{proof}

\begin{ex} Let
\begin{align*}\begin{split}
f(t,x,y)&=\frac{1}{\sqrt[4]{t(1-t)\,x\,y}},\\
g(t,x,y)&=\frac{1}{\sqrt[4]{t(1-t)\,x\,y}}
\end{split}\end{align*}
Clearly, $f$ and $g$ satisfy assumptions $(\mathbf{A}_{6})-(\mathbf{A}_{8})$. Hence, by Theorem \ref{tm232}, the system of SBVPs \eqref{k204} has a positive solution.
\end{ex}


\section{Singular systems of ODEs with four-point coupled BCs}\label{sec-couple-four-point}

In this section, we establish the existence of positive solutions for the system of SBVPs \eqref{4.0.1} \cite{asif3}. By a positive solution to the system of SBVPs \eqref{4.0.1}, we mean that $(x,y)\in (C[0,1]\cap C^{2}(0,1))\times (C[0,1]\cap C^{2}(0,1))$, $(x,y)$ satisfies \eqref{4.0.1}, $x>0$ and $y>0$ on $(0,1]$. For each $x\in C[0,1]$ we write $\|x\|=\max_{t\in[0,1]}|x(t)|$. Let
\begin{align*}
P=\big\{x\in C[0,1]:\min_{t\in[\max\{\xi,\eta\},1]}x(t)\geq\gamma\|x\|\big\},
\end{align*}
where
\begin{align*}
0<\gamma:=\frac{\min\{1,\alpha\xi,\alpha\beta\xi,\beta\eta,\alpha\beta\eta\}\min\{\xi,\eta,1-\xi,1-\eta\}}
{\max\{1,\alpha,\beta,\alpha\beta\xi,\alpha\beta\eta\}}<1.
\end{align*}
Clearly, $(C[0,1],\|\cdot\|)$ is a Banach space and $P$ is a cone of $C[0,1]$. Similarly, for each $(x,y)\in C[0,1]\times C[0,1]$ we write $\|(x,y)\|=\|x\|+\|y\|$. Clearly, $(C[0,1]\times C[0,1],\|\cdot\|)$ is a Banach space and $P\times P$ is a cone of $C[0,1]\times C[0,1]$. For any real constant $r>0$, define $\mathcal{O}_{r}=\left\{(x,y)\in C[0,1]\times
C[0,1]:\|(x,y)\|<r\right\}$.

\begin{lem}\label{lem2.2}
Let $u,v\in C[0,1]$, then the system of BVPs
\begin{equation}\label{4.1.1}\begin{split}
-x''(t)&=u(t),\hspace{0.4cm}t\in[0,1],\\
-y''(t)&=v(t),\hspace{0.4cm}t\in[0,1],\\
x(0)&=0,\,x(1)=\alpha y(\xi),\\
y(0)&=0,\,y(1)=\beta x(\eta),
\end{split}\end{equation} has integral representation
\begin{equation}\label{4.1.2}\begin{split}
x(t)=&\int_{0}^{1}F_{\xi\eta}(t,s)u(s)ds+\int_{0}^{1}G_{\alpha\beta\xi\eta}(t,s)v(s)ds,\\
y(t)=&\int_{0}^{1}F_{\eta\xi}(t,s)v(s)ds+\int_{0}^{1}G_{\beta\alpha\eta\xi}(t,s)u(s)ds,
\end{split}\end{equation} where
\begin{equation}\label{4.1.3}
F_{\xi\eta}(t,s)=
\begin{cases}
\frac{t(1-s)}{1-\alpha\beta\xi\eta}-\frac{\alpha\beta\xi
t(\eta-s)}{1-\alpha\beta\xi\eta}-(t-s),\hspace{0.4cm}&0\leq s\leq
t\leq1,\,s\leq\eta,\\
\frac{t(1-s)}{1-\alpha\beta\xi\eta}-\frac{\alpha\beta\xi
t(\eta-s)}{1-\alpha\beta\xi\eta},\hspace{0.4cm}&0\leq t\leq
s\leq1,\,s\leq\eta,\\
\frac{t(1-s)}{1-\alpha\beta\xi\eta}-(t-s),\hspace{0.4cm}&0\leq s\leq
t\leq1,\,s\geq\eta,\\
\frac{t(1-s)}{1-\alpha\beta\xi\eta},\hspace{0.4cm}&0\leq t\leq
s\leq1,\,s\geq\eta,
\end{cases}
\end{equation}
\begin{equation}\label{4.1.4}
G_{\alpha\beta\xi\eta}(t,s)=
\begin{cases}
\frac{\alpha\xi t(1-s)}{1-\alpha\beta\xi\eta}-\frac{\alpha
t(\xi-s)}{1-\alpha\beta\xi\eta},\hspace{2.3cm}&0\leq s,t\leq1,\,s\leq\xi,\\
\frac{\alpha\xi t(1-s)}{1-\alpha\beta\xi\eta},\hspace{0.4cm}&0\leq
s,t\leq1,\,s\geq\xi.
\end{cases}
\end{equation}
\end{lem}

\begin{lem}\label{lem5.4}
The functions $F_{\xi\eta}$ and $G_{\alpha\beta\xi\eta}$ satisfies
\begin{description}
\item[$(i)$]
$F_{\xi\eta}(t,s)\leq\frac{\max\{1,\alpha\beta\xi\}}{1-\alpha\beta\xi\eta}s(1-s),\hspace{0.4cm}t,s\in[0,1]$,
\item[$(ii)$] $G_{\alpha\beta\xi\eta}(t,s)\leq\frac{\alpha}{1-\alpha\beta\xi\eta}s(1-s),\hspace{0.4cm}t,s\in[0,1]$.
\end{description}
\end{lem}

\begin{rem}\label{rem2.5}In view of Lemma \ref{lem5.4}, we have
\begin{align*}F_{\eta\xi}(t,s)&\leq\frac{\max\{1,\alpha\beta\eta\}}{1-\alpha\beta\xi\eta}s(1-s),\hspace{0.4cm}t,s\in[0,1],\\
G_{\beta\alpha\eta\xi}(t,s)&\leq\frac{\beta}{1-\alpha\beta\xi\eta}s(1-s),\hspace{1.05cm}t,s\in[0,1].
\end{align*}
\end{rem}

\vskip 3em

\begin{lem}\label{lem2.6}
The functions $F_{\xi\eta}$ and $G_{\alpha\beta\xi\eta}$ satisfies
\begin{description}
\item[$(i)$]
$F_{\xi\eta}(t,s)\geq\frac{\min\{1,\alpha\beta\xi\}\min\{\eta,1-\eta\}}{1-\alpha\beta\xi\eta}s(1-s),\hspace{0.4cm}(t,s)\in[\eta,1]\times[0,1]$,
\item[$(ii)$] $G_{\alpha\beta\xi\eta}(t,s)\geq\frac{\alpha\xi\,\min\{\xi,1-\xi\}}{1-\alpha\beta\xi\eta}s(1-s),\hspace{0.4cm}(t,s)\in[\xi,1]\times[0,1]$.
\end{description}
\end{lem}

\begin{rem}\label{rem2.7}In view of Lemma \ref{lem2.6}, we have
\begin{align*}F_{\eta\xi}(t,s)&\geq\frac{\min\{1,\alpha\beta\eta\}\min\{\xi,1-\xi\}}{1-\alpha\beta\xi\eta}s(1-s),\hspace{0.4cm}(t,s)\in[\xi,1]\times[0,1],\\
G_{\beta\alpha\eta\xi}(t,s)&\geq\frac{\beta\eta\,\min\{\eta,1-\eta\}}{1-\alpha\beta\xi\eta}s(1-s),\hspace{1.9cm}(t,s)\in[\eta,1]\times[0,1].
\end{align*}
\end{rem}

\begin{rem}\label{rem2.8}From Lemma \ref{lem5.4} and Remark \ref{rem2.5}, for $t,s\in[0,1]$, we have
\begin{align*}F_{\xi\eta}(t,s)&\leq\mu s(1-s),F_{\eta\xi}(t,s)\leq\mu s(1-s),\\
G_{\alpha\beta\xi\eta}(t,s)&\leq\mu
s(1-s),G_{\beta\alpha\eta\xi}(t,s)\leq\mu s(1-s),
\end{align*}
where
$\mu=\frac{\max\{1,\alpha,\beta,\alpha\beta\xi,\alpha\beta\eta\}}{1-\alpha\beta\xi\eta}$.
Similarly, from Lemma \ref{lem2.6} and Remark \ref{rem2.7}, for
$(t,s)\in[\max\{\xi,\eta\},1]\times[0,1]$, we have
\begin{align*}
F_{\xi\eta}(t,s)&\geq\nu s(1-s),F_{\eta\xi}(t,s)\geq\nu s(1-s),\\
G_{\alpha\beta\xi\eta}(t,s)&\geq\nu
s(1-s),G_{\beta\alpha\eta\xi}(t,s)\geq\nu s(1-s),
\end{align*} where
$\nu=\frac{\min\{1,\alpha\xi,\alpha\beta\xi,\beta\eta,\alpha\beta\eta\}\min\{\xi,\eta,1-\xi,1-\eta\}}{1-\alpha\beta\xi\eta}$.
\end{rem}


In view of Lemma \ref{lem2.2}, the system of BVPs \eqref{4.0.1} can
be expressed as
\begin{equation}\label{4.1.7}\begin{split}
x(t)&=\int_{0}^{1}F_{\xi\eta}(t,s)f(s,x(s),y(s))ds+\int_{0}^{1}G_{\alpha\beta\xi\eta}(t,s)g(s,x(s),y(s))ds,\hspace{0.4cm}t\in[0,1],\\
y(t)&=\int_{0}^{1}F_{\eta\xi}(t,s)g(s,x(s),y(s))ds+\int_{0}^{1}G_{\beta\alpha\eta\xi}(t,s)f(s,x(s),y(s))ds,\hspace{0.4cm}t\in[0,1].
\end{split}\end{equation}
By a solution of the system of BVPs \eqref{4.0.1}, we mean a solution of the corresponding system of integral equations \eqref{4.1.7}. Define a map $T:P\times P\rightarrow P\times P$ by
\begin{equation}\label{4.1.8}T(x,y)=(A(x,y),B(x,y)),\end{equation}
where the maps $A,B:P\times P\rightarrow P$ are defined by
\begin{equation}\label{4.1.9}\begin{split}
A(x,y)(t)&=\int_{0}^{1}F_{\xi\eta}(t,s)f(s,x(s),y(s))ds+\int_{0}^{1}G_{\alpha\beta\xi\eta}(t,s)g(s,x(s),y(s))ds,\hspace{0.4cm}t\in[0,1],\\
B(x,y)(t)&=\int_{0}^{1}F_{\eta\xi}(t,s)g(s,x(s),y(s))ds+\int_{0}^{1}G_{\beta\alpha\eta\xi}(t,s)f(s,x(s),y(s))ds,\hspace{0.4cm}t\in[0,1].
\end{split}\end{equation}
Clearly, if $(x,y)\in P\times P$ is a fixed point of $T$, then
$(x,y)$ is a solution of the system of BVPs \eqref{4.0.1}.

\vskip 0.5em

Assume that

\begin{description}
\item[$(\mathbf{A}_{9})$] $f(\cdot,1,1),g(\cdot,1,1)\in C((0,1),(0,\infty))$ and satisfy
\begin{align*}a:=\int_{0}^{1}t(1-t)f(t,1,1)dt< +\infty,\hspace{0.4cm}b:=\int_{0}^{1}t(1-t)g(t,1,1)dt<+\infty,\end{align*}
\item[$(\mathbf{A}_{10})$] there exist real constants $\alpha_{i},\beta_{i}$ with $0\leq\alpha_{i}\leq\beta_{i}<1$, $i=1,2$;
$\beta_{1}+\beta_{2}<1$, such that for all $t\in(0,1)$,
$x,y\in[0,\infty)$,
\begin{align*}\begin{split}
c^{\,\beta_{1}}&f(t,x,y)\leq f(t,c\,x,y)\leq c^{\alpha_{1}}f(t,x,y),\hspace{0.4cm}0<c\leq1,\\
c^{\alpha_{1}}&f(t,x,y)\leq f(t,c\,x,y)\leq c^{\,\beta_{1}}f(t,x,y),\hspace{0.4cm}c\geq1,\\
c^{\,\beta_{2}}&f(t,x,y)\leq f(t,x,c\,y)\leq c^{\alpha_{2}}f(t,x,y),\hspace{0.4cm}0<c\leq1,\\
c^{\alpha_{2}}&f(t,x,y)\leq f(t,x,c\,y)\leq
c^{\,\beta_{2}}f(t,x,y),\hspace{0.4cm}c\geq1.
\end{split}\end{align*}
\item[$(\mathbf{A}_{11})$] there exist real constants $\gamma_{i},\rho_{i}$ with
$0\leq\gamma_{i}\leq\rho_{i}<1$, $i=1,2$; $\rho_{1}+\rho_{2}<1$,
such that for all $t\in(0,1)$, $x,y\in[0,\infty)$,
\begin{align*}\begin{split}
c^{\,\rho_{1}}&g(t,x,y)\leq g(t,c\,x,y)\leq c^{\gamma_{1}}g(t,x,y),\hspace{0.4cm}0<c\leq1,\\
c^{\gamma_{1}}&g(t,x,y)\leq g(t,c\,x,y)\leq c^{\,\rho_{1}}g(t,x,y),\hspace{0.4cm}c\geq1,\\
c^{\,\rho_{2}}&g(t,x,y)\leq g(t,x,c\,y)\leq c^{\gamma_{2}}g(t,x,y),\hspace{0.4cm}0<c\leq1,\\
c^{\gamma_{2}}&g(t,x,y)\leq g(t,x,c\,y)\leq
c^{\,\rho_{2}}g(t,x,y),\hspace{0.4cm}c\geq1.
\end{split}\end{align*}
\end{description}

\begin{lem}
Assume that $(\mathbf{A}_{9})-(\mathbf{A}_{11})$ hold. Then the
map $T:\overline{\mathcal{O}}_{r}\cap(P\times P)\rightarrow
P\times P$ is completely continuous.
\end{lem}

\begin{thm}\label{thm-coupl-four-point}
Assume that $(\mathbf{A}_{9})-(\mathbf{A}_{11})$ hold. Then the
system of BVPs \eqref{4.0.1} has a positive solution.
\end{thm}

\begin{thm}\label{}
Assume that $(\mathbf{A}_{6})-(\mathbf{A}_{8})$ hold. Then the
system of BVPs \eqref{4.0.1} with singularity at $t=0$, $t=1$, $x=0$ and/or $y=0$ has a positive solution.
\end{thm} 

%% file: Ch3.tex
\chapter[Singular Systems of Second-Order ODEs with Two-Point BCs]{Singular Systems of Second-Order Two-Point Boundary Value Problems}\label{ch3}

The existence of positive solutions for a nonlinear second-order two-point BVPs has received much attention; see for example the case of regular nonlinearities, \cite{lherbe1,lherbe2,hendersonwang,yli}, and the case of singular nonlinearities, see \cite{rpagarwal,AR,chu}. However, these results are for the case when nonlinear functions are independent of the first derivative. The BVPs involving the first derivative with regular nonlinear functions can be seen in \cite{guoge,johnny,gzhang}.

\vskip 0.5em

In \cite[Section~2.4]{ao}, Agarwal and O'Regan studied the existence of at least one positive solution for the following BVP with $a=1$ and $b=0$, \begin{equation}\begin{split}\label{3.0.1}
-y''(t)&=q(t)f(t,y(t),y'(t)),\hspace{0.4cm}t\in(0,1),\\
a\,y(0)&-b\,y'(0)=y'(1)=0,
\end{split}\end{equation}
where $f:[0,1]\times[0,\infty)\times(0,\infty)\rightarrow[0,\infty)$ is continuous and is allowed to be singular at $y'=0$; $q\in C(0,1)$ and $q>0$ on $(0,1)$. The existence of multiple positive solutions for second-order BVPs has also invited attention of many authors,
\cite{johnny,jiangxu,khan2,khan,zliu,thompson,liuqiu,zhang,lishen,zhangjiang}.
B. Yan et al. \cite{YRA} have studied the existence of multiple positive solutions of the BVP \eqref{3.0.1} with $a=1$ and $b=0$. Further, they generalized these results and established the existence of at least two positive solutions for BVP \eqref{3.0.1} with $a>0$ and $b>0$, \cite{YRA1}.

\vskip 0.5em

In Sections \ref{existence-one}, \ref{multiplicity-one}, \ref{existence-two} and \ref{multiplicity-two}, we study the existence and multiplicity of positive solutions to the following coupled systems of SBVPs
\begin{equation}\label{3.0.3}\begin{split}
-x''(t)&=p(t)f(t,y(t),x'(t)),\hspace{0.4cm}t\in(0,1),\\
-y''(t)&=q(t)g(t,x(t),y'(t)),\hspace{0.4cm}t\in(0,1),\\
x(0)&=y(0)=x'(1)=y'(1)=0,
\end{split}\end{equation}
and
\begin{equation}\label{3.0.4}\begin{split}
-x''(t)&=p(t)f(t,y(t),x'(t)),\hspace{0.4cm}t\in(0,1),\\
-y''(t)&=q(t)g(t,x(t),y'(t)),\hspace{0.4cm}t\in(0,1),\\
a_{1}&x(0)-b_{1}x'(0)=x'(1)=0,\\
a_{2}&y(0)-b_{2}y'(0)=y'(1)=0,
\end{split}\end{equation}
where the functions $f,g:[0,1]\times [0,\infty)\times (0,\infty)\rightarrow [0,\infty)$ are continuous and are allowed to be singular at $x'=0$, $y'=0$. Moreover, $p,q\in C(0,1)$ and positive on $(0,1)$, and the real constants $a_{i}\ (i=1,2)>0$, $b_{i}\ (i=1,2)>0$. By singularity of $f$ and $g$, we mean that the functions $f(t,x,y)$ and $g(t,x,y)$ are allowed to be unbounded at $y=0$.

\vskip 0.5em

Agarwal and O'Regan \cite[Section~2.10]{ao} have developed the
method of upper and lower solutions for the SBVP
\begin{equation}\begin{split}\label{1.1}
-y''(t)&=q(t)f(t,y(t),y'(t)),\hspace{0.4cm}t\in(0,1),\\
y(0)&=y(1)=0,
\end{split}\end{equation}
where $f:[0,1]\times(0,\infty)\times\mathbb{R}\rightarrow\mathbb{R}$
is continuous and singular at $y=0$ and the function $q\in C(0,1)$
is positive on $(0,1)$. Further, they have presented the method of
upper and lower solutions for more general problems in
\cite{oregan}. 

\vskip 0.5em

In Section \ref{finite}, we study the existence of $C^{1}$-positive
solutions for the following system of SBVPs
\begin{equation}\label{1.6}\begin{split}
-&x''(t)=p_{1}(t)f_{1}(t,x(t),y(t),x'(t)),\hspace{0.4cm}t\in(0,1),\\
-&y''(t)=p_{2}(t)f_{2}(t,x(t),y(t),y'(t)),\hspace{0.4cm}t\in(0,1),\\
&x(0)=x(1)=y(0)=y(1)=0,
\end{split}\end{equation}
where
$f_{1},f_{2}:[0,1]\times(0,\infty)\times(0,\infty)\times\mathbb{R}\rightarrow\mathbb{R}$
are continuous. Moreover, $f_{1}$, $f_{2}$ are allowed to change
sign and may be singular at $x=0$, $y=0$. Also, $p_{1},p_{2}\in
C(0,1)$ are positive on $(0,1)$.

\vskip 0.5em

Further in Section \ref{sec-couple-two-point}, we study more general
coupled system of ODEs and prove the
existence of $C^{1}$-positive solution to the following system of
ODEs subject to two-point coupled BCs
\begin{equation}\label{4.0.2}\begin{split}
-&x''(t)=p(t)f(t,x(t),y(t),x'(t)),\hspace{0.4cm}t\in(0,1),\\
-&y''(t)=q(t)g(t,x(t),y(t),y'(t)),\hspace{0.45cm}t\in(0,1),\\
&a_{1}y(0)-b_{1}x'(0)=0,\,y'(1)=0,\\
&a_{2}x(0)-b_{2}y'(0)=0,\,x'(1)=0,
\end{split}\end{equation}
where the nonlinearities $f,g:[0,1]\times
[0,\infty)\times[0,\infty)\times(0,\infty)\rightarrow [0,\infty)$
are continuous and are allowed to be singular at $x'=0$, $y'=0$.
Moreover, $p,q\in C(0,1)$, $p>0$ and $q>0$ on $(0,1)$.


\section{Existence of $C^{1}$-positive solutions}\label{existence-one}

In this section, we establish sufficient conditions for the existence of $C^{1}$-positive solutions to the system of SBVPs \eqref{3.0.3}. By a $C^{1}$-positive solution to the system of SBVPs \eqref{3.0.3}, we mean that $(x,y)\in(C^{1}[0,1]\cap C^{2}(0,1))\times(C^{1}[0,1]\cap
C^{2}(0,1))$ satisfying \eqref{3.0.3}, $x>0$ and $y>0$ on $(0,1]$, $x'>0$ and $y'>0$ on $[0,1)$. For each $x\in C[0,1]\cap C^{1}(0,1]$, we write
$\|x\|=\max_{t\in[0,1]}|x(t)|$ and $\|x\|_{1}=\sup_{t\in(0,1]}t|x'(t)|$. Moreover, for each $x\in\mathcal{E}:=\{x\in C[0,1]\cap C^{1}(0,1]:\|x\|_{1}<+\infty\}$, we write $\|x\|_{2}=\max\{\|x\|,\|x\|_{1}\}$. By Lemma \ref{lembanach}, $(\mathcal{E},\|\cdot\|_{2})$ is a Banach space. Moreover, for each $x\in C^{1}[0,1]$, we write $\|x\|_{3}=\max\{\|x\|,\|x'\|\}$. Clearly, $(C^{1}[0,1],\|\cdot\|_{3})$ is a Banach space.

\vskip 0.5em

Assume that
\begin{description}
\item[$(\mathbf{B}_{1})$] $p,q\in C(0,1)$, $p,q>0$ on $(0,1)$, $\int_{0}^{1}p(t)dt<+\infty$ and $\int_{0}^{1}q(t)dt<+\infty$;
\item[$(\mathbf{B}_{2})$] $f,g:[0,1]\times[0,\infty)\times(0,\infty)\rightarrow[0,\infty)$ are
continuous with $f(t,x,y)>0$ and $g(t,x,y)>0$ on
$[0,1]\times(0,\infty)\times(0,\infty)$;
\item[$(\mathbf{B}_{3})$] $f(t,x,y)\leq k_{1}(x)(u_{1}(y)+v_{1}(y))$ and $g(t,x,y)\leq k_{2}(x)(u_{2}(y)+v_{2}(y))$, where $u_{i}(i=1,2)>0$ are continuous
 and nonincreasing on $(0,\infty)$, $k_{i}(i=1,2)\geq 0$, $v_{i}(i=1,2)\geq 0$ are continuous and nondecreasing on $[0,\infty)$;
\item[$(\mathbf{B}_{4})$]
\begin{align*}
\sup_{c\in(0,\infty)}\frac{c}{I^{-1}(k_{1}(J^{-1}(k_{2}(c)\int_{0}^{1}q(s)ds))\int_{0}^{1}p(s)ds)}&>1,\\
\sup_{c\in(0,\infty)}\frac{c}{J^{-1}(k_{2}(I^{-1}(k_{1}(c)\int_{0}^{1}p(s)ds))\int_{0}^{1}q(s)ds)}&>1,
\end{align*}
where $I(\mu)=\int_{0}^{\mu}\frac{d\tau}{u_{1}(\tau)+v_{1}(\tau)},\,
J(\mu)=\int_{0}^{\mu}\frac{d\tau}{u_{2}(\tau)+v_{2}(\tau)}$, for
$\mu\in(0,\infty)$;
\item[$(\mathbf{B}_{5})$] $I(\infty)=\infty$ and $J(\infty)=\infty$;
\item[$(\mathbf{B}_{6})$] for real constants $E>0$ and $F>0$, there exist continuous functions $\varphi_{EF}$ and $\psi_{EF}$ defined on
$[0,1]$ and positive on $(0,1)$, and constants
$0\leq\delta_{1},\delta_{2}<1$ such that
\begin{align*}f(t,x,y)\geq\varphi_{EF}(t)x^{\delta_{1}},\,g(t,x,y)\geq\psi_{EF}(t)x^{\delta_{2}}\text{
on }[0,1]\times[0,E]\times[0,F];\end{align*}
\item[$(\mathbf{B}_{7})$]$\int_{0}^{1}p(t)u_{1}(C\int_{t}^{1}s^{\delta_{1}}p(s)\varphi_{EF}(s)ds)dt<+\infty$ and
$\int_{0}^{1}q(t)u_{2}(C\int_{t}^{1}s^{\delta_{2}}q(s)\psi_{EF}(s)ds)dt<+\infty$
for any real constant $C>0$.
\end{description}
\begin{rem}
Since $I$, $J$ are continuous, $I(0)=0$, $I(\infty)=\infty$,
$J(0)=0$, $J(\infty)=\infty$, and they are monotonically increasing.
Hence, $I$ and $J$ are invertible. Moreover, $I^{-1}$ and $J^{-1}$
are also monotonically increasing.
\end{rem}
\begin{thm}\label{th3.2}
Assume that $(\mathbf{B}_{1})-(\mathbf{B}_{7})$ hold. Then the system of BVPs \eqref{3.0.3} has a $C^{1}$-positive solution.
\end{thm}
\begin{proof}
In view of $(\mathbf{B}_{4})$, we can choose real constants
$M_{1}>0$ and $M_{2}>0$ such that
\begin{align*}
\frac{M_{1}}{I^{-1}(k_{1}(J^{-1}(k_{2}(M_{1})\int_{0}^{1}q(s)ds))\int_{0}^{1}p(s)ds)}>1,
\end{align*}
\begin{align*}
\frac{M_{2}}{J^{-1}(k_{2}(I^{-1}(k_{1}(M_{2})\int_{0}^{1}p(s)ds))\int_{0}^{1}q(s)ds)}>1.
\end{align*}
From the continuity of $k_{1}$, $k_{2}$, $I$ and $J$, choose
$\varepsilon>0$ small enough such that
\begin{equation}\label{3.1.1}
\frac{M_{1}}{I^{-1}(k_{1}(J^{-1}(k_{2}(M_{1})\int_{0}^{1}q(s)ds+J(\varepsilon)))\int_{0}^{1}p(s)ds+I(\varepsilon))}>1,
\end{equation}
\begin{equation}\label{3.1.2}
\frac{M_{2}}{J^{-1}(k_{2}(I^{-1}(k_{1}(M_{2})\int_{0}^{1}p(s)ds+I(\varepsilon)))\int_{0}^{1}q(s)ds+J(\varepsilon))}>1.
\end{equation}
Choose real constants $L_{1}>0$ and $L_{2}>0$ such that
\begin{equation}\label{3.1.3}I(L_{1})>k_{1}(M_{2})\int_{0}^{1}p(s)ds+I(\varepsilon),\end{equation}
\begin{equation}\label{3.1.4}J(L_{2})>k_{2}(M_{1})\int_{0}^{1}q(s)ds+J(\varepsilon).\end{equation}

Choose $n_{0}\in\{1,2,\cdots\}$ such that
$\frac{1}{n_{0}}<\varepsilon$. For each fixed
$n\in\{n_{0},n_{0}+1,\cdots\}$, define retractions
$\theta_{i}:\R\rightarrow[0,M_{i}]$ and
$\rho_{i}:\R\rightarrow[\frac{1}{n},L_{i}]$ by
\begin{align*}\theta_{i}(x)=\max\{0,\min\{x,M_{i}\}\}\text{ and
}\rho_{i}(x)=\max\{\frac{1}{n},\min\{x,L_{i}\}\},\,i=1,2.\end{align*}
Consider the modified system of BVPs
\begin{equation}\label{3.1.5}\begin{split}
-x''(t)&=p(t)f(t,\theta_{2}(y(t)),\rho_{1}(x'(t))),\hspace{0.4cm}t\in(0,1),\\
-y''(t)&=q(t)g(t,\theta_{1}(x(t)),\rho_{2}(y'(t))),\hspace{0.4cm}t\in(0,1),\\
x(0)&=y(0)=0,\,x'(1)=y'(1)=\frac{1}{n}.
\end{split}\end{equation}
Since
$f(t,\theta_{2}(y(t)),\rho_{1}(x'(t))),\,g(t,\theta_{1}(x(t)),\rho_{2}(y'(t)))$
are continuous and bounded on $[0,1]\times\R^{2}$, by Theorem
\ref{schauder}, it follows  that the modified system of BVPs
\eqref{3.1.5} has a solution $(x_{n},y_{n})\in(C^{1}[0,1]\cap
C^{2}(0,1))\times(C^{1}[0,1]\cap C^{2}(0,1))$.

\vskip 0.5em

Using \eqref{3.1.5} and $(\mathbf{B}_{2})$, we obtain
\begin{align*}x_{n}''(t)\leq0\text{ and }y_{n}''(t)\leq 0\text{ for
}t\in(0,1),\end{align*} which on  integration from $t$ to $1$, using
the BCs \eqref{3.1.5}, implies that
\begin{equation}\label{3.1.6}
x_{n}'(t)\geq\frac{1}{n}\text{ and }y_{n}'(t)\geq\frac{1}{n}\text{
for }t\in[0,1].
\end{equation} Integrating \eqref{3.1.6}  from $0$ to $t$, using the BCs \eqref{3.1.5}, we have
\begin{equation}\label{3.1.7}
x_{n}(t)\geq\frac{t}{n}\text{ and }y_{n}(t)\geq\frac{t}{n}\text{ for
}t\in[0,1].
\end{equation} From \eqref{3.1.6} and \eqref{3.1.7}, it follows that
\begin{align*}\|x_{n}\|=x_{n}(1)\text{ and }\|y_{n}\|=y_{n}(1).\end{align*}

Now, we show that
\begin{equation}\label{3.1.8}x_{n}'(t)< L_{1},\,\, y_{n}'(t)<L_{2},\hspace{0.4cm}t\in[0,1].\end{equation}
First, we prove $x_{n}'(t)<L_{1}$ for $t\in[0,1]$. Suppose
$x_{n}'(t_{1})\geq L_{1}$ for some $t_{1}\in[0,1]$. Using
\eqref{3.1.5} and $(\mathbf{B}_{3})$, we have
\begin{align*}-x_{n}''(t)\leq p(t)k_{1}(\theta_{2}(y_{n}(t)))(u_{1}(\rho_{1}(x_{n}'(t)))+v_{1}(\rho_{1}(x_{n}'(t)))),\hspace{0.4cm}t\in(0,1),\end{align*}
which implies that
\begin{align*}\frac{-x_{n}''(t)}{u_{1}(\rho_{1}(x_{n}'(t)))+v_{1}(\rho_{1}(x_{n}'(t)))}\leq k_{1}(M_{2})p(t),\hspace{0.4cm}t\in(0,1).\end{align*}
Integrating from $t_{1}$ to $1$, using the BCs \eqref{3.1.5}, we
obtain
\begin{align*}\int_{\frac{1}{n}}^{x_{n}'(t_{1})}\frac{dz}{u_{1}(\rho_{1}(z))+v_{1}(\rho_{1}(z))}\leq k_{1}(M_{2})\int_{t_{1}}^{1}p(t)dt,\end{align*}
which can also be written as
\begin{align*}\int_{\frac{1}{n}}^{L_{1}}\frac{dz}{u_{1}(z)+v_{1}(z)}+\int_{L_{1}}^{x_{n}'(t_{1})}\frac{dz}{u_{1}(L_{1})+v_{1}(L_{1})}
\leq k_{1}(M_{2})\int_{0}^{1}p(t)dt.\end{align*} Using the
increasing property of $I$, we obtain
\begin{align*}I(L_{1})+\frac{x_{n}'(t_{1})-L_{1}}{u_{1}(L_{1})+v_{1}(L_{1})}
\leq k_{1}(M_{2})\int_{0}^{1}p(t)dt+I(\varepsilon),\end{align*} a
contradiction to \eqref{3.1.3}. Hence, $x_{n}'(t)<L_{1}$ for
$t\in[0,1]$. Similarly, we can show that $y_{n}'(t)<L_{2}$ for $t\in[0,1]$.

\vskip 0.5em

Now, we show that
\begin{equation}\label{3.1.9}x_{n}(t)<M_{1},\hspace{0.2cm}y_{n}(t)<M_{2},\hspace{0.4cm}t\in[0,1].\end{equation}
Suppose $x_{n}(t_{2})\geq M_{1}$ for some $t_{2}\in[0,1]$. From
\eqref{3.1.5}, \eqref{3.1.8} and $(\mathbf{B}_{3})$, it follows
that
\begin{align*}-x_{n}''(t)&\leq p(t)k_{1}(\theta_{2}(y_{n}(t)))(u_{1}(x_{n}'(t))+v_{1}(x_{n}'(t))),\hspace{0.4cm}t\in(0,1),\\
-y_{n}''(t)&\leq
q(t)k_{2}(\theta_{1}(x_{n}(t)))(u_{2}(y_{n}'(t))+v_{2}(y_{n}'(t))),\hspace{0.4cm}t\in(0,1),\end{align*}
which implies that
\begin{align*}\frac{-x_{n}''(t)}{u_{1}(x_{n}'(t))+v_{1}(x_{n}'(t))}&\leq k_{1}(\theta_{2}(\|y_{n}\|))p(t),\hspace{0.4cm}t\in(0,1),\\
\frac{-y_{n}''(t)}{u_{2}(y_{n}'(t))+v_{2}(y_{n}'(t))}&\leq
k_{2}(M_{1})q(t),\hspace{1.3cm}t\in(0,1).\end{align*} Integrating
from $t$ to $1$, using the BCs \eqref{3.1.5}, we obtain
\begin{align*}\int_{\frac{1}{n}}^{x_{n}'(t)}\frac{dz}{u_{1}(z)+v_{1}(z)}&\leq
k_{1}(\theta_{2}(\|y_{n}\|))\int_{t}^{1}p(s)ds,\hspace{0.4cm}t\in[0,1],\\
\int_{\frac{1}{n}}^{y_{n}'(t)}\frac{dz}{u_{2}(z)+v_{2}(z)}&\leq
k_{2}(M_{1})\int_{t}^{1}q(s)ds,\hspace{1.3cm}t\in[0,1],\end{align*}
which implies that
\begin{align*}I(x_{n}'(t))-I(\frac{1}{n})&\leq
k_{1}(\theta_{2}(\|y_{n}\|))\int_{0}^{1}p(s)ds,\hspace{0.4cm}t\in[0,1],\\
J(y_{n}'(t))-J(\frac{1}{n})&\leq
k_{2}(M_{1})\int_{0}^{1}q(s)ds,\hspace{1.3cm}t\in[0,1].\end{align*}
The increasing property of $I$ and $J$ leads to
\begin{equation}\label{3.1.10}x_{n}'(t)\leq I^{-1}(k_{1}(\theta_{2}(\|y_{n}\|))\int_{0}^{1}p(s)ds+I(\varepsilon)),\hspace{0.4cm}t\in[0,1],\end{equation}
\begin{equation}\label{3.1.11}y_{n}'(t)\leq J^{-1}(k_{2}(M_{1})\int_{0}^{1}q(s)ds+J(\varepsilon)),\hspace{1.3cm}t\in[0,1].\end{equation}
Integrating \eqref{3.1.10} from $0$ to $t_{2}$ and \eqref{3.1.11}
from $0$ to $1$, using the BCs \eqref{3.1.5}, we obtain
\begin{equation}\label{3.1.12}M_{1}\leq x_{n}(t_{2})\leq I^{-1}(k_{1}(\theta_{2}(\|y_{n}\|))\int_{0}^{1}p(s)ds+I(\varepsilon)),\end{equation}
\begin{equation}\label{3.1.13}\|y_{n}\|\leq J^{-1}(k_{2}(M_{1})\int_{0}^{1}q(s)ds+J(\varepsilon)).\end{equation}
Either we have  $\|y_{n}\|<M_{2}$ or $\|y_{n}\|\geq M_{2}$. If
$\|y_{n}\|<M_{2}$, then from \eqref{3.1.12}, we have
\begin{equation}\label{3.1.14}M_{1}\leq I^{-1}(k_{1}(\|y_{n}\|)\int_{0}^{1}p(s)ds+I(\varepsilon)).\end{equation}
Now, by using \eqref{3.1.13} in \eqref{3.1.14} and the increasing
property of $k_{1}$ and $I^{-1}$, we obtain
\begin{align*}M_{1}\leq I^{-1}(k_{1}(J^{-1}(k_{2}(M_{1})\int_{0}^{1}q(s)ds+J(\varepsilon)))\int_{0}^{1}p(s)ds+I(\varepsilon)),\end{align*}
which implies that
\begin{align*}\frac{M_{1}}{I^{-1}(k_{1}(J^{-1}(k_{2}(M_{1})\int_{0}^{1}q(s)ds+J(\varepsilon)))\int_{0}^{1}p(s)ds+I(\varepsilon))}\leq 1,\end{align*}
a contradiction to \eqref{3.1.1}.

\vskip 0.5em

On the other hand, if $\|y_{n}\|\geq M_{2}$, then from
\eqref{3.1.12} and \eqref{3.1.13}, we have
\begin{equation}\label{3.1.15}M_{1}\leq x_{n}(t_{2})\leq I^{-1}(k_{1}(M_{2})\int_{0}^{1}p(s)ds+I(\varepsilon)),\end{equation}
\begin{equation}\label{3.1.16}M_{2}\leq J^{-1}(k_{2}(M_{1})\int_{0}^{1}q(s)ds+J(\varepsilon)).\end{equation}
Using \eqref{3.1.16} in \eqref{3.1.15} and the increasing property
of $k_{1}$ and $I^{-1}$, leads to
\begin{align*}M_{1}\leq I^{-1}(k_{1}(J^{-1}(k_{2}(M_{1})\int_{0}^{1}q(s)ds+J(\varepsilon)))\int_{0}^{1}p(s)ds+I(\varepsilon)),\end{align*}
which implies that
\begin{align*}\frac{M_{1}}{I^{-1}(k_{1}(J^{-1}(k_{2}(M_{1})\int_{0}^{1}q(s)ds+J(\varepsilon)))\int_{0}^{1}p(s)ds+I(\varepsilon))}\leq 1,\end{align*}
a contradiction to \eqref{3.1.1}. Hence, $x_{n}(t)<M_{1}$ for
$t\in[0,1]$. Similarly, we can show that $y_{n}(t)<M_{2}$ for $t\in[0,1]$.

\vskip 0.5em

Thus, in view of \eqref{3.1.5}-\eqref{3.1.9}, $(x_{n},y_{n})$ is a
solution of the following coupled system of BVPs
\begin{equation}\label{3.1.17}\begin{split}
-x''(t)&=p(t)f(t,y(t),x'(t)),\hspace{0.4cm}t\in(0,1),\\
-y''(t)&=q(t)g(t,x(t),y'(t)),\hspace{0.4cm}t\in(0,1),\\
x(0)&=y(0)=0,\,x'(1)=y'(1)=\frac{1}{n},\end{split}\end{equation}
satisfy
\begin{equation}\label{3.1.18}\begin{split}
\frac{t}{n}&\leq x_{n}(t)<M_{1},\,\frac{1}{n}\leq
x_{n}'(t)<L_{1},\hspace{0.4cm}t\in[0,1],\\
\frac{t}{n}&\leq y_{n}(t)<M_{2},\,\frac{1}{n}\leq
y_{n}'(t)<L_{2},\hspace{0.4cm}t\in[0,1].
\end{split}\end{equation}
Now, in view of $(\mathbf{B}_{6})$, there exist continuous
functions $\varphi_{M_{2}L_{1}}$ and $\psi_{M_{1}L_{2}}$ defined on
$[0,1]$ and positive on $(0,1)$, and real constants
$0\leq\delta_{1},\delta_{2}<1$ such that
\begin{equation}\label{3.1.19}\begin{split}f(t,y_{n}(t),x_{n}'(t))\geq \varphi_{M_{2}L_{1}}(t)(y_{n}(t))^{\delta_{1}},\hspace{0.4cm}(t,y_{n}(t),x_{n}'(t))\in[0,1]\times[0,M_{2}]\times[0,L_{1}],\\
g(t,x_{n}(t),y_{n}'(t))\geq\psi_{M_{1}L_{2}}(t)(x_{n}(t))^{\delta_{2}},\hspace{0.4cm}(t,x_{n}(t),y_{n}'(t))\in[0,1]\times[0,M_{1}]\times[0,L_{2}].\end{split}\end{equation}
We claim that
\begin{equation}\label{3.1.20}x_{n}'(t)\geq C_{2}^{\delta_{1}}\int_{t}^{1}s^{\delta_{1}}p(s)\varphi_{M_{2}L_{1}}(s)ds,\end{equation}
\begin{equation}\label{3.1.21}y_{n}'(t)\geq
C_{1}^{\delta_{2}}\int_{t}^{1}s^{\delta_{2}}q(s)\psi_{M_{1}L_{2}}(s)ds,\end{equation}
where
\begin{align*}\begin{split}
C_{1}=&\left(\int_{0}^{1}s^{\delta_{2}+1}q(s)\psi_{M_{1}L_{2}}(s)ds\right)^{\frac{\delta_{1}}{1-\delta_{1}\delta_{2}}}\left(\int_{0}^{1}s^{\delta_{1}+1}p(s)\varphi_{M_{2}L_{1}}(s)ds\right)^\frac{1}{1-\delta_{1}\delta_{2}},\\
C_{2}=&\left(\int_{0}^{1}s^{\delta_{1}+1}p(s)\varphi_{M_{2}L_{1}}(s)ds\right)^{\frac{\delta_{2}}{1-\delta_{1}\delta_{2}}}\left(\int_{0}^{1}s^{\delta_{2}+1}q(s)\psi_{M_{1}L_{2}}(s)ds\right)^{\frac{1}{1-\delta_{1}\delta_{2}}}.
\end{split}\end{align*}
To prove \eqref{3.1.20}, consider the following relation
\begin{equation}\label{3.1.22}
x_{n}(t)=\frac{t}{n}+\int_{0}^{t}sp(s)f(s,y_{n}(s),x_{n}'(s))ds+\int_{t}^{1}t
p(s)f(s,y_{n}(s),x_{n}'(s))ds,
\end{equation}
which implies that
\begin{align*}x_{n}(1)\geq\int_{0}^{1}sp(s)f(s,y_{n}(s),x_{n}'(s))ds.\end{align*}
Using \eqref{3.1.19} and Lemma \ref{lemaconcave}, we obtain
\begin{equation}\label{3.1.23}x_{n}(1)\geq(y_{n}(1))^{\delta_{1}}\int_{0}^{1}s^{\delta_{1}+1}p(s)\varphi_{M_{2}L_{1}}(s)ds.\end{equation}
Similarly, using \eqref{3.1.19} and Lemma \ref{lemaconcave}, we
obtain
\begin{align*}
y_{n}(1)\geq(x_{n}(1))^{\delta_{2}}\int_{0}^{1}s^{\delta_{2}+1}q(s)\psi_{M_{1}L_{2}}(s)ds,
\end{align*}
which in view of  \eqref{3.1.23} implies that
\begin{align*}
y_{n}(1)\geq(y_{n}(1))^{\delta_{1}\delta_{2}}\left(\int_{0}^{1}s^{\delta_{1}+1}p(s)\varphi_{M_{2}L_{1}}(s)ds\right)^{\delta_{2}}\int_{0}^{1}s^{\delta_{2}+1}q(s)\psi_{M_{1}L_{2}}(s)ds.
\end{align*}
Hence,
\begin{equation}\label{3.1.24}y_{n}(1)\geq C_{2}.\end{equation}
Now, from \eqref{3.1.22}, it follows that
\begin{align*}
x_{n}'(t)\geq\int_{t}^{1}p(s)f(s,y_{n}(s),x_{n}'(s))ds.
\end{align*}
Using \eqref{3.1.19}, Lemma \ref{lemaconcave} and \eqref{3.1.24}, we
obtain \eqref{3.1.20}. Similarly, we can prove \eqref{3.1.21}.

\vskip 0.5em

Now, using \eqref{3.1.17}, $(\mathbf{B}_{3})$, \eqref{3.1.18},
\eqref{3.1.20} and \eqref{3.1.21}, we have
\begin{equation}\label{3.1.25}\begin{split}
0\leq-x_{n}''(t)&\leq k_{1}(M_{2})p(t)(u_{1}(C_{2}^{\delta_{1}}\int_{t}^{1}s^{\delta_{1}}p(s)\varphi_{M_{2}L_{1}}(s)ds)+v_{1}(L_{1})),\hspace{0.4cm}t\in(0,1),\\
0\leq-y_{n}''(t)&\leq
k_{2}(M_{1})q(t)(u_{2}(C_{1}^{\delta_{2}}\int_{t}^{1}s^{\delta_{2}}q(s)\psi_{M_{1}L_{2}}(s)ds)+v_{2}(L_{2})),\hspace{0.4cm}t\in(0,1).
\end{split}\end{equation}
In view of \eqref{3.1.18}, \eqref{3.1.25}, $(\mathbf{B}_{1})$ and
$(\mathbf{B}_{7})$, it follows that the sequences
$\{(x_{n}^{(j)},y_{n}^{(j)})\}$ $(j=0,1)$ are uniformly bounded and
equicontinuous on $[0,1]$. Hence, by Theorem \ref{arzela}, there
exist subsequences $\{(x_{n_{k}}^{(j)},y_{n_{k}}^{(j)})\}$ $(j=0,1)$
of $\{(x_{n}^{(j)},y_{n}^{(j)})\}$ $(j=0,1)$ and $(x,y)\in
C^{1}[0,1]\times C^{1}[0,1]$ such that
$(x_{n_{k}}^{(j)},y_{n_{k}}^{(j)})$ converges uniformly to
$(x^{(j)},y^{(j)})$ on $[0,1]$  $(j=0,1)$. Also,
$x(0)=y(0)=x'(1)=y'(1)=0$. Moreover, from \eqref{3.1.20} and
\eqref{3.1.21}, with $n_{k}$ in place of $n$ and taking
$\lim_{n_{k}\rightarrow+\infty}$, we have
\begin{align*}x'(t)\geq C_{2}^{\delta_{1}}\int_{t}^{1}s^{\delta_{1}}p(s)\varphi_{M_{2}L_{1}}(s)ds,\\
y'(t)\geq
C_{1}^{\delta_{2}}\int_{t}^{1}s^{\delta_{2}}q(s)\psi_{M_{1}L_{2}}(s)ds,\end{align*}
which shows that $x'>0$ and $y'>0$ on $[0,1)$, $x>0$ and $y>0$ on
$(0,1]$. Further, $(x_{n_{k}},y_{n_{k}})$ satisfy
\begin{align*}
x_{n_{k}}'(t)&=x_{n_{k}}'(0)-\int_{0}^{t}p(s)f(s,y_{n_{k}}(s),x_{n_{k}}'(s))ds,\hspace{0.4cm}t\in[0,1],\\
y_{n_{k}}'(t)&=y_{n_{k}}'(0)-\int_{0}^{t}q(s)g(s,x_{n_{k}}(s),y_{n_{k}}'(s))ds,\hspace{0.4cm}t\in[0,1].
\end{align*}
Passing to the limit as $n_{k}\rightarrow\infty$, we obtain
\begin{align*}
x'(t)&=x'(0)-\int_{0}^{t}p(s)f(s,y(s),x'(s))ds,\hspace{0.4cm}t\in[0,1],\\
y'(t)&=y'(0)-\int_{0}^{t}q(s)g(s,x(s),y'(s))ds,\hspace{0.4cm}t\in[0,1],
\end{align*}
which implies that
\begin{align*}
-x''(t)=p(t)f(t,y(t),x'(t)),\hspace{0.4cm}t\in(0,1),\\
-y''(t)=q(t)g(t,x(t),y'(t)),\hspace{0.4cm}t\in(0,1).
\end{align*}Hence, $(x,y)$ is a $C^{1}$-positive solution of the system of SBVPs \eqref{3.0.3}.
\end{proof}

\begin{ex}
Consider the following coupled system of SBVPs
\begin{equation}\label{3.1.26}\begin{split}
-&x''(t)=t^{-\frac{1}{3}}(1-t)^{-\frac{2}{3}}(y(t))^{\frac{1}{3}}(x'(t))^{-\beta_{1}},\hspace{0.4cm}t\in(0,1),\\
-&y''(t)=t^{-\frac{2}{3}}(1-t)^{-\frac{1}{3}}(x(t))^{\frac{2}{3}}(y'(t))^{-\beta_{2}},\hspace{0.4cm}t\in(0,1),\\
&x(0)=y(0)=x'(1)=y'(1)=0,
\end{split}\end{equation}
where $0<\beta_{1}<1$ and $0<\beta_{2}<\frac{1}{2}$.

\vskip 0.5em

Choose $p(t)=t^{-\frac{1}{3}}(1-t)^{-\frac{2}{3}}$, $q(t)=t^{-\frac{2}{3}}(1-t)^{-\frac{1}{3}}$, $k_{1}(x)=x^{\frac{1}{3}}$, $k_{2}(x)=x^{\frac{2}{3}}$, $u_{1}(x)=x^{-\beta_{1}}$, $u_{2}(x)=x^{-\beta_{2}}$ and $v_{1}(x)=v_{2}(x)=0$.

\vskip 0.5em

Then, $I(z)=\frac{z^{\beta_{1}+1}}{\beta_{1}+1}$, $J(z)=\frac{z^{\beta_{2}+1}}{\beta_{2}+1}$, $I^{-1}(z)=(\beta_{1}+1)^{\frac{1}{\beta_{1}+1}}z^{\frac{1}{\beta_{1}+1}}$ and
$J^{-1}(z)=(\beta_{2}+1)^{\frac{1}{\beta_{2}+1}}z^{\frac{1}{\beta_{2}+1}}$.

\vskip 0.5em

Also, $\int_{0}^{1}p(t)dt=\int_{0}^{1}q(t)dt=\frac{2\pi}{\sqrt{3}}$.

\vskip 0.5em

Moreover,
\begin{align*}
\sup_{c\in(0,\infty)}&\frac{c}{I^{-1}(k_{1}(J^{-1}(k_{2}(c)\int_{0}^{1}q(s)ds))\int_{0}^{1}p(s)ds)}=\\
\sup_{c\in(0,\infty)}&\frac{c}{(\frac{2\pi}{\sqrt{3}})^{\frac{3\beta_{2}+4}{3(\beta_{1}+1)(\beta_{2}+1)}}(\beta_{1}+1)^{\frac{1}{\beta_{1}+1}}(\beta_{2}+1)^{\frac{1}{3(\beta_{1}+1)(\beta_{2}+1)}}c^{\frac{2}{9(\beta_{1}+1)(\beta_{2}+1)}}}=\infty,\\
\text{ and }\\
\sup_{c\in(0,\infty)}&\frac{c}{J^{-1}(k_{2}(I^{-1}(k_{1}(c)\int_{0}^{1}p(s)ds))\int_{0}^{1}q(s)ds)}=\\
\sup_{c\in(0,\infty)}&\frac{c}{(\frac{2\pi}{\sqrt{3}})^{\frac{3\beta_{1}+5}{3(\beta_{1}+1)(\beta_{2}+1)}}(\beta_{2}+1)^{\frac{1}{\beta_{2}+1}}(\beta_{1}+1)^{\frac{2}{3(\beta_{1}+1)(\beta_{2}+1)}}c^{\frac{2}{9(\beta_{1}+1)(\beta_{2}+1)}}}=\infty.
\end{align*}
Clearly, $(\mathbf{B}_{1})-(\mathbf{B}_{5})$ are satisfied.
For $\delta_{1}=\frac{1}{3}$, $\delta_{2}=\frac{2}{3}$,
$\varphi_{EF}(t)=F^{-\beta_{1}}$ and $\psi_{EF}(t)=F^{-\beta_{2}}$,
then $(\mathbf{B}_{6})$ holds. Further,
\begin{align*}
\int_{0}^{1}p(t)u_{1}(C\int_{t}^{1}s^{\delta_{1}}p(s)\varphi_{EF}(s)ds)dt&=3^{-\beta_{1}}C^{-\beta_{1}}F^{\beta_{1}^{2}}\int_{0}^{1}t^{-\frac{1}{3}}(1-t)^{-\frac{\beta_{1}+2}{3}}dt\\
&=3^{-\beta_{1}}C^{-\beta_{1}}F^{\beta_{1}^{2}}\frac{\Gamma(\frac{2}{3})\Gamma(\frac{1-\beta_{1}}{3})}{\Gamma(1-\frac{\beta_{1}}{3})},
\end{align*}

\begin{align*}
\int_{0}^{1}q(t)u_{2}(C\int_{t}^{1}s^{\delta_{2}}q(s)\psi_{EF}(s)ds)dt&=(\frac{3}{2})^{-\beta_{2}}C^{-\beta_{2}}F^{\beta_{2}^{2}}\int_{0}^{1}t^{-\frac{2}{3}}(1-t)^{-\frac{2\beta_{2}+1}{3}}dt\\
&=(\frac{3}{2})^{-\beta_{2}}C^{-\beta_{2}}F^{\beta_{2}^{2}}\frac{\Gamma(\frac{2}{3})\Gamma(\frac{1-2\beta_{2}}{3})}{\Gamma(1-\frac{2\beta_{2}}{3})},
\end{align*}
shows that $(\mathbf{B}_{7})$ also holds. Since, $(\mathbf{B}_{1})-(\mathbf{B}_{7})$ are satisfied. Therefore, by Theorem \ref{th3.2}, the system of BVPs \eqref{3.1.26} has a $C^{1}$-positive solution.
\end{ex}


\section{Existence of at least two positive
solutions}\label{multiplicity-one}

In this section, we establish sufficient conditions for the
existence of at least two positive solutions of the system of SBVPs \eqref{3.0.3}. By a positive solution $(x,y)$ of the
system of BVPs \eqref{3.0.3}, we mean that
$(x,y)\in\mathcal{E}\times\mathcal{E}$ satisfies \eqref{3.0.3}, $x>0$ and $y>0$ on $(0,1]$, $x'>0$ and $y'>0$ on
$[0,1)$. Define a cone $P$ of $\mathcal{E}$ by
\begin{align*}P=\{x\in\mathcal{E}:x(t)\geq t\|x\|\text{ for all }t\in[0,1],
x(1)\geq\|x\|_{1}\}.\end{align*} For each
$(x,y)\in\mathcal{E}\times\mathcal{E}$ we write
$\|(x,y)\|_{4}=\|x\|_{2}+\|y\|_{2}$. Clearly,
$(\mathcal{E}\times\mathcal{E},\|\cdot\|_{4})$ is a Banach space and
$P\times P$ is a cone of $\mathcal{E}\times\mathcal{E}$. We define a
partial ordering in $\mathcal{E}$, by $x\leq y$ if and only if
$x(t)\leq y(t)$, $t\in[0,1]$. We define a partial ordering in
$\mathcal{E}\times\mathcal{E}$, by $(x_{1},y_{1})\preceq
(x_{2},y_{2})$ if and only if $x_{1}\leq x_{2}$ and $y_{1}\leq
y_{2}$. For any real constant $r>0$, we define an open neighborhood
of $(0,0)\in\mathcal{E}\times\mathcal{E}$  as
\begin{align*}\mathcal{O}_{r}=\{(x,y)\in\mathcal{E}\times\mathcal{E}:\|(x,y)\|_{4}<r\}.\end{align*}

\vskip 0.5em

In view of $(\mathbf{B}_{4})$, there exist real constants
$R_{1}>0$ and $R_{2}>0$ such that
\begin{equation}\label{3.2.1}
\frac{R_{1}}{I^{-1}(k_{1}(J^{-1}(k_{2}(R_{1})\int_{0}^{1}q(s)ds))\int_{0}^{1}p(s)ds)}>1,
\end{equation}
\begin{equation}\label{3.2.2}
\frac{R_{2}}{J^{-1}(k_{2}(I^{-1}(k_{1}(R_{2})\int_{0}^{1}p(s)ds))\int_{0}^{1}q(s)ds)}>1.
\end{equation}
From the continuity of $k_{1}$, $k_{2}$, $I$ and $J$, we choose
$\varepsilon>0$ small enough such that
\begin{equation}\label{3.2.3}\frac{R_{1}}{I^{-1}(k_{1}(J^{-1}(k_{2}(R_{1}+\varepsilon)\int_{0}^{1}q(s)ds+J(\varepsilon))+\varepsilon)
\int_{0}^{1}p(s)ds+I(\varepsilon))}> 1,\end{equation}
\begin{equation}\label{3.2.4}\frac{R_{2}}{J^{-1}(k_{2}(I^{-1}(k_{1}(R_{2}+\varepsilon)\int_{0}^{1}p(s)ds+I(\varepsilon))+\varepsilon)
\int_{0}^{1}q(s)ds+J(\varepsilon))}> 1.\end{equation} Choose
$n_{0}\in\{1,2,\cdots\}$ such that $\frac{1}{n_{0}}<\varepsilon$ and
for each fixed $n\in\{n_{0},n_{0}+1,\cdots\}$, consider the system
of non-singular BVPs
\begin{equation}\label{3.2.5}\begin{split}
-x''(t)&=p(t)f(t,y(t)+\frac{t}{n},|x'(t)|+\frac{1}{n}),\hspace{0.4cm}t\in(0,1),\\
-y''(t)&=q(t)g(t,x(t)+\frac{t}{n},|y'(t)|+\frac{1}{n}),\hspace{0.4cm}t\in(0,1),\\
x(0)&=x'(1)=y(0)=y'(1)=0.
\end{split}\end{equation}
We write \eqref{3.2.5} as an equivalent system of integral equations
\begin{equation}\label{3.2.6}\begin{split}
x(t)&=\int_{0}^{1}G(t,s)p(s)f(s,y(s)+\frac{s}{n},|x'(s)|+\frac{1}{n})ds,
\hspace{0.4cm}t\in[0,1],\\
y(t)&=\int_{0}^{1}G(t,s)q(s)f(s,x(s)+\frac{s}{n},|y'(s)|+\frac{1}{n})ds,
\hspace{0.4cm}t\in[0,1],
\end{split}\end{equation}
where the Green's function is defined as
\begin{align*}
G(t,s)=\begin{cases}s,\,& 0\leq s\leq t\leq1,\\
t,\,& 0\leq t\leq s\leq1.\end{cases}
\end{align*}
By a solution of the system of BVPs \eqref{3.2.5}, we mean a solution of the corresponding system of integral equations \eqref{3.2.6}.

\vskip 0.5em

Define a map $T_{n}:\mathcal{E}\times\mathcal{E}\rightarrow\mathcal{E}\times\mathcal{E}$
by
\begin{equation}\label{3.2.7}T_{n}(x,y)=(A_{n}(x,y),B_{n}(x,y)),\end{equation}
where the maps
$A_{n},B_{n}:\mathcal{E}\times\mathcal{E}\rightarrow\mathcal{E}$ are
defined by
\begin{equation}\label{3.2.8}\begin{split}
A_{n}(x,y)(t)&=\int_{0}^{1}G(t,s)p(s)f(s,y(s)+\frac{s}{n},|x'(s)|+\frac{1}{n})ds,
\hspace{0.4cm}t\in[0,1],\\
B_{n}(x,y)(t)&=\int_{0}^{1}G(t,s)q(s)f(s,x(s)+\frac{s}{n},|y'(s)|+\frac{1}{n})ds,
\hspace{0.4cm}t\in[0,1].\end{split}\end{equation} Clearly, if
$(x_{n},y_{n})\in\mathcal{E}\times\mathcal{E}$ is a fixed point of
$T_{n}$; then $(x_{n},y_{n})$ is a solution of the system of BVPs
\eqref{3.2.5}.

\vskip 0.5em

Assume that
\begin{description}
\item[$(\mathbf{B}_{8})$] for any real constant $E>0$, there exist continuous functions $\varphi_{E}$ and $\psi_{E}$ defined on
$[0,1]$ and positive on $(0,1)$, and constants
$0\leq\delta_{1},\delta_{2}<1$ such that
\begin{align*}f(t,x,y)\geq\varphi_{E}(t)x^{\delta_{1}},\,g(t,x,y)\geq\psi_{E}(t)x^{\delta_{2}}\text{
on }[0,1]\times[0,E]\times[0,\infty);\end{align*}
\item[$(\mathbf{B}_{9})$] for any real constant $C>0$, $\int_{0}^{1}p(t)v_{1}(\frac{C}{t})dt<+\infty$, $\int_{0}^{1}q(t)v_{2}(\frac{C}{t})dt<+\infty$,\\
 $\int_{0}^{1}p(t)u_{1}(C\int_{t}^{1}s^{\delta_{1}}p(s)\varphi_{E}(s)ds)dt<+\infty$ and
$\int_{0}^{1}q(t)u_{2}(C\int_{t}^{1}s^{\delta_{2}}q(s)\psi_{E}(s)ds)dt<+\infty$.
\end{description}

\begin{lem}\label{4.1}
Assume that $(\mathbf{B}_{1})-(\mathbf{B}_{3})$ and
$(\mathbf{B}_{9})$ hold. Then the map $T_{n}:\overline{\mathcal{O}}_{r}\cap(P\times P)\rightarrow P\times
P$ is completely continuous.
\end{lem}

\begin{proof} Firstly, we show that $T_{n}(P\times P)\subseteq P\times
P$. For $(x,y)\in P\times P$, $t\in[0,1]$, using \eqref{3.2.8} and
Lemma \ref{lemmax1}, we obtain
\begin{equation}\label{3.2.9}\begin{split}
&A_{n}(x,y)(t)=\int_{0}^{1}G(t,s)p(s)f(s,y(s)+\frac{s}{n},|x'(s)|+\frac{1}{n})ds\\
&\geq
t\max_{\tau\in[0,1]}\int_{0}^{1}G(\tau,s)p(s)f(s,y(s)+\frac{s}{n},|x'(s)|+\frac{1}{n})ds=t\|A_{n}(x,y)\|
\end{split}\end{equation}
and
\begin{equation}\label{3.2.10}\begin{split}
&\|A_{n}(x,y)\|_{1}=\sup_{\tau\in(0,1]}\tau|A_{n}(x,y)'(\tau)|=\sup_{\tau\in(0,1]}\tau\int_{\tau}^{1}p(s)f(s,y(s)+\frac{s}{n},|x'(s)|+\frac{1}{n})ds\\
&\leq\max_{t\in[0,1]}\int_{0}^{1}G(t,s)p(s)f(s,y(s)+\frac{s}{n},|x'(s)|+\frac{1}{n})ds\leq
A_{n}(x,y)(1).
\end{split}\end{equation} From \eqref{3.2.9} and
\eqref{3.2.10}, $A_{n}(x,y)\in P$ for every $(x,y)\in P\times P$,
that is, $A_{n}(P\times P)\subseteq P$. Similarly, by using
\eqref{3.2.8} and Lemma \ref{lemmax1}, we can show that
$B_{n}(P\times P)\subseteq P$. Hence, $T_{n}(P\times P)\subseteq
P\times P$.

\vskip 0.5em

Now, we show that $T_{n}:\overline{\mathcal{O}}_{r}\cap(P\times
P)\rightarrow P\times P$ is uniformly bounded. For any
$(x,y)\in\overline{\mathcal{O}}_{r}\cap(P\times P)$, using
\eqref{3.2.8}, $(\mathbf{B}_{3})$, Lemma \ref{lemspace1},
$(\mathbf{B}_{1})$ and $(\mathbf{B}_{9})$, we have
\begin{equation}\label{3.2.11}\begin{split}
&\|A_{n}(x,y)\|=\max_{t\in[0,1]}\left|\int_{0}^{1}G(t,s)p(s)f(s,y(s)+\frac{s}{n},|x'(s)|+\frac{1}{n})ds\right|\\
&\leq\int_{0}^{1}p(s)k_{1}(y(s)+\frac{s}{n})(u_{1}(|x'(s)|+\frac{1}{n})+v_{1}(|x'(s)|+\frac{1}{n}))ds\\
&\leq\int_{0}^{1}p(s)k_{1}(y(s)+\frac{s}{n})(u_{1}(|x'(s)|+\frac{1}{n})+v_{1}(\frac{\|x\|_{2}}{s}+\frac{1}{n}))ds\\
&\leq\int_{0}^{1}p(s)k_{1}(y(s)+\frac{s}{n})(u_{1}(\frac{1}{n})+v_{1}(\frac{r}{s}+\frac{1}{n}))ds\\
&\leq
k_{1}(r+\frac{1}{n})\int_{0}^{1}p(s)(u_{1}(\frac{1}{n})+v_{1}((r+\frac{1}{n})\frac{1}{s}))ds<+\infty.
\end{split}\end{equation}
Also, for $(x,y)\in\overline{\mathcal{O}}_{r}\cap(P\times P)$, using
\eqref{3.2.8}, Lemma \ref{lemmax1} and $(\mathbf{B}_{3})$, we have
\begin{align*}\begin{split}
&\|A_{n}(x,y)\|_{1}=\sup_{\tau\in(0,1]}\tau|A_{n}(x,y)'(\tau)|=\sup_{\tau\in(0,1]}\tau\int_{\tau}^{1}p(s)f(s,y(s)+\frac{s}{n},|x'(s)|+\frac{1}{n})ds\\
&\leq\max_{t\in[0,1]}\int_{t}^{1}G(t,s)p(s)f(s,y(s)+\frac{s}{n},|x'(s)|+\frac{1}{n})ds\leq\int_{0}^{1}p(s)f(s,y(s)+\frac{s}{n},|x'(s)|+\frac{1}{n})ds\\
&\leq\int_{0}^{1}p(s)k_{1}(y(s)+\frac{s}{n})(u_{1}(|x'(s)|+\frac{1}{n})+v_{1}(|x'(s)|+\frac{1}{n}))ds.
\end{split}\end{align*}
Now, using Lemma \ref{lemspace1}, $(\mathbf{B}_{1})$ and
$(\mathbf{B}_{9})$, we obtain
\begin{equation}\label{3.2.12}\begin{split}
&\|A_{n}(x,y)\|_{1}\leq\int_{0}^{1}p(s)k_{1}(y(s)+\frac{s}{n})(u_{1}(\frac{1}{n})+v_{1}(\frac{\|x\|_{2}}{s}+\frac{1}{n}))ds\\
&\leq\int_{0}^{1}p(s)k_{1}(y(s)+\frac{s}{n})(u_{1}(\frac{1}{n})+v_{1}(\frac{r}{s}+\frac{1}{n}))ds\\
&\leq
k_{1}(r+\frac{1}{n})\int_{0}^{1}p(s)(u_{1}(\frac{1}{n})+v_{1}((r+\frac{1}{n})\frac{1}{s}))ds<+\infty.
\end{split}\end{equation}
From \eqref{3.2.11} and \eqref{3.2.12}, it follows that
$A_{n}(\overline{\mathcal{O}}_{r}\cap(P\times P))$ is uniformly
bounded under the norm $\|\cdot\|_{2}$. Similarly, by using
\eqref{3.2.8}, Lemma \ref{lemspace1}, Lemma \ref{lemmax1},
$(\mathbf{B}_{1})$, $(\mathbf{B}_{3})$ and $(\mathbf{B}_{9})$,
we can show that $B_{n}(\overline{\mathcal{O}}_{r}\cap(P\times P))$
is uniformly bounded under the norm $\|\cdot\|_{2}$. Hence,
$T_{n}(\overline{\mathcal{O}}_{r}\cap(P\times P))$ is uniformly
bounded.

\vskip 0.5em

Now, we show that $T_{n}(\overline{\mathcal{O}}_{r}\cap(P\times P))$
is equicontinuous. For
$(x,y)\in\overline{\mathcal{O}}_{r}\cap(P\times P)$,
$t_{1},t_{2}\in[0,1]$, using \eqref{3.2.8}, $(\mathbf{B}_{3})$ and
Lemma \ref{lemspace1}, we have
\begin{equation}\label{3.2.13}\begin{split}
&|A_{n}(x,y)(t_{1})-A_{n}(x,y)(t_{2})|=\left|\int_{0}^{1}(G(t_{1},s)-G(t_{2},s))p(s)f(s,y(s)+\frac{s}{n},|x'(s)|+\frac{1}{n})ds\right|\\
&\leq\int_{0}^{1}|G(t_{1},s)-G(t_{2},s)|p(s)f(s,y(s)+\frac{s}{n},|x'(s)|+\frac{1}{n})ds\\
&\leq\int_{0}^{1}|G(t_{1},s)-G(t_{2},s)|p(s)k_{1}(y(s)+\frac{s}{n})(u_{1}(|x'(s)|+\frac{1}{n})+v_{1}(|x'(s)|+\frac{1}{n}))ds\\
&\leq k_{1}(r+\frac{1}{n})\int_{0}^{1}|G(t_{1},s)-G(t_{2},s)|p(s)(u_{1}(\frac{1}{n})+v_{1}(\frac{\|x\|_{2}}{s}+\frac{1}{n}))ds\\
&\leq
k_{1}(r+\frac{1}{n})\int_{0}^{1}|G(t_{1},s)-G(t_{2},s)|p(s)(u_{1}(\frac{1}{n})+v_{1}(r+\frac{1}{n})\frac{1}{s})ds,
\end{split}\end{equation}
and
\begin{equation}\label{3.2.14}\begin{split}
&|A_{n}(x,y)'(t_{1})-A_{n}(x,y)'(t_{2})|=\left|\int_{t_{1}}^{t_{2}}p(s)f(s,y(s)+\frac{s}{n},|x'(s)|+\frac{1}{n})ds\right|\\
&\leq\int_{t_{1}}^{t_{2}}p(s)k_{1}(y(s)+\frac{s}{n})(u_{1}(|x'(s)|+\frac{1}{n})+v_{1}(|x'(s)|+\frac{1}{n}))ds\\
&\leq
k_{1}(r+\frac{1}{n})\int_{t_{1}}^{t_{2}}p(s)(u_{1}(\frac{1}{n})+v_{1}(\frac{\|x\|_{2}}{s}+\frac{1}{n}))ds\leq
k_{1}(r+\frac{1}{n})\int_{t_{1}}^{t_{2}}p(s)\\
&(u_{1}(\frac{1}{n})+v_{1}(\frac{r}{s}+\frac{1}{n}))ds\leq
k_{1}(r+\frac{1}{n})\int_{t_{1}}^{t_{2}}p(s)(u_{1}(\frac{1}{n})+v_{1}((r+\frac{1}{n})\frac{1}{s}))ds.
\end{split}\end{equation}
From \eqref{3.2.13}, \eqref{3.2.14}, $(\mathbf{B}_{1})$ and
$(\mathbf{B}_{9})$, it follows that
$A_{n}(\overline{\mathcal{O}}_{r}\cap(P\times P))$ is equicontinuous
under the norm $\|\cdot\|_{3}$. But, the norm $\|\cdot\|_{3}$ is
equivalent to the norm $\|\cdot\|_{2}$. Hence,
$A_{n}(\overline{\mathcal{O}}_{r}\cap(P\times P))$ is equicontinuous
under $\|\cdot\|_{2}$.

\vskip 0.5em

Similarly, using \eqref{3.2.8}, $(\mathbf{B}_{3})$ and Lemma
\ref{lemspace1}, we can show that
$B_{n}(\overline{\mathcal{O}}_{r}\cap(P\times P))$ is equicontinuous
under the norm $\|\cdot\|_{2}$. Consequently,
$T_{n}(\overline{\mathcal{O}}_{r}\cap(P\times P))$ is
equicontinuous. Hence, by Theorem \ref{arzela},
$T_{n}(\overline{\mathcal{O}}_{r}\cap(P\times P))$ is relatively
compact which implies that $T_{n}$ is a compact map.

\vskip 0.5em

Now, we show that $T_{n}$ is continuous. Let
$(x_{m},y_{m}),(x,y)\in\overline{\mathcal{O}}_{r}\cap(P\times P)$
such that $\|(x_{m},y_{m})-(x,y)\|_{4}\rightarrow0$ as
$m\rightarrow+\infty$. Using $(\mathbf{B}_{3})$ and Lemma
\ref{lemspace1}, we have
\begin{align*}\begin{split}
&\left|f(t,y_{m}(t)+\frac{t}{n},|x_{m}'(t)|+\frac{1}{n})\right|\leq k_{1}(y_{m}(t)+\frac{t}{n})(u_{1}(|x_{m}'(t)|+\frac{1}{n})+v_{1}(|x_{m}'(t)|+\frac{1}{n}))\\
&\leq k_{1}(r+\frac{1}{n})(u_{1}(\frac{1}{n})+v_{1}(\frac{\|x_{m}\|_{2}}{t}+\frac{1}{n}))\leq k_{1}(r+\frac{1}{n})(u_{1}(\frac{1}{n})+v_{1}(\frac{r}{t}+\frac{1}{n}))\\
&\leq
k_{1}(r+\frac{1}{n})(u_{1}(\frac{1}{n})+v_{1}((r+\frac{1}{n})\frac{1}{t})).
\end{split}\end{align*}
Using \eqref{3.2.8} and Lemma \ref{lemmax1}, we have
\begin{equation}\label{3.2.15}\begin{split}
&\|A_{n}(x_{m},y_{m})-A_{n}(x,y)\|=\\
&\max_{t\in[0,1]}\left|\int_{0}^{1}G(t,s)p(s)(f(s,y_{m}(s)+\frac{s}{n},|x_{m}'(s)|+\frac{1}{n})-f(s,y(s)+\frac{s}{n},|x'(s)|+\frac{1}{n}))ds\right|\\
&\leq\int_{0}^{1}p(s)\left|f(s,y_{m}(s)+\frac{s}{n},|x_{m}'(s)|+\frac{1}{n})-f(s,y(s)+\frac{s}{n},|x'(s)|+\frac{1}{n})\right|ds
\end{split}\end{equation}
and
\begin{equation}\label{3.2.16}\begin{split}
&\|A_{n}(x_{m},y_{m})'-A_{n}(x,y)'\|_{1}=\\
&\sup_{\tau\in(0,1]}\tau\left|\int_{\tau}^{1}p(s)(f(s,y_{m}(s)+\frac{s}{n},|x_{m}'(s)|+\frac{1}{n})-f(s,y(s)+\frac{s}{n},|x'(s)|+\frac{1}{n}))ds\right|\\
&\leq\max_{t\in[0,1]}\int_{0}^{1}G(t,s)p(s)\left|f(s,y_{m}(s)+\frac{s}{n},|x_{m}'(s)|+\frac{1}{n})-f(s,y(s)+\frac{s}{n},|x'(s)|+\frac{1}{n})\right|ds\\
&\leq\int_{0}^{1}p(s)\left|f(s,y_{m}(s)+\frac{s}{n},|x_{m}'(s)|+\frac{1}{n})-f(s,y(s)+\frac{s}{n},|x'(s)|+\frac{1}{n})\right|ds.
\end{split}\end{equation}
From \eqref{3.2.15} and \eqref{3.2.16}, using the Lebesgue dominated
convergence theorem, it follows that
\begin{align*}
\|A_{n}(x_{m},y_{m})-A_{n}(x,y)\|\rightarrow0,\hspace{0.2cm}\|A_{n}(x_{m},y_{m})'-A_{n}(x,y)'\|_{1}\rightarrow0\text{
as }m\rightarrow+\infty.\end{align*} Hence,
$\|A_{n}(x_{m},y_{m})-A_{n}(x,y)\|_{2}\rightarrow0$ as
$m\rightarrow\infty$.

\vskip 0.5em

Similarly, we can show that
$\|B_{n}(x_{m},y_{m})-B_{n}(x,y)\|_{2}\rightarrow0$ as
$m\rightarrow\infty$. Consequently,
$\|T_{n}(x_{m},y_{m})-T_{n}(x,y)\|_{4}\rightarrow0$ as
$m\rightarrow+\infty$, that is,
$T_{n}:\overline{\mathcal{O}}_{r}\cap(P\times P)\rightarrow P\times
P$ is continuous. Hence,
$T_{n}:\overline{\mathcal{O}}_{r}\cap(P\times P)\rightarrow P\times
P$ is completely continuous.
\end{proof}

Assume that
\begin{description}
\item[$(\mathbf{B}_{10})$] there exist $h_{1},h_{2}\in C([0,\infty)\times(0,\infty),[0,\infty))$ with $f(t,x,y)\geq h_{1}(x,y)$ and $g(t,x,y)\geq h_{2}(x,y)$ on $[0,1]\times[0,\infty)\times(0,\infty)$ such that
\begin{align*}\lim_{x\rightarrow+\infty}\frac{h_{i}(x,y)}{x}&=+\infty,\text{ uniformly for }\,y\in(0,\infty),\,i=1,2.
\end{align*}
\end{description}

\begin{thm}\label{th4.2}
Assume that $(\mathbf{B}_{1})-(\mathbf{B}_{5})$ and
$(\mathbf{B}_{8})-(\mathbf{B}_{10})$ hold. Then the system of BVPs \eqref{3.0.3} has at least two positive solutions.
\end{thm}

\begin{proof}
Let $R_{0}=R_{1}+R_{2}$ and define
$\mathcal{O}_{R_{0}}=\Omega_{R_{1}}\times\Omega_{R_{2}}$ where
\begin{align*}\Omega_{R_{1}}=\{x\in
E:\|x\|_{2}<R_{1}\},\,\Omega_{R_{2}}=\{x\in
E:\|x\|_{2}<R_{2}\}.\end{align*} We claim that
\begin{equation}\label{3.2.17}
(x,y)\neq\lambda T_{n}(x,y), \text{ for }\lambda\in(0,1],(x,y)\in\partial\mathcal{O}_{R_{0}}\cap (P\times P).
\end{equation}
Suppose there exist
$(x_{0},y_{0})\in\partial\mathcal{O}_{R}\cap(P\times P)$ and
$\lambda_{0}\in(0,1]$ such that $(x_{0},y_{0})=\lambda_{0}
T_{n}(x_{0},y_{0})$. Then,
\begin{equation}\label{3.2.18}\begin{split}
-x_{0}''(t)&=\lambda_{0}p(t)f(t,y_{0}(t)+\frac{t}{n},|x_{0}'(t)|+\frac{1}{n}),\hspace{0.4cm}t\in(0,1),\\
-y_{0}''(t)&=\lambda_{0}q(t)g(t,x_{0}(t)+\frac{t}{n},|y_{0}'(t)|+\frac{1}{n}),\hspace{0.4cm}t\in(0,1),\\
x_{0}(0)&=x_{0}'(1)=y_{0}(0)=y_{0}'(1)=0.
\end{split}\end{equation}
From \eqref{3.2.18} and $(\mathbf{B}_{2})$, we have $x_{0}''\leq
0$ and $y_{0}''\leq 0$ on $(0,1)$, integrating from $t$ to $1$,
using the BCs \eqref{3.2.18}, we obtain $x_{0}'(t)\geq0$ and
$y_{0}'(t)\geq0$ for $t\in[0,1]$. From \eqref{3.2.18} and
$(\mathbf{B}_{3})$, we have
\begin{align*}-x_{0}''(t)&\leq
p(t)k_{1}(y_{0}(t)+\frac{t}{n})(u_{1}(x_{0}'(t)+\frac{1}{n})+v_{1}(x_{0}'(t)+\frac{1}{n})),\hspace{0.4cm}t\in(0,1),\\
-y_{0}''(t)&\leq
q(t)k_{2}(x_{0}(t)+\frac{t}{n})(u_{2}(y_{0}'(t)+\frac{1}{n})+v_{2}(y_{0}'(t)+\frac{1}{n})),\hspace{0.4cm}t\in(0,1),\end{align*}
which implies that
\begin{align*}\frac{-x_{0}''(t)}{u_{1}(x_{0}'(t)+\frac{1}{n})+v_{1}(x_{0}'(t)+\frac{1}{n})}\leq p(t)k_{1}(y_{0}(t)+\frac{t}{n})&\leq
k_{1}(R_{2}+\varepsilon)p(t),\hspace{0.4cm}t\in(0,1),\\
\frac{-y_{0}''(t)}{u_{2}(y_{0}'(t)+\frac{1}{n})+v_{2}(y_{0}'(t)+\frac{1}{n})}\leq
q(t)k_{2}(x_{0}(t)+\frac{t}{n})&\leq
k_{2}(R_{1}+\varepsilon)q(t),\hspace{0.4cm}t\in(0,1).\end{align*}
Integrating from $t$ to $1$, using the BCs \eqref{3.2.18}, we obtain
\begin{align*}I(x_{0}'(t)+\frac{1}{n})-I(\frac{1}{n})&\leq k_{1}(R_{2}+\varepsilon)\int_{t}^{1}p(s)ds,\hspace{0.4cm}t\in[0,1],\\
J(y_{0}'(t)+\frac{1}{n})-J(\frac{1}{n})&\leq
k_{2}(R_{1}+\varepsilon)\int_{t}^{1}q(s)ds,\hspace{0.4cm}t\in[0,1],\end{align*}
which implies that
\begin{align*}x_{0}'(t)&\leq I^{-1}(k_{1}(R_{2}+\varepsilon)\int_{0}^{1}p(s)ds+I(\varepsilon)),\hspace{0.4cm}t\in[0,1],\\
y_{0}'(t)&\leq
J^{-1}(k_{2}(R_{1}+\varepsilon)\int_{0}^{1}q(s)ds+J(\varepsilon)),\hspace{0.4cm}t\in[0,1],\end{align*}
which on integration from $0$ to $1$, using the BCs \eqref{3.2.18}
and Lemma \ref{lemcone1}, leads to
\begin{equation}\label{3.2.19}R_{1}\leq I^{-1}(k_{1}(R_{2}+\varepsilon)\int_{0}^{1}p(s)ds+I(\varepsilon)),\end{equation}
\begin{equation}\label{3.2.20}R_{2}\leq J^{-1}(k_{2}(R_{1}+\varepsilon)\int_{0}^{1}q(s)ds+J(\varepsilon)).\end{equation}
Now, using \eqref{3.2.20} in \eqref{3.2.19} together with increasing
property of $k_{1}$ and $I^{-1}$, we have
\begin{align*}\frac{R_{1}}{I^{-1}(k_{1}(J^{-1}(k_{2}(R_{1}+\varepsilon)\int_{0}^{1}q(s)ds+J(\varepsilon))+\varepsilon)
\int_{0}^{1}p(s)ds+I(\varepsilon))}\leq 1,\end{align*} a
contradiction to \eqref{3.2.3}. Similarly, using \eqref{3.2.19} in
\eqref{3.2.20} together with increasing property of $k_{2}$ and
$J^{-1}$, we have
\begin{align*}\frac{R_{2}}{J^{-1}(k_{2}(I^{-1}(k_{1}(R_{2}+\varepsilon)\int_{0}^{1}p(s)ds+I(\varepsilon))+\varepsilon)
\int_{0}^{1}q(s)ds+J(\varepsilon))}\leq 1,\end{align*} a
contradiction to \eqref{3.2.4}. Hence, \eqref{3.2.17} is true and by
Lemma \ref{lemindexone}, the fixed point index
\begin{equation}\label{3.2.21}\ind(T_{n},\mathcal{O}_{R_{0}}\cap(P\times P),P\times P)=1.\end{equation}
Now, choose a $t_{0}\in(0,1)$ and define
\begin{equation}\label{3.2.22}N_{1}=\left(t_{0}\min_{t\in[t_{0},1]}\int_{t_{0}}^{1}G(t,s)p(s)ds\right)^{-1}+1\text{
and
}N_{2}=\left(t_{0}\min_{t\in[t_{0},1]}\int_{t_{0}}^{1}G(t,s)q(s)ds\right)^{-1}+1.\end{equation}
By $(\mathbf{B}_{10})$, there exist real constants with
$R_{1}^{*}>R_{1}$ and $R_{2}^{*}>R_{2}$ such that
\begin{equation}\label{3.2.23}\begin{split}
h_{1}(x,y)&\geq N_{1}x,\text{ for }x\geq
R_{1}^{*},y\in(0,\infty),\\
h_{2}(x,y)&\geq N_{2}x,\text{ for }x\geq
R_{2}^{*},y\in(0,\infty).
\end{split}\end{equation} Let
$R^{*}=\frac{R_{1}^{*}+R_{2}^{*}}{t_{0}}$ and define
$\mathcal{O}_{R^{*}}=\Omega_{R_{1}^{*}}\times \Omega_{R_{2}^{*}}$,
where
\begin{align*}\Omega_{R_{1}^{*}}=\{x\in E:\|x\|_{2}<\frac{R_{1}^{*}}{t_{0}}\},\,
\Omega_{R_{2}^{*}}=\{x\in
E:\|x\|_{2}<\frac{R_{2}^{*}}{t_{0}}\}.\end{align*} We show that
\begin{equation}\label{3.2.24}
T_{n}(x,y)\npreceq(x,y),\text{ for }(x,y)\in\partial
\mathcal{O}_{R^{*}}\cap(P\times P).
\end{equation}
Suppose $T_{n}(x_{0},y_{0})\preceq(x_{0},y_{0})$ for some
$(x_{0},y_{0})\in\partial\mathcal{O}_{R^{*}}\cap(P\times P)$. Then,
\begin{equation}\label{3.2.25}x_{0}(t)\geq A_{n}(x_{0},y_{0})(t)\text{ and }y_{0}(t)\geq B_{n}(x_{0},y_{0})(t)\text{ for }t\in[0,1].\end{equation}
By Lemma \ref{lemcone1}, we have
\begin{align*}x_{0}(t)\geq t\|x_{0}\|\geq
t_{0}\|x_{0}\|_{2}=t_{0}\frac{R_{1}^{*}}{t_{0}}=R_{1}^{*}\text{ for
}t\in[t_{0},1].\end{align*} Similarly, $y_{0}(t)\geq R_{2}^{*}$ for
$t\in[t_{0},1]$. Hence, $$|x_{0}(t)|+\frac{t}{n}\geq R_{1}^{*}\text{
and }|y_{0}(t)|+\frac{t}{n}\geq R_{2}^{*}\text{ for
}t\in[t_{0},1].$$ Now, using \eqref{3.2.25}, \eqref{3.2.23} and
$(\mathbf{B}_{10})$, we have
\begin{align*}\begin{split}
&x_{0}(t)\geq A_{n}(x_{0},y_{0})(t)\\
&=\int_{0}^{1}G(t,s)p(s)f(s,y_{0}(s)+\frac{s}{n},|x_{0}'(s)|+\frac{1}{n})ds\\
&\geq\int_{t_{0}}^{1}G(t,s)p(s)h_{1}(y_{0}(s)+\frac{s}{n},|x_{0}'(s)|+\frac{1}{n})ds\\
&\geq\int_{t_{0}}^{1}G(t,s)p(s)N_{1}(y_{0}(s)+\frac{s}{n})ds\\
&\geq\int_{t_{0}}^{1}G(t,s)p(s)dsN_{1}R_{2}^{*}\\
&\geq\min_{t\in[t_{0},1]}\int_{t_{0}}^{1}G(t,s)p(s)dsN_{1}R_{2}^{*}\\
&>\frac{R_{2}^{*}}{t_{0}},
\end{split}\end{align*} which implies that
$\|x_{0}\|_{2}=\|x_{0}\|>\frac{R_{2}^{*}}{t_{0}}$. Similarly, using
\eqref{3.2.23}, \eqref{3.2.25} and $(\mathbf{B}_{10})$, we have
$\|y_{0}\|_{2}>\frac{R_{1}^{*}}{t_{0}}$. Consequently, it follows
that, $\|(x_{0},y_{0})\|_{4}=\|x_{0}\|_{2}+\|y_{0}\|_{2}>R^{*}$, a
contradiction. Hence, \eqref{3.2.24} is true and by Lemma
\ref{lemindexzero}, the fixed point index
\begin{equation}\label{3.2.26}\ind(T_{n},\mathcal{O}_{R^{*}}\cap(P\times P),P\times P)=0.\end{equation}
From \eqref{3.2.21} and \eqref{3.2.26}, it follows that
\begin{equation}\label{3.2.27}\ind(T_{n},(\mathcal{O}_{R^{*}}\setminus\overline{\mathcal{O}}_{R})\cap(P\times P),P\times P)=-1.\end{equation}
Thus, in view of \eqref{3.2.21} and \eqref{3.2.27}, there exist
$(x_{n,1},y_{n,1})\in\mathcal{O}_{R}\cap(P\times P)$ and
$(x_{n,2},y_{n,2})\in(\mathcal{O}_{R^{*}}\setminus\overline{\mathcal{O}}_{R})\cap(P\times
P)$ such that $(x_{n,j},y_{n,j})=T_{n}(x_{n,j},y_{n,j}),\,(j=1,2)$
which implies that
\begin{equation}\label{3.2.28}\begin{split}x_{n,j}(t)&=\int_{0}^{1}G(t,s)p(s)f(t,y_{n,j}(s)+\frac{s}{n},|x_{n,j}'(s)|+\frac{1}{n})ds,\hspace{0.4cm}t\in[0,1],\\
y_{n,j}(t)&=\int_{0}^{1}G(t,s)q(s)g(s,x_{n,j}(s)+\frac{s}{n},|y_{n,j}'(s)|+\frac{1}{n})ds,\hspace{0.4cm}t\in[0,1],\,j=1,2.\end{split}\end{equation}
Using $(\mathbf{B}_{8})$ there exist continuous functions
$\varphi_{R_{2}+\varepsilon}$ and $\psi_{R_{1}+\varepsilon}$ defined
on $[0,1]$ and positive on $(0,1)$ and real constants
$0\leq\delta_{1},\delta_{2}<1$ such that
\begin{equation}\label{3.2.29}\begin{split}f(t,x,y)&\geq \varphi_{R_{2}+\varepsilon}(t)x^{\delta_{1}},\hspace{0.4cm}(t,x,y)\in[0,1]\times[0,R_{2}+\varepsilon]\times[0,\infty),\\
g(t,x,y)&\geq
\psi_{R_{1}+\varepsilon}(t)x^{\delta_{2}},\hspace{0.4cm}(t,x,y)\in[0,1]\times[0,R_{1}+\varepsilon]\times[0,\infty).\end{split}\end{equation}
By the definition of $P$, we have $x_{n,1}(t)\geq t\|x_{n,1}\|$ and
$y_{n,1}(t)\geq t\|y_{n,1}\|$ for $t\in[0,1]$. We show that
\begin{equation}\label{3.2.30}x_{n,1}'(t)\geq C_{4}^{\delta_{1}}\int_{t}^{1}s^{\delta_{1}}p(s)\varphi_{R_{2}+\varepsilon}(s)ds,\hspace{0.4cm}t\in[0,1],\end{equation}
\begin{equation}\label{3.2.31}y_{n,1}'(t)\geq
C_{3}^{\delta_{2}}\int_{t}^{1}s^{\delta_{2}}q(s)\psi_{R_{1}+\varepsilon}(s)ds,\hspace{0.4cm}t\in[0,1],\end{equation}
where
\begin{align*}\begin{split}
C_{3}&=\left(\int_{0}^{1}s^{\delta_{2}+1}q(s)\psi_{R_{1}+\varepsilon}(s)ds\right)^{\frac{\delta_{1}}{1-\delta_{1}\delta_{2}}}\left(\int_{0}^{1}s^{\delta_{1}+1}p(s)\varphi_{R_{2}+\varepsilon}(s)ds\right)^{\frac{1}{1-\delta_{1}\delta_{2}}},\\
C_{4}&=\left(\int_{0}^{1}s^{\delta_{1}+1}p(s)\varphi_{R_{2}+\varepsilon}(s)ds\right)^{\frac{\delta_{2}}{1-\delta_{1}\delta_{2}}}\left(\int_{0}^{1}s^{\delta_{2}+1}q(s)\psi_{R_{1}+\varepsilon}(s)ds\right)^{\frac{1}{1-\delta_{1}\delta_{2}}}.
\end{split}\end{align*}
In order to prove \eqref{3.2.30}, using \eqref{3.2.28} and
\eqref{3.2.29}, we consider
\begin{align*}\begin{split}
x_{n,1}(t)&=\int_{0}^{1}G(t,s)p(s)f(s,y_{n,1}(s)+\frac{s}{n},|x_{n,1}'(s)|+\frac{1}{n})ds\\
&\geq\int_{0}^{1}G(t,s)p(s)\varphi_{R_{2}+\varepsilon}(s)(y_{n,1}(s)+\frac{s}{n})^{\delta_{1}}ds\\
&\geq\|y_{n,1}\|^{\delta_{1}}\int_{0}^{1}G(t,s)s^{\delta_{1}}p(s)\varphi_{R_{2}+\varepsilon}(s)ds,
\end{split}\end{align*}
which shows that
\begin{equation}\label{3.2.32}
\|x_{n,1}\|\geq\|y_{n,1}\|^{\delta_{1}}\int_{0}^{1}s^{\delta_{1}+1}p(s)\varphi_{R_{2}+\varepsilon}(s)ds.
\end{equation}
Similarly, from \eqref{3.2.28} and \eqref{3.2.29}, we have
\begin{equation}\label{3.2.33}
\|y_{n,1}\|\geq\|x_{n,1}\|^{\delta_{2}}\int_{0}^{1}s^{\delta_{2}+1}q(s)\psi_{R_{1}+\varepsilon}(s)ds.
\end{equation}
Using \eqref{3.2.33} in \eqref{3.2.32}, we have
\begin{align*}
\|y_{n,1}\|\geq\left(\|y_{n,1}\|^{\delta_{1}}\int_{0}^{1}s^{\delta_{1}+1}p(s)\varphi_{R_{2}+\varepsilon}(s)ds\right)^{\delta_{2}}\int_{0}^{1}s^{\delta_{2}+1}q(s)\psi_{R_{1}+\varepsilon}(s)ds,
\end{align*}
which implies that
\begin{equation}\label{3.2.34}
\|y_{n,1}\|\geq\left(\int_{0}^{1}s^{\delta_{1}+1}p(s)\varphi_{R_{2}+\varepsilon}(s)ds\right)^{\frac{\delta_{2}}{1-\delta_{1}\delta_{2}}}
\left(\int_{0}^{1}s^{\delta_{2}+1}q(s)\psi_{R_{1}+\varepsilon}(s)ds\right)^{\frac{1}{1-\delta_{1}\delta_{2}}}=C_{4}.
\end{equation}
Using  \eqref{3.2.29} and \eqref{3.2.34} in the following relation
\begin{align*}
x_{n,1}'(t)=&\int_{t}^{1}p(s)f(s,y_{n,1}(s)+\frac{s}{n},|x_{n,1}'(s)|+\frac{1}{n})ds,
\end{align*}
we obtain \eqref{3.2.30}. Similarly, we can prove \eqref{3.2.31}.

\vskip 0.5em

Now, differentiating \eqref{3.2.28}, using $(\mathbf{B}_{3})$,
\eqref{3.2.30}, \eqref{3.2.31} and Lemma \ref{lemspace1}, we have
\begin{equation}\label{3.2.35}\begin{split}
0\leq-x_{n,1}''(t)&\leq p(t)k_{1}(R_{2}+\varepsilon)(u_{1}(C_{4}^{\delta_{1}}\int_{t}^{1}s^{\delta_{1}}p(s)\varphi_{R_{2}+\varepsilon}(s)ds)+v_{1}(\frac{R_{1}+1}{t})),\hspace{0.4cm}t\in(0,1),\\
0\leq-y_{n,1}''(t)&\leq q(t)
k_{2}(R_{1}+\varepsilon)(u_{2}(C_{3}^{\delta_{2}}\int_{t}^{1}s^{\delta_{2}}q(s)\psi_{R_{1}+\varepsilon}(s)ds)+v_{2}(\frac{R_{2}+1}{t})),\hspace{0.4cm}t\in(0,1).
\end{split}\end{equation}
Integration from $t$ to $1$, using the BCs \eqref{3.2.5}, leads to
\begin{align*}
x_{n,1}'(t)&\leq
k_{1}(R_{2}+\varepsilon)\int_{t}^{1}p(s)(u_{1}(C_{4}^{\delta_{1}}
\int_{s}^{1}\tau^{\delta_{1}}p(\tau)\varphi_{R_{2}+\varepsilon}(\tau)d\tau)+v_{1}(\frac{R_{1}+1}{s}))ds,\hspace{0.4cm}t\in[0,1],\\
y_{n,1}'(t)&\leq k_{2}(R_{1}+\varepsilon)\int_{t}^{1}q(s)
(u_{2}(C_{3}^{\delta_{2}}\int_{s}^{1}\tau^{\delta_{2}}q(\tau)\psi_{R_{1}+\varepsilon}(\tau)d\tau)+v_{2}(\frac{R_{2}+1}{s}))ds,\hspace{0.4cm}t\in[0,1],
\end{align*} which implies that
\begin{equation}\label{3.2.36}\begin{split}
x_{n,1}'(t)&\leq
k_{1}(R_{2}+\varepsilon)\int_{0}^{1}p(s)(u_{1}(C_{4}^{\delta_{1}}
\int_{s}^{1}\tau^{\delta_{1}}p(\tau)\varphi_{R_{2}+\varepsilon}(\tau)d\tau)+v_{1}(\frac{R_{1}+1}{s}))ds,\hspace{0.4cm}t\in[0,1],\\
y_{n,1}'(t)&\leq k_{2}(R_{1}+\varepsilon)\int_{0}^{1}q(s)
(u_{2}(C_{3}^{\delta_{2}}\int_{s}^{1}\tau^{\delta_{2}}q(\tau)\psi_{R_{1}+\varepsilon}(\tau)d\tau)+v_{2}(\frac{R_{2}+1}{s}))ds,\hspace{0.4cm}t\in[0,1].
\end{split}\end{equation}
In view of \eqref{3.2.30}, \eqref{3.2.31}, \eqref{3.2.36},
\eqref{3.2.35}, $(\mathbf{B}_{1})$ and $(\mathbf{B}_{9})$, the
sequences $\{(x_{n,1}^{(j)},y_{n,1}^{(j)})\}$ $(j=0,1)$ are
uniformly bounded and equicontinuous on $[0,1]$. Thus, by Theorem
\ref{arzela}, there exist subsequences
$\{(x_{n_{k},1}^{(j)},y_{n_{k},1}^{(j)})\}\,(j=0,1)$ of
 $\{(x_{n,1}^{(j)},y_{n,1}^{(j)})\}$ and functions $(x_{0,1},y_{0,1})\in\mathcal{E}\times\mathcal{E}$ such that $(x_{n_{k},1}^{(j)},y_{n_{k},1}^{(j)})$
 converges uniformly to $(x_{0,1}^{(j)},y_{0,1}^{(j)})$ on $[0,1]$. Also, $x_{0,1}(0)=y_{0,1}(0)=x_{0,1}'(1)=y_{0,1}'(1)=0$. Moreover, from
 \eqref{3.2.30} and \eqref{3.2.31}, with $n_{k}$ in place of $n$ and taking $\lim_{n_k\rightarrow+\infty}$, we have
\begin{align*}
x_{0,1}'(t)&\geq C_{4}^{\delta_{1}}\int_{t}^{1}s^{\delta_{1}}p(s)\varphi_{R_{2}+\varepsilon}(s)ds,\\
y_{0,1}'(t)&\geq
C_{3}^{\delta_{2}}\int_{t}^{1}s^{\delta_{2}}q(s)\psi_{R_{1}+\varepsilon}(s)ds,
\end{align*}
which implies that $x_{0,1}'>0$ and $y_{0,1}'>0$ on $[0,1)$,
$x_{0,1}>0$ and $y_{0,1}>0$ on $(0,1]$. Further,
\begin{equation}\label{3.2.37}\begin{split}
\left|f(t,y_{n_{k},1}(t)+\frac{t}{n},x_{n_{k},1}'(t)+\frac{1}{n_{k}})\right|&\leq
k_{1}(R_{2}+\varepsilon)(u_{1}(C_{4}^{\delta_{1}}\int_{t}^{1}s^{\delta_{1}}p(s)\varphi_{R_{2}+\varepsilon}(s)ds)+v_{1}(\frac{R_{1}+1}{t})),\\
\left|g(t,x_{n_{k},1}(t)+\frac{t}{n},y_{n_{k},1}'(t)+\frac{1}{n_{k}})\right|&\leq
k_{2}(R_{1}+\varepsilon)(u_{2}(C_{3}^{\delta_{2}}\int_{t}^{1}s^{\delta_{2}}q(s)\psi_{R_{1}+\varepsilon}(s)ds)+v_{2}(\frac{R_{2}+1}{t})),
\end{split}\end{equation}
\begin{equation}\label{3.2.38}\begin{split}
\lim_{n_k\rightarrow\infty}f(t,y_{n_{k},1}(t)+\frac{t}{n_{k}},x_{n_{k},1}'(t)+\frac{1}{n_{k}})&=f(t,y_{0,1}(t),x_{0,1}'(t)),\hspace{0.4cm}t\in(0,1],\\
\lim_{n_k\rightarrow\infty}g(t,x_{n_{k},1}(t)+\frac{t}{n_{k}},y_{n_{k},1}'(t)+\frac{1}{n_{k}})&=g(t,x_{0,1}(t),y_{0,1}'(t)),\hspace{0.4cm}t\in(0,1].
\end{split}\end{equation} Moreover, $(x_{n_{k},1},y_{n_{k},1})$
satisfies
\begin{align*}
x_{n_{k},1}(t)&=\int_{0}^{1}G(t,s)p(s)f(s,y_{n_{k},1}(s)+\frac{s}{n_{k}},x_{n_{k},1}'(s)+\frac{1}{n_{k}})ds,\hspace{0.4cm}t\in[0,1],\\
y_{n_{k},1}(t)&=\int_{0}^{1}G(t,s)q(s)g(s,x_{n_{k},1}(s)+\frac{s}{n_{k}},y_{n_{k},1}'(s)+\frac{1}{n_{k}})ds,\hspace{0.4cm}t\in[0,1],
\end{align*}
which in view of \eqref{3.2.37}, $(\mathbf{B}_{9})$, \eqref{3.2.38}, the
Lebesgue dominated convergence theorem and taking
$\lim_{n_k\rightarrow+\infty}$, leads to
\begin{align*}
x_{0,1}(t)&=\int_{0}^{1}G(t,s)p(s)f(s,y_{0,1}(s),x_{0,1}'(s))ds,\hspace{0.4cm}t\in[0,1],\\
y_{0,1}(t)&=\int_{0}^{1}G(t,s)q(s)g(s,x_{0,1}(s),y_{0,1}'(s))ds,\hspace{0.4cm}t\in[0,1],\end{align*}
which implies that $(x_{0,1},y_{0,1})\in C^{2}(0,1)\times
C^{2}(0,1)$ and
\begin{align*}
-x_{0,1}''(t)=p(t)f(t,y_{0,1}(t),x_{0,1}'(t)),\hspace{0.4cm}t\in(0,1),\\
-y_{0,1}''(t)=q(t)g(t,x_{0,1}(t),y_{0,1}'(t)),\hspace{0.4cm}t\in(0,1).
\end{align*}
Moreover, by \eqref{3.2.1} and \eqref{3.2.2}, we have
$\|x_{0,1}\|_{2}<R_{1}$ and $\|y_{0,1}\|_{2}<R_{2}$, that is,
$\|(x_{0,1},y_{0,1})\|_{3}<R_{0}$. By a similar proof the sequence
$\{(x_{n,2},y_{n,2})\}$ has a convergent subsequence
$\{(x_{n_{k},2},y_{n_{k},2})\}$ converging uniformly to
$(x_{0,2},y_{0,2})\in\mathcal{E}\times\mathcal{E}$ on $[0,1]$.
Moreover, $(x_{0,2},y_{0,2})$ is a solution to the system \eqref{3.0.3} with $x_{0,2}>0$ and $y_{0,2}>0$ on
$(0,1]$, $x'_{0,2}>0$ and $y'_{0,2}>0$ on $[0,1)$,
$R_{0}<\|(x_{0,2},y_{0,2})\|_{4}<R^{*}$.
\end{proof}

\begin{ex}Consider the following coupled system of SBVPs
\begin{equation}\begin{split}\label{3.2.39}
-&x''(t)=\mu_{1}(1+(y(t))^{\delta_{1}}+(y(t))^{\eta_{1}})(1+(x'(t))^{\alpha_{1}}+(x'(t))^{-\beta_{1}}),\hspace{0.4cm}t\in(0,1),\\
-&y''(t)=\mu_{2}(1+(x(t))^{\delta_{2}}+(x(t))^{\eta_{2}})(1+(y'(t))^{\alpha_{2}}+(y'(t))^{-\beta_{2}}),\hspace{0.4cm}t\in(0,1),\\
&x(0)=y(0)=x'(1)=y'(1)=0,
\end{split}\end{equation}
where $0\leq\delta_{i}<1$, $\eta_{i}>1$, $0<\alpha_{i}<1$,
$0<\beta_{i}<1$, and $\mu_{i}>0$, $i=1,2$.

\vskip 0.5em

Choose $p(t)=\mu_{1}$, $q(t)=\mu_{2}$, $k_{i}(x)=1+x^{\delta_{i}}+x^{\eta_{i}}$, $u_{i}(x)=x^{-\beta_{i}}$ and $v_{i}(x)=1+x^{\alpha_{i}}$, $i=1,2$. Also, $\varphi_{E}(t)=\mu_{1}$, $\psi_{E}(t)=\mu_{2}$ and $h_{i}(x,y)=\mu_{i}(1+x^{\eta_{i}})$, $i=1,2$.

\vskip 0.5em

Assume that $\mu_{1}$ is arbitrary real constant and
\begin{align*}
\mu_{2}<\min\{\inf_{c\in(0,\infty)}\frac{J(c)}{k_{2}(I^{-1}(\mu_{1}k_{1}(c)))},\,\inf_{c\in(0,\infty)}\frac{J((\mu_{1}^{-1}I(c))^{\delta_{1}^{-1}})}{k_{2}(c)},\,\inf_{c\in(0,\infty)}\frac{J((\mu_{1}^{-1}I(c))^{\eta_{1}^{-1}})}{k_{2}(c)}\}.
\end{align*}
Then,
\begin{align*}\begin{split}
&\sup_{c\in(0,\infty)}\frac{c}{I^{-1}(k_{1}(J^{-1}(k_{2}(c)\int_{0}^{1}q(s)ds))\int_{0}^{1}p(s)ds)}\\
&=\sup_{c\in(0,\infty)}\frac{c}{I^{-1}(\mu_{1}k_{1}(J^{-1}(\mu_{2}k_{2}(c))))}\\
&\geq\frac{c}{I^{-1}(\mu_{1}k_{1}(J^{-1}(\mu_{2}k_{2}(c))))},\hspace{4.6cm}c\in(0,\infty)\\
&=\frac{c}{I^{-1}(\mu_{1}(1+(J^{-1}(\mu_{2}k_{2}(c)))^{\delta_{1}}+(J^{-1}(\mu_{2}k_{2}(c)))^{\eta_{1}}))},\hspace{1cm}c\in(0,\infty)\\
&>1,
\end{split}\end{align*}
and
\begin{align*}\begin{split}
&\sup_{c\in(0,\infty)}\frac{c}{J^{-1}(k_{2}(I^{-1}(k_{1}(c)\int_{0}^{1}p(s)ds))\int_{0}^{1}q(s)ds)}\\
&=\sup_{c\in(0,\infty)}\frac{c}{J^{-1}(\mu_{2}k_{2}(I^{-1}(\mu_{1}k_{1}(c))))}\\
&=\frac{c}{J^{-1}(\mu_{2}k_{2}(I^{-1}(\mu_{1}k_{1}(c))))},\hspace{1.4cm}c\in(0,\infty)\\
&>1.\end{split}\end{align*} Moreover,
\begin{align*}
\int_{0}^{1}p(t)v_{1}(\frac{C}{t})dt&=\mu_{1}(1+\frac{C^{\alpha_{1}}}{1-\alpha_{1}})<+\infty,\\
\int_{0}^{1}p(t)u_{1}(C\int_{t}^{1}s^{\delta_{1}}p(s)\varphi_{E}(s)ds)dt&\leq\mu_{1}^{1-2\beta_{1}}C^{-\beta_{1}}(\delta_{1}+1)^{\beta_{1}}\int_{0}^{1}(1-t)^{-\beta_{1}}dt\\
&=\mu_{1}^{1-2\beta_{1}}C^{-\beta_{1}}(\delta_{1}+1)^{\beta_{1}}(1-\beta_{1})^{-1}<+\infty,\text{
etc.}
\end{align*}
Also,
\begin{align*}\lim_{x\rightarrow+\infty}\frac{h_{i}(x,y)}{x}=\lim_{x\rightarrow+\infty}\frac{\mu_{i}(1+x^{\eta_{i}})}{x}=+\infty,\,i=1,2.\end{align*}
Clearly, $(\mathbf{B}_{1})-(\mathbf{B}_{5})$ and
$(\mathbf{B}_{8})-(\mathbf{B}_{10})$ are satisfied. Hence, by
Theorem \ref{th4.2}, the system of BVPs \eqref{3.2.39} has at least
two positive solutions.
\end{ex}


\section{Existence of $C^{1}$-positive solutions with more general BCs}\label{existence-two}

In this section, we study the system of BVPs \eqref{3.0.4} and establish sufficient conditions for the existence
of $C^{1}$-positive solutions.
By a $C^{1}$-positive solution to the system of BVPs \eqref{3.0.4}, we mean $(x,y)\in (C^{1}[0,1]\cap C^{2}(0,1))\times (C^{1}[0,1]\cap C^{2}(0,1))$ satisfying \eqref{3.0.4}, $x>0$ and $y>0$ on $[0,1]$, $x'>0$ and $y'>0$ on $[0,1)$.

\vskip 0.5em

Assume that
\begin{description}
\item[$(\mathbf{B}_{11})$]
\begin{align*}\begin{split}
\sup_{c\in(0,\infty)}\frac{c}{(1+\frac{b_{1}}{a_{1}})I^{-1}(k_{1}((1+\frac{b_{2}}{a_{2}})J^{-1}(k_{2}(c)\int_{0}^{1}q(t)dt))\int_{0}^{1}p(t)dt)}&>1,\\
\sup_{c\in(0,\infty)}\frac{c}{(1+\frac{b_{2}}{a_{2}})J^{-1}(k_{2}((1+\frac{b_{1}}{a_{1}})I^{-1}(k_{1}(c)\int_{0}^{1}p(t)dt))\int_{0}^{1}q(t)dt)}&>1,
\end{split}\end{align*}
where $I(\mu)=\int_{0}^{\mu}\frac{d\tau}{u_{1}(\tau)+v_{1}(\tau)},\,
J(\mu)=\int_{0}^{\mu}\frac{d\tau}{u_{2}(\tau)+v_{2}(\tau)}$, for
$\mu\in(0,\infty)$;
\item[$(\mathbf{B}_{12})$] $\int_{0}^{1}p(t)u_{1}(C\int_{t}^{1}p(s)\varphi_{EF}(s)ds)dt<+\infty$ and
$\int_{0}^{1}q(t)u_{2}(C\int_{t}^{1}q(s)\psi_{EF}(s)ds)dt<+\infty$
for any real constant $C>0$.
\end{description}

\begin{thm}\label{th2.1}
Assume that $(\mathbf{B}_{1})-(\mathbf{B}_{3})$,
$(\mathbf{B}_{5})$, $(\mathbf{B}_{6})$, $(\mathbf{B}_{11})$ and
$(\mathbf{B}_{12})$ hold. Then the system of BVPs \eqref{3.0.4} has a $C^{1}$-positive solution.
\end{thm}
\begin{proof}
In view of $(\mathbf{B}_{11})$, we can choose real constants
$M_{3}>0$ and $M_{4}>0$ such that
\begin{align*}
\frac{M_{3}}{(1+\frac{b_{1}}{a_{1}})I^{-1}(k_{1}((1+\frac{b_{2}}{a_{2}})J^{-1}(k_{2}(M_{3})\int_{0}^{1}q(t)dt))\int_{0}^{1}p(t)dt)}>1,
\end{align*}
\begin{align*}
\frac{M_{4}}{(1+\frac{b_{2}}{a_{2}})J^{-1}(k_{2}((1+\frac{b_{1}}{a_{1}})I^{-1}(k_{1}(M_{4})\int_{0}^{1}p(t)dt))\int_{0}^{1}q(t)dt)}>1.
\end{align*}
From the continuity of $k_{1}$, $k_{2}$, $I$ and $J$, we choose
$\varepsilon>0$ small enough such that
\begin{equation}\label{3.3.1}\frac{M_{3}}{(1+\frac{b_{1}}{a_{1}})I^{-1}(k_{1}((1+\frac{b_{2}}{a_{2}})J^{-1}(k_{2}(M_{3})\int_{0}^{1}q(t)dt+J(\varepsilon)))\int_{0}^{1}p(t)dt+I(\varepsilon))}>1,\end{equation}
\begin{equation}\label{3.3.2}\frac{M_{4}}{(1+\frac{b_{2}}{a_{2}})J^{-1}(k_{2}((1+\frac{b_{1}}{a_{1}})I^{-1}(k_{1}(M_{4})\int_{0}^{1}p(t)dt+I(\varepsilon)))\int_{0}^{1}q(t)dt+J(\varepsilon))}>1.\end{equation}
Choose real constants $L_{3}>0$ and $L_{4}>0$ such that
\begin{equation}\label{3.3.3}I(L_{3})>k_{1}(M_{4})\int_{0}^{1}p(t)dt+I(\varepsilon),\end{equation}
\begin{equation}\label{3.3.4}J(L_{4})>k_{2}(M_{3})\int_{0}^{1}q(t)dt+J(\varepsilon).\end{equation}

Choose $n_{0}\in\{1,2,\cdots\}$ such that
$\frac{1}{n_{0}}<\varepsilon$. For each fixed
$n\in\{n_{0},n_{0}+1,\cdots\}$, define retractions
$\theta_{i}:\R\rightarrow[0,M_{i}]$ and
$\rho_{i}:\R\rightarrow[\frac{1}{n},L_{i}]$ by
\begin{align*}\theta_{i}(x)=\max\{0,\min\{x,M_{i}\}\}\text{ and }\rho_{i}(x)=\max\{\frac{1}{n},\min\{x,L_{i}\}\},i=3,4.\end{align*}
Consider the modified system of BVPs
\begin{equation}\label{3.3.5}\begin{split}
-x''(t)&=p(t)f(t,\theta_{4}(y(t)),\rho_{3}(x'(t))),\hspace{0.4cm}t\in(0,1),\\
-y''(t)&=q(t)g(t,\theta_{3}(x(t)),\rho_{4}(y'(t))),\hspace{0.4cm}t\in(0,1),\\
a_{1}x(&0)-b_{1}x'(0)=0,\,x'(1)=\frac{1}{n},\\
a_{2}y(&0)-b_{2}y'(0)=0,\,y'(1)=\frac{1}{n}.
\end{split}\end{equation}
Since $f(t,\theta_{4}(y(t)),\rho_{3}(x'(t)))$,
$g(t,\theta_{3}(x(t)),\rho_{4}(y'(t)))$ are continuous and bounded
on $[0,1]\times\R^{2}$, by Theorem \ref{schauder}, it follows that
the modified system of BVPs \eqref{3.3.5} has a solution
$(x_{n},y_{n})\in(C^{1}[0,1]\cap C^{2}(0,1))\times(C^{1}[0,1]\cap
C^{2}(0,1))$.

\vskip 0.5em

Using \eqref{3.3.5} and $(\mathbf{B}_{2})$, we obtain
\begin{align*}x_{n}''(t)\leq0\text{ and }y_{n}''(t)\leq0\text{ for }t\in(0,1),\end{align*}
which on integration from $t$ to $1$, and using the BCs
\eqref{3.3.5}, yields
\begin{equation}\label{3.3.6}x_{n}'(t)\geq\frac{1}{n}\text{ and }y_{n}'(t)\geq\frac{1}{n}\text{ for }t\in[0,1].\end{equation}
Integrating \eqref{3.3.6} from $0$ to $t$, using the BCs
\eqref{3.3.5} and \eqref{3.3.6}, we have
\begin{equation}\label{3.3.7}x_{n}(t)\geq(t+\frac{b_{1}}{a_{1}})\frac{1}{n}\text{ and }y_{n}(t)\geq(t+\frac{b_{2}}{a_{2}})\frac{1}{n}\text{ for }t\in[0,1].\end{equation}
From \eqref{3.3.6} and \eqref{3.3.7}, it follows that
\begin{align*}\|x_{n}\|=x_{n}(1)\text{ and }\|y_{n}\|=y_{n}(1).\end{align*}

Now, we show that
\begin{equation}\label{3.3.8}x_{n}'(t)<L_{3},\,y_{n}'(t)<L_{4},\hspace{0.4cm}t\in[0,1].\end{equation}
First, we prove $x_{n}'(t)<L_{3}$ for $t\in[0,1]$. Suppose
$x_{n}'(t_{1})\geq L_{3}$ for some $t_{1}\in[0,1]$. Using
\eqref{3.3.5} and $(\mathbf{B}_{3})$, we have
\begin{align*}-x_{n}''(t)\leq p(t)k_{1}(\theta_{4}(y_{n}(t)))(u_{1}(\rho_{3}(x_{n}'(t)))+v_{1}(\rho_{3}(x_{n}'(t)))),\hspace{0.4cm}t\in(0,1),\end{align*}
which implies that
\begin{align*}\frac{-x_{n}''(t)}{u_{1}(\rho_{3}(x_{n}'(t)))+v_{1}(\rho_{3}(x_{n}'(t)))}\leq k_{1}(M_{4})p(t),\hspace{0.4cm}t\in(0,1).\end{align*}
Integrating from $t_{1}$ to $1$, using the BCs \eqref{3.3.5}, we
obtain
\begin{align*}\int_{\frac{1}{n}}^{x_{n}'(t_{1})}\frac{dz}{u_{1}(\rho_{3}(z))+v_{1}(\rho_{3}(z))}\leq k_{1}(M_{4})\int_{t_{1}}^{1}p(t)dt,\end{align*}
which can also be written as
\begin{align*}\int_{\frac{1}{n}}^{L_{3}}\frac{dz}{u_{1}(\rho_{3}(z))+v_{1}(\rho_{3}(z))}+\int_{L_{3}}^{x_{n}'(t_{1})}\frac{dz}{u_{1}(\rho_{3}(z))+v_{1}(\rho_{3}(z))}
\leq k_{1}(M_{4})\int_{0}^{1}p(t)dt.\end{align*} Using the
increasing property of $I$, we obtain
\begin{align*}I(L_{3})+\frac{x_{n}'(t_{1})-L_{3}}{u_{1}(L_{3})+v_{1}(L_{3})}\leq k_{1}(M_{4})\int_{0}^{1}p(t)dt+I(\varepsilon),\end{align*}
a contradiction to \eqref{3.3.3}. Hence, $x_{n}'(t)<L_{3}$ for
$t\in[0,1]$. Similarly, we can show that $y_{n}'(t)<L_{4}$ for $t\in[0,1]$.

\vskip 0.5em

Now, we show that
\begin{equation}\label{3.3.9}x_{n}(t)<M_{3},\,y_{n}(t)<M_{4},\hspace{0.4cm}t\in[0,1].\end{equation}
Suppose $x_{n}(t_{2})\geq M_{3}$ for some $t_{2}\in[0,1]$. From
\eqref{3.3.5}, \eqref{3.3.8} and $(\mathbf{B}_{3})$, it follows
that
\begin{align*}\begin{split}
-x_{n}''(t)&\leq p(t)k_{1}(\theta_{4}(y_{n}(t)))(u_{1}(x_{n}'(t))+v_{1}(x_{n}'(t))),\hspace{0.4cm}t\in(0,1),\\
-y_{n}''(t)&\leq
q(t)k_{2}(\theta_{3}(x_{n}(t)))(u_{2}(y_{n}'(t))+v_{2}(y_{n}'(t))),\hspace{0.4cm}t\in(0,1),
\end{split}\end{align*}
which implies that
\begin{align*}\begin{split}
\frac{-x_{n}''(t)}{u_{1}(x_{n}'(t))+v_{1}(x_{n}'(t))}&\leq k_{1}(\theta_{4}(\|y_{n}\|))p(t),\hspace{0.4cm}t\in(0,1),\\
\frac{-y_{n}''(t)}{u_{2}(y_{n}'(t))+v_{2}(y_{n}'(t))}&\leq
k_{2}(M_{3})q(t),\hspace{1.3cm}t\in(0,1).
\end{split}\end{align*}
Integrating from $t$ to $1$, using the BCs \eqref{3.3.5}, we obtain
\begin{align*}\begin{split}
\int_{\frac{1}{n}}^{x_{n}'(t)}\frac{dz}{u_{1}(z)+v_{1}(z)}&\leq k_{1}(\theta_{4}(\|y_{n}\|))\int_{t}^{1}p(s)ds,\hspace{0.4cm}t\in[0,1],\\
\int_{\frac{1}{n}}^{y_{n}'(t)}\frac{dz}{u_{2}(z)+v_{2}(z)}&\leq
k_{2}(M_{3})\int_{t}^{1}q(s)ds,\hspace{1.3cm}t\in[0,1],
\end{split}\end{align*}
which implies that
\begin{align*}\begin{split}
I(x_{n}'(t))-I(\frac{1}{n})&\leq k_{1}(\theta_{4}(\|y_{n}\|))\int_{0}^{1}p(s)ds,\hspace{0.4cm}t\in[0,1],\\
J(y_{n}'(t))-J(\frac{1}{n})&\leq
k_{2}(M_{3})\int_{0}^{1}q(s)ds,\hspace{1.3cm}t\in[0,1].
\end{split}\end{align*}
The increasing property of $I$ and $J$ leads to
\begin{equation}\label{3.3.10}x_{n}'(t)\leq I^{-1}(k_{1}(\theta_{4}(\|y_{n}\|))\int_{0}^{1}p(s)ds+I(\varepsilon)),\hspace{0.4cm}t\in[0,1],\end{equation}
\begin{equation}\label{3.3.11}y_{n}'(t)\leq J^{-1}(k_{2}(M_{3})\int_{0}^{1}q(s)ds+J(\varepsilon)),\hspace{1.4cm}t\in[0,1].\end{equation}
Integrating \eqref{3.3.10} from $0$ to $t_{2}$ and \eqref{3.3.11}
from $0$ to $1$, using the BCs \eqref{3.3.5}, \eqref{3.3.10} and
\eqref{3.3.11}, we obtain
\begin{equation}\label{3.3.12}
M_{3}\leq x_{n}(t_{2})\leq(1+\frac{b_{1}}{a_{1}})I^{-1}(k_{1}(\theta_{4}(\|y_{n}\|))\int_{0}^{1}p(s)ds+I(\varepsilon)),
\end{equation}
\begin{equation}\label{3.3.13}\|y_{n}\|\leq(1+\frac{b_{2}}{a_{2}})J^{-1}(k_{2}(M_{3})\int_{0}^{1}q(s)ds+J(\varepsilon)).\end{equation}
Either we have $\|y_{n}\|<M_{4}$ or $\|y_{n}\|\geq M_{4}$. If
$\|y_{n}\|<M_{4}$, then from \eqref{3.3.12}, we have
\begin{equation}\label{3.3.14}M_{3}\leq(1+\frac{b_{1}}{a_{1}})I^{-1}(k_{1}(\|y_{n}\|)\int_{0}^{1}p(s)ds+I(\varepsilon)),\end{equation}
Now, by using \eqref{3.3.13} in \eqref{3.3.14} and the increasing
property of $k_{1}$ and $I^{-1}$, we obtain
\begin{align*}M_{3}\leq(1+\frac{b_{1}}{a_{1}})I^{-1}(k_{1}((1+\frac{b_{2}}{a_{2}})J^{-1}(k_{2}(M_{3})\int_{0}^{1}q(s)ds+J(\varepsilon)))
\int_{0}^{1}p(s)ds+I(\varepsilon)),\end{align*} which implies that
\begin{align*}\frac{M_{3}}{(1+\frac{b_{1}}{a_{1}})I^{-1}(k_{1}((1+\frac{b_{2}}{a_{2}})J^{-1}(k_{2}(M_{3})\int_{0}^{1}q(s)ds+J(\varepsilon)))
\int_{0}^{1}p(s)ds+I(\varepsilon))}\leq1,\end{align*} a
contradiction to \eqref{3.3.1}.

\vskip 0.5em

On the other hand, if $\|y_{n}\|\geq M_{4}$, then from
\eqref{3.3.12} and \eqref{3.3.13}, we have
\begin{equation}\label{3.3.15}M_{3}\leq(1+\frac{b_{1}}{a_{1}})I^{-1}(k_{1}(M_{4})\int_{0}^{1}p(s)ds+I(\varepsilon)),\end{equation}
\begin{equation}\label{3.3.16}M_{4}\leq(1+\frac{b_{2}}{a_{2}})J^{-1}(k_{2}(M_{3})\int_{0}^{1}q(s)ds+J(\varepsilon)).\end{equation}
Using \eqref{3.3.16} in \eqref{3.3.15} and the increasing property
of $k_{1}$ and $I^{-1}$, we obtain
\begin{align*}
M_{3}\leq(1+\frac{b_{1}}{a_{1}})I^{-1}(k_{1}((1+\frac{b_{2}}{a_{2}})J^{-1}(k_{2}(M_{3})\int_{0}^{1}q(s)ds+J(\varepsilon)))\int_{0}^{1}p(s)ds+I(\varepsilon)),
\end{align*}
which implies that
\begin{align*}
\frac{M_{3}}{(1+\frac{b_{1}}{a_{1}})I^{-1}(k_{1}((1+\frac{b_{2}}{a_{2}})J^{-1}(k_{2}(M_{3})\int_{0}^{1}q(s)ds+J(\varepsilon)))\int_{0}^{1}p(s)ds+I(\varepsilon))}\leq1,
\end{align*}
a contradiction to \eqref{3.3.1}. Hence, $x_{n}(t)<M_{3}$ for
$t\in[0,1]$. Similarly, we can show that $y_{n}(t)<M_{4}$ for $t\in[0,1]$.

\vskip 0.5em

Thus, in view of \eqref{3.3.5}-\eqref{3.3.9}, $(x_{n},y_{n})$ is a
solution of the following coupled system of BVPs
\begin{equation}\label{3.3.17}\begin{split}
-x''(t)&=p(t)f(t,y(t),x'(t)),\hspace{0.4cm}t\in(0,1),\\
-y''(t)&=q(t)g(t,x(t),y'(t)),\hspace{0.4cm}t\in(0,1),\\
a_{1}x(&0)-b_{1}x'(0)=0,\,x'(1)=\frac{1}{n},\\
a_{2}y(&0)-b_{2}y'(0)=0,\,y'(1)=\frac{1}{n},
\end{split}\end{equation}
satisfy
\begin{equation}\label{3.3.18}\begin{split}
(t+\frac{b_{1}}{a_{1}})\frac{1}{n}&\leq
x_{n}(t)<M_{3},\,\frac{1}{n}\leq
x_{n}'(t)<L_{3},\hspace{0.4cm}t\in[0,1],\\
(t+\frac{b_{2}}{a_{2}})\frac{1}{n}&\leq
y_{n}(t)<M_{4},\,\frac{1}{n}\leq
y_{n}'(t)<L_{4},\hspace{0.4cm}t\in[0,1].
\end{split}\end{equation}
Now, in view of $(\mathbf{B}_{6})$, there exist continuous
functions $\varphi_{M_{4}L_{3}}$ and $\psi_{M_{3}L_{4}}$ defined on
$[0,1]$ and positive on $(0,1)$, and real constants
$0\leq\delta_{1},\delta_{2}<1$ such that
\begin{equation}\label{3.3.19}\begin{split}
f(t,y_{n}(t),x_{n}'(t))&\geq\varphi_{M_{4}L_{3}}(t)(y_{n}(t))^{\delta_{1}},\hspace{0.2cm}(t,y_{n}(t),x_{n}'(t))\in[0,1]\times[0,M_{4}]\times[0,L_{3}],\\
g(t,x_{n}(t),y_{n}'(t))&\geq\psi_{M_{3}L_{4}}(t)(x_{n}(t))^{\delta_{2}},\hspace{0.2cm}(t,x_{n}(t),y_{n}'(t))\in[0,1]\times[0,M_{3}]\times[0,L_{4}].
\end{split}\end{equation}
We claim that
\begin{equation}\label{3.3.20}x_{n}'(t)\geq C_{6}^{\delta_{1}}\int_{t}^{1}p(s)\varphi_{M_{4}L_{3}}(s)ds,\end{equation}
\begin{equation}\label{3.3.21}y_{n}'(t)\geq C_{5}^{\delta_{2}}\int_{t}^{1}q(s)\psi_{M_{3}L_{4}}(s)ds,\end{equation}
where
\begin{align*}
C_{5}=&\left(\frac{b_{1}}{a_{1}}\right)^{\frac{1}{1-\delta_{1}\delta_{2}}}
\left(\frac{b_{2}}{a_{2}}\right)^{\frac{\delta_{1}}{1-\delta_{1}\delta_{2}}}
\left(\int_{0}^{1}p(t)\varphi_{M_{4}L_{3}}(t)dt\right)^{\frac{1}{1-\delta_{1}
\delta_{2}}}
\left(\int_{0}^{1}q(t)\psi_{M_{3}L_{4}}(t)dt\right)^{\frac{\delta_{1}}{1-\delta_{1}\delta_{2}}},\\
C_{6}=&\left(\frac{b_{1}}{a_{1}}\right)^{\frac{\delta_{2}}{1-\delta_{1}\delta_{2}}}\left(\frac{b_{2}}{a_{2}}\right)^{\frac{1}{1-\delta_{1}\delta_{2}}}
\left(\int_{0}^{1}p(t)\varphi_{M_{4}L_{3}}(t)dt\right)^{\frac{\delta_{2}}{1-\delta_{1}\delta_{2}}}
\left(\int_{0}^{1}q(t)\psi_{M_{3}L_{4}}(t)dt\right)^{\frac{1}{1-\delta_{1}
\delta_{2}}}.
\end{align*}
To prove \eqref{3.3.20}, consider the following relation
\begin{equation}\label{3.3.22}\begin{split}
x_{n}(t)=&(t+\frac{b_{1}}{a_{1}})\frac{1}{n}+\frac{1}{a_{1}}\int_{0}^{t}(a_{1}s+b_{1})p(s)f(s,y_{n}(s),x_{n}'(s))ds\\
&+\frac{1}{a_{1}}\int_{t}^{1}(a_{1}t+b_{1})p(s)f(s,y_{n}(s),x_{n}'(s))ds,\hspace{0.4cm}t\in[0,1],
\end{split}\end{equation} which implies that
\begin{align*}
x_{n}(0)=\frac{b_{1}}{a_{1}}\frac{1}{n}+\frac{b_{1}}{a_{1}}\int_{0}^{1}p(s)f(s,y_{n}(s),x_{n}'(s))ds.
\end{align*}
Using \eqref{3.3.19} and \eqref{3.3.18}, we obtain
\begin{equation}\label{3.3.23}\begin{split}
x_{n}(0)\geq&\frac{b_{1}}{a_{1}}\int_{0}^{1}p(s)\varphi_{M_{4}L_{3}}(s)(y_{n}(s))^{\delta_{1}}ds\geq(y_{n}(0))^{\delta_{1}}\frac{b_{1}}{a_{1}}\int_{0}^{1}p(s)\varphi_{M_{4}L_{3}}(s)ds.
\end{split}\end{equation} Similarly, using \eqref{3.3.19} and \eqref{3.3.18}, we obtain
\begin{align*}y_{n}(0)\geq(x_{n}(0))^{\delta_{2}}\frac{b_{2}}{a_{2}}\int_{0}^{1}q(s)\psi_{M_{3}L_{4}}(s)ds,\end{align*}
which in view of \eqref{3.3.23} implies that
\begin{align*}
y_{n}(0)\geq(y_{n}(0))^{\delta_{1}\delta_{2}}\left(\frac{b_{1}}{a_{1}}\int_{0}^{1}p(s)\varphi_{M_{4}L_{3}}(s)ds\right)^{\delta_{2}}
\frac{b_{2}}{a_{2}}\int_{0}^{1}q(s)\psi_{M_{3}L_{4}}(s)ds.
\end{align*}
Hence,
\begin{equation}\label{3.3.24}y_{n}(0)\geq C_{6}.\end{equation}
Now, from  \eqref{3.3.22}, it follows that
\begin{align*}x_{n}'(t)\geq\int_{t}^{1}p(s)f(s,y_{n}(s),x_{n}'(s))ds,\end{align*}
and using \eqref{3.3.19} and \eqref{3.3.24}, we obtain
\eqref{3.3.20}. Similarly, we can prove \eqref{3.3.21}.

\vskip 0.5em

Now, using \eqref{3.3.17}, $(\mathbf{B}_{3})$, \eqref{3.3.18},
\eqref{3.3.20} and \eqref{3.3.21}, we have
\begin{equation}\label{3.3.25}\begin{split}
0\leq-x_{n}''(t)&\leq k_{1}(M_{4})p(t)(u_{1}(C_{6}^{\delta_{1}}\int_{t}^{1}p(s)\varphi_{M_{4}L_{3}}(s)ds)+v_{1}(L_{3})),\hspace{0.3cm}t\in(0,1),\\
0\leq-y_{n}''(t)&\leq
k_{2}(M_{3})q(t)(u_{2}(C_{5}^{\delta_{2}}\int_{t}^{1}q(s)\psi_{M_{3}}(s)ds)+v_{2}(L_{4})),\hspace{0.65cm}t\in(0,1).
\end{split}\end{equation}
In view of \eqref{3.3.18}, \eqref{3.3.25}, $(\mathbf{B}_{1})$ and
$(\mathbf{B}_{12})$, it follows that the sequences
$\{(x_{n}^{(j)},y_{n}^{(j)})\}$ $(j=0,1)$ are uniformly bounded and
equicontinuous on $[0,1]$. Hence, by Theorem \eqref{arzela}, there
exist subsequences $\{(x_{n_{k}}^{(j)},y_{n_{k}}^{(j)})\}\ \
(j=0,1)$ of $\{(x_{n}^{(j)},y_{n}^{(j)})\}$ $(j=0,1)$ and $(x,y)\in
C^{1}[0,1]\times C^{1}[0,1]$ such that
$(x_{n_{k}}^{(j)},y_{n_{k}}^{(j)})$ converges uniformly to
$(x^{(j)},y^{(j)})$ on $[0,1]$ $(j=0,1)$. Also,
$a_{1}x(0)-b_{1}x'(0)=a_{2}y(0)-b_{2}y'(0)=x'(1)=y'(1)=0$. Moreover,
from \eqref{3.3.20} and \eqref{3.3.21}, with $n_{k}$ in place of $n$
and taking $\lim_{n_k\rightarrow+\infty}$, we have
\begin{align*}
x'(t)\geq&C_{6}^{\delta_{1}}\int_{t}^{1}p(s)\varphi_{M_{4}L_{3}}(s)ds,\\
y'(t)\geq&C_{5}^{\delta_{2}}\int_{t}^{1}q(s)\psi_{M_{3}L_{4}}(s)ds,
\end{align*}
which shows that $x'>0$ and $y'>0$ on $[0,1)$, $x>0$ and $y>0$ on
$[0,1]$. Further, $(x_{n_{k}},y_{n_{k}})$ satisfy
\begin{align*}x_{n_{k}}'(t)=&x_{n_{k}}'(0)-\int_{0}^{t}p(s)f(s,y_{n_{k}}(s),x_{n_{k}}'(s))ds,\hspace{0.4cm}t\in[0,1],\\
y_{n_{k}}'(t)=&y_{n_{k}}'(0)-\int_{0}^{t}q(s)f(s,x_{n_{k}}(s),y_{n_{k}}'(s))ds,\hspace{0.4cm}t\in[0,1].\end{align*}
Passing to the limit as $n_{k}\rightarrow\infty$, we obtain
\begin{align*}x'(t)=&x'(0)-\int_{0}^{t}p(s)f(s,y(s),x'(s))ds,\hspace{0.4cm}t\in[0,1],\\
y'(t)=&y'(0)-\int_{0}^{t}q(s)f(s,x(s),y'(s))ds,\hspace{0.4cm}t\in[0,1],\end{align*}
which implies that
\begin{align*}-x''(t)=&p(t)f(t,y(t),x'(t)),\hspace{0.4cm}t\in(0,1),\\
-y''(t)=&q(t)f(t,x(t),y'(t)),\hspace{0.4cm}t\in(0,1).\end{align*}
Hence, $(x,y)$ is a $C^{1}$-positive solution of the system of SBVPs \eqref{3.0.4}.
\end{proof}

\begin{ex}
Consider the following coupled system of singular BVPs
\begin{equation}\label{3.3.26}\begin{split}
-&x''(t)=(1-t)^{-\frac{3}{4}}(y(t))^{\frac{1}{4}}(x'(t))^{-\beta_{1}},\hspace{0.47cm}t\in(0,1),\\
-&y''(t)=(1-t)^{-\frac{1}{4}}(x(t))^{\frac{3}{4}}(y'(t))^{-\beta_{2}},\hspace{0.47cm}t\in(0,1),\\
&x(0)-x'(0)=y(0)-y'(0)=x'(1)=y'(1)=0,
\end{split}\end{equation}
where $0<\beta_{1}<1$ and $0<\beta_{2}<1$.

\vskip 0.5em

Choose $p(t)=(1-t)^{-\frac{3}{4}}$, $q(t)=(1-t)^{-\frac{1}{4}}$,
$k_{1}(x)=x^{\frac{1}{4}}$, $k_{2}(x)=x^{\frac{3}{4}}$,
$u_{1}(x)=x^{-\beta_{1}}$, $u_{2}(x)=x^{-\beta_{2}}$ and
$v_{1}(x)=v_{2}(x)=0$. Also $\delta_{1}=\frac{1}{4}$, $\delta_{2}=\frac{3}{4}$,
$\varphi_{EF}(t)=F^{-\beta_{1}}$ and $\psi_{EF}(t)=F^{-\beta_{2}}$. Then,
$I(z)=\frac{z^{\beta_{1}+1}}{\beta_{1}+1}$,
$J(z)=\frac{z^{\beta_{2}+1}}{\beta_{2}+1}$,
$I^{-1}(z)=(\beta_{1}+1)^{\frac{1}{\beta_{1}+1}}z^{\frac{1}{\beta_{1}+1}}$
and
$J^{-1}(z)=(\beta_{2}+1)^{\frac{1}{\beta_{2}+1}}z^{\frac{1}{\beta_{2}+1}}$. Then, $\int_{0}^{1}p(t)dt=4$ and $\int_{0}^{1}q(t)dt=\frac{4}{3}$.

\vskip 0.5em

Clearly, $(\mathbf{B}_{1})-(\mathbf{B}_{3})$,
$(\mathbf{B}_{5})$ and $(\mathbf{B}_{6})$ are satisfied.
Moreover,
\begin{align*}\begin{split}
&\sup_{c\in(0,\infty)}\frac{c}{(1+\frac{b_{1}}{a_{1}})I^{-1}(k_{1}((1+\frac{b_{2}}{a_{2}})J^{-1}(k_{2}(c)\int_{0}^{1}q(t)dt))\int_{0}^{1}p(t)dt)}=\\
&\sup_{c\in(0,\infty)}\frac{c}{2^{1+\frac{9}{4(\beta_{1}+1)}}(\frac{4}{3})^{\frac{1}{4(\beta_{1}+1)(\beta_{2}+1)}}
(\beta_{1}+1)^{\frac{1}{\beta_{1}+1}}(\beta_{2}+1)^{\frac{1}{4(\beta_{1}+1)(\beta_{2}+1)}}
c^{\frac{3}{16(\beta_{1}+1)(\beta_{2}+1)}}}=\infty,
\end{split}\end{align*}

\begin{align*}\begin{split}
&\sup_{x\in(0,\infty)}\frac{c}{(1+\frac{b_{2}}{a_{2}})J^{-1}(k_{2}((1+\frac{b_{1}}{a_{1}})I^{-1}(k_{1}(c)\int_{0}^{1}p(t)dt))\int_{0}^{1}q(t)dt)}=\\
&\sup_{c\in(0,\infty)}\frac{c}{2^{1+\frac{3}{4(\beta_{2}+1)}(1+\frac{2}{\beta_{1}+1})}(\frac{4}{3})^{\frac{1}{\beta_{2}+1}}
(\beta_{2}+1)^{\frac{1}{\beta_{2}+1}}(\beta_{1}+1)^{\frac{3}{4(\beta_{1}+1)(\beta_{2}+1)}}
c^{\frac{3}{16(\beta_{1}+1)(\beta_{2}+1)}}}=\infty,
\end{split}\end{align*}

\begin{align*}\begin{split}
&\int_{0}^{1}p(t)u_{1}(C\int_{t}^{1}p(s)\varphi_{EF}(s)ds)dt=4^{1-\beta_{1}}C^{-\beta_{1}}F^{\beta_{1}^{2}}\int_{0}^{1}(1-t)^{-\frac{\beta_{1}+3}{4}}dt=\frac{4^{1-\beta_{1}}C^{-\beta_{1}}F^{\beta_{1}^{2}}}{1-\beta_{1}},\\
&\int_{0}^{1}q(t)u_{2}(C\int_{t}^{1}q(s)\psi_{EF}(s)ds)dt=(\frac{4}{3})^{-\beta_{2}}C^{-\beta_{2}}F^{\beta_{2}^{2}}\int_{0}^{1}(1-t)^{-\frac{3\beta_{2}+1}{4}}dt\\
&\hspace{5.55cm}=\frac{(4/3)^{1-\beta_{2}}C^{-\beta_{2}}F^{\beta_{2}^{2}}}{1-\beta_{2}},
\end{split}\end{align*} which shows that $(\mathbf{B}_{11})$ and $(\mathbf{B}_{12})$ also holds.

\vskip 0.5em

Since, $(\mathbf{B}_{1})-(\mathbf{B}_{3})$, $(\mathbf{B}_{5})$, $(\mathbf{B}_{6})$, $(\mathbf{B}_{11})$ and $(\mathbf{B}_{12})$ are satisfied. Therefore, by Theorem \ref{th2.1}, the system of BVPs \eqref{3.3.26} has at least one $C^{1}$-positive solution.
\end{ex}


\section{Existence of at least two positive solutions with more general BCs}\label{multiplicity-two}

In this section, we establish at lest two $C^{1}$-positive solutions
to the system of SBVPs \eqref{3.0.4}. For each $(x,y)\in C^{1}[0,1]\times C^{1}[0,1]$, we write $\|(x,y)\|_{5}=\|x\|_{3}+\|y\|_{3}$. Clearly, $(C^{1}[0,1]\times C^{1}[0,1],\|\cdot\|_{5})$ is a Banach space. We define a partial ordering in $C^{1}[0,1]$, by $x\leq y$ if and only if $x(t)\leq
y(t)$, $t\in[0,1]$. We define a partial ordering in $C^{1}[0,1]\times C^{1}[0,1]$, by $(x_{1},y_{1})\preceq (x_{2},y_{2})$ if and only if $x_{1}\leq x_{2}$ and $y_{1}\leq y_{2}$. Let
\begin{align*}
P_{i}=\{x\in C^{1}[0,1]:x(t)\geq \gamma_{i}\|x\|\text{ for all }t\in[0,1], x(0)\geq\frac{b_{i}}{a_{i}}\|x'\|\},
\end{align*}
where $\gamma_{i}=\frac{b_{i}}{a_{i}+b_{i}}$, $i=1,2$. Clearly, $P_{i}$ $(i=1,2)$ are cones of $C^{1}[0,1]$ and $P_{1}\times P_{2}$ is a cone of $C^{1}[0,1]\times C^{1}[0,1]$. For any real constant $r>0$, we define an open neighborhood of $(0,0)\in C^{1}[0,1]\times C^{1}[0,1]$ as
\begin{align*}
\mathcal{O}_{r}=\{(x,y)\in C^{1}[0,1]\times C^{1}[0,1]:\|(x,y)\|_{5}<r\}.
\end{align*}

\vskip 0.5em

In view of $(\mathbf{B}_{11})$, there exist real constants
$R_{3}>0$ and $R_{4}>0$ such that
\begin{equation}\label{3.4.1}
\frac{R_{3}}{(1+\frac{b_{1}}{a_{1}})I^{-1}(k_{1}((1+\frac{b_{2}}{a_{2}})J^{-1}(k_{2}(R_{3})\int_{0}^{1}q(t)dt))\int_{0}^{1}p(t)dt)}>1,
\end{equation}
\begin{equation}\label{3.4.2}
\frac{R_{4}}{(1+\frac{b_{2}}{a_{2}})J^{-1}(k_{2}((1+\frac{b_{1}}{a_{1}})I^{-1}(k_{1}(R_{4})\int_{0}^{1}p(t)dt))\int_{0}^{1}q(t)dt)}>1.
\end{equation}
From the continuity of $k_{1}$, $k_{2}$, $I$ and $J$, we choose
$\varepsilon>0$ small enough such that
\begin{equation}\label{3.4.3}\frac{R_{3}}{(1+\frac{b_{1}}{a_{1}})I^{-1}(k_{1}((1+\frac{b_{2}}{a_{2}})J^{-1}(k_{2}(R_{3}+\varepsilon)
\int_{0}^{1}q(t)dt+J(\varepsilon))+\varepsilon)\int_{0}^{1}p(t)dt+I(\varepsilon))}>1,\end{equation}
\begin{equation}\label{3.4.4}\frac{R_{4}}{(1+\frac{b_{2}}{a_{2}})J^{-1}(k_{2}((1+\frac{b_{1}}{a_{1}})I^{-1}(k_{1}(R_{4}+\varepsilon)
\int_{0}^{1}p(t)dt+I(\varepsilon))+\varepsilon)\int_{0}^{1}q(t)dt+J(\varepsilon))}>1.\end{equation}
Choose $n_{0}\in\{1,2,\cdots\}$ such that
$\max\{\frac{1}{n_{0}}(1+\frac{b_{1}}{a_{1}}),\frac{1}{n_{0}}(1+\frac{b_{2}}{a_{2}})\}<\varepsilon$
and for each fixed $n\in\{n_{0},n_{0}+1,\cdots\}$, consider the
system of non-singular BVPs
\begin{equation}\label{3.4.5}\begin{split}-x''(t)&=p(t)f(t,y(t)+\frac{1}{n}(t+\frac{b_{2}}{a_{2}}),|x'(t)|+\frac{1}{n}),\hspace{0.4cm}t\in(0,1),\\
-y''(t)&=q(t)g(t,x(t)+\frac{1}{n}(t+\frac{b_{1}}{a_{1}}),|y'(t)|+\frac{1}{n}),\hspace{0.4cm}t\in(0,1),\\a_{1}x(0)&-b_{1}x'(0)=a_{2}y(0)-b_{2}y'(0)=x'(1)=y'(1)=0.
\end{split}\end{equation}
We write \eqref{3.4.5} as an equivalent system of integral equations
\begin{equation}\label{3.4.6}\begin{split}
x(t)&=\int_{0}^{1}G_{1}(t,s)p(s)f(s,y(s)+\frac{1}{n}(s+\frac{b_{2}}{a_{2}}),|x'(s)|+\frac{1}{n})ds,
\hspace{0.4cm}t\in[0,1],\\
y(t)&=\int_{0}^{1}G_{2}(t,s)q(s)f(s,x(s)+\frac{1}{n}(s+\frac{b_{1}}{a_{1}}),|y'(s)|+\frac{1}{n})ds,
\hspace{0.4cm}t\in[0,1],
\end{split}\end{equation}
where
\begin{align*}
G_{i}(t,s)=\frac{1}{a_{i}}\begin{cases}b_{i}+a_{i}s,\,& 0\leq s\leq t\leq1,\\
b_{i}+a_{i}t,\,& 0\leq t\leq s\leq1,\,i=1,2.\end{cases}
\end{align*}

By a solution of the system of BVPs \eqref{3.4.5}, we mean a
solution of the corresponding system of integral equations
\eqref{3.4.6}.

\vskip 0.5em

Define a map $T_{n}:C^{1}[0,1]\times C^{1}[0,1]\rightarrow
C^{1}[0,1]\times C^{1}[0,1]$ by
\begin{equation}\label{3.4.7}T_{n}(x,y)=(A_{n}(x,y),B_{n}(x,y)),\end{equation}
where the maps $A_{n},B_{n}:C^{1}[0,1]\times C^{1}[0,1]\rightarrow C^{1}[0,1]$ are defined by
\begin{equation}\label{3.4.8}\begin{split}
A_{n}(x,y)(t)&=\int_{0}^{1}G_{1}(t,s)p(s)f(s,y(s)+\frac{1}{n}(s+\frac{b_{2}}{a_{2}}),|x'(s)|+\frac{1}{n})ds,\hspace{0.4cm}t\in[0,1],\\
B_{n}(x,y)(t)&=\int_{0}^{1}G_{2}(t,s)q(s)f(s,x(s)+\frac{1}{n}(s+\frac{b_{1}}{a_{1}}),|y'(s)|+\frac{1}{n})ds,\hspace{0.4cm}t\in[0,1].
\end{split}\end{equation}
Clearly, if $(x_{n},y_{n})\in C^{1}[0,1]\times C^{1}[0,1]$ is a fixed point of $T_{n}$; then $(x_{n},y_{n})$ is a solution of the system of BVPs \eqref{3.4.5}.

\begin{lem}\label{2.1}
Assume that $(\mathbf{B}_{1})-(\mathbf{B}_{3})$ hold. Then the
map $T_{n}:\overline{\mathcal{O}}_{r}\cap(P_{1}\times P_{2})\rightarrow P_{1}\times P_{2}$ is completely continuous.
\end{lem}

\begin{proof} Firstly, we show that $T_{n}(P_{1}\times P_{2})\subseteq P_{1}\times P_{2}$. For $(x,y)\in P_{1}\times P_{2}$, $t\in[0,1]$, using
\eqref{3.4.8} and Lemma \ref{lemmax2}, we obtain
\begin{equation}\label{3.4.9}\begin{split}
A_{n}(x,y)(t)&=\int_{0}^{1}G_{1}(t,s)p(s)f(s,y(s)+\frac{1}{n}(s+\frac{b_{2}}{a_{2}}),|x'(s)|+\frac{1}{n})ds\\
&\geq\gamma_{1}\max_{\tau\in[0,1]}\int_{0}^{1}G_{1}(\tau,s)p(s)f(s,y(s)+\frac{1}{n}(s+\frac{b_{2}}{a_{2}}),|x'(s)|+\frac{1}{n})ds\\
&=\gamma_{1}\|A_{n}(x,y)\|\end{split}\end{equation}
and
\begin{equation}\label{3.4.10}\begin{split}
A_{n}(x,y)(0)&=\int_{0}^{1}G_{1}(0,s)p(s)f(s,y(s)+\frac{1}{n}(s+\frac{b_{2}}{a_{2}}),|x'(s)|+\frac{1}{n})ds\\
&=\frac{b_{1}}{a_{1}}\max_{\tau\in[0,1]}\int_{\tau}^{1}p(s)f(s,y(s)+\frac{1}{n}(s+\frac{b_{2}}{a_{2}}),|x'(s)|+\frac{1}{n})ds\\
&=\frac{b_{1}}{a_{1}}\max_{\tau\in[0,1]}|A_{n}(x,y)'(\tau)|\\
&=\frac{b_{1}}{a_{1}}\|A_{n}(x,y)'\|.
\end{split}\end{equation}
From \eqref{3.4.9} and \eqref{3.4.10}, $A_{n}(x,y)\in P_{1}$ for
every $(x,y)\in P_{1}\times P_{2}$, that is, $A_{n}(P_{1}\times
P_{2})\subseteq P_{1}$. Similarly, by using \eqref{3.4.8} and Lemma
\ref{lemmax2}, we can show that $B_{n}(P_{1}\times P_{2})\subseteq
P_{2}$. Hence, $T_{n}(P_{1}\times P_{2})\subseteq P_{1}\times
P_{2}$.

\vskip 0.5em

Now, we show that $T_{n}:\overline{\mathcal{O}}_{r}\cap(P_{1}\times
P_{2})\rightarrow P_{1}\times P_{2}$ is uniformly bounded. For any
$(x,y)\in\overline{\mathcal{O}}_{r}\cap(P_{1}\times P_{2})$, using
\eqref{3.4.8}, Lemma \ref{lemmax2}, $(\mathbf{B}_{1})$ and
$(\mathbf{B}_{3})$, we have
\begin{equation}\label{3.4.11}\begin{split}
&\|A_{n}(x,y)\|=\max_{t\in[0,1]}\left|\int_{0}^{1}G_{1}(t,s)p(s)f(s,y(s)+\frac{1}{n}(s+\frac{b_{2}}{a_{2}}),|x'(s)|+\frac{1}{n})ds\right|\\
&\leq\frac{1}{a_{1}}\int_{0}^{1}(a_{1}s+b_{1})p(s)k_{1}(y(s)+\frac{1}{n}(s+\frac{b_{2}}{a_{2}}))(u_{1}(|x'(s)|+\frac{1}{n})+v_{1}(|x'(s)|+\frac{1}{n}))ds\\
&\leq\frac{1}{a_{1}}k_{1}(r+\frac{1}{n}(1+\frac{b_{2}}{a_{2}}))(u_{1}(\frac{1}{n})+v_{1}(r+\frac{1}{n}))\int_{0}^{1}(a_{1}s+b_{1})p(s)ds<+\infty,
\end{split}\end{equation}
\begin{equation}\label{3.4.12}\begin{split}
&\|A_{n}(x,y)'\|=\max_{\tau\in[0,1]}|A_{n}(x,y)'(\tau)|\\
&=\max_{\tau\in[0,1]}\int_{\tau}^{1}p(s)f(s,y(s)+\frac{1}{n}(s+\frac{b_{2}}{a_{2}}),|x'(s)|+\frac{1}{n})ds\\
&=\int_{0}^{1}p(s)f(s,y(s)+\frac{1}{n}(s+\frac{b_{2}}{a_{2}}),|x'(s)|+\frac{1}{n})ds\\
&\leq\int_{0}^{1}p(s)k_{1}(y(s)+\frac{1}{n}(s+\frac{b_{2}}{a_{2}}))(u_{1}(|x'(s)|+\frac{1}{n})+v_{1}(|x'(s)|+\frac{1}{n}))ds\\
&\leq k_{1}(r+\frac{1}{n}(1+\frac{b_{2}}{a_{2}}))(u_{1}(\frac{1}{n})+v_{1}(r+\frac{1}{n}))\int_{0}^{1}p(s)ds<+\infty.
\end{split}\end{equation}
From \eqref{3.4.11} and \eqref{3.4.12}, it follows that
$A_{n}(\overline{\mathcal{O}}_{r}\cap(P_{1}\times P_{2}))$ is
uniformly bounded under $\|\cdot\|_{3}$. Similarly, by using
\eqref{3.4.8}, Lemma \ref{lemmax2}, $(\mathbf{B}_{1})$ and
$(\mathbf{B}_{3})$, we can show that
$B_{n}(\overline{\mathcal{O}}_{r}\cap(P_{1}\times P_{2}))$ is
uniformly bounded under $\|\cdot\|_{3}$. Hence,
$T_{n}(\overline{\mathcal{O}}_{r}\cap(P_{1}\times P_{2}))$ is
uniformly bounded.

\vskip 0.5em

Now, we show that $T_{n}(\overline{\mathcal{O}}_{r}\cap(P_{1}\times
P_{2}))$ is equicontinuous. For any
$(x,y)\in\overline{\mathcal{O}}_{r}\cap(P_{1}\times P_{2})$ and
$t_{1},t_{2}\in[0,1]$, using \eqref{3.4.8} and $(\mathbf{B}_{3})$,
we have
\begin{equation}\label{3.4.13}\begin{split}
&|A_{n}(x,y)(t_{1})-A_{n}(x,y)(t_{2})|\\
&=\left|\int_{0}^{1}(G_{1}(t_{1},s)-G_{1}(t_{2},s))p(s)f(s,y(s)+\frac{1}{n}(s+\frac{b_{2}}{a_{2}}),|x'(s)|+\frac{1}{n})ds\right|\\
&\leq\int_{0}^{1}|G_{1}(t_{1},s)-G_{1}(t_{2},s)|p(s)f(s,y(s)+\frac{1}{n}(s+\frac{b_{2}}{a_{2}}),|x'(s)|+\frac{1}{n})ds\\
&\leq\int_{0}^{1}|G_{1}(t_{1},s)-G_{1}(t_{2},s)|p(s)k_{1}(y(s)+\frac{1}{n}(s+\frac{b_{2}}{a_{2}}))(u_{1}(|x'(s)|+\frac{1}{n})+v_{1}(|x'(s)|+\frac{1}{n}))ds\\
&\leq
k_{1}(r+\frac{1}{n}(1+\frac{b_{2}}{a_{2}}))(u_{1}(\frac{1}{n})+v_{1}(r+\frac{1}{n}))\int_{0}^{1}|G(t_{1},s)-G(t_{2},s)|p(s)ds,
\end{split}\end{equation}
\begin{equation}\label{3.4.14}\begin{split}
&|A_{n}(x,y)'(t_{1})-A_{n}(x,y)'(t_{2})|=\left|\int_{t_{1}}^{t_{2}}p(s)f(s,y(s)+\frac{1}{n}(s+\frac{b_{2}}{a_{2}}),|x'(s)|+\frac{1}{n})ds\right|\\
&\leq\int_{t_{1}}^{t_{2}}p(s)k_{1}(y(s)+\frac{1}{n}(s+\frac{b_{2}}{a_{2}}))(u_{1}(|x'(s)|+\frac{1}{n})+v_{1}(|x'(s)|+\frac{1}{n}))ds\\
&\leq
k_{1}(r+\frac{1}{n}(1+\frac{b_{2}}{a_{2}}))(u_{1}(\frac{1}{n})+v_{1}(r+\frac{1}{n}))\int_{t_{1}}^{t_{2}}p(s)ds.
\end{split}\end{equation}
From \eqref{3.4.13}, \eqref{3.4.14} and $(\mathbf{B}_{1})$, it
follows that $A_{n}(\overline{\mathcal{O}}_{r}\cap(P_{1}\times
P_{2}))$ is equicontinuous under the norm $\|\cdot\|_{3}$.
Similarly, using \eqref{3.4.8} and $(\mathbf{B}_{3})$, we can show
that $B_{n}(\overline{\mathcal{O}}_{r}\cap(P_{1}\times P_{2}))$ is
equicontinuous under $\|\cdot\|_{3}$. Consequently,
$T_{n}(\overline{\mathcal{O}}_{r}\cap(P_{1}\times P_{2}))$ is
equicontinuous. Hence, by Theorem \ref{arzela},
$T_{n}(\overline{\mathcal{O}}_{r}\cap(P_{1}\times P_{2}))$ is
relatively compact which implies that $T_{n}$ is a compact map. Further, we show that $T_{n}$ is continuous. Let $(x_{m},y_{m}),(x,y)\in\overline{\mathcal{O}}_{r}\cap(P_{1}\times
P_{2})$ such that
\begin{align*}
\|(x_{m},y_{m})-(x,y)\|_{5}\rightarrow0\text{ as
}m\rightarrow+\infty.
\end{align*}
Using $(\mathbf{B}_{3})$, we have
\begin{align*}\begin{split}
&\left|f(t,y_{m}(t)+\frac{1}{n}(t+\frac{b_{2}}{a_{2}}),|x_{m}'(t)|+\frac{1}{n})\right|\leq
k_{1}(y_{m}(t)+\frac{1}{n}(t+\frac{b_{2}}{a_{2}}))\\
&(u_{1}(|x_{m}'(t)|+\frac{1}{n})+v_{1}(|x_{m}'(t)|+\frac{1}{n}))\leq
k_{1}(r+\frac{1}{n}(1+\frac{b_{2}}{a_{2}}))(u_{1}(\frac{1}{n})+v_{1}(r+\frac{1}{n})).
\end{split}\end{align*}
Using \eqref{3.4.8} and Lemma \ref{lemmax2}, we have
\begin{equation}\label{3.4.15}\begin{split}
&\|A_{n}(x_{m},y_{m})-A_{n}(x,y)\|\\
&=\max_{t\in[0,1]}\left|\int_{0}^{1}G_{1}(t,s)p(s)(f(s,y_{m}(s)+\frac{1}{n}(s+\frac{b_{2}}{a_{2}}),|x_{m}'(s)|+\frac{1}{n})-f(s,y(s)+\frac{1}{n}(s+\frac{b_{2}}{a_{2}}),|x'(s)|+\frac{1}{n}))ds\right|\\
&\leq\frac{1}{a_{1}}\int_{0}^{1}(a_{1}s+b_{1})p(s)\left|f(s,y_{m}(s)+\frac{1}{n}(s+\frac{b_{2}}{a_{2}}),|x_{m}'(s)|+\frac{1}{n})-f(s,y(s)+\frac{1}{n}(s+\frac{b_{2}}{a_{2}}),|x'(s)|+\frac{1}{n})\right|ds
\end{split}\end{equation}
and
\begin{equation}\label{3.4.16}\begin{split}
&\|A_{n}(x_{m},y_{m})'-A_{n}(x,y)'\|\\
&=\max_{\tau\in[0,1]}\left|\int_{\tau}^{1}p(s)(f(s,y_{m}(s)+\frac{1}{n}(s+\frac{b_{2}}{a_{2}}),|x_{m}'(s)|+\frac{1}{n})-f(s,y(s)+\frac{1}{n}(s+\frac{b_{2}}{a_{2}}),|x'(s)|+\frac{1}{n}))ds\right|\\
&\leq\int_{0}^{1}p(s)\left|f(s,y_{m}(s)+\frac{1}{n}(s+\frac{b_{2}}{a_{2}}),|x_{m}'(s)|+\frac{1}{n})-f(s,y(s)+\frac{1}{n}(s+\frac{b_{2}}{a_{2}}),|x'(s)|+\frac{1}{n})\right|ds.
\end{split}\end{equation}
From \eqref{3.4.15} and \eqref{3.4.16}, using Lebesgue dominated
convergence theorem, it follows that
\begin{align*}
\|A_{n}(x_{m},y_{m})-A_{n}(x,y)\|\rightarrow0,\hspace{0.2cm}\|A_{n}(x_{m},y_{m})'-A_{n}(x,y)'\|\rightarrow0,\text{
as }m\rightarrow+\infty.\end{align*} Hence,
$\|A_{n}(x_{m},y_{m})-A_{n}(x,y)\|_{3}\rightarrow0$ as
$m\rightarrow\infty$. Similarly, we can show that $\|B_{n}(x_{m},y_{m})-B_{n}(x,y)\|_{3}\rightarrow0$ as $m\rightarrow\infty$. Consequently, $\|T_{n}(x_{m},y_{m})-T_{n}(x,y)\|_{5}\rightarrow0$ as
$m\rightarrow+\infty$, that is, $T_{n}:\overline{\mathcal{O}}_{r}\cap(P_{1}\times P_{2})\rightarrow
P_{1}\times P_{2}$ is continuous. Hence, $T_{n}:\overline{\mathcal{O}}_{r}\cap(P_{1}\times P_{2})\rightarrow P_{1}\times P_{2}$ is completely continuous.
\end{proof}

Assume that
\begin{description}
\item[$(\mathbf{B}_{13})$] for any real constant $C>0$, $\int_{0}^{1}p(t)u_{1}(C\int_{t}^{1}p(s)\varphi_{E}(s)ds)dt<+\infty$
and\\
$\int_{0}^{1}q(t)u_{2}(C\int_{t}^{1}q(s)\psi_{E}(s)ds)dt<+\infty$.
\end{description}

\begin{thm}\label{th2.2}
Assume that $(\mathbf{B}_{1})-(\mathbf{B}_{3})$,
$(\mathbf{B}_{5})$, $(\mathbf{B}_{8})$, $(\mathbf{B}_{10})$,
$(\mathbf{B}_{11})$ and $(\mathbf{B}_{13})$ hold. Then the system of SBVPs \eqref{3.0.4} has at least two $C^{1}$-positive solutions.
\end{thm}

\begin{proof}
Let $R_{5}=R_{3}+R_{4}$ and define
$\mathcal{O}_{R_{5}}=\Omega_{R_{3}}\times\Omega_{R_{4}}$ where
\begin{align*}\Omega_{R_{3}}=\{x\in
E:\|x\|_{3}<R_{3}\},\,\Omega_{R_{4}}=\{x\in
E:\|x\|_{3}<R_{4}\}.\end{align*} We claim that
\begin{equation}\label{3.4.17}(x,y)\neq\lambda T_{n}(x,y), \text{ for }\lambda\in(0,1],
(x,y)\in\partial\mathcal{O}_{R_{5}}\cap (P_{1}\times
P_{2}).\end{equation} Suppose there exist
$(x_{0},y_{0})\in\partial\mathcal{O}_{R_{5}}\cap(P_{1}\times P_{2})$
and $\lambda_{0}\in(0,1]$ such that
\begin{align*}(x_{0},y_{0})=\lambda_{0} T_{n}(x_{0},y_{0}).\end{align*}
Then,
\begin{equation}\label{3.4.18}\begin{split}
-x_{0}''(t)&=\lambda_{0}p(t)f(t,y_{0}(t)+\frac{1}{n}(t+\frac{b_{2}}{a_{2}}),|x_{0}'(t)|+\frac{1}{n}),\hspace{0.4cm}t\in(0,1),\\
-y_{0}''(t)&=\lambda_{0}q(t)g(t,x_{0}(t)+\frac{1}{n}(t+\frac{b_{1}}{a_{1}}),|y_{0}'(t)|+\frac{1}{n}),\hspace{0.4cm}t\in(0,1),\\
a_{1}x(0)&-b_{1}x'(0)=a_{2}y(0)-b_{2}y'(0)=x'(1)=y'(1)=0.
\end{split}\end{equation}
From \eqref{3.4.18} and $(\mathbf{B}_{2})$, we have $x_{0}''\leq
0$ and $y_{0}''\leq 0$ on $(0,1)$, integrating from $t$ to $1$,
using the BCs \eqref{3.4.18}, we obtain $x_{0}'(t)\geq0$ and
$y_{0}'(t)\geq0$ for $t\in[0,1]$. From \eqref{3.4.18} and
$(\mathbf{B}_{3})$, we have
\begin{align*}-x_{0}''(t)&\leq
p(t)k_{1}(y_{0}(t)+\frac{1}{n}(t+\frac{b_{2}}{a_{2}}))(u_{1}(x_{0}'(t)+\frac{1}{n})+v_{1}(x_{0}'(t)+\frac{1}{n})),\hspace{0.4cm}t\in(0,1),\\
-y_{0}''(t)&\leq
q(t)k_{2}(x_{0}(t)+\frac{1}{n}(t+\frac{b_{1}}{a_{1}}))(u_{2}(y_{0}'(t)+\frac{1}{n})+v_{2}(y_{0}'(t)+\frac{1}{n})),\hspace{0.4cm}t\in(0,1),\end{align*}
which implies that
\begin{align*}\frac{-x_{0}''(t)}{u_{1}(x_{0}'(t)+\frac{1}{n})+v_{1}(x_{0}'(t)+\frac{1}{n})}\leq p(t)k_{1}(y_{0}(t)+\frac{1}{n}(t+\frac{b_{2}}{a_{2}}))&\leq
k_{1}(R_{4}+\varepsilon)p(t),\hspace{0.4cm}t\in(0,1),\\
\frac{-y_{0}''(t)}{u_{2}(y_{0}'(t)+\frac{1}{n})+v_{2}(y_{0}'(t)+\frac{1}{n})}\leq
q(t)k_{2}(x_{0}(t)+\frac{1}{n}(t+\frac{b_{1}}{a_{1}}))&\leq
k_{2}(R_{3}+\varepsilon)q(t),\hspace{0.4cm}t\in(0,1).\end{align*}
Integrating from $t$ to $1$, using the BCs \eqref{3.4.18}, we obtain
\begin{align*}I(x_{0}'(t)+\frac{1}{n})-I(\frac{1}{n})&\leq k_{1}(R_{4}+\varepsilon)\int_{t}^{1}p(s)ds,\hspace{0.4cm}t\in[0,1],\\
J(y_{0}'(t)+\frac{1}{n})-J(\frac{1}{n})&\leq
k_{2}(R_{3}+\varepsilon)\int_{t}^{1}q(s)ds,\hspace{0.4cm}t\in[0,1],\end{align*}
which implies that
\begin{equation}\label{3.4.19}x_{0}'(t)\leq I^{-1}(k_{1}(R_{4}+\varepsilon)\int_{0}^{1}p(s)ds+I(\varepsilon)),\hspace{0.4cm}t\in[0,1],\end{equation}
\begin{equation}\label{3.4.20}y_{0}'(t)\leq J^{-1}(k_{2}(R_{3}+\varepsilon)\int_{0}^{1}q(s)ds+J(\varepsilon)),\hspace{0.4cm}t\in[0,1].\end{equation}
Integrating from $0$ to $t$, using the BCs \eqref{3.4.18}, leads to
\begin{align*}x_{0}(t)\leq \frac{b_{1}}{a_{1}}
x_{0}'(0)+&I^{-1}(k_{1}(R_{4}+\varepsilon)\int_{0}^{1}p(s)ds+I(\varepsilon)),\hspace{0.4cm}t\in[0,1],\\
y_{0}(t)\leq
\frac{b_{2}}{a_{2}}y_{0}'(0)+&J^{-1}(k_{2}(R_{3}+\varepsilon)\int_{0}^{1}q(s)ds+J(\varepsilon)),\hspace{0.4cm}t\in[0,1].
\end{align*}
Using \eqref{3.4.19} and \eqref{3.4.20}, we have
\begin{equation}\label{3.4.21}x_{0}(t)\leq (1+\frac{b_{1}}{a_{1}})I^{-1}(k_{1}(R_{4}+\varepsilon)\int_{0}^{1}p(s)ds+I(\varepsilon)),\hspace{0.4cm}t\in[0,1],\end{equation}
\begin{equation}\label{3.4.22}y_{0}(t)\leq (1+\frac{b_{2}}{a_{2}})J^{-1}(k_{2}(R_{3}+\varepsilon)\int_{0}^{1}q(s)ds+J(\varepsilon)),\hspace{0.4cm}t\in[0,1].\end{equation}
From \eqref{3.4.19}-\eqref{3.4.22}, it follows that
\begin{equation}\label{3.4.23}R_{3}\leq (1+\frac{b_{1}}{a_{1}})I^{-1}(k_{1}(R_{4}+\varepsilon)\int_{0}^{1}p(s)ds+I(\varepsilon)),\end{equation}
\begin{equation}\label{3.4.24}R_{4}\leq (1+\frac{b_{2}}{a_{2}})J^{-1}(k_{2}(R_{3}+\varepsilon)\int_{0}^{1}q(s)ds+J(\varepsilon)).\end{equation}
Now, using \eqref{3.4.24} in \eqref{3.4.23} together with increasing
property of $k_{1}$ and $I^{-1}$, we have
\begin{align*}\frac{R_{3}}{(1+\frac{b_{1}}{a_{1}})I^{-1}(k_{1}((1+\frac{b_{2}}{a_{2}})J^{-1}(k_{2}(R_{3}+\varepsilon)\int_{0}^{1}q(s)ds+J(\varepsilon))+\varepsilon)\int_{0}^{1}p(s)ds+I(\varepsilon))}\leq 1,\end{align*} a
contradiction to \eqref{3.4.3}. Similarly, using \eqref{3.4.23} in
\eqref{3.4.24} together with increasing property of $k_{2}$ and
$J^{-1}$, we have
\begin{align*}\frac{R_{4}}{(1+\frac{b_{2}}{a_{2}})J^{-1}(k_{2}((1+\frac{b_{1}}{a_{1}})I^{-1}(k_{1}(R_{4}+\varepsilon)\int_{0}^{1}p(s)ds+I(\varepsilon))+\varepsilon)\int_{0}^{1}q(s)ds+J(\varepsilon))}\leq 1,\end{align*} a
contradiction to \eqref{3.4.4}. Hence, \eqref{3.4.17} is true and by
Lemma \ref{lemindexone}, the fixed point index
\begin{equation}\label{3.4.25}\ind(T_{n},\mathcal{O}_{R_{5}}\cap(P_{1}\times P_{2}),P_{1}\times P_{2})=1.\end{equation}
Now, choose a $t_{0}\in(0,1)$ and define
\begin{equation}\label{3.4.26}N_{3}=\frac{1+\gamma_{2}^{-1}\varrho_{2}^{-1}}{\max_{t\in[0,1]}\int_{0}^{1}G_{1}(t,s)p(s)ds}\text{
and
}N_{4}=\frac{1+\gamma_{1}^{-1}\varrho_{1}^{-1}}{\max_{t\in[0,1]}\int_{0}^{1}G_{2}(t,s)q(s)ds}.\end{equation}
By $(\mathbf{B}_{10})$, there exist real constants with
$R_{3}^{*}>R_{3}$ and $R_{4}^{*}>R_{4}$ such that
\begin{equation}\label{3.4.27}\begin{split}
h_{1}(x,y)&\geq N_{3}x,\text{ for }x\geq
R_{3}^{*},y\in(0,\infty),\\
h_{2}(x,y)&\geq N_{4}x,\text{ for }x\geq
R_{4}^{*},y\in(0,\infty).\end{split}\end{equation} Let
$R^{**}=\frac{R_{3}^{*}}{\gamma_{1}\varrho_{1}}+\frac{R_{4}^{*}}{\gamma_{2}\varrho_{2}}$
and define $\mathcal{O}_{R^{**}}=\Omega_{R_{3}^{*}}\times
\Omega_{R_{4}^{*}}$, where
\begin{align*}
\Omega_{R_{3}^{*}}=\{x\in
C^{1}[0,1]:\|x\|_{3}<\frac{R_{3}^{*}}{\gamma_{1}\varrho_{1}}\},\,
\Omega_{R_{4}^{*}}=\{x\in
C^{1}[0,1]:\|x\|_{3}<\frac{R_{4}^{*}}{\gamma_{2}\varrho_{2}}\}.
\end{align*}
We show that
\begin{equation}\label{3.4.28}
T_{n}(x,y)\npreceq(x,y),\text{ for }(x,y)\in\partial
\mathcal{O}_{R^{**}}\cap(P_{1}\times P_{2}).
\end{equation}
Suppose $T_{n}(x_{0},y_{0})\preceq(x_{0},y_{0})$ for some
$(x_{0},y_{0})\in\partial\mathcal{O}_{R^{**}}\cap(P_{1}\times
P_{2})$. Then,
\begin{equation}\label{3.4.29}x_{0}(t)\geq A_{n}(x_{0},y_{0})(t)\text{ and }y_{0}(t)\geq B_{n}(x_{0},y_{0})(t)\text{ for }t\in[0,1].\end{equation}
By Lemma \ref{lemcone2}, we have
\begin{align*}x_{0}(t)\geq\gamma_{1}\varrho_{1}\|x_{0}\|_{3}=R_{3}^{*},\hspace{0.4cm}t\in[0,1].\end{align*}
Similarly, $y_{0}(t)\geq R_{4}^{*}$ for $t\in[0,1]$. Hence,
\begin{align*}|x_{0}(t)|+\frac{1}{n}(t+\frac{b_{1}}{a_{1}})\geq R_{3}^{*}\text{ and
}|y_{0}(t)|+\frac{1}{n}(t+\frac{b_{2}}{a_{2}})\geq R_{4}^{*}\text{
for }t\in[0,1].\end{align*} Now, using \eqref{3.4.29},
$(\mathbf{B}_{10})$ and \eqref{3.4.27}, we have
\begin{align*}\begin{split}
&\|x_{0}\|\geq x_{0}(t)\\
&\geq A_{n}(x_{0},y_{0})(t)\\
&=\int_{0}^{1}G_{1}(t,s)p(s)f(s,y_{0}(s)+\frac{1}{n}(s+\frac{b_{2}}{a_{2}}),|x_{0}'(s)|+\frac{1}{n})ds\\
&\geq\int_{0}^{1}G_{1}(t,s)p(s)h_{1}(y_{0}(s)+\frac{1}{n}(s+\frac{b_{2}}{a_{2}}),|x_{0}'(s)|+\frac{1}{n})ds\\
&\geq
N_{3}\int_{0}^{1}G_{1}(t,s)p(s)(y_{0}(s)+\frac{1}{n}(s+\frac{b_{2}}{a_{2}}))ds\\
&\geq
N_{3}R_{4}^{*}\int_{0}^{1}G_{1}(t,s)p(s)ds,\hspace{0.4cm}t\in[0,1],
\end{split}\end{align*} in view of \eqref{3.4.27} we have
\begin{align*}\begin{split}
\|x_{0}\|&\geq
N_{3}R_{4}^{*}\max_{t\in[0,1]}\int_{0}^{1}G_{1}(t,s)p(s)ds>\frac{R_{4}^{*}}{\gamma_{2}\varrho_{2}}.
\end{split}\end{align*}
Thus
$\|x_{0}\|_{3}\geq\|x_{0}\|>\frac{R_{4}^{*}}{\gamma_{2}\varrho_{2}}$.
Similarly, using \eqref{3.4.29}, $(\mathbf{B}_{10})$,
\eqref{3.4.27} and \eqref{3.4.26}, we have
$\|y_{0}\|_{3}>\frac{R_{3}^{*}}{\gamma_{1}\varrho_{1}}$.
Consequently, it follows that,
$\|(x_{0},y_{0})\|_{5}=\|x_{0}\|_{3}+\|y_{0}\|_{3}>R^{**}$, a
contradiction. Hence, \eqref{3.4.28} is true and by Lemma
\ref{lemindexzero}, the fixed point index
\begin{equation}\label{3.4.30}\ind(T_{n},\mathcal{O}_{R^{**}}\cap(P_{1}\times P_{2}),P_{1}\times P_{2})=0.\end{equation}
From \eqref{3.4.25} and \eqref{3.4.30}, it follows that
\begin{equation}\label{3.4.31}\ind(T_{n},(\mathcal{O}_{R^{**}}\setminus\overline{\mathcal{O}}_{R})\cap(P_{1}\times P_{2}),P_{1}\times P_{2})=-1.\end{equation}
Thus, in view of \eqref{3.4.25} and \eqref{3.4.31}, there exist
$(x_{n,1},y_{n,1})\in\mathcal{O}_{R}\cap(P_{1}\times P_{2})$ and
$(x_{n,2},y_{n,2})\in(\mathcal{O}_{R^{**}}\setminus\overline{\mathcal{O}}_{R})\cap(P_{1}\times
P_{2})$ such that
$(x_{n,j},y_{n,j})=T_{n}(x_{n,j},y_{n,j}),\,(j=1,2)$ which implies
that
\begin{equation}\label{3.4.32}\begin{split}
x_{n,j}(t)&=\int_{0}^{1}G_{1}(t,s)p(s)f(t,y_{n,j}(s)+\frac{1}{n}(s+\frac{b_{2}}{a_{2}}),|x_{n,j}'(s)|+\frac{1}{n})ds,\hspace{0.4cm}t\in[0,1],\\
y_{n,j}(t)&=\int_{0}^{1}G_{2}(t,s)q(s)g(s,x_{n,j}(s)+\frac{1}{n}(s+\frac{b_{1}}{a_{1}}),|y_{n,j}'(s)|+\frac{1}{n})ds,\hspace{0.4cm}t\in[0,1],\,j=1,2.
\end{split}\end{equation}
Using $(\mathbf{B}_{8})$ there exist continuous functions
$\varphi_{R_{4}+\varepsilon}$ and $\psi_{R_{3}+\varepsilon}$ defined
on $[0,1]$ and positive on $(0,1)$ and real constants
$0\leq\delta_{1},\delta_{2}<1$ such that
\begin{equation}\label{3.4.33}\begin{split}f(t,x,y)&\geq \varphi_{R_{4}+\varepsilon}(t)x^{\delta_{1}},\hspace{0.4cm}(t,x,y)\in[0,1]\times[0,R_{4}+\varepsilon]\times[0,R_{4}+\varepsilon],\\
g(t,x,y)&\geq\psi_{R_{3}+\varepsilon}(t)x^{\delta_{2}},\hspace{0.4cm}(t,x,y)\in[0,1]\times[0,R_{3}+\varepsilon]\times[0,R_{3}+\varepsilon].\end{split}\end{equation}
By the Lemma \ref{lemmax2}, we have
$x_{n,1}(t)\geq\gamma_{1}\|x_{n,1}\|$ and
$y_{n,1}(t)\geq\gamma_{2}\|y_{n,1}\|$ for $t\in[0,1]$. We show that
\begin{equation}\label{3.4.34}x_{n,1}'(t)\geq C_{8}^{\delta_{1}}\gamma_{2}^{\delta_{1}}\int_{t}^{1}p(s)\varphi_{R_{4}+\varepsilon}(s)ds,\hspace{0.4cm}t\in[0,1],\end{equation}
\begin{equation}\label{3.4.35}y_{n,1}'(t)\geq C_{7}^{\delta_{2}}\gamma_{1}^{\delta_{2}}\int_{t}^{1}q(s)\psi_{R_{3}+\varepsilon}(s)ds,\hspace{0.4cm}t\in[0,1],\end{equation}
where
\begin{align*}\begin{split}
C_{7}&=\gamma_{1}^{\frac{1+\delta_{1}\delta_{2}}{1-\delta_{1}\delta_{2}}}\gamma_{2}^{\frac{2\delta_{1}}{1-\delta_{1}\delta_{2}}}
\left(\max_{\tau\in[0,1]}\int_{0}^{1}G_{1}(\tau,s)p(s)\varphi_{R_{4}+\varepsilon}(s)ds\right)^{\frac{1}{1-\delta_{1}\delta_{2}}}\\
&\left(\max_{\tau\in[0,1]}\int_{0}^{1}G_{2}(\tau,s)q(s)\psi_{R_{3}+\varepsilon}(s)ds\right)^{\frac{\delta_{1}}{1-\delta_{1}\delta_{2}}},\\
C_{8}&=\gamma_{1}^{\frac{2\delta_{2}}{1-\delta_{1}\delta_{2}}}\gamma_{2}^{\frac{1+\delta_{1}\delta_{2}}{1-\delta_{1}\delta_{2}}}
\left(\max_{\tau\in[0,1]}\int_{0}^{1}G_{1}(\tau,s)p(s)\varphi_{R_{4}+\varepsilon}(s)ds\right)^{\frac{\delta_{2}}{1-\delta_{1}\delta_{2}}}\\
&\left(\max_{\tau\in[0,1]}\int_{0}^{1}G_{2}(\tau,s)q(s)\psi_{R_{3}+\varepsilon}(s)ds\right)^{\frac{1}{1-\delta_{1}\delta_{2}}}.
\end{split}\end{align*}
To prove \eqref{3.4.34}, using \eqref{3.4.32} and \eqref{3.4.33}, we
have
\begin{align*}\begin{split}
x_{n,1}(t)&=\int_{0}^{1}G_{1}(t,s)p(s)f(s,y_{n,1}(s)+\frac{1}{n}(s+\frac{b_{2}}{a_{2}}),|x_{n,1}'(s)|+\frac{1}{n})ds\\
&\geq\int_{0}^{1}G_{1}(t,s)p(s)\varphi_{R_{4}+\varepsilon}(s)(y_{n,1}(s)+\frac{1}{n}(s+\frac{b_{2}}{a_{2}}))^{\delta_{1}}ds\\
&\geq\gamma_{2}^{\delta_{1}}\|y_{n,1}\|^{\delta_{1}}\int_{0}^{1}G_{1}(t,s)p(s)\varphi_{R_{4}+\varepsilon}(s)ds,\hspace{0.4cm}t\in[0,1],
\end{split}\end{align*}
which implies that
\begin{align*}
x_{n,1}(t)\geq\gamma_{1}\gamma_{2}^{\delta_{1}}\|y_{n,1}\|^{\delta_{1}}\max_{\tau\in[0,1]}\int_{0}^{1}G_{1}(\tau,s)p(s)
\varphi_{R_{4}+\varepsilon}(s)ds,\hspace{0.4cm}t\in[0,1].
\end{align*}
Hence,
\begin{equation}\label{3.4.36}
\|x_{n,1}\|\geq\gamma_{1}\gamma_{2}^{\delta_{1}}\|y_{n,1}\|^{\delta_{1}}\max_{\tau\in[0,1]}\int_{0}^{1}G_{1}(\tau,s)p(s)\varphi_{R_{4}+\varepsilon}(s)ds.
\end{equation}
Similarly, from \eqref{3.4.32} and \eqref{3.4.33}, we have
\begin{equation}\label{3.4.37}
\|y_{n,1}\|\geq\gamma_{1}^{\delta_{2}}\gamma_{2}\|x_{n,1}\|^{\delta_{2}}\max_{\tau\in[0,1]}\int_{0}^{1}G_{2}(\tau,s)q(s)\psi_{R_{3}+\varepsilon}(s)ds.
\end{equation}
Using \eqref{3.4.36} in \eqref{3.4.37}, we have
\begin{align*}
\|y_{n,1}\|&\geq\gamma_{1}^{2\delta_{2}}\gamma_{2}^{1+\delta_{1}\delta_{2}}\|y_{n,1}\|^{\delta_{1}\delta_{2}}
\left(\max_{\tau\in[0,1]}\int_{0}^{1}G_{1}(\tau,s)p(s)\varphi_{R_{4}+\varepsilon}(s)ds\right)^{\delta_{2}}\\
&\max_{\tau\in[0,1]}\int_{0}^{1}G_{2}(\tau,s)q(s)\psi_{R_{3}+\varepsilon}(s)ds,
\end{align*}
which implies that
\begin{equation}\label{3.4.38}\begin{split}
\|y_{n,1}\|&\geq\gamma_{1}^{\frac{2\delta_{2}}{1-\delta_{1}\delta_{2}}}\gamma_{2}^{\frac{1+\delta_{1}\delta_{2}}{1-\delta_{1}\delta_{2}}}
\left(\max_{\tau\in[0,1]}\int_{0}^{1}G_{1}(\tau,s)p(s)\varphi_{R_{4}+\varepsilon}(s)ds\right)^{\frac{\delta_{2}}{1-\delta_{1}\delta_{2}}}\\
&\left(\max_{\tau\in[0,1]}\int_{0}^{1}G_{2}(\tau,s)q(s)\psi_{R_{3}+\varepsilon}(s)ds\right)^{\frac{1}{1-\delta_{1}\delta_{2}}}=C_{8}.
\end{split}\end{equation} Using \eqref{3.4.33} and \eqref{3.4.38} in the
following relation
\begin{align*}
x_{n,1}'(t)=&\int_{t}^{1}p(s)f(s,y_{n,1}(s)+\frac{1}{n}(s+\frac{b_{2}}{a_{2}}),|x_{n,1}'(s)|+\frac{1}{n})ds
\end{align*}
we obtain \eqref{3.4.34}. Similarly, we can prove \eqref{3.4.35}.

\vskip 0.5em

Now, differentiating \eqref{3.4.32}, using $(\mathbf{B}_{3})$,
\eqref{3.4.34} and \eqref{3.4.35}, we have
\begin{equation}\label{3.4.39}\begin{split}
0\leq-x_{n,1}''(t)&\leq
p(t)k_{1}(R_{4}+\varepsilon)(u_{1}(C_{8}^{\delta_{1}}\gamma_{2}^{\delta_{1}}\int_{t}^{1}p(s)\varphi_{R_{4}+\varepsilon}(s)ds)
+v_{1}(R_{3}+\varepsilon)),\hspace{0.4cm}t\in(0,1),\\
0\leq-y_{n,1}''(t)&\leq q(t)
k_{2}(R_{3}+\varepsilon)(u_{2}(C_{7}^{\delta_{2}}\gamma_{1}^{\delta_{2}}\int_{t}^{1}q(s)\psi_{R_{3}+\varepsilon}(s)ds)+v_{2}(R_{4}+\varepsilon)),\hspace{0.4cm}t\in(0,1),
\end{split}\end{equation}
which on integration from $t$ to $1$, using the BCs \eqref{3.4.5},
leads to
\begin{align*}
x_{n,1}'(t)&\leq
k_{1}(R_{4}+\varepsilon)\int_{t}^{1}p(s)(u_{1}(C_{8}^{\delta_{1}}\gamma_{2}^{\delta_{1}}\int_{s}^{1}p(\tau)\varphi_{R_{4}+\varepsilon}(\tau)d\tau)
+v_{1}(R_{3}+\varepsilon))ds,\hspace{0.4cm}t\in[0,1],\\
y_{n,1}'(t)&\leq k_{2}(R_{3}+\varepsilon)\int_{t}^{1}q(s)
(u_{2}(C_{7}^{\delta_{2}}\gamma_{1}^{\delta_{2}}\int_{s}^{1}q(\tau)\psi_{R_{3}+\varepsilon}(\tau)d\tau)+v_{2}(R_{4}+\varepsilon))ds,\hspace{0.4cm}t\in[0,1],
\end{align*}
which implies that
\begin{equation}\label{3.4.40}\begin{split}
x_{n,1}'(t)&\leq
k_{1}(R_{4}+\varepsilon)\int_{0}^{1}p(s)(u_{1}(C_{8}^{\delta_{1}}\gamma_{2}^{\delta_{1}}\int_{s}^{1}p(\tau)\varphi_{R_{4}+\varepsilon}(\tau)d\tau)
+v_{1}(R_{3}+\varepsilon))ds,\hspace{0.4cm}t\in[0,1],\\
y_{n,1}'(t)&\leq
k_{2}(R_{3}+\varepsilon)\int_{0}^{1}q(s)(u_{2}(C_{7}^{\delta_{2}}\gamma_{1}^{\delta_{2}}\int_{s}^{1}q(\tau)\psi_{R_{3}+\varepsilon}(\tau)d\tau)+v_{2}(R_{4}+\varepsilon))ds,\hspace{0.4cm}t\in[0,1].
\end{split}\end{equation}
In view of \eqref{3.4.34}, \eqref{3.4.35}, \eqref{3.4.40},
\eqref{3.4.39}, $(\mathbf{B}_{1})$ and $(\mathbf{B}_{13})$, the
sequences $\{(x_{n,1}^{(j)},y_{n,1}^{(j)})\}$ $(j=0,1)$ are
uniformly bounded and equicontinuous on $[0,1]$. Thus, by Theorem
\ref{arzela}, there exist subsequences
$\{(x_{n_{k},1}^{(j)},y_{n_{k},1}^{(j)})\}$ $(j=0,1)$ of
$\{(x_{n,1}^{(j)},y_{n,1}^{(j)})\}$ and functions $(x_{0,1},y_{0,1})\in C^{1}[0,1]\times C^{1}[0,1]$ such that $(x_{n_{k},1}^{(j)},y_{n_{k},1}^{(j)})$ converges uniformly to
$(x_{0,1}^{(j)},y_{0,1}^{(j)})$ on $[0,1]$. Also,
$a_{1}x_{0,1}(0)-b_{1}x_{0,1}'(0)=a_{2}y_{0,1}(0)-b_{2}y_{0,1}'(0)=x_{0,1}'(1)=y_{0,1}'(1)=0$.
Moreover, from \eqref{3.4.34} and \eqref{3.4.35}, with $n_{k}$ in
place of $n$ and taking $\lim_{n_k\rightarrow+\infty}$, we have
\begin{align*}
x_{0,1}'(t)&\geq C_{8}^{\delta_{1}}\gamma_{2}^{\delta_{1}} \int_{t}^{1}p(s)\varphi_{R_{4}+\varepsilon}(s)ds,\\
y_{0,1}'(t)&\geq
C_{7}^{\delta_{2}}\gamma_{1}^{\delta_{2}}\int_{t}^{1}q(s)\psi_{R_{3}+\varepsilon}(s)ds,
\end{align*}
which implies that $x_{0,1}'>0$ and $y_{0,1}'>0$ on $[0,1)$,
$x_{0,1}>0$ and $y_{0,1}>0$ on $[0,1]$. Further,
\begin{equation}\label{3.4.41}\begin{split}
&\left|f(t,y_{n_{k},1}(t)+\frac{1}{n_{k}}(t+\frac{b_{2}}{a_{2}}),x_{n_{k},1}'(t)+\frac{1}{n_{k}})\right|\\
&\leq p(t)k_{1}(R_{4}+\varepsilon)(u_{1}(C_{8}^{\delta_{1}}\gamma_{2}^{\delta_{1}}\int_{t}^{1}p(s)\varphi_{R_{4}+\varepsilon}(s)ds)+v_{1}(R_{3}+\varepsilon)),\\
&\left|g(t,x_{n_{k},1}(t)+\frac{1}{n_{k}}(t+\frac{b_{1}}{a_{1}}),y_{n_{k},1}'(t)+\frac{1}{n_{k}})\right|\\
&\leq
q(t)k_{2}(R_{3}+\varepsilon)(u_{2}(C_{7}^{\delta_{2}}\gamma_{1}^{\delta_{2}}\int_{t}^{1}q(s)\psi_{R_{3}+\varepsilon}(s)ds)+v_{2}(R_{4}+\varepsilon)),
\end{split}\end{equation}

\begin{equation}\label{3.4.42}\begin{split}
\lim_{n_k\rightarrow\infty}f(t,y_{n_{k},1}(t)+\frac{1}{n_{k}}(t+\frac{b_{2}}{a_{2}}),x_{n_{k},1}'(t)+\frac{1}{n_{k}})&=f(t,y_{0,1}(t),x_{0,1}'(t)),\hspace{0.4cm}t\in(0,1),\\
\lim_{n_k\rightarrow\infty}g(t,x_{n_{k},1}(t)+\frac{1}{n_{k}}(t+\frac{b_{1}}{a_{1}}),y_{n_{k},1}'(t)+\frac{1}{n_{k}})&=g(t,x_{0,1}(t),y_{0,1}'(t)),\hspace{0.4cm}t\in(0,1).
\end{split}\end{equation}
Moreover, $(x_{n_{k},1},y_{n_{k},1})$ satisfies
\begin{align*}
x_{n_{k},1}(t)&=\int_{0}^{1}G_{1}(t,s)p(s)f(s,y_{n_{k},1}(s)+\frac{1}{n_{k}}(s+\frac{b_{2}}{a_{2}}),x_{n_{k},1}'(s)+\frac{1}{n_{k}})ds,\hspace{0.4cm}t\in[0,1],\\
y_{n_{k},1}(t)&=\int_{0}^{1}G_{2}(t,s)q(s)g(s,x_{n_{k},1}(s)+\frac{1}{n_{k}}(s+\frac{b_{1}}{a_{1}}),y_{n_{k},1}'(s)+\frac{1}{n_{k}})ds,\hspace{0.4cm}t\in[0,1],
\end{align*}
in view of \eqref{3.4.41}, $(\mathbf{B}_{13})$, \eqref{3.4.42}, the
Lebesgue dominated convergence theorem and taking
$\lim_{n_k\rightarrow+\infty}$, we have
\begin{align*}
x_{0,1}(t)&=\int_{0}^{1}G_{1}(t,s)p(s)f(s,y_{0,1}(s),x_{0,1}'(s))ds,\hspace{0.4cm}t\in[0,1],\\
y_{0,1}(t)&=\int_{0}^{1}G_{2}(t,s)q(s)g(s,x_{0,1}(s),y_{0,1}'(s))ds,\hspace{0.4cm}t\in[0,1],\end{align*}
which implies that $(x_{0,1},y_{0,1})\in C^{2}(0,1)\times
C^{2}(0,1)$ and
\begin{align*}
-x_{0,1}''(t)=p(t)f(t,y_{0,1}(t),x_{0,1}'(t)),\hspace{0.4cm}t\in(0,1),\\
-y_{0,1}''(t)=q(t)g(t,x_{0,1}(t),y_{0,1}'(t)),\hspace{0.4cm}t\in(0,1).
\end{align*}
Moreover, by \eqref{3.4.1} and \eqref{3.4.2}, we have
$\|x_{0,1}\|_{3}<R_{3}$ and $\|y_{0,1}\|_{3}<R_{4}$, that is,
$\|(x_{0,1},y_{0,1})\|_{5}<R_{5}$. By a similar proof the sequence
$\{(x_{n,2},y_{n,2})\}$ has a convergent subsequence
$\{(x_{n_{k},2},y_{n_{k},2})\}$ converging uniformly to
$(x_{0,2},y_{0,2})\in C^{1}[0,1]\times C^{1}[0,1]$ on $[0,1]$.
Moreover, $(x_{0,2},y_{0,2})$ is a solution to the system of SBVPs \eqref{3.0.4} with $x_{0,2}>0$ and $y_{0,2}>0$ on
$[0,1]$, $x'_{0,2}>0$ and $y'_{0,2}>0$ on $[0,1)$,
$R_{5}<\|(x_{0,2},y_{0,2})\|_{5}<R^{**}$.
\end{proof}

%
%
\begin{ex}Consider the following coupled system of SBVPs
\begin{equation}\label{3.4.43}\begin{split}
-&x''(t)=\mu_{1}(1+(y(t))^{\delta_{1}}+(y(t))^{\eta_{1}})(1+(x'(t))^{\alpha_{1}}+(x'(t))^{-\beta_{1}}),\hspace{0.4cm}t\in(0,1),\\
-&y''(t)=\mu_{2}(1+(x(t))^{\delta_{2}}+(x(t))^{\eta_{2}})(1+(y'(t))^{\alpha_{2}}+(y'(t))^{-\beta_{2}}),\hspace{0.4cm}t\in(0,1),\\
&x(0)-x'(0)=y(0)-y'(0)=x'(1)=y'(1)=0,
\end{split}\end{equation}
where $0\leq\delta_{i}<1$, $\eta_{i}>1$, $0<\alpha_{i}<1$,
$0<\beta_{i}<1$, and $\mu_{i}>0$, $i=1,2$.

\vskip 0.5em

Choose $p(t)=\mu_{1}$, $q(t)=\mu_{2}$,
$k_{i}(x)=1+x^{\delta_{i}}+x^{\eta_{i}}$, $u_{i}(x)=x^{-\beta_{i}}$
and $v_{i}(x)=1+x^{\alpha_{i}}$, $i=1,2$. Assume that $\mu_{1}$ is
arbitrary and
\begin{equation*}\small{\begin{split}
\mu_{2}<\min\{\inf_{c\in(0,\infty)}\frac{J(2^{-1}(\mu_{1}^{-1}I(\frac{c}{2}))^{\delta_{1}^{-1}})}{k_{2}(c)},\,\inf_{c\in(0,\infty)}\frac{J(2^{-1}(\mu_{1}^{-1}I(\frac{c}{2}))^{\eta_{1}^{-1}})}{k_{2}(c)},\,\inf_{c\in(0,\infty)}\frac{J(\frac{c}{2})}{k_{2}(2I^{-1}(\mu_{1}k_{1}(c)))}\}.
\end{split}}\end{equation*}

We choose $\varphi_{E}(t)=\mu_{1}$, $\psi_{E}(t)=\mu_{2}$ and
$h_{i}(x,y)=\mu_{i}(1+x^{\eta_{i}})$, $i=1,2$. Then,
\begin{align*}
\lim_{x\rightarrow+\infty}\frac{h_{i}(x,y)}{x}=\lim_{x\rightarrow+\infty}\frac{\mu_{i}(1+x^{\eta_{i}})}{x}=+\infty,\,i=1,2.
\end{align*}

Moreover,
\begin{align*}\begin{split}
&\sup_{c\in(0,\infty)}\frac{c}{(1+\frac{b_{1}}{a_{1}})I^{-1}(k_{1}((1+\frac{b_{2}}{a_{2}})J^{-1}(k_{2}(c)\int_{0}^{1}q(t)dt))\int_{0}^{1}p(t)dt)}\\
&=\sup_{c\in(0,\infty)}\frac{c}{2I^{-1}(\mu_{1}k_{1}(2J^{-1}(\mu_{2}k_{2}(c))))}\\
&\geq\frac{c}{2I^{-1}(\mu_{1}k_{1}(2J^{-1}(\mu_{2}k_{2}(c))))},\hspace{4.8cm}c\in(0,\infty)\\
&=\frac{c}{2I^{-1}(\mu_{1}(1+(2J^{-1}(\mu_{2}k_{2}(c)))^{\delta_{1}}+(2J^{-1}(\mu_{2}k_{2}(c)))^{\eta_{1}}))},\hspace{1cm}c\in(0,\infty)\\
&>1,
\end{split}\end{align*}
and
\begin{align*}\begin{split}
&\sup_{c\in(0,\infty)}\frac{c}{(1+\frac{b_{2}}{a_{2}})J^{-1}(k_{2}((1+\frac{b_{1}}{a_{1}})I^{-1}(k_{1}(c)\int_{0}^{1}p(t)dt))\int_{0}^{1}q(t)dt)}\\
&=\sup_{c\in(0,\infty)}\frac{c}{2J^{-1}(\mu_{2}k_{2}(2I^{-1}(\mu_{1}k_{1}(c))))}\\
&=\frac{c}{2J^{-1}(\mu_{2}k_{2}(2I^{-1}(\mu_{1}k_{1}(c))))},\hspace{0.4cm}c\in(0,\infty)\\
&>1.\end{split}\end{align*}

Further,
\begin{align*}
&\int_{0}^{1}p(t)u_{1}(C\int_{t}^{1}p(s)\varphi_{E}(s)ds)dt=\mu_{1}^{1-2\beta_{1}}C^{-\beta_{1}}\int_{0}^{1}(1-t)^{-\beta_{1}}dt=\frac{\mu_{1}^{1-2\beta_{1}}C^{-\beta_{1}}}{1-\beta_{1}},\\
&\int_{0}^{1}q(t)u_{2}(C\int_{t}^{1}q(s)\psi_{E}(s)ds)dt=\mu_{2}^{1-2\beta_{2}}C^{-\beta_{2}}\int_{0}^{1}(1-t)^{-\beta_{2}}dt=\frac{\mu_{1}^{1-2\beta_{2}}C^{-\beta_{2}}}{1-\beta_{2}}.
\end{align*}
Clearly, $(\mathbf{B}_{1})-(\mathbf{B}_{3})$,
$(\mathbf{B}_{5})$, $(\mathbf{B}_{8})$, $(\mathbf{B}_{10})$,
$(\mathbf{B}_{11})$ and $(\mathbf{B}_{13})$ are satisfied. Hence,
by Theorem \ref{th2.2}, the system of BVPs \eqref{3.4.43} has at
least two $C^{1}$-positive solutions.
\end{ex}


\section{System of ODEs with two-point coupled BCs}\label{sec-couple-two-point}

In this section, we establish existence of at least one
$C^{1}$-positive solution for the system of BVPs \eqref{4.0.2}. By a
$C^{1}$-positive solution to the system of BVPs \eqref{4.0.2}, we
mean that $(x,y)\in (C^{1}[0,1]\cap C^{2}(0,1))\times
(C^{1}[0,1]\cap C^{2}(0,1))$, $(x,y)$ satisfies \eqref{4.0.2}, $x>0$
and $y>0$ on $[0,1]$, $x'>0$ and $y'>0$ on $[0,1)$.

Assume that
\begin{description}
\item[$(\mathbf{B}_{14})$]
\begin{align*}
\sup_{c\in(0,\infty)}\frac{c}{(1+\frac{b_{1}}{a_{1}})I^{-1}(h_{1}(c)k_{1}(c)\int_{0}^{1}p(t)dt)+(1+\frac{b_{2}}{a_{2}})J^{-1}(h_{2}(c)k_{2}(c)\int_{0}^{1}q(t)dt)}>1,
\end{align*}
where $I(\mu)=\int_{0}^{\mu}\frac{d\tau}{u_{1}(\tau)+v_{1}(\tau)}$,
$J(\mu)=\int_{0}^{\mu}\frac{d\tau}{u_{2}(\tau)+v_{2}(\tau)}$, for
$\mu>0$;
\item[$(\mathbf{B}_{16})$] for real constants $M>0$ and $L>0$ there exist continuous functions $\varphi_{ML}$ and $\psi_{ML}$ defined on
$[0,1]$ and positive on $(0,1)$, and constants
$0\leq\gamma_{1},\delta_{1},\gamma_{2},\delta_{2}<1$ satisfying
$(1-\gamma_{1})(1-\gamma_{2})\neq\delta_{1}\delta_{2}$, such that
$f(t,x,y,z)\geq\varphi_{ML}(t)x^{\gamma_{1}}y^{\delta_{1}}$ and
$g(t,x,y,z)\geq\psi_{ML}(t)x^{\gamma_{2}}y^{\delta_{2}}$ on
$[0,1]\times[0,M]\times[0,M]\times[0,L]$;
\item[$(\mathbf{B}_{17})$] $\int_{0}^{1}p(t)u_{1}(C\int_{t}^{1}s^{\delta_{1}}p(s)\varphi_{ML}(s)ds)dt<+\infty$ and
$\int_{0}^{1}q(t)u_{2}(C\int_{t}^{1}s^{\delta_{2}}q(s)\psi_{ML}(s)ds)dt<+\infty$
for any real constant $C>0$.
\end{description}
\begin{thm}\label{th5.2}
Under the hypothesis $(\mathbf{B}_{1})-(\mathbf{B}_{17})$, the
system of BVPs \eqref{4.0.2} has at least one $C^{1}$-positive
solution.
\end{thm}
\begin{proof}
In view of $(\mathbf{B}_{14})$, we can choose real constant
$M_{5}>0$ such that
\begin{align*}
\frac{M_{5}}{(1+\frac{b_{1}}{a_{1}})I^{-1}(h_{1}(M_{5})k_{1}(M_{5})\int_{0}^{1}p(s)ds)
+(1+\frac{b_{2}}{a_{2}})J^{-1}(h_{2}(M_{5})k_{2}(M_{5})\int_{0}^{1}q(s)ds)}>1.
\end{align*}
From the continuity of $I$ and $J$, we choose $\varepsilon>0$ small
enough such that
\begin{equation}\label{4.2.1}
\small{\frac{M_{5}}{(1+\frac{b_{1}}{a_{1}})I^{-1}(h_{1}(M_{5})k_{1}(M_{5})\int_{0}^{1}p(s)ds+I(\varepsilon))+(1+\frac{b_{2}}{a_{2}})J^{-1}(h_{2}(M_{5})k_{2}(M_{5})\int_{0}^{1}q(s)ds+J(\varepsilon))}>1.}
\end{equation}
Choose a real constant $L_{5}>0$ such that
\begin{equation}\label{4.2.2}
L_{5}>\max\{I^{-1}(h_{1}(M_{5})k_{1}(M_{5})\int_{0}^{1}p(t)dt+I(\varepsilon)),J^{-1}(h_{2}(M_{5})k_{2}(M_{5})\int_{0}^{1}q(t)dt+J(\varepsilon))\}
\end{equation}

Choose $n_{0}\in\{1,2,\cdots\}$ such that
$\frac{1}{n_{0}}<\varepsilon$. For each
$n\in\{n_{0},n_{0}+1,\cdots\}$, define retractions
$\theta_{5}:\R\rightarrow[0,M_{5}]$ and
$\rho_{5}:\R\rightarrow[\frac{1}{n},L_{5}]$ by
\begin{align*}\theta_{5}(x)=\max\{0,\min\{x,M_{5}\}\}\text{ and
}\rho_{5}(x)=\max\{\frac{1}{n},\min\{x,L_{5}\}\}.\end{align*}
Consider the modified system of BVPs
\begin{equation}\label{4.2.3}\begin{split}
-&x''(t)=p(t)f(t,\theta_{5}(x(t)),\theta_{5}(y(t)),\rho_{5}(x'(t))),\hspace{0.4cm}t\in(0,1),\\
-&y''(t)=q(t)g(t,\theta_{5}(x(t)),\theta_{5}(y(t)),\rho_{5}(y'(t))),\hspace{0.4cm}t\in(0,1),\\
&a_{1}y(0)-b_{1}x'(0)=0,\,y'(1)=\frac{1}{n}\\
&a_{2}x(0)-b_{2}y'(0)=0,\,x'(1)=\frac{1}{n}.
\end{split}\end{equation}
Since
$f(t,\theta_{5}(x(t)),\theta_{5}(y(t)),\rho_{5}(x'(t))),\,g(t,\theta_{5}(x(t)),\theta_{5}(y(t)),\rho_{5}(y'(t)))$
are continuous and bounded on $[0,1]\times \R^{3}$, by Theorem
\ref{schauder}, it follows that the modified system of BVPs
\eqref{4.2.3} has a solution $(x_{n},y_{n})\in(C^{1}[0,1]\cap
C^{2}(0,1))\times(C^{1}[0,1]\cap C^{2}(0,1))$.

\vskip 0.5em

Using \eqref{4.2.3} and $(\mathbf{B}_{2})$, we obtain
\begin{align*}x_{n}''(t)\leq0\text{ and }y_{n}''(t)\leq 0\text{ for
}t\in(0,1),\end{align*} which on  integration from $t$ to $1$, using
the BCs \eqref{4.2.3}, implies that
\begin{equation}\label{4.2.4}
x_{n}'(t)\geq\frac{1}{n}\text{ and }y_{n}'(t)\geq\frac{1}{n}\text{
for }t\in[0,1].\end{equation} Integrating \eqref{4.2.4}  from $0$ to
$t$, using the BCs \eqref{4.2.3} and \eqref{4.2.4}, we have
\begin{equation}\label{4.2.5}
x_{n}(t)\geq(t+\frac{b_{2}}{a_{2}})\frac{1}{n}\text{ and
}y_{n}(t)\geq(t+\frac{b_{1}}{a_{1}})\frac{1}{n}\text{ for
}t\in[0,1].
\end{equation} From \eqref{4.2.4} and \eqref{4.2.5}, it
follows that
\begin{equation}\label{4.2.6}
\|x_{n}\|=x_{n}(1)\text{ and }\|y_{n}\|=y_{n}(1).
\end{equation}

Now, we show that
\begin{equation}\label{4.2.7}x_{n}'(t)<L_{5},\,\, y_{n}'(t)<L_{5},\hspace{0.4cm}t\in[0,1].\end{equation}
First, we prove $x_{n}'(t)<L_{5}$ for $t\in[0,1]$. Suppose
$x_{n}'(t_{1})\geq L_{5}$ for some $t_{1}\in[0,1]$. Using
\eqref{4.2.3} and $(\mathbf{B}_{3})$, we have
\begin{align*}-x_{n}''(t)\leq p(t)h_{1}(\theta_{5}(x_{n}(t)))k_{1}(\theta_{5}(y_{n}(t)))(u_{1}(\rho_{5}(x_{n}'(t)))+v_{1}(\rho_{5}(x_{n}'(t)))),\hspace{0.4cm}t\in(0,1),\end{align*}
which implies that
\begin{align*}\frac{-x_{n}''(t)}{u_{1}(\rho_{5}(x_{n}'(t)))+v_{1}(\rho_{5}(x_{n}'(t)))}\leq h_{1}(M_{5})k_{1}(M_{5})p(t),\hspace{0.4cm}t\in(0,1).\end{align*}
Integrating from $t_{1}$ to $1$, using the BCs \eqref{4.2.3}, we
obtain
\begin{align*}\int_{\frac{1}{n}}^{x_{n}'(t_{1})}\frac{dz}{u_{1}(\rho_{5}(z))+v_{1}(\rho_{5}(z))}\leq h_{1}(M_{5})k_{1}(M_{5})\int_{t_{1}}^{1}p(t)dt,\end{align*}
which can also be written as
\begin{align*}\int_{\frac{1}{n}}^{L_{5}}\frac{dz}{u_{1}(z)+v_{1}(z)}+\int_{L_{5}}^{x_{n}'(t_{1})}\frac{dz}{u_{1}(L_{5})+v_{1}(L_{5})}
\leq h_{1}(M_{5})k_{1}(M_{5})\int_{0}^{1}p(t)dt.\end{align*} Using
the increasing property of $I$, we obtain
\begin{align*}I(L_{5})+\frac{x_{n}'(t_{1})-L_{5}}{u_{1}(L_{5})+v_{1}(L_{5})}
\leq
h_{1}(M_{5})k_{1}(M_{5})\int_{0}^{1}p(t)dt+I(\varepsilon),\end{align*}
and using the increasing property of $I^{-1}$, leads to
\begin{align*}L_{5}\leq I^{-1}(h_{1}(M_{5})k_{1}(M_{5})\int_{0}^{1}p(t)dt+I(\varepsilon)).\end{align*}
Which is a contradiction to \eqref{4.2.2}. Hence, $x_{n}'(t)<L_{5}$
for $t\in[0,1]$. Similarly, we can show that $y_{n}'(t)<L_{5}$ for $t\in[0,1]$.

\vskip 0.5em

Now, we show that
\begin{equation}\label{4.2.8}\|x_{n}\|+\|y_{n}\|<M_{5}.\end{equation}
Suppose $\|x_{n}\|+\|y_{n}\|\geq M_{5}$. From \eqref{4.2.3},
\eqref{4.2.4}, \eqref{4.2.7} and $(\mathbf{B}_{3})$, it follows
that
\begin{align*}
-x_{n}''(t)&\leq p(t)h_{1}(\theta_{5}(x_{n}(t)))k_{1}(\theta_{5}(y_{n}(t)))(u_{1}(x_{n}'(t))+v_{1}(x_{n}'(t))),\hspace{0.4cm}t\in(0,1),\\
-y_{n}''(t)&\leq
q(t)h_{2}(\theta_{5}(x_{n}(t)))k_{2}(\theta_{5}(y_{n}(t)))(u_{2}(y_{n}'(t))+v_{2}(y_{n}'(t))),\hspace{0.4cm}t\in(0,1),\end{align*}
which implies that
\begin{align*}
\frac{-x_{n}''(t)}{u_{1}(x_{n}'(t))+v_{1}(x_{n}'(t))}&\leq h_{1}(M_{5})k_{1}(M_{5})p(t),\hspace{0.4cm}t\in(0,1),\\
\frac{-y_{n}''(t)}{u_{2}(y_{n}'(t))+v_{2}(y_{n}'(t))}&\leq
h_{2}(M_{5})k_{2}(M_{5})q(t),\hspace{0.4cm}t\in(0,1).\end{align*}
Integrating from $t$ to $1$, using the BCs \eqref{4.2.3}, we obtain
\begin{align*}
\int_{\frac{1}{n}}^{x_{n}'(t)}\frac{dz}{u_{1}(z)+v_{1}(z)}&\leq
h_{1}(M_{5})k_{1}(M_{5})\int_{t}^{1}p(s)ds,\hspace{0.4cm}t\in[0,1],\\
\int_{\frac{1}{n}}^{y_{n}'(t)}\frac{dz}{u_{2}(z)+v_{2}(z)}&\leq
h_{2}(M_{5})k_{2}(M_{5})\int_{t}^{1}q(s)ds,\hspace{0.4cm}t\in[0,1],
\end{align*}
which can also be written as
\begin{align*}
I(x_{n}'(t))-I(\frac{1}{n})&\leq
h_{1}(M_{5})k_{1}(M_{5})\int_{0}^{1}p(s)ds,\hspace{0.4cm}t\in[0,1],\\
J(y_{n}'(t))-J(\frac{1}{n})&\leq
h_{2}(M_{5})k_{2}(M_{5})\int_{0}^{1}q(s)ds,\hspace{0.4cm}t\in[0,1].
\end{align*}
The increasing property of $I$ and $J$ leads to
\begin{equation}\label{4.2.9}\begin{split}
x_{n}'(t)&\leq
I^{-1}(h_{1}(M_{5})k_{1}(M_{5})\int_{0}^{1}p(s)ds+I(\varepsilon)),\hspace{0.4cm}t\in[0,1],\\
y_{n}'(t)&\leq
J^{-1}(h_{2}(M_{5})k_{2}(M_{5})\int_{0}^{1}q(s)ds+J(\varepsilon)),\hspace{0.35cm}t\in[0,1].
\end{split}\end{equation}
Integrating from $0$ to $t$, using the BCs \eqref{4.2.3} and
\eqref{4.2.9}, we obtain
\begin{equation}\label{4.2.10}\begin{split}
&x_{n}(t)\leq
I^{-1}(h_{1}(M_{5})k_{1}(M_{5})\int_{0}^{1}p(s)ds+I(\varepsilon))\\
&+\frac{b_{2}}{a_{2}}J^{-1}(h_{2}(M_{5})k_{2}(M_{5})\int_{0}^{1}q(s)ds+J(\varepsilon)),\hspace{0.4cm}t\in[0,1],\\
&y_{n}(t)\leq\frac{b_{1}}{a_{1}}I^{-1}(h_{1}(M_{5})k_{1}(M_{5})\int_{0}^{1}p(s)ds+I(\varepsilon))\\
&+J^{-1}(h_{2}(M_{5})k_{2}(M_{5})\int_{0}^{1}q(s)ds+J(\varepsilon)),\hspace{0.8cm}t\in[0,1].
\end{split}\end{equation}
From \eqref{4.2.10} and \eqref{4.2.6}, it follows that
\begin{align*}&M_{5}\leq\|x_{n}\|+\|y_{n}\|\leq(1+\frac{b_{1}}{a_{1}})I^{-1}(h_{1}(M_{5})k_{1}(M_{5})\int_{0}^{1}p(s)ds+I(\varepsilon))\\
&+(1+\frac{b_{2}}{a_{2}})J^{-1}(h_{2}(M_{5})k_{2}(M_{5})\int_{0}^{1}q(s)ds+J(\varepsilon)),\end{align*}
which implies that
\begin{equation*}\small{\begin{split}
\frac{M_{5}}{(1+\frac{b_{1}}{a_{1}})I^{-1}(h_{1}(M_{5})k_{1}(M_{5})\int_{0}^{1}p(s)ds+I(\varepsilon))
+(1+\frac{b_{2}}{a_{2}})J^{-1}(h_{2}(M_{5})k_{2}(M_{5})\int_{0}^{1}q(s)ds+J(\varepsilon))}\leq1,
\end{split}}\end{equation*}
a contradiction to \eqref{4.2.1}. Hence,
$\|x_{n}\|+\|y_{n}\|<M_{5}$.

\vskip 0.5em

Thus, in view of \eqref{4.2.3}-\eqref{4.2.8}, $(x_{n},y_{n})$ is a
solution of the following coupled system of BVPs
 \begin{equation}\label{4.2.11}\begin{split}
-&x''(t)=p(t)f(t,x(t),y(t),x'(t)),\hspace{0.4cm}t\in(0,1),\\
-&y''(t)=q(t)g(t,x(t),y(t),y'(t)),\hspace{0.4cm}t\in(0,1),\\
&a_{1}y(0)-b_{1}x'(0)=0,\,x'(1)=\frac{1}{n},\\
&a_{2}x(0)-b_{2}y'(0)=0,\,x'(1)=\frac{1}{n},
\end{split}\end{equation}
satisfying
\begin{equation}\label{4.2.12}\begin{split}
(t+\frac{b_{2}}{a_{2}})\frac{1}{n}&\leq
x_{n}(t)<M_{5},\,\frac{1}{n}\leq
x_{n}'(t)<L_{5},\hspace{0.4cm}t\in[0,1],\\
(t+\frac{b_{1}}{a_{1}})\frac{1}{n}&\leq
y_{n}(t)<M_{5},\,\frac{1}{n}\leq
y_{n}'(t)<L_{5},\hspace{0.4cm}t\in[0,1].\end{split}\end{equation} We
claim that
\begin{equation}\label{4.2.13}x_{n}'(t)\geq C_{9}^{\gamma_{1}} C_{10}^{\delta_{1}}\int_{t}^{1}p(s)\varphi_{M_{5}L_{5}}(s)ds,\end{equation}
\begin{equation}\label{4.2.14}y_{n}'(t)\geq
C_{9}^{\gamma_{2}}C_{10}^{\delta_{2}}\int_{t}^{1}q(s)\psi_{M_{5}L_{5}}(s)ds,\end{equation}
where
\begin{align*}\begin{split}
C_{9}=&\left(\frac{b_{1}}{a_{1}}\int_{0}^{1}p(s)\varphi_{M_{5}L_{5}}(s)ds\right)^{\frac{\delta_{1}}{(1-\gamma_{1})(1-\gamma_{2})-\delta_{1}\delta_{2}}}\left(\frac{b_{2}}{a_{2}}\int_{0}^{1}q(s)\psi_{M_{5}L_{5}}(s)ds\right)^{\frac{1-\gamma_{2}}{(1-\gamma_{1})(1-\gamma_{2})-\delta_{1}\delta_{2}}},\\
C_{10}=&\left(\frac{b_{1}}{a_{1}}\int_{0}^{1}p(s)\varphi_{M_{5}L_{5}}(s)ds\right)^{\frac{1-\gamma_{1}}{(1-\gamma_{1})(1-\gamma_{2})-\delta_{1}\delta_{2}}}\left(\frac{b_{2}}{a_{2}}\int_{0}^{1}q(s)\psi_{M_{5}L_{5}}(s)ds\right)^{\frac{\delta_{2}}{(1-\gamma_{1})(1-\gamma_{2})-\delta_{1}\delta_{2}}}.
\end{split}\end{align*}
To prove \eqref{4.2.13}, consider the following relation
\begin{equation}\label{4.2.15}\begin{split}
&x_{n}(t)=(t+\frac{b_{2}}{a_{2}})\frac{1}{n}+\int_{0}^{t}sp(s)f(s,x_{n}(s),y_{n}(s),x_{n}'(s))ds\\
&+\int_{t}^{1}tp(s)f(s,x_{n}(s),y_{n}(s),x_{n}'(s))ds+\frac{b_{2}}{a_{2}}\int_{t}^{1}q(s)g(s,x_{n}(s),y_{n}(s),y_{n}'(s))ds,
\end{split}\end{equation} which implies that
\begin{align*}
x_{n}(0)=\frac{b_{2}}{a_{2}}\frac{1}{n}+\frac{b_{2}}{a_{2}}\int_{0}^{1}q(s)g(s,x_{n}(s),y_{n}'(s))ds.
\end{align*}
Using $(\mathbf{B}_{16})$ and \eqref{4.2.12}, we obtain
\begin{equation*}\small{\begin{split}
x_{n}(0)\geq\frac{b_{2}}{a_{2}}\int_{0}^{1}q(s)\psi_{M_{5}L_{5}}(s)(x_{n}(s))^{\gamma_{1}}(y_{n}(s))^{\delta_{1}}ds\geq(x_{n}(0))^{\gamma_{1}}(y_{n}(0))^{\delta_{1}}\frac{b_{2}}{a_{2}}\int_{0}^{1}q(s)\psi_{M_{5}L_{5}}(s)ds,
\end{split}}\end{equation*}
which implies that
\begin{equation}\label{4.2.16}
x_{n}(0)\geq(y_{n}(0))^{\frac{\delta_{1}}{1-\gamma_{1}}}\left(\frac{b_{2}}{a_{2}}\int_{0}^{1}q(s)\psi_{M_{5}L_{5}}(s)ds\right)^{\frac{1}{1-\gamma_{1}}}.
\end{equation}
Similarly, using $(\mathbf{B}_{16})$ and \eqref{4.2.12}, we obtain
\begin{equation}\label{4.2.17}
y_{n}(0)\geq(x_{n}(0))^{\frac{\delta_{2}}{1-\gamma_{2}}}\left(\frac{b_{1}}{a_{1}}\int_{0}^{1}p(s)\varphi_{M_{5}L_{5}}(s)ds\right)^{\frac{1}{1-\gamma_{2}}}.
\end{equation}
Now, using \eqref{4.2.17} in \eqref{4.2.16}, we have
\begin{align*}
x_{n}(0)\geq&(x_{n}(0))^{\frac{\delta_{1}\delta_{2}}{(1-\gamma_{1})(1-\gamma_{2})}}\left(\frac{b_{1}}{a_{1}}\int_{0}^{1}p(s)\varphi_{M_{5}L_{5}}(s)ds\right)^{\frac{\delta_{1}}{(1-\gamma_{1})(1-\gamma_{2})}}\left(\frac{b_{2}}{a_{2}}\int_{0}^{1}q(s)\psi_{M_{5}L_{5}}(s)ds\right)^{\frac{1}{1-\gamma_{1}}}.
\end{align*}
Hence,
\begin{equation}\label{4.2.18}
x_{n}(0) \geq C_{9}.
\end{equation}
Similarly, using \eqref{4.2.16} in \eqref{4.2.17}, we obtain
\begin{equation}\label{4.2.19}
y_{n}(0) \geq C_{10}.
\end{equation}
Now, from \eqref{4.2.15}, it follows that
\begin{align*}
x_{n}'(t)\geq\int_{t}^{1}p(s)f(s,x_{n}(s),y_{n}(s),x_{n}'(s))ds.
\end{align*}
and using $(\mathbf{B}_{16})$, \eqref{4.2.12}, \eqref{4.2.18} and
\eqref{4.2.19}, we obtain \eqref{4.2.13}. Similarly, we can prove \eqref{4.2.14}.

\vskip 0.5em

Now, using \eqref{4.2.11}, $(\mathbf{B}_{3})$, \eqref{4.2.12},
\eqref{4.2.13} and \eqref{4.2.14}, we have
\begin{equation}\label{4.2.20}\begin{split}
0\leq-x_{n}''(t)&\leq h_{1}(M_{5})k_{1}(M_{5})p(t)(u_{1}(C_{9}^{\gamma_{1}}C_{10}^{\delta_{1}}\int_{t}^{1}p(s)\varphi_{M_{5}L_{5}}(s)ds)+v_{1}(L_{5})),\hspace{0.4cm}t\in(0,1),\\
0\leq-y_{n}''(t)&\leq
h_{2}(M_{5})k_{2}(M_{5})q(t)(u_{2}(C_{9}^{\gamma_{2}}C_{10}^{\delta_{2}}\int_{t}^{1}q(s)\psi_{M_{5}L_{5}}(s)ds)+v_{2}(L_{5})),\hspace{0.4cm}t\in(0,1).
\end{split}\end{equation}
In view of \eqref{4.2.12}, \eqref{4.2.20}, $(\mathbf{B}_{1})$ and
$(\mathbf{B}_{17})$, it follows that the sequences
$\{(x_{n}^{(j)},y_{n}^{(j)})\}$ $(j=0,1)$ are uniformly bounded and
equicontinuous on $[0,1]$. Hence, by Theorem \eqref{arzela}, there
exist subsequences $\{(x_{n_{k}}^{(j)},y_{n_{k}}^{(j)})\}$ $(j=0,1)$
of $\{(x_{n}^{(j)},y_{n}^{(j)})\}$ $(j=0,1)$ and $(x,y)\in
C^{1}[0,1]\times C^{1}[0,1]$ such that
$(x_{n_{k}}^{(j)},y_{n_{k}}^{(j)})$ converges uniformly to
$(x^{(j)},y^{(j)})$ on $[0,1]$  $(j=0,1)$. Also,
$a_{2}x(0)-b_{2}y'(0)=a_{1}y(0)-b_{1}x'(0)=x'(1)=y'(1)=0$. Moreover,
from \eqref{4.2.13} and \eqref{4.2.14}, with $n_{k}$ in place of $n$
and taking $\lim_{n_{k}\rightarrow+\infty}$, we have
\begin{align*}x'(t)\geq C_{9}^{\gamma_{1}}C_{10}^{\delta_{1}}\int_{t}^{1}p(s)\varphi_{M_{5}L_{5}}(s)ds,\\
y'(t)\geq
C_{9}^{\gamma_{2}}C_{10}^{\delta_{2}}\int_{t}^{1}q(s)\psi_{M_{5}L_{5}}(s)ds,\end{align*}
which shows that $x'>0$ and $y'>0$ on $[0,1)$, $x>0$ and $y>0$ on
$[0,1]$. Further, $(x_{n_{k}},y_{n_{k}})$ satisfy
\begin{align*}
x_{n_{k}}'(t)&=x_{n_{k}}'(0)-\int_{0}^{t}p(s)f(s,x_{n_{k}}(s),y_{n_{k}}(s),x_{n_{k}}'(s))ds,\hspace{0.4cm}t\in[0,1],\\
y_{n_{k}}'(t)&=y_{n_{k}}'(0)-\int_{0}^{t}q(s)g(s,x_{n_{k}}(s),y_{n_{k}}(s),y_{n_{k}}'(s))ds,\hspace{0.4cm}t\in[0,1].
\end{align*}
Passing to the limit as $n_{k}\rightarrow\infty$, we obtain
\begin{align*}
x'(t)&=x'(0)-\int_{0}^{t}p(s)f(s,x(s),y(s),x'(s))ds,\hspace{0.4cm}t\in[0,1],\\
y'(t)&=y'(0)-\int_{0}^{t}q(s)g(s,x(s),y(s),y'(s))ds,\hspace{0.4cm}t\in[0,1],
\end{align*}
which implies that
\begin{align*}
-x''(t)=p(t)f(t,x(t),y(t),x'(t)),\hspace{0.4cm}t\in(0,1),\\
-y''(t)=q(t)g(t,x(t),y(t),y'(t)),\hspace{0.4cm}t\in(0,1).
\end{align*}
Hence, $(x,y)$ is a $C^{1}$-positive solution of the system of BVPs
\eqref{4.0.2}.
\end{proof}

\begin{ex}Consider the following coupled system of SBVPs
\begin{equation}\begin{split}\label{4.2.21}
-&x''(t)=\nu^{\beta_{1}+1}(x(t))^{\gamma_{1}}(y(t))^{\delta_{1}}(x'(t))^{-\beta_{1}},\hspace{0.4cm}t\in(0,1),\\
-&y''(t)=\nu^{\beta_{2}+1}(x(t))^{\gamma_{2}}(y(t))^{\delta_{2}}(y'(t))^{-\beta_{2}},\hspace{0.4cm}t\in(0,1),\\
&x(0)-y'(0)=y(0)-x'(0)=x'(1)=y'(1)=0,
\end{split}\end{equation}
where $0\leq\gamma_{1},\gamma_{2},\delta_{1},\delta_{2}<1$
satisfying $(1-\gamma_{1})(1-\gamma_{2})\neq\delta_{1}\delta_{2}$,
$0<\beta_{1}<1$, $0<\beta_{2}<1$ and $\nu>0$ such that
\begin{align*}\nu<\sup_{c\in(0,\infty)}\frac{c}{2\sum_{i=1}^{2}(\beta_{i}+1)^{\frac{1}{\beta_{i}+1}}c^{\frac{{\gamma_{i}+\delta_{i}}}{\beta_{i}+1}}}.\end{align*}

Choose $p(t)=q(t)=1$, $h_{1}(x)=\nu^{\beta_{1}+1} x^{\gamma_{1}}$,
$h_{2}(x)=\nu^{\beta_{2}+1} x^{\gamma_{2}}$,
$k_{1}(x)=x^{\delta_{1}}$, $k_{2}(x)=x^{\delta_{2}}$,
$u_{1}(x)=x^{-\beta_{1}}$, $u_{2}(x)=x^{-\beta_{2}}$ and
$v_{1}(x)=v_{2}(x)=0$, $\varphi_{ML}(t)=L^{-\beta_{1}}$,
$\psi_{ML}(t)=L^{-\beta_{2}}$.

Then,
$I(\nu)=\frac{\nu^{\beta_{1}+1}}{\beta_{1}+1}$,
$J(\nu)=\frac{\nu^{\beta_{2}+1}}{\beta_{2}+1}$,
$I^{-1}(\nu)=(\beta_{1}+1)^{\frac{1}{\beta_{1}+1}}\nu^{\frac{1}{\beta_{1}+1}}$
and
$J^{-1}(\nu)=(\beta_{2}+1)^{\frac{1}{\beta_{2}+1}}z^{\frac{1}{\beta_{2}+1}}$. Also,
\begin{align*}
\sup_{c\in(0,\infty)}&\frac{c}{(1+\frac{b_{1}}{a_{1}})I^{-1}(h_{1}(c)k_{1}(c)\int_{0}^{1}p(t)dt)+(1+\frac{b_{2}}{a_{2}})J^{-1}(h_{2}(c)k_{2}(c)\int_{0}^{1}q(t)dt)}=\\
\sup_{c\in(0,\infty)}&\frac{c}{2\nu\sum_{i=1}^{2}(\beta_{i}+1)^{\frac{1}{\beta_{i}+1}}c^{\frac{{\gamma_{i}+\delta_{i}}}{\beta_{i}+1}}}>1.
\end{align*}
Clearly, $(\mathbf{B}_{1})-(\mathbf{B}_{17})$ are satisfied.
Hence, by Theorem \ref{th5.2}, the system of BVPs \eqref{4.2.21} has
at least one $C^{1}$-positive solution.
\end{ex}


\section{Existence of at least one $C^{1}$-positive
solution via lower and upper solutions}\label{finite}

In this section, we establish existence of at least one
$C^{1}$-positive solution of the system of BVPs \eqref{1.6} \cite{asif1}. By a
$C^{1}$-positive solution to the system of BVPs \eqref{1.6}, we
means $(x,y)\in (C^{1}[0,1]\cap C^{2}(0,1))\times(C^{1}[0,1]\cap
C^{2}(0,1))$ satisfying \eqref{1.6}, $x>0$ and $y>0$ on $(0,1)$.

\vskip 0.5em

Let $\{\rho_{n}\}_{n=1}^{\infty}$ be a nonincreasing sequence of
real constants such that $\lim_{n\rightarrow\infty}\rho_{n}=0$.
Assume that the following holds:
\begin{description}
\item[$(\mathbf{B}_{18})$] $p_{i}\in C(0,1)$, $p_{i}>0$ on $(0,1)$ and $\int_{0}^{1}p_{i}(t)dt<+\infty$, $i=1,2$;
\item[$(\mathbf{B}_{19})$] $f_{i}:[0,1]\times(0,\infty)\times(0,\infty)\times\mathbb{R}\rightarrow\mathbb{R}$ are
continuous, $i=1,2$;
\item[$(\mathbf{B}_{20})$] there exist $(\beta_{1},\beta_{2})\in
(C^{1}[0,1]\cap C^{2}(0,1))\times(C^{1}[0,1]\cap C^{2}(0,1))$ and
$n_{0}\in\{1,2,\cdots\}$ such that $\beta_{1}(t)\geq\rho_{n_{0}}$,
$\beta_{2}(t)\geq\rho_{n_{0}}$ for $t\in[0,1]$ and
\begin{align*}
-\beta_{1}''(t)&\geq p_{1}(t)f_{1}(t,\beta_{1}(t),\beta_{2}(t),\beta_{1}'(t)),\,\,\,\,t\in(0,1),\\
-\beta_{2}''(t)&\geq
p_{2}(t)f_{2}(t,\beta_{1}(t),\beta_{2}(t),\beta_{2}'(t)),\,\,\,\,t\in(0,1);
\end{align*}
\item[$(\mathbf{B}_{21})$] there exist $(\alpha_{1},\alpha_{2})\in (C^{1}[0,1]\cap C^{2}(0,1))\times(C^{1}[0,1]\cap C^{2}(0,1))$ with
$\alpha_{1}(0)=\alpha_{1}(1)=\alpha_{2}(0)=\alpha_{2}(1)=0$,
$\alpha_{1}>0$ and $\alpha_{2}>0$ on $(0,1)$ such that for
$(t,x,y)\in(0,1)\times\{x\in(0,\infty):x<\alpha_{1}(t)\}\times\{y\in(0,\infty):y\leq\beta_{2}(t)\}$,
\begin{align*}
-\alpha_{1}''(t)&<p_{1}(t)f_{1}(t,x,y,\alpha_{1}'(t)),
\end{align*}
for
$(t,x,y)\in(0,1)\times\{x\in(0,\infty):x\leq\beta_{1}(t)\}\times\{y\in(0,\infty):y<\alpha_{2}(t)\}$,
\begin{align*}
-\alpha_{2}''(t)&<p_{2}(t)f_{2}(t,x,y,\alpha_{2}'(t));
\end{align*}
\item[$(\mathbf{B}_{22})$] for each $n\in\{n_{0},n_{0}+1,\cdots\}$, $0\leq t\leq 1$, $\rho_{n}\leq x\leq \beta_{1}(t)$, $\rho_{n}\leq y\leq
\beta_{2}(t)$, we have $f_{1}(t,\rho_{n},y,0)\geq 0$ and
$f_{2}(t,x,\rho_{n},0)\geq 0$;
\item[$(\mathbf{B}_{23})$] $|f_{i}(t,x,y,z)|\leq (h_{i}(x)+k_{i}(x))(u_{i}(y)+v_{i}(y))\psi_{i}(|z|)$, where $h_{i},u_{i}>0$ are continuous and nonincreasing
on $(0,\infty)$, $k_{i},v_{i}\geq 0$, $\psi_{i}>0$ are continuous on
$[0,\infty)$ with $\frac{k_{i}}{h_{i}},\frac{v_{i}}{u_{i}}$
nondecreasing on $(0,\infty)$, $i=1,2$;
\item[$(\mathbf{B}_{24})$]$\int_{0}^{1}p_{i}(t)h_{i}(\alpha_{1}(t))u_{i}(\alpha_{2}(t))dt<+\infty$, $i=1,2$;
\item[$(\mathbf{B}_{25})$]
\begin{align*}
\int_{0}^{\infty}\frac{du}{\psi_{i}(u)}>\left[1+\frac{k_{i}(b_{1})}{h_{i}(b_{1})}\right]\left[1+\frac{v_{i}(b_{2})}{u_{i}(b_{2})}\right]\int_{0}^{1}p_{i}(t)h_{i}(\alpha_{1}(t))u_{i}(\alpha_{2}(t))dt,
\end{align*}
where $b_{i}=\max\{\beta_{i}(t):t\in[0,1]\}$, $i=1,2$.
\end{description}
\begin{thm}\label{th1}
Assume that $(\mathbf{B}_{18})-(\mathbf{B}_{25})$ hold. Then the system of BVPs \eqref{1.6} has a $C^{1}$-positive solution.
\end{thm}

%% file: Ch4.tex
\chapter[Singular Systems of Second-Order ODEs on Infinite Intervals]{Singular Systems of BVPs on Infinite Intervals}\label{ch5}

Recently, the theory on existence of solutions to nonlinear BVPs on unbounded
domain has attracted the attention of many authors, see for example
\cite{agarwalinfinite,eloe,kw1,liange,liu1,rma,yl} and the
references therein. For BVPs defined on half-line, an excellent
resource is produced by Agarwal and O'Regan \cite{agarwalinfinite}
that have been received considerable attentions.

\vskip 0.5em

Agarwal and O'Regan \cite[Section~2.15]{ao} studied the existence of positive
solutions to the following BVP
\begin{equation}\begin{split}\label{s1.3}
-x''(t)&=\phi(t)f(t,x(t)),\hspace{0.4cm}t\in(0,\infty),\\
x(0)&=0,\,\lim_{t\rightarrow\infty}x'(t)=0,
\end{split}\end{equation}
where $f(t,x)$ is singular at $x=0$. Further, in
\cite[Section~1.11]{agarwalinfinite} they establish the existence
results for \eqref{s1.3} when $f$ includes first derivative also. However in
\cite{wei,cheng}, it was assumed that the nonlinearities are
positive which leads to concave solutions.

\vskip 0.5em

In Sections \ref{infinite-one} and \ref{infinite-two}, we
study the existence of $C^{1}$-positive solutions to the following
coupled systems of SBVPs
\begin{equation}\label{bc3}\begin{split}
-&x''(t)=p_{1}(t)f_{1}(t,x(t),y(t),x'(t)),\hspace{0.4cm}t\in\R_{0}^{+},\\
-&y''(t)=p_{2}(t)f_{2}(t,x(t),y(t),y'(t)),\hspace{0.4cm}t\in\R_{0}^{+},\\
&x(0)=y(0)=\lim_{t\rightarrow\infty}y'(t)=\lim_{t\rightarrow\infty}x'(t)=0
\end{split}\end{equation}
and
\begin{equation}\label{s1.6}\begin{split}
-&x''(t)=p_{1}(t)f_{1}(t,x(t),y(t),x'(t)),\hspace{0.4cm}t\in\R_{0}^{+},\\
-&y''(t)=p_{2}(t)f_{2}(t,x(t),y(t),y'(t)),\hspace{0.4cm}t\in\R_{0}^{+},\\
&a_{1}x(0)-b_{1}x'(0)=\lim_{t\rightarrow\infty}x'(t)=0,\\
&a_{2}y(0)-b_{2}y'(0)=\lim_{t\rightarrow\infty}y'(t)=0,
\end{split}\end{equation}
where the functions $f_{i}:\R^{+}\times\R^{2}\times
\R_{0}\rightarrow \R$ are continuous and allowed to change sign.
Further, the nonlinearities $f_{i}$ $(i=1,2)$ are allowed to be
singular at $x'=0$ and $y'=0$. Also, $p_{i}\in C(\R_{0}^{+})$,
$p_{i}(i=1,2)>0$ on $\R_{0}^{+}$ and the constants
$a_{i},b_{i}(i=1,2)>0$; here $\R=(-\infty,\infty)$,
$\R_{0}=\R\setminus\{0\}$, $\R^{+}=[0,\infty)$,
$\R_{0}^{+}=\R^{+}\setminus\{0\}$.


\section{Systems of BVPs on infinite intervals}\label{infinite-one}
In this section, we establish the existence of $C^{1}$-positive
solutions for the system of BVPs \eqref{bc3}. We say,
$(x,y)\in(C^{1}(\R^{+})\cap
C^{2}(\R_{0}^{+}))\times(C^{1}(\R^{+})\cap C^{2}(\R_{0}^{+}))$ is a
$C^{1}$-positive solution of the system of BVPs \eqref{bc3}, if $(x,y)$ satisfies \eqref{bc3},
$x>0$ and $y>0$ on $\mathbb{R}_{0}^{+}$, $x'>0$ and $y'>0$ on
$\R^{+}$. 

\vskip 0.5em

Assume that
\begin{description}
\item[$(\mathbf{C}_{1})$] $p_{i}\in C(\R_{0}^{+})$, $p_{i}>0$ on $\R_{0}^{+}$, $\int_{0}^{\infty}p_{i}(t)dt<+\infty$, $i=1,2$;
\item[$(\mathbf{C}_{2})$] $f_{i}:\R^{+}\times\R^{2}\times\R_{0}\rightarrow\R$ is continuous, $i=1,2$;
\item[$(\mathbf{C}_{3})$] $|f_{i}(t,x,y,z)|\leq h_{i}(|x|)k_{i}(|y|)(u_{i}(|z|)+v_{i}(|z|))$, where
$u_{i}>0$ is continuous
 and nonincreasing on $\R_{0}^{+}$, $h_{i},\,k_{i},\,v_{i}\geq 0$ are continuous and nondecreasing on
 $\R^{+}$, $i=1,2$;
\item[$(\mathbf{C}_{4})$] there exist a constant $M>0$ such that
$\frac{M}{\omega(M)}>1$, where
$\omega(M)=\lim_{\varepsilon\rightarrow 0}\omega_{\varepsilon}(M)$,
\begin{align*}\omega_{\varepsilon}(M)&=\sum_{i=1}^{2}\int_{0}^{\infty}[J_{i}^{-1}\big(h_{i}(M)k_{i}(M)\int_{t}^{\infty}p_{i}(s)ds+J_{i}(\varepsilon)\big)]dt\\
&\hspace{0.4cm}+\sum_{i=1}^{2}J_{i}^{-1}\big(h_{i}(M)k_{i}(M)\int_{0}^{\infty}p_{i}(s)ds+J_{i}(\varepsilon)\big),\\
J_{i}(\mu)&=\int_{0}^{\mu}\frac{d\tau}{u_{i}(\tau)+v_{i}(\tau)},\text{
for }\mu>0,\, i=1,2;
\end{align*}
\item[$(\mathbf{C}_{5})$] $J_{1}(\infty)=\infty$ and $J_{2}(\infty)=\infty$;
\item[$(\mathbf{C}_{6})$] $f_{i}$ is positive on $\R^{+}\times(0,M]^{3}$,
$i=1,2$;
\item[$(\mathbf{C}_{7})$] there exist continuous functions $\varphi_{M}$ and $\psi_{M}$ defined on
$\R^{+}$ and positive on $\R_{0}^{+}$, and constants
$0\leq\gamma_{1},\gamma_{2},\delta_{1},\delta_{2}<1$ satisfying
$(1-\gamma_{1})(1-\gamma_{2})\neq\delta_{1}\delta_{2}$, such that
$f_{1}(t,x,y,z)\geq\varphi_{M}(t)x^{\gamma_{1}}y^{\delta_{1}}$ and
$f_{2}(t,x,y,z)\geq\psi_{M}(t)x^{\gamma_{2}}y^{\delta_{2}}$ on
$\R^{+}\times[0,M]^{3}$.
\end{description}
\subsection{Existence of positive solutions on finite intervals} Choose
$m\in N_{0}\setminus\{0\}$, where $N_{0}:=\{0,1,\cdots\}$, and
consider the following system of BVPs on finite interval
\begin{equation}\label{t2.1}\begin{split}
-&x''(t)=p_{1}(t)f_{1}(t,x(t),y(t),x'(t)),\hspace{0.4cm}t\in(0,m),\\
-&y''(t)=p_{2}(t)f_{2}(t,x(t),y(t),y'(t)),\hspace{0.4cm}t\in(0,m),\\
&x(0)=y(0)=x'(m)=y'(m)=0.
\end{split}\end{equation} First, we show
that system of BVPs \eqref{t2.1} has a $C^{1}$-positive solution. We
say, $(x,y)\in(C^{1}[0,m]\cap C^{2}(0,m))\times(C^{1}[0,m]\cap
C^{2}(0,m))$, a $C^{1}$-positive solution of the system of BVPs
\eqref{t2.1}, if $(x,y)$ satisfies \eqref{t2.1}, $x>0$ and $y>0$ on
$(0,m]$, $x'>0$ and $y'>0$ on $[0,m)$.
\begin{thm}\label{th11}
Assume that $(\mathbf{C}_{1})-(\mathbf{C}_{7})$ hold. Then the system of BVPs \eqref{t2.1} has a $C^{1}$-positive solution.
\end{thm}
\begin{proof}
In view of $(\mathbf{C}_{4})$, we choose $\varepsilon>0$ small
enough such that
\begin{equation}\label{t2.3}\frac{M}{\omega_{\varepsilon}(M)}>1.\end{equation}
Choose $n_{0}\in\{1,2,\cdots\}$ such that
$\frac{1}{n_{0}}<\varepsilon$. For each $n\in
N:=\{n_{0},n_{0}+1,\cdots\}$, define retractions
$\theta:\R\rightarrow[0,M]$ and $\rho:\R\rightarrow[\frac{1}{n},M]$
as
\begin{align*}\theta(x)=\max\{0,\min\{x,M\}\}\text{ and
}\rho(x)=\max\{\frac{1}{n},\min\{x,M\}\}.\end{align*} Consider the
modified system of BVPs
\begin{equation}\label{t2.5}\begin{split}
-&x''(t)=p_{1}(t)f_{1}^{*}(t,x(t),y(t),x'(t)),\hspace{0.4cm}t\in(0,m),\\
-&y''(t)=p_{2}(t)f_{2}^{*}(t,x(t),y(t),x'(t)),\hspace{0.4cm}t\in(0,m),\\
&x(0)=y(0)=0,\,x'(m)=y'(m)=\frac{1}{n},
\end{split}\end{equation}where
$f_{1}^{*}(t,x,y,x')=f_{1}(t,\theta(x),\theta(y),\rho(x'))$ and
$f_{2}^{*}(t,x,y,y')=f_{2}(t,\theta(x),\theta(y),\rho(y'))$.
Clearly, $f_{i}^{*}\,(i=1,2)$ are continuous and bounded on
$[0,m]\times \R^{3}$. Hence, by Theorem \ref{schauder}, the modified
system of BVPs \eqref{t2.5} has a solution
$(x_{m,n},y_{m,n})\in(C^{1}[0,m]\cap
C^{2}(0,m))\times(C^{1}[0,m]\cap C^{2}(0,m))$.

\vskip 0.5em

Using \eqref{t2.5}, $(\mathbf{C}_{1})$ and $(\mathbf{C}_{6})$,
we obtain
\begin{align*}
x_{m,n}''\leq0 \text{ and }y_{m,n}''\leq 0\text{ on }\in(0,m).
\end{align*}
Integrating from $t$ to $m$ and using the BCs \eqref{t2.5}, we
obtain
\begin{equation}\label{t2.6}
x_{m,n}'(t)\geq\frac{1}{n}\text{ and
}y_{m,n}'(t)\geq\frac{1}{n}\text{ for }t\in[0,m].
\end{equation}
Integrating \eqref{t2.6} from $0$ to $t$, using the BCs
\eqref{t2.5}, we have
\begin{equation}\label{t2.7}
x_{m,n}(t)\geq \frac{t}{n}\text{ and }y_{m,n}(t)\geq
\frac{t}{n}\text{ for }t\in[0,m].
\end{equation}
From \eqref{t2.6} and \eqref{t2.7}, it follows that
\begin{align*}\|x_{m,n}\|_{{}_{7,m}}=x_{m,n}(m)\text{ and
}\|y_{m,n}\|_{{}_{7,m}}=y_{m,n}(m),\text{ where
}\|u\|_{{}_{7,m}}=\max_{t\in[0,m]}|u(t)|.
\end{align*}
Now, we show that the following hold
\begin{equation}\label{t2.8}\|x_{m,n}'\|_{{}_{7,m}}<M\text{ and }\|y_{m,n}'\|_{{}_{7,m}}<M.\end{equation}
 Suppose $x_{m,n}'(t_{1})\geq M$ for some $t_{1}\in[0,m]$. Using
\eqref{t2.5} and $(\mathbf{C}_{3})$, we have
\begin{align*}-x_{m,n}''(t)\leq p_{1}(t)h_{1}(\theta(x_{m,n}(t)))k_{1}(\theta(y_{m,n}(t)))(u_{1}(\rho(x_{m,n}'(t)))
+v_{1}(\rho(x_{m,n}'(t)))),\hspace{0.1cm}t\in(0,m),\end{align*}
which implies that
\begin{align*}\frac{-x_{m,n}''(t)}{u_{1}(\rho(x_{m,n}'(t)))+v_{1}(\rho(x_{m,n}'(t)))}\leq h_{1}(M)k_{1}(M)p_{1}(t),\hspace{0.4cm}t\in(0,m).\end{align*}
Integrating from $t_{1}$ to $m$, using the BCs \eqref{t2.5}, we
obtain
\begin{align*}\int_{\frac{1}{n}}^{x_{m,n}'(t_{1})}\frac{dz}{u_{1}(\rho(z))+v_{1}(\rho(z))}\leq h_{1}(M)k_{1}(M)\int_{t_{1}}^{m}p_{1}(t)dt,\end{align*}
which can also be written as
\begin{align*}\int_{\frac{1}{n}}^{M}\frac{dz}{u_{1}(z)+v_{1}(z)}+\int_{M}^{x_{m,n}'(t_{1})}\frac{dz}{u_{1}(M)+v_{1}(M)}
\leq h_{1}(M)k_{1}(M)\int_{0}^{\infty}p_{1}(t)dt.\end{align*} Using
the increasing property of $J_{1}$, we obtain
\begin{align*}J_{1}(M)+\frac{x_{m,n}'(t_{1})-M}{u_{1}(M)+v_{1}(M)}\leq h_{1}(M)k_{1}(M)
\int_{0}^{\infty}p_{1}(t)dt+J_{1}(\varepsilon),\end{align*} and the
increasing property of $J_{1}^{-1}$ yields
\begin{align*}M\leq
J_{1}^{-1}(h_{1}(M)k_{1}(M)\int_{0}^{\infty}p_{1}(t)dt+J_{1}(\varepsilon))\leq\omega_{\varepsilon}(M),\end{align*}
a contradiction to \eqref{t2.3}. Hence, $\|x_{m,n}'\|_{{}_{7,m}}<M$.

\vskip 0.5em

Similarly, we can show that $\|y_{m,n}'\|_{{}_{7,m}}<M$.

\vskip 0.5em

Now, we show that
\begin{equation}\label{t2.9}\|x_{m,n}\|_{{}_{7,m}}<M\text{ and }\|y_{m,n}\|_{{}_{7,m}}<M.\end{equation}
Suppose $\|x_{m,n}\|_{{}_{7,m}}\geq M$. From \eqref{t2.5},
\eqref{t2.6}, \eqref{t2.8} and $(\mathbf{C}_{3})$, it follows that
\begin{align*}
-x_{m,n}''(t)\leq
p_{1}(t)h_{1}(\theta(x_{m,n}(t)))k_{1}(\theta(y_{m,n}(t)))(u_{1}(x_{m,n}'(t))+v_{1}(x_{m,n}'(t))),
\end{align*}
which implies that
\begin{align*}\frac{-x_{m,n}''(t)}{u_{1}(x_{m,n}'(t))+v_{1}(x_{m,n}'(t))}&\leq h_{1}(M)k_{1}(M)p_{1}(t),\hspace{0.4cm}t\in(0,m).\end{align*}
Integrating from $t$ to $m$, using the BCs \eqref{t2.5}, we obtain
\begin{align*}\int_{\frac{1}{n}}^{x_{m,n}'(t)}\frac{dz}{u_{1}(z)+v_{1}(z)}\leq
h_{1}(M)k_{1}(M)\int_{t}^{m}p_{1}(s)ds,\hspace{0.4cm}t\in[0,m],\end{align*}
which can also be written as
\begin{align*}J_{1}(x_{m,n}'(t))-J_{1}(\frac{1}{n})\leq
h_{1}(M)k_{1}(M)\int_{t}^{\infty}p_{1}(s)ds,\hspace{0.4cm}t\in[0,m].\end{align*}
The increasing property of $J_{1}$ and $J_{1}^{-1}$, leads to
\begin{align*}
x_{m,n}'(t)\leq
J_{1}^{-1}(h_{1}(M)k_{1}(M)\int_{t}^{\infty}p_{1}(s)ds+J_{1}(\varepsilon)),\hspace{0.4cm}t\in[0,m].
\end{align*}
Now, integrating from $0$ to $m$, using the BCs \eqref{t2.5}, we
obtain
\begin{align*}
M\leq\|x_{m,n}\|_{{}_{7,m}}\leq
\int_{0}^{m}[J_{1}^{-1}(h_{1}(M)k_{1}(M)\int_{t}^{\infty}p_{1}(s)ds+J_{1}(\varepsilon))]dt,
\end{align*}
which implies that
\begin{align*}
M\leq\int_{0}^{\infty}[J_{1}^{-1}(h_{1}(M)k_{1}(M)\int_{t}^{\infty}p_{1}(s)ds+J_{1}(\varepsilon))]dt\leq\omega_{\varepsilon}(M),
\end{align*}
a contradiction to \eqref{t2.3}. Therefore,
$\|x_{m,n}\|_{{}_{7,m}}<M$.

\vskip 0.5em

Similarly, we can show that $\|y_{m,n}\|_{{}_{7,m}}<M$.

\vskip 0.5em

Hence, in view of \eqref{t2.5}-\eqref{t2.9}, $(x_{m,n},y_{m,n})$ is
a solution of the following coupled system of BVPs
\begin{equation}\label{t2.11}\begin{split}
-&x''(t)=p_{1}(t)f_{1}(t,x(t),y(t),x'(t)),\hspace{0.4cm}t\in(0,m),\\
-&y''(t)=p_{2}(t)f_{2}(t,x(t),y(t),y'(t)),\hspace{0.4cm}t\in(0,m),\\
&x(0)=y(0)=0,\,x'(m)=y'(m)=\frac{1}{n},
\end{split}\end{equation}
satisfying
\begin{equation}\label{t2.12}\begin{split}
\frac{t}{n}&\leq x_{m,n}(t)<M,\,\frac{1}{n}\leq
x_{m,n}'(t)<M,\hspace{0.4cm}t\in[0,m],\\
\frac{t}{n}&\leq y_{m,n}(t)<M,\,\frac{1}{n}\leq
y_{m,n}'(t)<M,\hspace{0.4cm}t\in[0,m].
\end{split}\end{equation}
Now, we show that
\begin{equation}\label{t2.13}
\{x_{m,n}'\}_{n\in N}\text{ and }\{y_{m,n}'\}_{n\in N} \text{ are
equicontinuous on }[0,m].
\end{equation}
From \eqref{t2.11}, \eqref{t2.12} and $(\mathbf{C}_{3})$, it
follows that
\begin{align*}
-x_{m,n}''(t)\leq&p_{1}(t)h_{1}(M)k_{1}(M)(u_{1}(x_{m,n}'(t))+v_{1}(x_{m,n}'(t))),\hspace{0.4cm}t\in(0,m),\\
-y_{m,n}''(t)\leq&p_{2}(t)h_{2}(M)k_{2}(M)(u_{2}(y_{m,n}'(t))+v_{2}(y_{m,n}'(t))),\hspace{0.4cm}t\in(0,m),
\end{align*}
which implies that
\begin{align*}
\frac{-x_{m,n}''(t)}{u_{1}(x_{m,n}'(t))+v_{1}(x_{m,n}'(t))}&\leq h_{1}(M)k_{1}(M)p_{1}(t),\hspace{0.4cm}t\in(0,m),\\
\frac{-y_{m,n}''(t)}{u_{2}(y_{m,n}'(t))+v_{2}(y_{m,n}'(t))}&\leq
h_{2}(M)k_{2}(M)p_{2}(t),\hspace{0.4cm}t\in(0,m).
\end{align*}
Thus for $t_{1},t_{2}\in[0,m]$, we have
\begin{equation}\label{t2.14}\begin{split}
|J_{1}(x_{m,n}'(t_{1}))-J_{1}(x_{m,n}'(t_{2}))|&\leq h_{1}(M)k_{1}(M)\left|\int_{t_{1}}^{t_{2}}p_{1}(t)dt\right|,\\
|J_{2}(y_{m,n}'(t_{1}))-J_{2}(y_{m,n}'(t_{2}))|&\leq
h_{2}(M)k_{2}(M)\left|\int_{t_{1}}^{t_{2}}p_{2}(t)dt\right|.
\end{split}\end{equation}
In view of \eqref{t2.14}, $(\mathbf{C}_{1})$, uniform continuity
of $J_{i}^{-1}$ over $[0,J_{i}(M)]$ $(i=1,2)$ and
\begin{align*}
|x_{m,n}'(t_{1})-x_{m,n}'(t_{2})|&=|J_{1}^{-1}(J_{1}(x_{m,n}'(t_{1})))-J_{1}^{-1}(J_{1}(x_{m,n}'(t_{2})))|,\\
|y_{m,n}'(t_{1})-y_{m,n}'(t_{2})|&=|J_{2}^{-1}(J_{2}(y_{m,n}'(t_{1})))-J_{2}^{-1}(J_{2}(y_{m,n}'(t_{2})))|,
\end{align*}
we obtain \eqref{t2.13}.

\vskip 0.5em

From \eqref{t2.12} and \eqref{t2.13}, it follows that the sequences
$\{(x_{m,n}^{(j)},y_{m,n}^{(j)})\}_{n\in N}$ $(j=0,1)$ are uniformly
bounded and equicontinuous on $[0,m]$. Hence, by Theorem
\eqref{arzela}, there exist subsequence $N_{*}$ of $N$ and
$(x_{m},y_{m})\in C^{1}[0,m]\times C^{1}[0,m]$ such that for each
$j=0,1$, the sequences $(x_{m,n}^{(j)},y_{m,n}^{(j)})$ converges
uniformly to $(x_{m}^{(j)},y_{m}^{(j)})$ on $[0,m]$ as
$n\rightarrow\infty$ through $N_{*}$. From the BCs \eqref{t2.11}, we
have $x_{m}(0)=y_{m}(0)=x_{m}'(m)=y_{m}'(m)=0$. Next, we show that
$x_{m}>0$ and $y_{m}>0$ on $(0,m]$, $x_{m}'>0$ and $y_{m}'>0$ on
$[0,m)$.

\vskip 0.5em

We claim that
\begin{equation}\label{t6.1.12}
x_{m,n}(t)\geq
C_{11}^{\gamma_{1}}C_{12}^{\delta_{1}}\int_{0}^{\min\{t,1\}}\tau^{1+\gamma_{1}+\delta_{1}}p_{1}(\tau)\varphi_{M}(\tau)d\tau\equiv\Phi_{M}(t),\hspace{0.2cm}t\in[0,m],
\end{equation}
\begin{equation}\label{t6.1.13}
y_{m,n}(t)\geq
C_{11}^{\gamma_{2}}C_{12}^{\delta_{2}}\int_{0}^{\min\{t,1\}}\tau^{1+\gamma_{2}+\delta_{2}}p_{2}(\tau)\psi_{M}(\tau)d\tau\equiv\Psi_{M}(t),\hspace{0.2cm}t\in[0,m],
\end{equation}
\begin{equation}\label{t2.15}
x_{m,n}'(t)\geq
\int_{t}^{m}p_{1}(s)\varphi_{M}(s)(\Phi_{M}(s))^{\gamma_{1}}(\Psi_{M}(s))^{\delta_{1}}ds,\hspace{2cm}t\in[0,m],
\end{equation}
\begin{equation}\label{t2.16}
y_{m,n}'(t)\geq
\int_{t}^{m}p_{2}(s)\psi_{M}(s)(\Phi_{M}(s))^{\gamma_{2}}(\Psi_{M}(s))^{\delta_{2}}ds,\hspace{2cm}t\in[0,m],
\end{equation}
where
\begin{equation}\nonumber\small{\begin{split}
&C_{11}=\left(\int_{0}^{1}\tau^{1+\gamma_{1}+\delta_{1}}p_{1}(\tau)\varphi_{M}(\tau)d\tau\right)^{\frac{1-\gamma_{2}}{(1-\gamma_{1})(1-\gamma_{2})-\delta_{1}\delta_{2}}}\left(\int_{0}^{1}\tau^{1+\gamma_{2}+\delta_{2}}p_{2}(\tau)\psi_{M}(\tau)d\tau\right)^{\frac{\delta_{1}}{(1-\gamma_{1})(1-\gamma_{2})-\delta_{1}\delta_{2}}},\\
&C_{12}=\left(\int_{0}^{1}\tau^{1+\gamma_{1}+\delta_{1}}p_{1}(\tau)\varphi_{M}(\tau)d\tau\right)^{\frac{\delta_{2}}{(1-\gamma_{1})(1-\gamma_{2})-\delta_{1}\delta_{2}}}\left(\int_{0}^{1}\tau^{1+\gamma_{2}+\delta_{2}}p_{2}(\tau)\psi_{M}(\tau)d\tau\right)^{\frac{1-\gamma_{1}}{(1-\gamma_{1})(1-\gamma_{2})-\delta_{1}\delta_{2}}}.
\end{split}}\end{equation}
First we prove \eqref{t6.1.12}. Let $z(t)=x_{m,n}(t)-tx_{m,n}(1)$
for $t\in[0,1]$. Then, $z(0)=z(1)=0$, $z''(t)\leq0$ for $t\in[0,1]$.
So, $z(t)\geq0$ for $t\in[0,1]$, that is
\begin{equation}\label{t115}x_{m,n}(t)\geq tx_{m,n}(1),\hspace{0.4cm}t\in[0,1].\end{equation}
Similarly,
\begin{equation}\label{t116}y_{m,n}(t)\geq ty_{m,n}(1),\hspace{0.4cm}t\in[0,1].\end{equation}
Now, consider the following relation
\begin{equation}\label{t2.17}\begin{split}
x_{m,n}(t)&=\frac{t}{n}+\int_{0}^{t}sp_{1}(s)f_{1}(s,x_{m,n}(s),y_{m,n}(s),x_{m,n}'(s))ds\\
&+\int_{t}^{m}tp_{1}(s)f_{1}(s,x_{m,n}(s),y_{m,n}(s),x_{m,n}'(s))ds,\hspace{0.4cm}t\in[0,m].
\end{split}\end{equation}
In view of $(\mathbf{C}_{7})$, using \eqref{t115} and
\eqref{t116}, for $t\in[0,m]$, we have
\begin{equation}\label{t2115}\begin{split}
x_{m,n}(t)&\geq\int_{0}^{t}sp_{1}(s)f_{1}(s,x_{m,n}(s),y_{m,n}(s),x_{m,n}'(s))ds\\
&\geq\int_{0}^{\min\{t,1\}}sp_{1}(s)f_{1}(s,x_{m,n}(s),y_{m,n}(s),x_{m,n}'(s))ds\\
&\geq\int_{0}^{\min\{t,1\}}sp_{1}(s)(x_{m,n}(s))^{\gamma_{1}}(y_{m,n}(s))^{\delta_{1}}\varphi_{M}(s)ds\\
&\geq(x_{m,n}(1))^{\gamma_{1}}(y_{m,n}(1))^{\delta_{1}}\int_{0}^{\min\{t,1\}}s^{1+\gamma_{1}+\delta_{1}}p_{1}(s)\varphi_{M}(s)ds,
\end{split}\end{equation} which implies that
\begin{equation}\label{t2.18}
x_{m,n}(1)\geq(y_{m,n}(1))^{\frac{\delta_{1}}{1-\gamma_{1}}}\left(\int_{0}^{1}s^{1+\gamma_{1}+\delta_{1}}p_{1}(s)\varphi_{M}(s)ds\right)^{\frac{1}{1-\gamma_{1}}}.
\end{equation} Similarly,
\begin{equation}\label{t2.19}
y_{m,n}(1)\geq(x_{m,n}(1))^{\frac{\delta_{2}}{1-\gamma_{2}}}\left(\int_{0}^{1}s^{1+\gamma_{2}+\delta_{2}}p_{2}(s)\psi_{M}(s)ds\right)^{\frac{1}{1-\gamma_{2}}}.
\end{equation}
Now, using \eqref{t2.19} in \eqref{t2.18}, we have
\begin{align*}
(x_{m,n}(1))^{1-\frac{\delta_{1}\delta_{2}}{(1-\gamma_{1})(1-\gamma_{2})}}&\geq\left(\int_{0}^{1}s^{1+\gamma_{1}+\delta_{1}}p_{1}(s)\varphi_{M}(s)ds\right)^{\frac{1}{1-\gamma_{1}}}\\
&\hspace{0.4cm}\left(\int_{0}^{1}s^{1+\gamma_{2}+\delta_{2}}p_{2}(s)\psi_{M}(s)ds\right)^{\frac{\delta_{1}}{(1-\gamma_{1})(1-\gamma_{2})}}.
\end{align*}
Hence,
\begin{equation}\label{t2.20}
x_{m,n}(1)\geq C_{11}.
\end{equation}
Similarly, using \eqref{t2.18} in \eqref{t2.19}, we obtain
\begin{equation}\label{t2.21}
y_{m,n}(1)\geq C_{12}.
\end{equation}
Thus, from \eqref{t2115}, using \eqref{t2.20} and \eqref{t2.21}, we
get \eqref{t6.1.12}.

\vskip 0.5em

Similarly, we can prove \eqref{t6.1.13}.

\vskip 0.5em

Now, we prove \eqref{t2.15}. From \eqref{t2.17}, it follows that
\begin{align*}
x_{m,n}'(t)\geq\int_{t}^{m}p_{1}(s)f_{1}(s,x_{m,n}(s),y_{m,n}(s),x_{m,n}'(s))ds.
\end{align*}
Using $(\mathbf{C}_{7})$, \eqref{t6.1.12} and \eqref{t6.1.13}, we
obtain \eqref{t2.15}.

\vskip 0.5em

Similarly, we can prove \eqref{t2.16}.

\vskip 0.5em

From \eqref{t6.1.12}-\eqref{t2.16}, passing to the limit
$n\rightarrow\infty$ through $N_{*}$, we obtain
\begin{equation}\label{t2.22}\begin{split}
x_{m}(t)&\geq\Phi_{M}(t),\,\,y_{m}(t)\geq\Psi_{M}(t),\hspace{3cm}t\in[0,m],\\
x_{m}'(t)&\geq\int_{t}^{m}p_{1}(s)\varphi_{M}(s)(\Phi_{M}(s))^{\gamma_{1}}(\Psi_{M}(s))^{\delta_{1}}ds,\hspace{0.4cm}t\in[0,m],\\
y_{m}'(t)&\geq\int_{t}^{m}p_{2}(s)\psi_{M}(s)(\Phi_{M}(s))^{\gamma_{2}}(\Psi_{M}(s))^{\delta_{2}}ds,\hspace{0.4cm}t\in[0,m].
\end{split}\end{equation} Consequently, $x_{m}>0$ and $y_{m}>0$ on $(0,m]$, $x_{m}'>0$ and $y_{m}'>0$ on
$[0,m)$.

\vskip 0.5em

Moreover, $(x_{m,n},y_{m,n})$ satisfy
\begin{align*}
x_{m,n}'(t)&=x_{m,n}'(0)-\int_{0}^{t}p_{1}(s)f_{1}(s,x_{m,n}(s),y_{m,n}(s),x_{m,n}'(s))ds,\hspace{0.4cm}t\in[0,m],\\
y_{m,n}'(t)&=y_{m,n}'(0)-\int_{0}^{t}p_{2}(s)f_{2}(s,x_{m,n}(s),y_{m,n}(s),y_{m,n}'(s))ds,\hspace{0.4cm}t\in[0,m].
\end{align*}
Letting $n\rightarrow\infty$ through $N_{*}$, we obtain
\begin{align*}
x_{m}'(t)&=x_{m}'(0)-\int_{0}^{t}p_{1}(s)f_{1}(s,x_{m}(s),y_{m}(s),x_{m}'(s))ds,\hspace{0.4cm}t\in[0,m],\\
y_{m}'(t)&=y_{m}'(0)-\int_{0}^{t}p_{2}(s)f_{2}(s,x_{m}(s),y_{m}(s),y_{m}'(s))ds,\hspace{0.4cm}t\in[0,m],
\end{align*}
which imply that
\begin{equation}\label{t2.23}\begin{split}
-x_{m}''(t)&=p_{1}(t)f_{1}(t,x_{m}(t),y_{m}(t),x_{m}'(t)),\hspace{0.4cm}t\in(0,m),\\
-y_{m}''(t)&=p_{2}(t)f_{2}(t,x_{m}(t),y_{m}(t),y_{m}'(t)),\hspace{0.4cm}t\in(0,m).
\end{split}\end{equation}Hence, $(x_{m},y_{m})$ is a $C^{1}$-positive solution of the system of BVPs \eqref{t2.1}.
\end{proof}

\subsection{Existence of positive solutions on an infinite interval}
\begin{thm}\label{tht1}
Assume that $(\mathbf{C}_{1})-(\mathbf{C}_{7})$ hold. Then the system of BVPs \eqref{bc3} has a $C^{1}$-positive solution.
\end{thm}
\begin{proof}
By Theorem \ref{th11}, for each $m\in N_{0}\setminus\{0\}$, the
system of BVPs \eqref{t2.1} has a $C^{1}$-positive solution
$(x_{m},y_{m})$ defined on $[0,m]$. By applying diagonalization
argument we will show that the system of BVPs \eqref{bc3} has a $C^{1}$-positive solution. For this purpose we
define a continuous extension $(\overline{x}_{m},\overline{y}_{m})$
of $(x_{m},y_{m})$ by
\begin{equation}\label{t2.24}
\overline{x}_{m}(t)=
\begin{cases}
x_{m}(t),\,&t\in[0,m],\\
x_{m}(m),\,&t\in[m,\infty),
\end{cases}
\hspace{0.5cm}\overline{y}_{m}(t)=
\begin{cases}
y_{m}(t),\,&t\in[0,m],\\
y_{m}(m),\,&t\in[m,\infty).
\end{cases}
\end{equation}
Clearly, $\overline{x}_{m},\overline{y}_{m}\in C^{1}[0,\infty)$ and
satisfy
\begin{equation}\label{t2.25}\begin{split}
0&\leq \overline{x}_{m}(t)<M,\,0\leq
\overline{x}_{m}'(t)<M,\hspace{0.4cm}t\in[0,\infty),\\
0&\leq \overline{y}_{m}(t)<M,\,0\leq
\overline{y}_{m}'(t)<M,\hspace{0.4cm}t\in[0,\infty).
\end{split}\end{equation}
We claim that
\begin{equation}\label{t2.26}
\{\overline{x}_{m}'\}_{m\in N_{0}\setminus\{0\}}\text{ and
}\{\overline{y}_{m}'\}_{m\in N_{0}\setminus\{0\}}\text{ are
equicontinuous on }[0,1].
\end{equation}
Using \eqref{t2.23}, \eqref{t2.24}, \eqref{t2.25} and
$(\mathbf{C}_{3})$, we obtain
\begin{align*}
-\overline{x}_{m}''(t)&\leq p_{1}(t)h_{1}(M)k_{1}(M)(u_{1}(\overline{x}_{m}'(t))+v_{1}(\overline{x}_{m}'(t))),\hspace{0.4cm}t\in(0,1),\\
-\overline{y}_{m}''(t)&\leq
p_{2}(t)h_{2}(M)k_{2}(M)(u_{2}(\overline{y}_{m}'(t))+v_{2}(\overline{y}_{m}'(t))),\hspace{0.4cm}t\in(0,1),
\end{align*}
which implies that
\begin{align*}
\frac{-\overline{x}_{m}''(t)}{u_{1}(\overline{x}_{m}'(t))+v_{1}(\overline{x}_{m}'(t))}&\leq h_{1}(M)k_{1}(M)p_{1}(t),\hspace{0.4cm}t\in(0,1),\\
\frac{-\overline{y}_{m}''(t)}{u_{2}(\overline{y}_{m}'(t))+v_{2}(\overline{y}_{m}'(t))}&\leq
h_{2}(M)k_{2}(M)p_{2}(t),\hspace{0.4cm}t\in(0,1).
\end{align*}
Hence, for $t_{1},t_{2}\in[0,1]$, we have
\begin{equation}\label{t2.27}\begin{split}
|J_{1}(\overline{x}_{m}'(t_{1}))-J_{1}(\overline{x}_{m}'(t_{2}))|&\leq h_{1}(M)k_{1}(M)\left|\int_{t_{1}}^{t_{2}}p_{1}(t)dt\right|,\\
|J_{2}(\overline{y}_{m}'(t_{1}))-J_{2}(\overline{y}_{m}'(t_{2}))|&\leq
h_{2}(M)k_{2}(M)\left|\int_{t_{1}}^{t_{2}}p_{2}(t)dt\right|.
\end{split}\end{equation}
In view of \eqref{t2.27}, $(\mathbf{C}_{1})$, uniform continuity
of $J_{i}^{-1}$ over $[0,J_{i}(L)]$ $(i=1,2)$ and
\begin{align*}
|\overline{x}_{m}'(t_{1})-\overline{x}_{m}'(t_{2})|&=|J_{1}^{-1}(J_{1}(\overline{x}_{m}'(t_{1})))-J_{1}^{-1}(J_{1}(\overline{x}_{m}'(t_{2})))|,\\
|\overline{y}_{m}'(t_{1})-\overline{y}_{m}'(t_{2})|&=|J_{2}^{-1}(J_{2}(y_{m}'(t_{1})))-J_{2}^{-1}(J_{2}(y_{m}'(t_{2})))|,
\end{align*}
we establish \eqref{t2.26}.

\vskip 0.5em

From \eqref{t2.25} and \eqref{t2.26}, it follows that the sequences
$\{(\overline{x}_{m}^{(j)},\overline{y}_{m}^{(j)})\}$ $(j=0,1)$ are
uniformly bounded and equicontinuous on $[0,1]$. Hence, by Theorem
\ref{arzela}, there exist subsequence $N_{1}$ of
$N_{0}\setminus\{0\}$ and $(u_{1},v_{1})\in C^{1}[0,1]\times
C^{1}[0,1]$ such that for each $j=0,1$, the sequence
$(\overline{x}_{m}^{(j)},\overline{y}_{m}^{(j)})$ converges
uniformly to $(u_{1}^{(j)},v_{1}^{(j)})$ on $[0,1]$ as
$m\rightarrow\infty$ through $N_{1}$. Also from BCs \eqref{t2.1}, we
have $u_{1}(0)=v_{1}(0)=0$.

\vskip 0.5em

Moreover, from \eqref{t2.22} and \eqref{t2.24}, for each $m\in
N_{0}\setminus\{0\}$, we have
\begin{align*}\begin{split}
\overline{x}_{m}(t)&\geq\Phi_{M}(t),\,\,\overline{y}_{m}(t)\geq\Psi_{M}(t),\hspace{2.85cm}t\in[0,1],\\
\overline{x}_{m}'(t)&\geq\int_{t}^{1}p_{1}(s)\varphi_{M}(s)(\Phi_{M}(s))^{\gamma_{1}}(\Psi_{M}(s))^{\delta_{1}}ds,\hspace{0.4cm}t\in[0,1],\\
\overline{y}_{m}'(t)&\geq\int_{t}^{1}p_{2}(s)\psi_{M}(s)(\Phi_{M}(s))^{\gamma_{2}}(\Psi_{M}(s))^{\delta_{2}}ds,\hspace{0.4cm}t\in[0,1].
\end{split}\end{align*}
Passing to the limit $m\rightarrow\infty$ through $N_{1}$, we obtain
\begin{align*}\begin{split}
u_{1}(t)&\geq\Phi_{M}(t),\,\,v_{1}(t)\geq\Psi_{M}(t),\hspace{3cm}t\in[0,1],\\
u_{1}'(t)&\geq\int_{t}^{1}p_{1}(s)\varphi_{M}(s)(\Phi_{M}(s))^{\gamma_{1}}(\Psi_{M}(s))^{\delta_{1}}ds,\hspace{0.4cm}t\in[0,1],\\
v_{1}'(t)&\geq\int_{t}^{1}p_{2}(s)\psi_{M}(s)(\Phi_{M}(s))^{\gamma_{2}}(\Psi_{M}(s))^{\delta_{2}}ds,\hspace{0.4cm}t\in[0,1],
\end{split}\end{align*} which shows that $u_{1}>0$ and $v_{1}>0$ on $(0,1]$, $u_{1}'>0$ and $v_{1}'>0$ on
$[0,1)$.

\vskip 0.5em

By the same process as above,  we can show that
\begin{equation}\label{t2.30a}
\{\overline{x}_{m}'\}_{m\in N_{1}\setminus\{1\}}\text{ and
}\{\overline{y}_{m}'\}_{m\in N_{1}\setminus\{1\}}\text{ are
equicontinuous families on }[0,2].
\end{equation}
Further, in view of \eqref{t2.25} and \eqref{t2.30a}, it follows
that the sequences
$\{(\overline{x}_{m}^{(j)},\overline{y}_{m}^{(j)})\}$ $(j=0,1)$ are
uniformly bounded and equicontinuous on $[0,2]$. Hence, by Theorem
\ref{arzela}, there exist subsequence $N_{2}$ of
$N_{1}\setminus\{1\}$ and $(u_{2},v_{2})\in C^{1}[0,2]\times
C^{1}[0,2]$ such that for each $j=0,1$, the sequence
$(\overline{x}_{m}^{(j)},\overline{y}_{m}^{(j)})$ converges
uniformly to $(u_{2}^{(j)},v_{2}^{(j)})$ on $[0,2]$ as
$m\rightarrow\infty$ through $N_{2}$. Also from BCs \eqref{t2.1},
$u_{2}(0)=v_{2}(0)=0$. Moreover, in view of \eqref{t2.22} and
\eqref{t2.24}, for each $m\in N_{1}\setminus\{1\}$, we have
\begin{align*}\begin{split}
\overline{x}_{m}(t)&\geq\Phi_{M}(t),\,\,\overline{y}_{m}(t)\geq\Psi_{M}(t),\hspace{2.85cm}t\in[0,2],\\
\overline{x}_{m}'(t)&\geq\int_{t}^{2}p_{1}(s)\varphi_{M}(s)(\Phi_{M}(s))^{\gamma_{1}}(\Psi_{M}(s))^{\delta_{1}}ds,\hspace{0.4cm}t\in[0,2],\\
\overline{y}_{m}'(t)&\geq\int_{t}^{2}p_{2}(s)\psi_{M}(s)(\Phi_{M}(s))^{\gamma_{2}}(\Psi_{M}(s))^{\delta_{2}}ds,\hspace{0.4cm}t\in[0,2].
\end{split}\end{align*}
Now, the $\lim_{m\rightarrow\infty}$ through $N_{2}$ leads to
\begin{align*}\begin{split}
u_{2}(t)&\geq\Phi_{M}(t),\,\,v_{2}(t)\geq\Psi_{M}(t),\hspace{3cm}t\in[0,2],\\
u_{2}'(t)&\geq\int_{t}^{2}p_{1}(s)\varphi_{M}(s)(\Phi_{M}(s))^{\gamma_{1}}(\Psi_{M}(s))^{\delta_{1}}ds,\hspace{0.4cm}t\in[0,2],\\
v_{2}'(t)&\geq\int_{t}^{2}p_{2}(s)\psi_{M}(s)(\Phi_{M}(s))^{\gamma_{2}}(\Psi_{M}(s))^{\delta_{2}}ds,\hspace{0.4cm}t\in[0,2],
\end{split}\end{align*} which shows that $u_{2}>0$ and $v_{2}>0$ on $(0,2]$, $u_{2}'>0$ and $v_{2}'>0$ on
$[0,2)$. Note that, $u_{2}=u_{1}$ and $v_{2}=v_{1}$ on $[0,1]$ as
$N_{2}\subseteq N_{1}$.

\vskip 0.5em

In general, for each $k\in N_{0}\setminus\{0\}$, there exists a
subsequence $N_{k}$ of $N_{k-1}\setminus\{k-1\}$ and
$(u_{k},v_{k})\in C^{1}[0,k]\times C^{1}[0,k]$ such that
$(\overline{x}_{m}^{(j)},\overline{y}_{m}^{(j)})$ converges
uniformly to $(u_{k}^{(j)},v_{k}^{(j)})$ $(j=0,1)$ on $[0,k]$, as
$m\rightarrow\infty$ through $N_{k}$. Also, $u_{k}(0)=v_{k}(0)=0$,
$u_{k}=u_{k-1}$ and $v_{k}=v_{k-1}$ on $[0,k-1]$ as $N_{k}\subseteq
N_{k-1}$. Moreover,
\begin{align*}\begin{split}
u_{k}(t)&\geq\Phi_{M}(t),\,\,v_{k}(t)\geq\Psi_{M}(t),\hspace{3cm}t\in[0,k],\\
u_{k}'(t)&\geq\int_{t}^{k}p_{1}(s)\varphi_{M}(s)(\Phi_{M}(s))^{\gamma_{1}}(\Psi_{M}(s))^{\delta_{1}}ds,\hspace{0.4cm}t\in[0,k],\\
v_{k}'(t)&\geq\int_{t}^{k}p_{2}(s)\psi_{M}(s)(\Phi_{M}(s))^{\gamma_{1}}(\Psi_{M}(s))^{\delta_{1}}ds,\hspace{0.4cm}t\in[0,k],
\end{split}\end{align*} which shows that $u_{k}>0$ and $v_{k}>0$ on $(0,k]$, $u_{k}'>0$ and $v_{k}'>0$ on
$[0,k)$.

\vskip 0.5em

Define functions $x,y:\R^{+}\rightarrow\R^{+}$ as:

\vskip 0.5em

For fixed $\tau\in\R_{0}^{+}$ and $k\in N_{0}\setminus\{0\}$ with
$\tau\leq k$, $x(\tau)=u_{k}(\tau)$ and $y(\tau)=v_{k}(\tau)$. Then,
$x$ and $y$ are well defined as, $x(t)=u_{k}(t)>0$ and
$y(t)=v_{k}(t)>0$ for $t\in(0,k]$. We can do this for each
$\tau\in\R_{0}^{+}$. Thus, $(x,y)\in C^{1}(\R^{+})\times
C^{1}(\R^{+})$ with $x>0$ and $y>0$ on $\R_{0}^{+}$, $x'>0$ and
$y'>0$ on $\R^{+}$.

\vskip 0.5em

Now, we show that $(x,y)$ is a solution of system of BVPs \eqref{bc3}. Choose a fixed $\tau\in\R^{+}$ and $k\in
N_{0}\setminus\{0\}$ such that $k\geq\tau$. Then,
$(\overline{x}_{m}(\tau),\overline{y}_{m}(\tau))$ where $m\in
N_{k}$, satisfy
\begin{align*}
\overline{x}_{m}'(\tau)&=\overline{x}_{m}'(0)-\int_{0}^{\tau}p_{1}(s)f_{1}(s,\overline{x}_{m}(s),\overline{y}_{m}(s),\overline{x}_{m}'(s))ds,\\
\overline{y}_{m}'(\tau)&=\overline{y}_{m}'(0)-\int_{0}^{\tau}p_{2}(s)f_{2}(s,\overline{x}_{m}(s),\overline{y}_{m}(s),\overline{y}_{m}'(s))ds.
\end{align*}
Passing to the limit $m\rightarrow\infty$ through $N_{k}$, we obtain
\begin{align*}
u_{k}'(\tau)&=u_{k}'(0)-\int_{0}^{\tau}p_{1}(s)f_{1}(s,u_{k}(s),v_{k}(s),u_{k}'(s))ds,\\
v_{k}'(\tau)&=v_{k}'(0)-\int_{0}^{\tau}p_{2}(s)f_{2}(s,u_{k}(s),v_{k}(s),v_{k}'(s))ds.
\end{align*}
Hence,
\begin{align*}
x'(\tau)&=x'(0)-\int_{0}^{\tau}p_{1}(s)f_{1}(s,x(s),y(s),x'(s))ds,\\
y'(\tau)&=y'(0)-\int_{0}^{\tau}p_{2}(s)f_{2}(s,x(s),y(s),y'(s))ds,
\end{align*}
which implies that
\begin{align*}
-x''(\tau)&=p_{1}(\tau)f_{1}(\tau,x(\tau),y(\tau),x'(\tau)),\\
-y''(\tau)&=p_{2}(\tau)f_{2}(\tau,x(\tau),y(\tau),y'(\tau)).
\end{align*}
We can do this for each $\tau\in\R^{+}$. Consequently,
\begin{align*}
-x''(t)&=p_{1}(t)f_{1}(t,x(t),y(t),x'(t)),\hspace{0.4cm}t\in\R_{0}^{+},\\
-y''(t)&=p_{2}(t)f_{2}(t,x(t),y(t),y'(t)),\hspace{0.4cm}t\in\R_{0}^{+}.
\end{align*} Thus, $(x,y)\in C^{2}(\R_{0}^{+})\times C^{2}(\R_{0}^{+})$,
$x(0)=y(0)=0$.

\vskip 0.5em

It remains to show that
\begin{align*}\lim_{t\rightarrow\infty}x'(t)=\lim_{t\rightarrow\infty}y'(t)=0.\end{align*}
First, we show that $\lim_{t\rightarrow\infty}x'(t)=0$. Suppose
$\lim_{t\rightarrow\infty}x'(t)=\varepsilon_{0}$, for some
$\varepsilon_{0}>0$. Then, $x'(t)\geq\varepsilon_{0}$ for all
$t\in[0,\infty)$. Choose $k\in N_{0}\setminus\{0\}$, then for $m\in
N_{k}$, in view of \eqref{t2.24}, we have
\begin{align*}x'(t)=u_{k}'(t)=\lim_{m\rightarrow\infty}\overline{x}_{m}'(t)=\lim_{m\rightarrow\infty}x_{m}'(t),\hspace{0.4cm}t\in[0,k],\end{align*}
which leads to
\begin{align*}x'(k)=\lim_{m\rightarrow\infty}x_{m}'(k).\\ \end{align*}
Thus for every $\varepsilon>0$, there exist $m^{*}\in N_{k}$ such
that $|x_{m}'(k)-x'(k)|<\varepsilon$ for all $m\geq m^{*}$. Without
loss of generality assume that $m^{*}=k$, then
$|x_{k}'(k)-x'(k)|<\varepsilon$, that is, $|x'(k)|<\varepsilon$.
Which is a contradiction whenever $\varepsilon=\varepsilon_{0}$.
Hence, $\lim_{t\rightarrow\infty}x'(t)=0$. Similarly, we can prove
$\lim_{t\rightarrow\infty}y'(t)=0$. Thus, $(x,y)$ is a
$C^{1}$-positive solution of system of BVPs \eqref{bc3}.
\end{proof}

\begin{ex}Let
\begin{align*}f_{i}(t,x,y,z)=\nu^{\alpha_{i}+1} e^{-t}(M+1-x)(M+1-y)|x|^{\gamma_{i}}|y|^{\delta_{i}}|z|^{-\alpha_{i}},\,i=1,2,\end{align*}
 where $\nu>0$, $M>0$, $\alpha_{i}>0$,
$0\leq\gamma_{i},\delta_{i}<1$, $i=1,2$.

\vskip 0.5em

Assume that $(1-\gamma_{1})(1-\gamma_{2})\neq\delta_{1}\delta_{2}$
and
\begin{align*}
\nu<\frac{M}{\sum_{i=1}^{2}(\alpha_{i}+2)(\alpha_{i}+1)^{\frac{1}{\alpha_{i}+1}}(2M+1)^{\frac{2}{\alpha_{i}+1}}M^{\frac{\gamma_{i}+\delta_{i}}{\alpha_{i}+1}}}.
\end{align*}
Taking $p_{i}(t)=e^{-t}$,
$h_{i}(x)=\nu^{\alpha_{i}+1}(M+1+x)x^{\gamma_{i}}$,
$k_{i}(y)=(M+1+y)y^{\delta_{i}}$, $u_{i}(z)=z^{-\alpha_{i}}$ and
$v_{i}(z)=0$, $i=1,2$. Choose
$\varphi_{M}(t)=\nu^{\alpha_{1}+1}M^{-\alpha_{1}}e^{-t}$ and
$\psi_{M}(t)=\nu^{\alpha_{2}+1}M^{-\alpha_{2}}e^{-t}$. Then,
$J_{i}(\mu)=\frac{\mu^{\alpha_{i}+1}}{\alpha_{i}+1}$ and
$J_{i}^{-1}(\mu)=(\alpha_{i}+1)^{\frac{1}{\alpha_{i}+1}}\mu^{\frac{1}{\alpha_{i}+1}}$,
$i=1,2$.

\vskip 0.5em

Also,
\begin{equation*}\small{\begin{split}
\frac{M}{\omega(M)}&=\frac{M}{\sum_{i=1}^{2}\int_{0}^{\infty}J_{i}^{-1}\big(h_{i}(M)k_{i}(M)\int_{t}^{\infty}p_{i}(s)ds\big)dt+\sum_{i=1}^{2}J_{i}^{-1}\big(h_{i}(M)k_{i}(M)\int_{0}^{\infty}p_{i}(s)ds\big)}\\
&=\frac{M}{\sum_{i=1}^{2}\int_{0}^{\infty}J_{i}^{-1}\big(\nu^{\alpha_{i}+1}(2M+1)^{2}M^{\gamma_{i}+\delta_{i}}e^{-t}\big)dt+\sum_{i=1}^{2}J_{i}^{-1}\big(\nu^{\alpha_{i}+1}(2M+1)^{2}M^{\gamma_{i}+\delta_{i}}\big)}\\
&=\frac{M}{\nu\sum_{i=1}^{2}(\alpha_{i}+2)(\alpha_{i}+1)^{\frac{1}{\alpha_{i}+1}}(2M+1)^{\frac{2}{\alpha_{i}+1}}M^{\frac{\gamma_{i}+\delta_{i}}{\alpha_{i}+1}}}>1.
\end{split}}\end{equation*}
Clearly, $(\mathbf{C}_{1})-(\mathbf{C}_{7})$ are satisfied.
Hence, by Theorem \ref{tht1}, the system of BVPs \eqref{bc3} has at least one $C^{1}$-positive solution.
\end{ex}


\section{Systems of BVPs on infinite intervals with more general
BCs}\label{infinite-two}

We say, $(x,y)\in(C^{1}(\R^{+})\cap
C^{2}(\R_{0}^{+}))\times(C^{1}(\R^{+})\cap C^{2}(\R_{0}^{+}))$ is a
$C^{1}$-positive solution of the system of BVPs \eqref{s1.6}, if $(x,y)$ satisfies \eqref{s1.6},
$x>0$, $y>0$, $x'>0$ and $y'>0$ on $\R^{+}$.

\vskip 0.5em

Assume that

\begin{description}
\item[$(\mathbf{C}_{8})$] there exist a constant $M>0$ such that
$\frac{M}{\omega(M)}>1$, where
$\omega(M)=\lim_{\varepsilon\rightarrow 0}\omega_{\varepsilon}(M)$,
\begin{align*}\omega_{\varepsilon}(M)=&\sum_{i=1}^{2}\int_{0}^{\infty}[J_{i}^{-1}\big(h_{i}(M)k_{i}(M)\int_{t}^{\infty}p_{i}(s)ds+
J_{i}(\varepsilon)\big)]dt\\&+\sum_{i=1}^{2}\Big(1+\frac{b_{i}}{a_{i}}\Big)J_{i}^{-1}\big(h_{i}(M)k_{i}(M)\int_{0}^{\infty}p_{i}(s)ds+J_{i}(\varepsilon)\big),\\
J_{i}(\mu)=&\int_{0}^{\mu}\frac{dz}{u_{i}(z)+v_{i}(z)},\text{ for
}\mu>0,\, i=1,2.
\end{align*}
\end{description}

\subsection{Existence of positive solutions on finite intervals} Choose
$m\in N_{0}\setminus\{0\}$, where $N_{0}:=\{0,1,\cdots\}$, and
consider the system of BVPs on finite interval
\begin{equation}\label{s2.1}\begin{split}
-&x''(t)=p_{1}(t)f_{1}(t,x(t),y(t),x'(t)),\hspace{0.4cm}t\in(0,m),\\
-&y''(t)=p_{2}(t)f_{2}(t,x(t),y(t),y'(t)),\hspace{0.4cm}t\in(0,m),\\
&a_{1}x(0)-b_{1}x'(0)=x'(m)=0,\\
&a_{2}y(0)-b_{2}y'(0)=y'(m)=0.\end{split}\end{equation} First we
show that the system of BVPs \eqref{s2.1} has a $C^{1}$-positive
solution. We say, $(x,y)\in(C^{1}[0,m]\cap
C^{2}(0,m))\times(C^{1}[0,m]\cap C^{2}(0,m))$, a $C^{1}$-positive
solution of the system of BVPs \eqref{s2.1}, if $(x,y)$ satisfies
\eqref{s2.1}, $x>0$ and $y>0$ on $[0,m]$, $x'>0$ and $y'>0$ on
$[0,m)$.
\begin{thm}\label{ths11}
Assume that $(\mathbf{C}_{1})-(\mathbf{C}_{3})$ and $(\mathbf{C}_{5})-(\mathbf{C}_{8})$ hold. Then the system of BVPs \eqref{s2.1} has a $C^{1}$-positive solution.
\end{thm}
\begin{proof}
In view of $(\mathbf{C}_{8})$, we choose $\varepsilon>0$ small
enough such that
\begin{equation}\label{s2.3}\frac{M}{\omega_{\varepsilon}(M)}>1.\end{equation}
Choose $n_{0}\in\{1,2,\cdots\}$ such that
$\frac{1}{n_{0}}<\varepsilon$. For each $n\in
N:=\{n_{0},n_{0}+1,\cdots\}$, define retractions
$\theta:\R\rightarrow[0,M]$ and $\rho:\R\rightarrow[\frac{1}{n},M]$
as
\begin{align*}\theta(x)=\max\{0,\min\{x,M\}\}\text{ and
}\rho(x)=\max\{\frac{1}{n},\min\{x,M\}\}.\end{align*} Consider the
modified system of BVPs
\begin{equation}\label{s2.5}\begin{split}
-&x''(t)=p_{1}(t)f_{1}^{*}(t,x(t),y(t),x'(t)),\hspace{0.4cm}t\in(0,m),\\
-&y''(t)=p_{2}(t)f_{2}^{*}(t,x(t),y(t),x'(t)),\hspace{0.4cm}t\in(0,m),\\
&a_{1}x(0)-b_{1}x'(0)=0,\,x'(m)=\frac{1}{n},\\
&a_{2}y(0)-b_{2}y'(0)=0,\,y'(m)=\frac{1}{n},
\end{split}\end{equation}where
$f_{1}^{*}(t,x,y,x')=f_{1}(t,\theta(x),\theta(y),\rho(x'))$ and
$f_{2}^{*}(t,x,y,y')=f_{2}(t,\theta(x),\theta(y),\rho(y'))$.
Clearly, $f_{i}^{*}\,(i=1,2)$ are continuous and bounded on
$[0,m]\times \R^{3}$. Hence, by Theorem \ref{schauder}, the modified
system of BVPs \eqref{s2.5} has a solution
$(x_{m,n},y_{m,n})\in(C^{1}[0,m]\cap
C^{2}(0,m))\times(C^{1}[0,m]\cap C^{2}(0,m))$.

\vskip 0.5em

Using \eqref{s2.5}, $(\mathbf{C}_{1})$ and $(\mathbf{C}_{6})$,
we obtain
\begin{align*}
x_{m,n}''\leq0 \text{ and }y_{m,n}''\leq 0\text{ on }\in(0,m).
\end{align*}
Integrating from $t$ to $m$ and using the BCs \eqref{s2.5}, we
obtain
\begin{equation}\label{s2.6}
x_{m,n}'(t)\geq\frac{1}{n}\text{ and
}y_{m,n}'(t)\geq\frac{1}{n}\text{ for }t\in[0,m].
\end{equation}
Integrating \eqref{s2.6} from $0$ to $t$, using the BCs \eqref{s2.5}
and \eqref{s2.6}, we have
\begin{equation}\label{s2.7}
x_{m,n}(t)\geq(t+\frac{b_{1}}{a_{1}})\frac{1}{n}\text{ and
}y_{m,n}(t)\geq(t+\frac{b_{2}}{a_{2}})\frac{1}{n}\text{ for
}t\in[0,m].
\end{equation}
From \eqref{s2.6} and \eqref{s2.7}, it follows that
\begin{align*}\|x_{m,n}\|_{{}_{7,m}}=x_{m,n}(m)\text{ and
}\|y_{m,n}\|_{{}_{7,m}}=y_{m,n}(m).
\end{align*}
Now, we show that the following hold
\begin{equation}\label{s2.8}\|x_{m,n}'\|_{{}_{7,m}}<M\text{ and }\|y_{m,n}'\|_{{}_{7,m}}<M.\end{equation}
 Suppose $x_{m,n}'(t_{1})\geq M$ for some $t_{1}\in[0,m]$. Using
\eqref{s2.5} and $(\mathbf{C}_{3})$, we have
\begin{align*}-x_{m,n}''(t)\leq p_{1}(t)h_{1}(\theta(x_{m,n}(t)))k_{1}(\theta(y_{m,n}(t)))(u_{1}(\rho(x_{m,n}'(t)))
+v_{1}(\rho(x_{m,n}'(t)))),\hspace{0.1cm}t\in(0,m),\end{align*}
which implies that
\begin{align*}\frac{-x_{m,n}''(t)}{u_{1}(\rho(x_{m,n}'(t)))+v_{1}(\rho(x_{m,n}'(t)))}\leq h_{1}(M)k_{1}(M)p_{1}(t),\hspace{0.4cm}t\in(0,m).\end{align*}
Integrating from $t_{1}$ to $m$, using the BCs \eqref{s2.5}, we
obtain
\begin{align*}\int_{\frac{1}{n}}^{x_{m,n}'(t_{1})}\frac{dz}{u_{1}(\rho(z))+v_{1}(\rho(z))}\leq h_{1}(M)k_{1}(M)\int_{t_{1}}^{m}p_{1}(t)dt,\end{align*}
which can also be written as
\begin{align*}\int_{\frac{1}{n}}^{M}\frac{dz}{u_{1}(z)+v_{1}(z)}+\int_{M}^{x_{m,n}'(t_{1})}\frac{dz}{u_{1}(M)+v_{1}(M)}
\leq h_{1}(M)k_{1}(M)\int_{0}^{\infty}p_{1}(t)dt.\end{align*} Using
the increasing property of $J_{1}$, we obtain
\begin{align*}J_{1}(M)+\frac{x_{m,n}'(t_{1})-M}{u_{1}(M)+v_{1}(M)}\leq h_{1}(M)k_{1}(M)
\int_{0}^{\infty}p_{1}(t)dt+J_{1}(\varepsilon),\end{align*} and the
increasing property of $J_{1}^{-1}$ yields
\begin{align*}M\leq
J_{1}^{-1}(h_{1}(M)k_{1}(M)\int_{0}^{\infty}p_{1}(t)dt+J_{1}(\varepsilon))\leq\omega_{\varepsilon}(M)\end{align*}
a contradiction to \eqref{s2.3}. Hence, $\|x_{m,n}'\|_{{}_{7,m}}<M$.

\vskip 0.5em

Similarly, we can show that $\|y_{m,n}'\|_{{}_{7,m}}<M$.

\vskip 0.5em

Now, we show that
\begin{equation}\label{s2.9}\|x_{m,n}\|_{{}_{7,m}}<M\text{ and }\|y_{m,n}\|_{{}_{7,m}}<M.\end{equation}
Suppose $\|x_{m,n}\|_{{}_{7,m}}\geq M$. From \eqref{s2.5},
\eqref{s2.6}, \eqref{s2.8} and $(\mathbf{C}_{3})$, it follows that
\begin{align*}
-x_{m,n}''(t)\leq
p_{1}(t)h_{1}(\theta(x_{m,n}(t)))k_{1}(\theta(y_{m,n}(t)))(u_{1}(x_{m,n}'(t))+v_{1}(x_{m,n}'(t))),\hspace{0.4cm}t\in(0,m),
\end{align*}
which implies that
\begin{align*}\frac{-x_{m,n}''(t)}{u_{1}(x_{m,n}'(t))+v_{1}(x_{m,n}'(t))}&\leq h_{1}(M)k_{1}(M)p_{1}(t),\hspace{0.4cm}t\in(0,m).\end{align*}
Integrating from $t$ to $m$, using the BCs \eqref{s2.5}, we obtain
\begin{align*}\int_{\frac{1}{n}}^{x_{m,n}'(t)}\frac{dz}{u_{1}(z)+v_{1}(z)}\leq
h_{1}(M)k_{1}(M)\int_{t}^{m}p_{1}(s)ds,\hspace{0.4cm}t\in[0,m],\end{align*}
which can also be written as
\begin{align*}J_{1}(x_{m,n}'(t))-J_{1}(\frac{1}{n})\leq
h_{1}(M)k_{1}(M)\int_{t}^{\infty}p_{1}(s)ds,\hspace{0.4cm}t\in[0,m].\end{align*}
The increasing property of $J_{1}$ and $J_{1}^{-1}$, leads to
\begin{equation}\label{s2.10}x_{m,n}'(t)\leq J_{1}^{-1}(h_{1}(M)k_{1}(M)\int_{t}^{\infty}p_{1}(s)ds+J_{1}(\varepsilon))
,\hspace{0.4cm}t\in[0,m],\end{equation} Now, integrating from $0$ to
$m$, using the BCs \eqref{s2.5} and \eqref{s2.10}, we obtain
\begin{align*}
M\leq\|x_{m,n}\|_{{}_{7,m}}\leq
\int_{0}^{m}&[J_{1}^{-1}(h_{1}(M)k_{1}(M)\int_{t}^{\infty}p_{1}(s)ds+J_{1}(\varepsilon))]dt\\
+\frac{b_{1}}{a_{1}}&J_{1}^{-1}(h_{1}(M)k_{1}(M)\int_{0}^{\infty}p_{1}(s)ds+J_{1}(\varepsilon)),
\end{align*}
which implies that
\begin{align*}
M\leq&
\int_{0}^{\infty}[J_{1}^{-1}(h_{1}(M)k_{1}(M)\int_{t}^{\infty}p_{1}(s)ds+J_{1}(\varepsilon))]dt\\
&+\frac{b_{1}}{a_{1}}J_{1}^{-1}(h_{1}(M)k_{1}(M)\int_{0}^{\infty}p_{1}(s)ds+J_{1}(\varepsilon))\leq\omega_{\varepsilon}(M),
\end{align*}
a contradiction to \eqref{s2.3}. Therefore,
$\|x_{m,n}\|_{{}_{7,m}}<M$.

\vskip 0.5em

Similarly, we can show that $\|y_{m,n}\|_{{}_{7,m}}<M$.

\vskip 0.5em

Hence, in view of \eqref{s2.5}-\eqref{s2.9}, $(x_{m,n},y_{m,n})$ is
a solution of the following coupled system of BVPs
\begin{equation}\label{s2.11}\begin{split}
-&x''(t)=p_{1}(t)f_{1}(t,x(t),y(t),x'(t)),\hspace{0.4cm}t\in(0,m),\\
-&y''(t)=p_{2}(t)f_{2}(t,x(t),y(t),y'(t)),\hspace{0.4cm}t\in(0,m),\\
&a_{1}x(0)-b_{1}x'(0)=0,\,x'(m)=\frac{1}{n},\\
&a_{2}y(0)-b_{2}y'(0)=0,\,y'(m)=\frac{1}{n},\end{split}\end{equation}
satisfying
\begin{equation}\label{s2.12}\begin{split}
(t+\frac{b_{1}}{a_{1}})\frac{1}{n}&\leq
x_{m,n}(t)<M,\,\frac{1}{n}\leq
x_{m,n}'(t)<M,\hspace{0.4cm}t\in[0,m],\\
(t+\frac{b_{2}}{a_{2}})\frac{1}{n}&\leq
y_{m,n}(t)<M,\,\frac{1}{n}\leq
y_{m,n}'(t)<M,\hspace{0.4cm}t\in[0,m].
\end{split}\end{equation}
Now, we show that
\begin{equation}\label{s2.13}
\{x_{m,n}'\}_{n\in N}\text{ and }\{y_{m,n}'\}_{n\in N} \text{ are
equicontinuous on }[0,m].
\end{equation}
From \eqref{s2.11}, \eqref{s2.12} and $(\mathbf{C}_{3})$, it
follows that
\begin{align*}
-x_{m,n}''(t)\leq&p_{1}(t)h_{1}(M)k_{1}(M)(u_{1}(x_{m,n}'(t))+v_{1}(x_{m,n}'(t))),\hspace{0.4cm}t\in(0,m),\\
-y_{m,n}''(t)\leq&p_{2}(t)h_{2}(M)k_{2}(M)(u_{2}(y_{m,n}'(t))+v_{2}(y_{m,n}'(t))),\hspace{0.4cm}t\in(0,m),
\end{align*}
which implies that
\begin{align*}
\frac{-x_{m,n}''(t)}{u_{1}(x_{m,n}'(t))+v_{1}(x_{m,n}'(t))}&\leq h_{1}(M)k_{1}(M)p_{1}(t),\hspace{0.4cm}t\in(0,m),\\
\frac{-y_{m,n}''(t)}{u_{2}(y_{m,n}'(t))+v_{2}(y_{m,n}'(t))}&\leq
h_{2}(M)k_{2}(M)p_{2}(t),\hspace{0.4cm}t\in(0,m).
\end{align*}
Thus for $t_{1},t_{2}\in[0,m]$, we have
\begin{equation}\label{s2.14}\begin{split}
|J_{1}(x_{m,n}'(t_{1}))-J_{1}(x_{m,n}'(t_{2}))|&\leq h_{1}(M)k_{1}(M)\left|\int_{t_{1}}^{t_{2}}p_{1}(t)dt\right|,\\
|J_{2}(y_{m,n}'(t_{1}))-J_{2}(y_{m,n}'(t_{2}))|&\leq
h_{2}(M)k_{2}(M)\left|\int_{t_{1}}^{t_{2}}p_{2}(t)dt\right|.
\end{split}\end{equation}
In view of \eqref{s2.14}, $(\mathbf{C}_{1})$, uniform continuity
of $J_{i}^{-1}$ over $[0,J_{i}(M)]$ $(i=1,2)$ and
\begin{align*}
|x_{m,n}'(t_{1})-x_{m,n}'(t_{2})|&=|J_{1}^{-1}(J_{1}(x_{m,n}'(t_{1})))-J_{1}^{-1}(J_{1}(x_{m,n}'(t_{2})))|,\\
|y_{m,n}'(t_{1})-y_{m,n}'(t_{2})|&=|J_{2}^{-1}(J_{2}(y_{m,n}'(t_{1})))-J_{2}^{-1}(J_{2}(y_{m,n}'(t_{2})))|,
\end{align*}
we obtain \eqref{s2.13}.

\vskip 0.5em

From \eqref{s2.12} and \eqref{s2.13}, it follows that the sequences
$\{(x_{m,n}^{(j)},y_{m,n}^{(j)})\}_{n\in N}$ $(j=0,1)$ are uniformly
bounded and equicontinuous on $[0,m]$. Hence, by Theorem
\eqref{arzela}, there exist subsequence $N_{*}$ of $N$ and
$(x_{m},y_{m})\in C^{1}[0,m]\times C^{1}[0,m]$ such that for each
$j=0,1$ the sequence $(x_{m,n}^{(j)},y_{m,n}^{(j)})$ converges
uniformly to $(x_{m}^{(j)},y_{m}^{(j)})$ on $[0,m]$ as
$n\rightarrow\infty$ through $N_{*}$. From the BCs \eqref{s2.11}, we
have
$a_{1}x_{m}(0)-b_{1}x_{m}'(0)=a_{2}y_{m}(0)-b_{2}y_{m}'(0)=x_{m}'(m)=y_{m}'(m)=0$.
Next, we show that $x_{m}>0$ and $y_{m}>0$ on $[0,m]$, $x_{m}'>0$
and $y_{m}'>0$ on $[0,m)$.

\vskip 0.5em

We claim that
\begin{equation}\label{s2.15}
x_{m,n}'(t)\geq
C_{13}^{\gamma_{1}}C_{14}^{\delta_{1}}\int_{t}^{m}p_{1}(s)\varphi_{M}(s)ds,\hspace{0.4cm}t\in[0,m],
\end{equation}
\begin{equation}\label{s2.16}
y_{m,n}'(t)\geq
C_{13}^{\gamma_{2}}C_{14}^{\delta_{2}}\int_{t}^{m}p_{2}(s)\psi_{M}(s)ds,\hspace{0.4cm}t\in[0,m],
\end{equation}
where
\begin{align*}\begin{split}
&C_{13}=\left(\frac{b_{1}}{a_{1}}\int_{0}^{1}p_{1}(s)\varphi_{M}(s)ds\right)^{\frac{1-\gamma_{2}}{(1-\gamma_{1})(1-\gamma_{2})-\delta_{1}\delta_{2}}}\left(\frac{b_{2}}{a_{2}}\int_{0}^{1}p_{2}(s)\psi_{M}(s)ds\right)^{\frac{\delta_{1}}{(1-\gamma_{1})(1-\gamma_{2})-\delta_{1}\delta_{2}}},\\
&C_{14}=\left(\frac{b_{1}}{a_{1}}\int_{0}^{1}p_{1}(s)\varphi_{M}(s)ds\right)^{\frac{\delta_{2}}{(1-\gamma_{1})(1-\gamma_{2})-\delta_{1}\delta_{2}}}\left(\frac{b_{2}}{a_{2}}\int_{0}^{1}p_{2}(s)\psi_{M}(s)ds\right)^{\frac{1-\gamma_{1}}{(1-\gamma_{1})(1-\gamma_{2})-\delta_{1}\delta_{2}}}.\end{split}\end{align*}
To prove \eqref{s2.15}, consider the following relation
\begin{equation}\label{s2.17}\begin{split}
x_{m,n}(t)&=(t+\frac{b_{1}}{a_{1}})\frac{1}{n}+\frac{1}{a_{1}}\int_{0}^{t}(a_{1}s+b_{1})p_{1}(s)f_{1}(s,x_{m,n}(s),y_{m,n}(s),x_{m,n}'(s))ds\\
&\hspace{0.4cm}+\frac{1}{a_{1}}\int_{t}^{m}(a_{1}t+b_{1})
p_{1}(s)f_{1}(s,x_{m,n}(s),y_{m,n}(s),x_{m,n}'(s))ds,\hspace{0.4cm}t\in[0,m],
\end{split}\end{equation} which implies that
\begin{align*}
x_{m,n}(0)=\frac{b_{1}}{a_{1}}\frac{1}{n}+\frac{b_{1}}{a_{1}}\int_{0}^{m}p_{1}(s)f_{1}(s,x_{m,n}(s),y_{m,n}(s),x_{m,n}'(s))ds.
\end{align*}
Using $(\mathbf{C}_{7})$ and \eqref{s2.12}, we obtain
\begin{align*}\begin{split}
x_{m,n}(0)\geq&(x_{m,n}(0))^{\gamma_{1}}(y_{m,n}(0))^{\delta_{1}}\frac{b_{1}}{a_{1}}\int_{0}^{m}p_{1}(s)\varphi_{M}(s)ds\\
\geq&(x_{m,n}(0))^{\gamma_{1}}(y_{m,n}(0))^{\delta_{1}}\frac{b_{1}}{a_{1}}\int_{0}^{1}p_{1}(s)\varphi_{M}(s)ds,
\end{split}\end{align*}
which implies that
\begin{equation}\label{s2.18}
x_{m,n}(0)\geq(y_{m,n}(0))^{\frac{\delta_{1}}{1-\gamma_{1}}}\left(\frac{b_{1}}{a_{1}}\int_{0}^{1}p_{1}(s)\varphi_{M}(s)ds\right)^{\frac{1}{1-\gamma_{1}}}.
\end{equation} Similarly,
\begin{equation}\label{s2.19}
y_{m,n}(0)\geq(x_{m,n}(0))^{\frac{\delta_{2}}{1-\gamma_{2}}}\left(\frac{b_{2}}{a_{2}}\int_{0}^{1}p_{2}(s)\psi_{M}(s)ds\right)^{\frac{1}{1-\gamma_{2}}}.
\end{equation}
Now, using \eqref{s2.19} in \eqref{s2.18}, we have
\begin{align*}
(x_{m,n}(0))^{1-\frac{\delta_{1}\delta_{2}}{(1-\gamma_{1})(1-\gamma_{2})}}\geq\left(\frac{b_{1}}{a_{1}}\int_{0}^{1}p_{1}(s)\varphi_{M}(s)ds\right)^{\frac{1}{1-\gamma_{1}}}\left(\frac{b_{2}}{a_{2}}\int_{0}^{1}p_{2}(s)\psi_{M}(s)ds\right)^{\frac{\delta_{1}}{(1-\gamma_{1})(1-\gamma_{2})}}.
\end{align*}
Hence,
\begin{equation}\label{s2.20}
x_{m,n}(0)\geq C_{13}.
\end{equation}
Similarly, using \eqref{s2.18} in \eqref{s2.19}, we obtain
\begin{equation}\label{s2.21}
y_{m,n}(0)\geq C_{14}.
\end{equation}

Now, from \eqref{s2.17}, it follows that
\begin{align*}
x_{m,n}'(t)\geq\int_{t}^{m}p_{1}(s)f_{1}(s,y_{m,n}(s),x_{m,n}'(s))ds.
\end{align*}
Using $(\mathbf{C}_{6})$, \eqref{s2.12}, \eqref{s2.20} and
\eqref{s2.21}, we obtain \eqref{s2.15}.

\vskip 0.5em

Similarly, we can prove \eqref{s2.16}.

\vskip 0.5em

From \eqref{s2.15} and \eqref{s2.16}, passing to the limit
$n\rightarrow\infty$ through $N_{*}$, we obtain
\begin{equation}\label{s2.22}\begin{split}
x_{m}'(t)&\geq C_{13}^{\gamma_{1}}C_{14}^{\delta_{1}}\int_{t}^{m}p_{1}(s)\varphi_{M}(s)ds,\hspace{0.4cm}t\in[0,m],\\
y_{m}'(t)&\geq
C_{13}^{\gamma_{2}}C_{14}^{\delta_{2}}\int_{t}^{m}p_{2}(s)\psi_{M}(s)ds,\hspace{0.4cm}t\in[0,m].
\end{split}\end{equation} Consequently, $x_{m}'>0,\,y_{m}'>0$ on
$[0,m)$ and $x_{m}>0,\,y_{m}>0$ on $[0,m]$.

\vskip 0.5em

Moreover, $x_{m,n},y_{m,n}$ satisfy
\begin{align*}
x_{m,n}'(t)&=x_{m,n}'(0)-\int_{0}^{t}p_{1}(s)f_{1}(s,x_{m,n}(s),y_{m,n}(s),x_{m,n}'(s))ds,\hspace{0.4cm}t\in[0,m],\\
y_{m,n}'(t)&=y_{m,n}'(0)-\int_{0}^{t}p_{2}(s)f_{2}(s,x_{m,n}(s),y_{m,n}(s),y_{m,n}'(s))ds,\hspace{0.4cm}t\in[0,m].
\end{align*}
Letting $n\rightarrow\infty$ through $N_{*}$, we obtain
\begin{align*}
x_{m}'(t)&=x_{m}'(0)-\int_{0}^{t}p_{1}(s)f_{1}(s,x_{m}(s),y_{m}(s),x_{m}'(s))ds,\hspace{0.4cm}t\in[0,m],\\
y_{m}'(t)&=y_{m}'(0)-\int_{0}^{t}p_{2}(s)f_{2}(s,x_{m}(s),y_{m}(s),y_{m}'(s))ds,\hspace{0.4cm}t\in[0,m],
\end{align*}
which imply that
\begin{equation}\label{s2.23}\begin{split}
-x_{m}''(t)&=p_{1}(t)f_{1}(t,x_{m}(t),y_{m}(t),x_{m}'(t)),\hspace{0.4cm}t\in(0,m),\\
-y_{m}''(t)&=p_{2}(t)f_{2}(t,x_{m}(t),y_{m}(t),y_{m}'(t)),\hspace{0.4cm}t\in(0,m).
\end{split}\end{equation}Hence, $(x_{m},y_{m})$ is a $C^{1}$-positive solution of \eqref{2.1}.
\end{proof}

\subsection{Existence of positive solutions on an infinite interval}

\begin{thm}\label{ths1}
Assume that $(\mathbf{C}_{1})-(\mathbf{C}_{3})$ and $(\mathbf{C}_{5})-(\mathbf{C}_{8})$ hold. Then the system of BVPs \eqref{s1.6} has a $C^{1}$-positive solution.
\end{thm}
\begin{proof}
By Theorem \ref{ths11}, for each $m\in N_{0}\setminus\{0\}$, the
system of BVPs \eqref{s2.1} has a $C^{1}$-positive solution
$(x_{m},y_{m})$ defined on $[0,m]$. By applying diagonalization
argument we will show that the system of BVPs \eqref{s1.6} has a $C^{1}$-positive solution. For this we define a
continuous extension $(\overline{x}_{m},\overline{y}_{m})$ of
$(x_{m},y_{m})$ by
\begin{equation}\label{s2.24}
\overline{x}_{m}(t)=
\begin{cases}
x_{m}(t),\,&t\in[0,m],\\
x_{m}(m),\,&t\in[m,\infty),
\end{cases}
\hspace{0.5cm}\overline{y}_{m}(t)=
\begin{cases}
y_{m}(t),\,&t\in[0,m],\\
y_{m}(m),\,&t\in[m,\infty).
\end{cases}
\end{equation}
Clearly, $\overline{x}_{m},\overline{y}_{m}\in C^{1}[0,\infty)$ and
satisfy,
\begin{equation}\label{s2.25}\begin{split}
0&\leq \overline{x}_{m}(t)<M,\,0\leq
\overline{x}_{m}'(t)<M,\hspace{0.4cm}t\in[0,\infty),\\
0&\leq \overline{y}_{m}(t)<M,\,0\leq
\overline{y}_{m}'(t)<M,\hspace{0.4cm}t\in[0,\infty).
\end{split}\end{equation}
We claim that
\begin{equation}\label{s2.26}
\{\overline{x}_{m}'\}_{m\in N_{0}\setminus\{0\}}\text{ and
}\{\overline{y}_{m}'\}_{m\in N_{0}\setminus\{0\}}\text{ are
equicontinuous on }[0,1].
\end{equation}
Using \eqref{s2.23}, \eqref{s2.24}, \eqref{s2.25} and
$(\mathbf{C}_{3})$, we obtain
\begin{align*}
-\overline{x}_{m}''(t)&\leq p_{1}(t)h_{1}(M)k_{1}(M)(u_{1}(\overline{x}_{m}'(t))+v_{1}(\overline{x}_{m}'(t))),\hspace{0.4cm}t\in(0,1),\\
-\overline{y}_{m}''(t)&\leq
p_{2}(t)h_{2}(M)k_{2}(M)(u_{2}(\overline{y}_{m}'(t))+v_{2}(\overline{y}_{m}'(t))),\hspace{0.4cm}t\in(0,1),
\end{align*}
which implies that
\begin{align*}
\frac{-\overline{x}_{m}''(t)}{u_{1}(\overline{x}_{m}'(t))+v_{1}(\overline{x}_{m}'(t))}&\leq h_{1}(M)k_{1}(M)p_{1}(t),\hspace{0.4cm}t\in(0,1),\\
\frac{-\overline{y}_{m}''(t)}{u_{2}(\overline{y}_{m}'(t))+v_{2}(\overline{y}_{m}'(t))}&\leq
h_{2}(M)k_{2}(M)p_{2}(t),\hspace{0.4cm}t\in(0,1).
\end{align*}
Hence, for $t_{1},t_{2}\in[0,1]$, we have
\begin{equation}\label{s2.27}\begin{split}
|J_{1}(\overline{x}_{m}'(t_{1}))-J_{1}(\overline{x}_{m}'(t_{2}))|&\leq h_{1}(M)k_{1}(M)\left|\int_{t_{1}}^{t_{2}}p_{1}(t)dt\right|,\\
|J_{2}(\overline{y}_{m}'(t_{1}))-J_{2}(\overline{y}_{m}'(t_{2}))|&\leq
h_{2}(M)k_{2}(M)\left|\int_{t_{1}}^{t_{2}}p_{2}(t)dt\right|.
\end{split}\end{equation}
In view of \eqref{s2.27}, $(\mathbf{C}_{1})$, uniform continuity
of $J_{i}^{-1}$ over $[0,J_{i}(L)]$ $(i=1,2)$, and
\begin{align*}
|\overline{x}_{m}'(t_{1})-\overline{x}_{m}'(t_{2})|&=|J_{1}^{-1}(J_{1}(\overline{x}_{m}'(t_{1})))-J_{1}^{-1}(J_{1}(\overline{x}_{m}'(t_{2})))|,\\
|\overline{y}_{m}'(t_{1})-\overline{y}_{m}'(t_{2})|&=|J_{2}^{-1}(J_{2}(y_{m}'(t_{1})))-J_{2}^{-1}(J_{2}(y_{m}'(t_{2})))|,
\end{align*}
we establish \eqref{s2.26}.

\vskip 0.5em

From \eqref{s2.25} and \eqref{s2.26}, it follows that the sequences
$\{(\overline{x}_{m}^{(j)},\overline{y}_{m}^{(j)})\} \,(j=0,1)$ are
uniformly bounded and equicontinuous on $[0,1]$. Hence, by Theorem
\eqref{arzela}, there exist subsequence $N_{1}$ of
$N_{0}\setminus\{0\}$ and $(u_{1},v_{1})\in C^{1}[0,1]\times
C^{1}[0,1]$ such that for each $j=0,1$, the sequence
$(\overline{x}_{m}^{(j)},\overline{y}_{m}^{(j)})$ converges
uniformly to $(u_{1}^{(j)},v_{1}^{(j)})$ on $[0,1]$ as
$m\rightarrow\infty$ through $N_{1}$. Also from BCs \eqref{s2.1}, we
have $a_{1}u_{1}(0)-b_{1}u_{1}'(0)=a_{2}v_{1}(0)-b_{2}v_{1}'(0)=0$.

\vskip 0.5em

Moreover, from \eqref{s2.22} and \eqref{s2.24}, for each $m\in
N_{0}\setminus\{0\}$, we have
\begin{align*}\begin{split}
\overline{x}_{m}'(t)&\geq C_{13}^{\gamma_{1}}C_{14}^{\delta_{1}}\int_{t}^{1}p_{1}(s)\varphi_{M}(s)ds,\hspace{0.4cm}t\in[0,1],\\
\overline{y}_{m}'(t)&\geq
C_{13}^{\gamma_{2}}C_{14}^{\delta_{2}}\int_{t}^{1}p_{2}(s)\psi_{M}(s)ds,\hspace{0.4cm}t\in[0,1],
\end{split}\end{align*}
as limit $m\rightarrow\infty$ through $N_{1}$, we obtain
\begin{align*}\begin{split}
u_{1}'(t)&\geq C_{13}^{\gamma_{1}}C_{14}^{\delta_{1}}\int_{t}^{1}p_{1}(s)\varphi_{M}(s)ds,\hspace{0.4cm}t\in[0,1],\\
v_{1}'(t)&\geq
C_{13}^{\gamma_{2}}C_{14}^{\delta_{2}}\int_{t}^{1}p_{2}(s)\psi_{M}(s)ds,\hspace{0.4cm}t\in[0,1],
\end{split}\end{align*} which shows that $u_{1}'>0$ and $v_{1}'>0$ on
$[0,1)$, $u_{1}>0$ and $v_{1}>0$ on $[0,1]$.

\vskip 0.5em

By the same process as above,  we can show that
\begin{equation}\label{s2.30a}
\{\overline{x}_{m}'\}_{m\in N_{1}\setminus\{1\}}\text{ and
}\{\overline{y}_{m}'\}_{m\in N_{1}\setminus\{1\}}\text{ are
equicontinuous families on }[0,2].
\end{equation}
Now, in view of \eqref{s2.25} and \eqref{s2.30a}, it follows that
the sequences $\{(\overline{x}_{m}^{(j)},\overline{y}_{m}^{(j)})\}$
$(j=0,1)$ are uniformly bounded and equicontinuous on $[0,2]$.
Hence, by Theorem \eqref{arzela}, there exist subsequence $N_{2}$ of
$N_{1}\setminus\{1\}$ and $(u_{2},v_{2})\in C^{1}[0,2]\times
C^{1}[0,2]$ such that for each $j=0,1$, the sequence
$(\overline{x}_{m}^{(j)},\overline{y}_{m}^{(j)})$ converges
uniformly to $(u_{2}^{(j)},v_{2}^{(j)})$ on $[0,2]$ as
$m\rightarrow\infty$ through $N_{2}$. Also,
$a_{1}u_{2}(0)-b_{1}u_{2}'(0)=a_{2}v_{2}(0)-b_{2}v_{2}'(0)=0$.
Moreover, in view of \eqref{s2.22} and \eqref{s2.24}, for each $m\in
N_{1}\setminus\{1\}$, we have
\begin{align*}\begin{split}
\overline{x}_{m}'(t)&\geq C_{13}^{\gamma_{1}}C_{14}^{\delta_{1}}\int_{t}^{2}p_{1}(s)\varphi_{M}(s)ds,\hspace{0.4cm}t\in[0,2],\\
\overline{y}_{m}'(t)&\geq
C_{13}^{\gamma_{2}}C_{14}^{\delta_{2}}\int_{t}^{2}p_{2}(s)\psi_{M}(s)ds,\hspace{0.4cm}t\in[0,2].
\end{split}\end{align*}
Now, the $\lim_{m\rightarrow\infty}$ through $N_{2}$ leads to
\begin{align*}\begin{split}
u_{2}'(t)&\geq C_{13}^{\gamma_{1}}C_{14}^{\delta_{1}}\int_{t}^{2}p_{1}(s)\varphi_{M}(s)ds,\hspace{0.4cm}t\in[0,2],\\
v_{1}'(t)&\geq
C_{13}^{\gamma_{2}}C_{14}^{\delta_{2}}\int_{t}^{2}p_{2}(s)\psi_{M}(s)ds,\hspace{0.4cm}t\in[0,2],
\end{split}\end{align*} which shows that $u_{2}'>0$ and $v_{2}'>0$ on
$[0,2)$, $u_{2}>0$ and $v_{2}>0$ on $[0,2]$. Note that,
$u_{2}=u_{1}$ and $v_{2}=v_{1}$ on $[0,1]$ as $N_{2}\subseteq
N_{1}$.

\vskip 0.5em

In general, for each $k\in N_{0}\setminus\{0\}$, there exists a
subsequence $N_{k}$ of $N_{k-1}\setminus\{k-1\}$ and
$(u_{k},v_{k})\in C^{1}[0,k]\times C^{1}[0,k]$ such that
$(\overline{x}_{m}^{(j)},\overline{y}_{m}^{(j)})$ converges
uniformly to $(u_{k}^{(j)},v_{k}^{(j)})$ $(j=0,1)$ on $[0,k]$, as
$m\rightarrow\infty$ through $N_{k}$. Also,
$a_{1}u_{k}(0)-b_{1}u_{k}'(0)=a_{2}v_{k}(0)-b_{2}v_{k}'(0)=0$,
$u_{k}=u_{k-1}$ and $v_{k}=v_{k-1}$ on $[0,k-1]$ as $N_{k}\subseteq
N_{k-1}$. Moreover,
\begin{align*}\begin{split}
u_{k}'(t)&\geq C_{13}^{\gamma_{1}}C_{14}^{\delta_{1}}\int_{t}^{k}p_{1}(s)\varphi_{M}(s)ds,\hspace{0.4cm}t\in[0,k],\\
v_{k}'(t)&\geq
C_{13}^{\gamma_{2}}C_{14}^{\delta_{2}}\int_{t}^{k}p_{2}(s)\psi_{M}(s)ds,\hspace{0.4cm}t\in[0,k],
\end{split}\end{align*} which shows that $u_{k}'>0$ and $v_{k}'>0$ on
$[0,k)$, $u_{k}>0$ and $v_{k}>0$ on $[0,k]$.

\vskip 0.5em

Define functions $x,y:\R^{+}\rightarrow\R^{+}$ as:

\vskip 0.5em

For fixed $\tau\in\R_{0}^{+}$ and $k\in N_{0}\setminus\{0\}$ with
$\tau\leq k$, $x(\tau)=u_{k}(\tau)$ and $y(\tau)=v_{k}(\tau)$. Then,
$x$ and $y$ are well defined as, $x(t)=u_{k}(t)>0$ and
$y(t)=v_{k}(t)>0$ for $t\in[0,k]$. We can do this for each
$\tau\in\R_{0}^{+}$. Thus, $(x,y)\in C^{1}(\R^{+})\times
C^{1}(\R^{+})$ with $x>0$, $y>0$, $x'>0$ and $y'>0$ on $\R^{+}$.

\vskip 0.5em

Now, we show that $(x,y)$ is a solution of system of BVPs \eqref{s1.6}. Choose a fixed $\tau\in\R^{+}$ and $k\in
N_{0}\setminus\{0\}$ such that $k\geq\tau$. Then,
$(\overline{x}_{m}(\tau),\overline{y}_{m}(\tau))$ where $m\in
N_{k}$, satisfy
\begin{align*}
\overline{x}_{m}'(\tau)&=\overline{x}_{m}'(0)-\int_{0}^{\tau}p_{1}(s)f_{1}(s,\overline{x}_{m}(s),\overline{y}_{m}(s),\overline{x}_{m}'(s))ds,\\
\overline{y}_{m}'(\tau)&=\overline{y}_{m}'(0)-\int_{0}^{\tau}p_{2}(s)f_{2}(s,\overline{x}_{m}(s),\overline{y}_{m}(s),\overline{y}_{m}'(s))ds.
\end{align*}
Passing to the limit $m\rightarrow\infty$, we obtain
\begin{align*}
u_{k}'(\tau)&=u_{k}'(0)-\int_{0}^{\tau}p_{1}(s)f_{1}(s,u_{k}(s),v_{k}(s),u_{k}'(s))ds,\\
v_{k}'(\tau)&=v_{k}'(0)-\int_{0}^{\tau}p_{2}(s)f_{2}(s,u_{k}(s),v_{k}(s),v_{k}'(s))ds.
\end{align*}
Hence,
\begin{align*}
x'(\tau)&=x'(0)-\int_{0}^{\tau}p_{1}(s)f_{1}(s,x(s),y(s),x'(s))ds,\\
y'(\tau)&=y'(0)-\int_{0}^{\tau}p_{2}(s)f_{2}(s,x(s),y(s),y'(s))ds,
\end{align*}
which implies that
\begin{align*}
-x''(\tau)&=p_{1}(\tau)f_{1}(\tau,x(\tau),y(\tau),x'(\tau)),\\
-y''(\tau)&=p_{2}(\tau)f_{2}(\tau,x(\tau),y(\tau),y'(\tau)).
\end{align*}
We can do this for each $\tau\in\R^{+}$. Consequently,
\begin{align*}
-x''(t)&=p_{1}(t)f_{1}(t,x(t),y(t),x'(t)),\hspace{0.4cm}t\in\R_{0}^{+},\\
-y''(t)&=p_{2}(t)f_{2}(t,x(t),y(t),y'(t)),\hspace{0.4cm}t\in\R_{0}^{+}.
\end{align*} Thus, $(x,y)\in C^{2}(\R_{0}^{+})\times C^{2}(\R_{0}^{+})$,
$a_{1}x(0)-b_{1}x'(0)=a_{2}y(0)-b_{2}y'(0)=0$.

\vskip 0.5em

It remains to show that
\begin{align*}\lim_{t\rightarrow\infty}x'(t)=\lim_{t\rightarrow\infty}y'(t)=0.\end{align*}
First, we show that $\lim_{t\rightarrow\infty}x'(t)=0$. Suppose
$\lim_{t\rightarrow\infty}x'(t)=\varepsilon_{0}$, for some
$\varepsilon_{0}>0$. Then, $x'(t)\geq\varepsilon_{0}$ for all
$t\in[0,\infty)$. Choose $k\in N_{0}\setminus\{0\}$, then for $m\in
N_{k}$, in view of \eqref{s2.24}, we have
\begin{align*}x'(t)=u_{k}'(t)=\lim_{m\rightarrow\infty}\overline{x}_{m}'(t)=\lim_{m\rightarrow\infty}x_{m}'(t),\hspace{0.4cm}t\in[0,k],\end{align*}
which leads to
\begin{align*}x'(k)=\lim_{m\rightarrow\infty}x_{m}'(k).\end{align*}
Thus for every $\varepsilon>0$, there exist $m^{*}\in N_{k}$ such
that $|x_{m}'(k)-x'(k)|<\varepsilon$ for all $m\geq m^{*}$. Without
loss of generality assume that $m^{*}=k$, then
$|x_{k}'(k)-x'(k)|<\varepsilon$, that is, $|x'(k)|<\varepsilon$.
Which is a contradiction whenever $\varepsilon=\varepsilon_{0}$.
Hence, $\lim_{t\rightarrow\infty}x'(t)=0$. Similarly, we can prove
$\lim_{t\rightarrow\infty}y'(t)=0$. Thus, $(x,y)$ is a
$C^{1}$-positive solution of the system of BVPs \eqref{s1.6}.
\end{proof}

\begin{ex}Let
\begin{align*}f_{i}(t,x,y,z)=\nu^{\alpha_{i}+1} e^{-t}(M+1-x)(M+1-y)|x|^{\gamma_{i}}|y|^{\delta_{i}}|z|^{-\alpha_{i}},\,i=1,2,\end{align*}
 where $\nu>0$, $M>0$, $\alpha_{i}>0$,
$0\leq\gamma_{i},\delta_{i}<1$, $i=1,2$.

\vskip 0.5em

Assume that $(1-\gamma_{1})(1-\gamma_{2})\neq\delta_{1}\delta_{2}$
and
\begin{align*}
\nu<\frac{M}{\sum_{i=1}^{2}\Big(\frac{b_{i}}{a_{i}}+\alpha_{i}+2\Big)(\alpha_{i}+1)^{\frac{1}{\alpha_{i}+1}}(2M+1)^{\frac{2}{\alpha_{i}+1}}M^{\frac{\gamma_{i}+\delta_{i}}{\alpha_{i}+1}}}.
\end{align*}
Taking $p_{i}(t)=e^{-t}$,
$h_{i}(x)=\nu^{\alpha_{i}+1}(M+1+x)x^{\gamma_{i}}$,
$k_{i}(y)=(M+1+y)y^{\delta_{i}}$, $u_{i}(z)=z^{-\alpha_{i}}$ and
$v_{i}(z)=0$, $i=1,2$. Choose
$\varphi_{M}(t)=\nu^{\alpha_{1}+1}M^{-\alpha_{1}}e^{-t}$ and
$\psi_{M}(t)=\nu^{\alpha_{2}+1}M^{-\alpha_{2}}e^{-t}$. Then,
$J_{i}(\mu)=\frac{\mu^{\alpha_{i}+1}}{\alpha_{i}+1}$ and
$J_{i}^{-1}(\mu)=(\alpha_{i}+1)^{\frac{1}{\alpha_{i}+1}}\mu^{\frac{1}{\alpha_{i}+1}}$,
$i=1,2$.

\vskip 0.5em

Also,
\begin{equation*}\small{\begin{split}
&\frac{M}{\omega(M)}=\frac{M}{\sum_{i=1}^{2}\int_{0}^{\infty}J_{i}^{-1}(h_{i}(M)k_{i}(M)\int_{t}^{\infty}p_{i}(s)ds)dt+\sum_{i=1}^{2}(1+\frac{b_{i}}{a_{i}})J_{i}^{-1}(h_{i}(M)k_{i}(M)\int_{0}^{\infty}p_{i}(s)ds)}\\
&=\frac{M}{\sum_{i=1}^{2}\int_{0}^{\infty}J_{i}^{-1}(\nu^{\alpha_{i}+1}(2M+1)^{2}M^{\gamma_{i}+\delta_{i}}e^{-t})dt+\sum_{i=1}^{2}(1+\frac{b_{i}}{a_{i}})J_{i}^{-1}(\nu^{\alpha_{i}+1}(2M+1)^{2}M^{\gamma_{i}+\delta_{i}})}\\
&=\frac{M}{\nu\sum_{i=1}^{2}(\frac{b_{i}}{a_{i}}+\alpha_{i}+2)(\alpha_{i}+1)^{\frac{1}{\alpha_{i}+1}}(2M+1)^{\frac{2}{\alpha_{i}+1}}M^{\frac{\gamma_{i}+\delta_{i}}{\alpha_{i}+1}}}>1.
\end{split}}\end{equation*} Clearly, $(\mathbf{C}_{1})-(\mathbf{C}_{3})$ and
$(\mathbf{C}_{5})-(\mathbf{C}_{8})$ are satisfied. Hence, by
Theorem \ref{ths1}, the system of BVPs \eqref{s1.6}
has at least one $C^{1}$-positive solution.
\end{ex}

%% file: Ch5.tex
\chapter{Concluding Remarks}

In Chapter \ref{ch2}, Section \ref{existenceone}, we have established four different results (Theorem \ref{thmfirst}, Theorem \ref{thmsecond}, Theorem \ref{thmthird} and Theorem \ref{thmfourth}) for the existence of at least one positive solutions to the system of SBVPs \eqref{k203} under
the new assumption on the nonlinearities $f$ and $g$. In
Theorem \ref{thmfirst}, we provide the existence of at least one
positive solution for the system of SBVPs \eqref{k203} under the assumptions $(\mathbf{A}_{1})-(\mathbf{A}_{3})$, where
$(\mathbf{A}_{1})$ is about integrability of nonlinearities while
$(\mathbf{A}_{2})$ and $(\mathbf{A}_{3})$ are natural assumptions
satisfied by a class of singular nonlinearities. Our next result, Theorem \ref{thmsecond}, is obtained by replacing $(\mathbf{A}_{3})$ with $(\mathbf{A}_{4})$ in Theorem \ref{thmfirst}. Theorem \ref{thmthird} is obtained by replacing $(\mathbf{A}_{2})$ with $(\mathbf{A}_{5})$ in Theorem
\ref{thmfirst}. Moreover, Theorem \ref{thmfourth} can be obtained
either by replacing $(\mathbf{A}_{2})$ with $(\mathbf{A}_{5})$ in
Theorem \ref{thmsecond} or by replacing $(\mathbf{A}_{3})$ with
$(\mathbf{A}_{4})$ in Theorem \ref{thmthird}. Further in Section \ref{existencetwo}, Theorem \ref{tm232}, the existence of positive solutions to SBVPs \eqref{k204} is provided under the assumption $(\mathbf{A}_{6})-(\mathbf{A}_{8})$, where the assumption $(\mathbf{A}_{6})$ is integrability condition on nonlinearities while $(\mathbf{A}_{7})$ and $(\mathbf{A}_{8})$ are sublinear conditions. In Section \ref{sec-couple-four-point}, we discuss the four-point coupled system of SBVPs \eqref{4.0.1}. In Theorem \ref{thm-coupl-four-point}, by employing
the Guo-Krasnosel'skii fixed point theorem for a completely
continuous map on a positive cone, it is shown that the system
\eqref{4.0.1} has a positive solution under the assumptions
$(\mathbf{A}_{9})-(\mathbf{A}_{11})$, where $(\mathbf{A}_{9})$
is integrability condition while $(\mathbf{A}_{10})$ and
$(\mathbf{A}_{11})$ are sublinear conditions on nonlinearities $f$ and $g$.

\vskip 0.5em

In Chapter \ref{ch3}, Section \ref{existence-one}, we establish the
existence results for a coupled system of SBVPs \eqref{3.0.3}. In Theorem \ref{th3.2}, we prove the existence of at
least one $C^{1}$-positive solution for the system of SBVPs
\eqref{3.0.3} under the assumptions $(\mathbf{B}_{1})-(\mathbf{B}_{7})$. The assumptions $(\mathbf{B}_{1})$ and $(\mathbf{B}_{7})$ are some integrability
conditions, $(\mathbf{B}_{2})$ is necessary because, otherwise,
positive solution $(x,y)$ will not satisfy the condition $x'>0$ and
$y'>0$ on $[0,1)$, and therefore, $(x,y)\notin C^{2}(0,1)\cap
C^{2}(0,1)$, $(\mathbf{B}_{3})$ is a natural assumption when
$f(t,x,y)$ and $g(t,x,y)$ have singularity at $y=0$,
$(\mathbf{B}_{4})$ is required to bound the solution, whereas
$(\mathbf{B}_{5})$ is necessary for invertibility of the maps $I$
and $J$, and the solution is positive due to $(\mathbf{B}_{6})$.
By replacing the assumptions $(\mathbf{B}_{6})$ and
$(\mathbf{B}_{7})$ of Theorem \ref{th3.2} with
$(\mathbf{B}_{8})$ and $(\mathbf{B}_{9})$, and including $(\mathbf{B}_{10})$, we obtained the existence of at least two positive solutions for the system of SBVPs \eqref{3.0.3}, that is, Theorem
\ref{th4.2} of Section \ref{multiplicity-one}. The assumption $(\mathbf{B}_{10})$ is required for the existence of at least two
solutions. By replacing the assumption $(\mathbf{B}_{4})$ and
$(\mathbf{B}_{7})$ of Theorem \ref{th3.2} with
$(\mathbf{B}_{11})$ and $(\mathbf{B}_{12})$ we get our next
result, that is, Theorem \ref{th2.1} of Section \ref{existence-two},
which provide existence of at least one $C^{1}$-positive solutions
to the SBVPs \eqref{3.0.4}. Theorem \ref{th2.2} of
Section \ref{multiplicity-two} is obtained by replacing the assumptions $(\mathbf{B}_{4})$ and $(\mathbf{B}_{9})$ of Theorem
\ref{th4.2} with $(\mathbf{B}_{11})$ and $(\mathbf{B}_{13})$,
which is a criteria for the existence of at least two
$C^{1}$-positive solutions for the system of SBVPs \eqref{3.0.4}. Moreover in Section \ref{sec-couple-two-point}, Theorem
\ref{th5.2}, we studied the existence of $C^{1}$-positive solutions
to the system of SBVPs \eqref{4.0.2} under the assumptions $(\mathbf{B}_{1})-(\mathbf{B}_{3})$, $(\mathbf{B}_{5})$,
$(\mathbf{B}_{14})$, $(\mathbf{B}_{16})$ and $(\mathbf{B}_{17})$.
The assumption $(\mathbf{B}_{14})$ is a replacement of
$(\mathbf{B}_{11})$ in the case of two-point coupled BCs
\eqref{4.0.2}, $(\mathbf{B}_{16})$ is a generalization of
$(\mathbf{B}_{6})$ in case system of SBVPs \eqref{4.0.2} while
$(\mathbf{B}_{17})$ is just similar to $(\mathbf{B}_{7})$. In Section \ref{finite}, we develop the notion of upper and lower solutions for the system of SBVPs \eqref{1.6}. Theorem \ref{th1} guarantees the existence of $C^{1}$-positive solutions for the system \eqref{1.6} under the assumptions $(\mathbf{B}_{18})-(\mathbf{B}_{25})$, where $(\mathbf{B}_{18})$ is equivalent to $(\mathbf{B}_{1})$, $(\mathbf{B}_{19})$ is just a continuity condition on nonlinearities $f_{i}\ (i=1,2)$, $(\mathbf{B}_{20})$ and $(\mathbf{B}_{21})$ defines upper and
lower solutions, $(\mathbf{B}_{22})$ is a condition about the
concavity of solutions, $(\mathbf{B}_{23})$ is a natural assumption
when the functions $f$, $g$ are singular with respect to $x=0$ and
$y=0$, $(\mathbf{B}_{24})$ is about integrability condition and
$(\mathbf{B}_{25})$ is required to bound the derivative of solution.

\vskip 0.5em

In Chapter \ref{ch5}, Section \ref{infinite-one}, we establish the existence of
$C^{1}$-positive solutions to the coupled system of SBVPs \eqref{bc3}. Theorem \ref{tht1} offer $C^{1}$-positive solutions to the system of SBVPs \eqref{bc3} under the assumption $(\mathbf{C}_{1})-(\mathbf{C}_{7})$, where $(\mathbf{C}_{1})$ is some integrability condition,
$(\mathbf{C}_{2})$ and $(\mathbf{C}_{6})$ are weaker than $(\mathbf{B}_{2})$, $(\mathbf{C}_{3})$ is more general than $(\mathbf{B}_{3})$ when nonlinearities are sign-changing, $(\mathbf{C}_{4})$ is required to bound the solution which is much simpler than $(\mathbf{B}_{4})$, $(\mathbf{C}_{5})$ is just
$(\mathbf{B}_{5})$, and $(\mathbf{C}_{7})$ is required to prove that the solution is positive. By replacing the assumption $(\mathbf{C}_{4})$ in
Theorem \ref{tht1} with $(\mathbf{C}_{8})$ we obtain Theorem
\ref{ths1} of Section \ref{infinite-two}, which is a criteria for
the existence of at least one $C^{1}$-positive solutions to the
system of SBVPs \eqref{s1.6}.